\documentclass[11pt,a4paper,reqno]{amsart}

\usepackage{amssymb}
\usepackage{vmargin}
\usepackage{amscd}
\usepackage{stmaryrd}
\usepackage{mathrsfs}
\usepackage[all]{xy}
\usepackage{enumitem}
\usepackage{xr-hyper}
\usepackage{hyperref}
\hypersetup{colorlinks,
allcolors=black
}

\setmargins{32mm}{20mm}{14.6cm}{22cm}{1cm}{1cm}{1cm}{1cm}

\newcommand\da{\!\downarrow\!}

\newcommand\la{\leftarrow}

\newcommand\id{\mathrm{id}}

\newcommand\ten{\otimes}
\newcommand\hten{\hat{\otimes}}

\newcommand\DD{\mathrm{D}}

\renewcommand\H{\mathrm{H}}
\newcommand\z{\mathrm{Z}}

\newcommand\N{\mathbb{N}}
\newcommand\Z{\mathbb{Z}}
\newcommand\Q{\mathbb{Q}}
\newcommand\Ql{\mathbb{Q}_{\ell}}
\newcommand\Zl{\mathbb{Z}_{\ell}}
\newcommand\bFl{\mathbb{F}_{\ell}}

\newcommand\vv{\mathbb{V}}

\newcommand\bA{\mathbb{A}}

\newcommand\bD{\mathbb{D}}
\newcommand\bE{\mathbb{E}}

\newcommand\bG{\mathbb{G}}
\newcommand\bH{\mathbb{H}}

\newcommand\bL{\mathbb{L}}

\newcommand\bT{\mathbb{T}}

\newcommand\C{\mathcal{C}}

\newcommand\cC{\mathcal{C}}

\newcommand\cF{\mathcal{F}}

\newcommand\cO{\mathcal{O}}
\newcommand\cP{\mathcal{P}}

\newcommand\cW{\mathcal{W}}

\newcommand\cZ{\mathcal{Z}}

\renewcommand\O{\mathscr{O}}

\newcommand\sD{\mathscr{D}}

\newcommand\sH{\mathscr{H}}

\newcommand\sL{\mathscr{L}}

\newcommand\sO{\mathscr{O}}
\newcommand\sP{\mathscr{P}}

\newcommand\sT{\mathscr{T}}

\newcommand\fX{\mathfrak{X}}

\renewcommand\L{\Lambda}

\newcommand\m{\mathfrak{m}}

\newcommand\g{\mathfrak{g}}

\newcommand\sfB{\mathsf{B}}

\renewcommand\hom{\mathscr{H}\!\mathit{om}}
\newcommand\sHom{\mathscr{H}\!\mathit{om}}

\newcommand\cHom{\mathcal{H}\!\mathit{om}}

\newcommand\Ho{\mathrm{Ho}}

\newcommand\Alg{\mathrm{Alg}}
\newcommand\CAlg{\mathrm{CAlg}}

\newcommand\Mod{\mathrm{Mod}}

\newcommand\Hom{\mathrm{Hom}}

\newcommand\map{\mathrm{map}}

\newcommand\HHom{\underline{\mathrm{Hom}}}
\newcommand\EEnd{\underline{\mathrm{End}}}

\newcommand\Ext{\mathrm{Ext}}
\newcommand\EExt{\mathbb{E}\mathrm{xt}}
\newcommand\End{\mathrm{End}}

\newcommand\Iso{\mathrm{Iso}}

\newcommand\cone{\mathrm{cone}}

\newcommand{\llb}{\llbracket}
\newcommand{\rrb}{\rrbracket}

\newcommand\Gal{\mathrm{Gal}}

\newcommand\coker{\mathrm{coker\,}}

\newcommand\Co{\mathrm{Co}}
\newcommand\CoS{\mathrm{CoS}}

\newcommand\Ab{\mathrm{Ab}}

\newcommand\loc{\mathrm{loc}}

\DeclareMathOperator\Spec{\mathrm{Spec}\,}

\newcommand\Set{\mathrm{Set}}

\newcommand\aff{\mathrm{aff}}
\newcommand\Affd{\mathrm{Affd}}

\newcommand\Sp{\mathrm{Sp}}
\newcommand\PreSp{\mathrm{PreSp}}

\newcommand\Pol{\mathrm{Pol}}

\newcommand\Lag{\mathrm{Lag}}

\newcommand\nondeg{\mathrm{nondeg}}

\newcommand\ad{\mathrm{ad}}

\newcommand\<{\langle}
\renewcommand\>{\rangle}
\newcommand\Lim{\varprojlim}
\newcommand\LLim{\varinjlim}

\newcommand\ho{\mathrm{ho}\!}
\newcommand\into{\hookrightarrow}
\newcommand\onto{\twoheadrightarrow}

\newcommand\xra{\xrightarrow}
\newcommand\xla{\xleftarrow}

\newcommand\pr{\mathrm{pr}}

\newcommand\alg{\mathrm{alg}}

\newcommand\bt{\bullet}
\newcommand\by{\times}

\DeclareMathOperator\mc{\mathrm{MC}}
\DeclareMathOperator\mmc{\underline{\mathrm{MC}}}

\DeclareMathOperator\Gg{\mathrm{Gg}}

\newcommand\Perf{\mathrm{Perf}}

\newcommand\Symm{\mathrm{Symm}}

\newcommand\GL{\mathrm{GL}}

\newcommand\et{\acute{\mathrm{e}}\mathrm{t}}

\newcommand\an{\mathrm{an}}

\newcommand\Tot{\mathrm{Tot}\,}

\newcommand\tr{\mathrm{tr}}

\newcommand\ind{\mathrm{ind}}
\newcommand\pro{\mathrm{pro}}

\newcommand\pd{\partial}
\newcommand\dc{d^{\mathrm{c}}}

\newcommand\half{\frac{1}{2}}

\newcommand\Coh{\cC\mathit{oh}}

\newcommand\cris{\mathrm{cris}}

\newcommand\rig{\mathrm{rig}}

\newcommand\gr{\mathrm{gr}}

\newcommand\cts{\mathrm{cts}}

\newcommand\Fil{\mathrm{Fil}}

\renewcommand\alg{\mathrm{alg}}

\newcommand\dR{\mathrm{dR}}
\newcommand\DR{\mathrm{DR}}
\newcommand\Hod{\mathrm{Hod}}

\newcommand\co{\colon\thinspace}

\newcommand\oR{\mathbf{R}}

\newcommand\oL{\mathbf{L}}

\newcommand\oSpec{\mathbf{Spec}\,}

\newcommand\uleft\underleftarrow
\newcommand\uline\underline
\newcommand\uright\underrightarrow

\newcommand{\tps}{\texorpdfstring}

\newtheorem{theorem}{Theorem}[section]
\newtheorem{proposition}[theorem]{Proposition}
\newtheorem{corollary}[theorem]{Corollary}

\newtheorem{lemma}[theorem]{Lemma}
\newtheorem*{theorem*}{Theorem}
\newtheorem*{proposition*}{Proposition}
\newtheorem*{corollary*}{Corollary}
\newtheorem*{lemma*}{Lemma}
\newtheorem*{conjecture*}{Conjecture}

\theoremstyle{definition}
\newtheorem{definition}[theorem]{Definition}
\newtheorem*{definition*}{Definition}

\newtheorem*{notation*}{Notation}

\theoremstyle{remark}
\newtheorem{example}[theorem]{Example}
\newtheorem{examples}[theorem]{Examples}
\newtheorem{remark}[theorem]{Remark}
\newtheorem{remarks}[theorem]{Remarks}

\newtheorem*{example*}{Example}
\newtheorem*{examples*}{Examples}
\newtheorem*{remark*}{Remark}
\newtheorem*{remarks*}{Remarks}
\newtheorem*{exercise*}{Exercise}
\newtheorem*{property*}{Property}
\newtheorem*{properties*}{Properties}

\begin{document}

\begin{abstract}
We develop a characterisation of non-Archimedean derived analytic geometry based on dg enhancements of dagger algebras. This allows us to formulate derived analytic moduli functors for many types of pro-\'etale sheaves, and to construct shifted symplectic structures on them by transgression using arithmetic duality theorems. In order to handle  duality functors involving Tate twists, we introduce  weighted shifted symplectic structures on formal  weighted moduli stacks, with the usual moduli stacks given by taking $\bG_m$-invariants.

In particular,   this establishes the existence of shifted symplectic and Lagrangian structures on derived moduli stacks of $\ell$-adic constructible complexes on smooth  varieties via Poincar\'e duality, and on weighted derived moduli stacks of $\ell$-adic Galois representations via Tate and Poitou--Tate duality; the latter proves a conjecture of Minhyong Kim.
Unweighted shifted symplectic and Lagrangian structures are also established for $\ell$-adic Galois representations of  cyclotomic fields $K(\mu_{\ell^{\infty}})$, subject to additional constraints related to Iwasawa theory; these derived moduli stacks yield a non-abelian analogue of Selmer complexes, with $0$-shifted symplectic structures related to generalised Cassels--Tate pairings.
\end{abstract}

\title[Shifted symplectic structures on $\ell$-adic derived analytic moduli stacks]{Shifted symplectic structures on derived analytic moduli of $\ell$-adic local systems and Galois representations}
\author{J.P.Pridham}


\maketitle

\section*{Introduction}

In \cite[\S 10]{KimArithGauge}, Kim outlined an approach to interpreting Selmer groups as   Lagrangian intersections of suitable moduli spaces, and proposed various generalisations. The main purpose of this paper is to provide the necessary foundations to make  constructions of this sort precise in a derived  geometric setting. This involves two substantial new pieces of theory:  a characterisation of  non-Archimedean derived analytic geometry well-suited to the study of $\ell$-adic sheaves, and a theory of shifted symplectic structures which can handle the Tate twists featuring in arithmetic duality theorems.

Section \ref{daggersn} develops a theory of derived non-Archimedean analytic geometry built from differential graded enhancements of the dagger algebras of \cite{GrosseKloenne}, i.e. rings of overconvergent functions. These give rise to the same theory as that proposed in \cite{DStein} using rings with entire functional calculus. By Corollary \ref{globalEFCdaggercor}, this theory is equivalent to the dagger-analytic analogue of the non-Archimedean  derived analytic spaces  from \cite{PortaYuNonArch,lurieDAG5},  and agrees with that theory on spaces without boundary, despite involving much less data; it satisfies  similar representability theorems. In \S \ref{proetdaggersn},  pro-\'etale sheaves are  associated to derived dagger algebras using their natural topology, a construction which is hard to mimic in other models of analytic geometry. Section \ref{sympsn} then translates the theory of derived analytic symplectic geometry from \cite{DStein} to this dagger analytic context, where it takes a very natural form.

The principle underpinning  Section \ref{proetsn} is that given a derived algebraic moduli functor $F$ over $\Ql$ or a similar field,  and any scheme $X$, we can naturally  form a derived analytic moduli functor of $F$-valued (hyper)sheaves on the pro-\'etale site of $X$. For instance, when $F=B\GL_r$ this gives us a derived analytic moduli stack of locally free $\Ql$-sheaves of rank $r$. Moreover, if $F$ is $n$-shifted symplectic\footnote{for instance $B\GL_r$ is $2$-shifted symplectic} and $X$ is smooth of dimension $m$ over an algebraically closed field, then Poincar\'e duality endows  the derived moduli stack of suitable $F$-valued sheaves on $X$ with an $(n-2m)$-shifted symplectic structure when $X$ is proper (Examples \ref{Spex}), or an $(n+1-2m)$-shifted Lagrangian structure in general (Examples \ref{SpexLag}). For derived moduli of local systems, this gives a derived generalisation of the symplectic structures established on the non-singular locus in \cite{pappasVolSymplAdicLocSys}.

Following Mazur's analogy under which arithmetic objects can be thought of as manifolds of the same cohomological dimension, we also have shifted symplectic structures over more general bases. We would like to think of  a finite field as being like a curve, a local field like a surface, and a number field like a $3$-fold with boundary, but because of the Tate twists arising in arithmetic duality, we should not think of the latter two as being orientable. Specifically, we use the following cohomological  duality results. For a finite field $k$ and for a non-Archimedean local field $K$, there are quasi-isomorphisms
\begin{align*}
 \oR\Gamma_{\pro\et}(k,\uline{V})^* &\simeq \oR\Gamma_{\pro\et}(k,\uline{V}^*)[1] \tag{as $\Gal(k)\cong \hat{\Z}$}\\
 \oR\Gamma_{\pro\et}(K,\uline{V})^* &\simeq \oR\Gamma_{\pro\et}(K,\uline{V}^*(1))[2]\tag{local Tate}
\end{align*}
for
finite-dimensional continuous Galois $\Ql$-representations
$V$. For $K$ as above with residue characteristic prime to $\ell$, and for $F$  a number field, with $S$ a finite set of primes including all those dividing $\ell$, the second pairing above induces exact triangles
\begin{align*}
 \oR\Gamma_{\pro\et}(\cO_K,\uline{V}) \to  \oR\Gamma_{\pro\et}(K,\uline{V}) \to \oR\Gamma_{\pro\et}(\cO_K,\uline{V}^*(1))^*[-2] \tag{local Tate}\\
 \oR\Gamma_{\pro\et}(\cO_{F,S},\uline{V}) \to  \prod_{v \in S}\oR\Gamma_{\pro\et}(F_v,\uline{V}) \to \oR\Gamma_{\pro\et}(\cO_{F,S},\uline{V}^*(1))^*[-2],\tag{Poitou--Tate}
\end{align*}
where  $\cO_K$ is the ring of integers of $K$  and $\cO_{F,S}$ the localisation of $\cO_F$ at $S$.

We give two different approaches to the problem of non-orientability: attaching roots of unity or twisting with weights. The introduction of roots of unity might loosely be thought of as passing to the orientable cover, but because that cover is now infinite ($\Zl$ containing many more units than $\Z$), duality drops a dimension and becomes more subtle (\S \ref{cycdualitysn}), with non-degeneracy related to Iwasawa theory (Remark \ref{Iwasawarmk}). Introducing weights instead leads to twisted shifted symplectic structures, with no drop in dimension, and the derived moduli stacks carrying $\bG_m$-actions rather than $\Gal(k(\mu_{\ell^{\infty}})/k)$-actions. The dimensional discrepancy can then be understood as $\bG_m$ having cohomological dimension $0$ (being reductive) and $\Zl$ cohomological dimension $1$.

Applications of the first approach  include a more subtle family of shifted symplectic structures for schemes $X$ over cyclotomic fields $k(\mu_{\ell^{\infty}})$ when $k$ is a local field with finite residue field. In this case, Example \ref{Spex}.(\ref{Iwasawa2})
uses the duality theory of \S \ref{cycdualitysn} to give an $(n-1-2m)$-shifted symplectic structure on
the derived moduli stack of suitable $F$-valued sheaves on $X$, when $X$ is proper and smooth of dimension $m$ and $F$ carries an $n$-shifted symplectic structure. In particular, this gives a $1$-shifted symplectic structure on the derived moduli stack of suitable Galois representations of $k(\mu_{\ell^{\infty}})$, with unramified representations forming a shifted Lagrangian if the residue characteristic is prime to $\ell$, by   Example \ref{SpexLag}.(\ref{Iwlocalex}).

When $k$ is a number field, these duality theories also give rise to an $(n-1-2m)$-shifted isotropic structure on the derived moduli stack of  $F$-valued sheaves when $X$ is proper and smooth of dimension $m$ over $\cO_{k(\mu_{\ell^{\infty}}),S}$, as in Example \ref{SpexLag}.(\ref{IwglobalZex}).
Imposing constraints on ramification at the places in $S$ not dividing $\ell$, together with  Lagrangian conditions analogous to Greenberg's local conditions at the places in $S$ dividing $\ell$, then leads to a derived isotropic intersection  as in Example 
\ref{IwintersectGLCex}
providing an $(n-2-2m)$-shifted symplectic structure on suitable substacks. When $m=0$, this derived intersection can be thought of as a non-abelian generalisation of a Selmer complex (Remark \ref{selmer1rmk}). 

For these constructions, duality over  cyclotomic extensions in some respects  generalises  the technique of putting structures on derived moduli stacks over the algebraic closure of a finite field. In our case, the derived moduli stack over $k$ or $\cO_{k,S}$ would be recovered as the homotopy fixed points of a $\Zl$-action, rather than a $\hat{\Z}$-action. We can still regard this as the mapping torus of  Frobenius, so it might open up the possibility of extending arithmetic quantum field theory constructions such as \cite[Remark D.5.4]{BenZviSakellaridisVenkatesh} 
or  \cite{ArinkinGaitsgoryKazhdanRaskinRozenblyumYarshavsky} to mixed characteristic and number fields.

Section \ref{weightedsn}  develops the second approach,  a generalisation of the theory of shifted symplectic structures to address cases where the dualising bundle is non-trivial. The considerations introduced here work equally well in algebraic and analytic settings, applying to moduli of maps from spaces which are not Calabi--Yau but have a dualising line bundle, and   this weighted theory  can be thought of as a special case of the theory of $\sP$-shifted symplectic structures from the seminal manuscript \cite{BouazizGrojnowski}. The idea is to characterise the  moduli space as the $\bG_m$-invariant locus of a natural $\bG_m$-equivariant formal thickening, with that thickening carrying a shifted symplectic structure of non-zero weight with respect to the $\bG_m$-action. The resulting constructions  have a similar flavour to Iwasawa theory, but with $\bG_m$-actions rather than $\Zl$-actions. 
 Remark \ref{locformrmk} describes local forms for weighted shifted symplectic spaces in terms of twisted shifted cotangent bundles, and \S \ref{wrepsn} establishes weighted representability results. 

Section \ref{wproetsn} contains the remaining major applications of the paper, constructing weighted shifted symplectic (Examples \ref{wSpex}) and Lagrangian (Examples \ref{wSpexLag}) structures on a wide range of derived analytic moduli stacks of pro-\'etale sheaves. Poincar\'e duality for smooth proper schemes and local Tate duality for local fields 
tends to give rise to symplectic structures for moduli of sheaves, while Poincar\'e duality for smooth schemes and Poitou--Tate duality for number fields give rise to Lagrangian structures. 

In Section \ref{crissn}, crystalline constructions are introduced, with derived moduli of filtered $\phi$-objects giving a weighted Lagrangian in the  derived moduli stack of  Galois representations of an $\ell$-adic local field. By taking Lagrangian intersections with respect to the previous constructions, we  construct shifted symplectic structures on derived moduli stacks of structures over number fields with local constraints (\S \ref{wSpexLagcrissn}). In particular, this includes a $(-1)$-shifted symplectic structure on the weighted derived stack of $\ell$-adic Galois representations which are unramified away from $\ell$ and crystalline at $\ell$, a Selmer-type construction similar to those  envisaged by Kim. 

In Section \ref{poisssn},
we initiate the theory of shifted Poisson structures in the weighted setting. The main result is Theorem \ref{compatthmLie} and its generalisation sketched in \S \ref{Artinpoisssn}, extending the equivalence between  shifted symplectic and non-degenerate shifted Poisson structures to incorporate weights   (and even to handle $\sP$-shifted structures in Remark \ref{PshiftedPoissonrmk}).  Quantisation approaches in the weighted setting are summarised in  \S \ref{quantsn}; the only modification needed to enable them is to give the formal parameter $\hbar$ the same weight as the Poisson structure. As a consequence, the shifted symplectic examples of Section \ref{wproetsn} will all give rise to weighted shifted Poisson structures, many of which (depending on shift) will admit quantisations of various flavours.


I would like to thank the anonymous referee for helpful comments.

\tableofcontents

\smallskip\paragraph{\it Notation and terminology}

Throughout the paper, we will usually denote chain differentials by $\delta$. The graded vector space underlying a chain (resp. cochain) complex $V$ is denoted by $V_{\#}$ (resp. $V^{\#}$). When we have to work with chain and cochain structures separately, we denote shifts as subscripts and superscripts, respectively, so $(V_{[i]})_n:= V_{i+n}$ and $(V^{[i]})^n:= V^{i+n}$. 

When we need to compare chain and cochain complexes, we  silently make use of the equivalence  $u$ from chain complexes to cochain complexes given by $(uV)^i := V_{-i}$. On suspensions, this has the effect that $u(V_{[n]}) = (uV)^{[-n]}$; accordingly, we often write $V[n]:=V^{[n]}$ for cochain complexes and $V[n]:= V_{[-n]}$ for chain complexes when no confusion is likely.

Given an associative algebra $A$ in chain complexes, and  $A$-modules $M,N$ in chain complexes, we write $\HHom_A(M,N)$ for the cochain complex given by
\[
 \HHom_A(M,N)^i= \Hom_{A_{\#}}(M_{\#[i]},N_{\#}),
\]
with differential $ f\mapsto \delta_N \circ f \pm f \circ \delta_M$.

We refer to  commutative algebras in chain complexes as CDGAs (i.e. commutative differential graded algebras); these are assumed unital unless stated otherwise.

Here and elsewhere, we use the symbol $\pm$ to denote the  sign in the total complex of a double complex, or in induced constructions such as tensor powers of, and monomial operations on, chain complexes, noting that internal tensor products are total complexes of external tensor products.  The sign is determined by the property that $\pm$ takes the value $+$ when all inputs have degree $0$.   The symbol $\mp$ then denotes the opposite sign. 


\section{Affinoid dagger dg spaces}\label{daggersn}
Let $K$ be a field of characteristic $0$, 
complete with respect to a non-Archimedean valuation. The non-Archimedean hypothesis is not strictly necessary, and in particular our statements  all have complex analogues, which however  mostly rely on  different references.

We will be working with affinoid dagger spaces in the sense of \cite{GrosseKloenne}. 

\begin{definition}
 Given strictly positive real numbers $r_1,\ldots, r_n$, recall that the Tate algebra 
 $
K\<\frac{x_1}{r_1}, \ldots ,\frac{x_n}{r_n}\>
 $
is the ring of power series $\sum_{\nu\in \N_0^n} a_{\nu}x^{\nu}$ for $a_{\nu} \in K$ such that the set
\[
 \{|a_{\nu}|r_1^{\nu_1}\ldots r_n^{\nu_n}\}_{\nu}
\]
is bounded above.

This is a Banach algebra, with norm $\| \sum_{\nu} a_{\nu}x^{\nu}\|=\sup_{\nu}|a_{\nu}|r_1^{\nu_1}\ldots r_n^{\nu_n}$.
\end{definition}
The Tate algebra is thought of as analytic functions on the closed polydisc of radii $(r_1, \ldots,r_n)$.

An affinoid algebra is then  a quotient algebra $A=K\<\frac{x_1}{r_1}, \ldots ,\frac{x_n}{r_n}\>/I$ for some  $(r_1, r_2, \ldots, r_n)$  and some ideal $I$. The associated affinoid  space $\Sp(A)$ is then  the set of maximal ideals of $A$, equipped with the obvious structure sheaf on open  affinoid subdomains.
 
As in \cite{GrosseKloenne}, this idea is adapted as follows to study  overconvergent analytic functions on the closed polydisc.
\begin{definition}
 Given strictly positive real numbers $r_1,\ldots, r_n$, recall that the Monsky--Washnitzer (MW) algebra 
 $
K\<\frac{x_1}{r_1}, \ldots ,\frac{x_n}{r_n}\>^{\dagger}
 $
is the nested union
\[
 \bigcup_{\rho_i>r_i} K\<\frac{x_1}{\rho_1}, \ldots ,\frac{x_1}{\rho_n}\> = \bigcup_{\rho_i>r_i} T(\rho_1, \ldots, \rho_n)
\]
of Tate algebras.

Explicitly, elements are power series $\sum_{\nu} a_{\nu}x^{\nu}$ for $a_{\nu} \in K$, such that 
\[
 |a_{\nu}|\rho_1^{\nu_1}\ldots \rho_n^{\nu_n} \xra{|\nu| \to \infty} 0
\]
for some $\rho_i>r_i$.
\end{definition}

A dagger algebra is then defined in \cite{GrosseKloenne} to be a quotient algebra $A=K\<\frac{x_1}{r_1}, \ldots ,\frac{x_n}{r_n}\>^{\dagger}/I$ for some $r_i > 0$ and some ideal $I$. The associated affinoid dagger space $\Sp(A)$ is then  the set of maximal ideals of $A$, equipped with the obvious structure sheaf on open  affinoid subdomains.

\subsection{Definitions}

\subsubsection{Dagger dg affinoids}

The following is the natural generalisation to the dagger affinoid setting of the dg schemes of \cite{Kon,Quot}:

\begin{definition}\label{dgdaggerdef}
 Define an affinoid dagger dg space $X$ over $K$ to consist of an affinoid dagger space $X^0$ over $K$  together with an $\sO_{X^0}$-CDGA $\sO_{X,\ge 0}$ in coherent sheaves on $X^0$, with $\sO_{X,0}=\sO_{X^0}$.
 
 Explicitly, this means that we have 
 coherent sheaves $\{\sO_{X,i}\}_{i\ge 0}$, an associative graded-commutative product  $\sO_{X,i} \ten_{\sO_{X^0}} \sO_{X,j} \xra{\cdot} \sO_{X,i+j}$, and a derivation $\delta: \O_{X,i}\to \O_{X,i-1}$ with $\delta \circ \delta =0$.
 
 A morphism $f \co X \to Y$ of  affinoid dagger dg spaces consists of a morphism $f^0 \co X^0 \to Y^0$ of affinoid dagger spaces, together with a morphism $f^{\sharp} \co (f^0)^*\sO_Y\to \sO_X$ of CDGAs in coherent sheaves on $X^0$. 
 
We define a  dagger dg algebra $A$ to be a $K$-algebra of the form $\Gamma(X^0,\sO_X)$ for $X$ an affinoid dagger dg space. Equivalently, this is a $K$-CDGA $(A_{\ge 0},\delta)$ such that $A_0$ is a dagger algebra and the $A_0$-modules $A_m$ are all finite. 
 \end{definition}

 Note that the functor $X \mapsto \Gamma(X^0,\sO_X)$ gives a contravariant equivalence between affinoid dagger dg spaces and  dagger dg algebras.
 
 \begin{definition}
  Given an affinoid dagger dg space $X$, we define the underived truncation $\pi^0X \subset X^0$ to be the closed affinoid subspace defined by the ideal sheaf $\delta \sO_{X,1} \subset \sO_{X,0}$.
   \end{definition}

 \begin{remark}
  Although the corresponding construction for dg schemes is denoted by $\pi_0X$ in \cite{Quot}, we use the notation $\pi^0X$ to avoid confusion with path components of the simplicial constructions we will see later, and  because superscripts are more appropriate than subscripts to denote equalisers.
 \end{remark}

 \begin{definition}
  We say that a morphism $f \co X \to Y$ of  affinoid dagger dg spaces is a quasi-isomorphism if it induces an isomorphism $\pi^0f\co \pi^0X \to \pi^0Y$ on underived truncations together with isomorphisms $f^{-1}\sH_n(\sO_Y) \cong \sH_n(\sO_X)$ of sheaves on $X^0$.  
 \end{definition}

 The category of coherent sheaves on a dagger affinoid $\Sp(B)$ is  equivalent to the category of finite $B$-modules via  the global sections functor, by \cite[Theorem 2.18 and Proposition 3.1]{GrosseKloenne}. We may therefore rephrase Definition \ref{dgdaggerdef} as saying that an affinoid dagger dg space consists of a $K$-CDGA $A_{\ge 0}$ with  $A_0$ a dagger algebra (i.e. a quotient of some MW algebra $K\<x_1, \ldots,x_n\>^{\dagger}$ of overconvergent functions on a closed polydisc). Morphisms are then just CDGA morphisms, 
 and quasi-isomorphisms of affinoid dagger dg spaces just correspond to quasi-isomorphisms of these $K$-CDGAs. 
 

%
\begin{definition}\label{qufreedef}
  Say that a dagger dg algebra $A$ is quasi-free if $A_0$ is isomorphic to an MW algebra $K\<x_1, \ldots,x_n\>^{\dagger}$ for some $n$, and $A$ is freely generated as a graded-commutative $A_0$-algebra.
 
 Say that a morphism $f \co R \to A$ of  dagger dg algebras is quasi-free if $A_0$ is isomorphic over $R_0$  to a relative MW algebra $R_0\<x_1, \ldots,x_n\>^{\dagger}$, and if $A$ is freely generated as a  graded-commutative algebra over $A_0\ten_{R_0}R$.
 
 Here, $ R_0\<x_1, \ldots,x_n\>^{\dagger}$ is the ring of overconvergent functions on the relative polydisc $\Sp(R_0) \by  \Sp(K\<x_1, \ldots,x_n\>^{\dagger})$, given by $ K\<y_1,\ldots, y_m, x_1, \ldots,x_n\>^{\dagger}/(I)$ when $R_0= K\<y_1,\ldots, y_m\>^{\dagger}/I$.
   \end{definition}

 \begin{lemma}\label{coffactnlemma}
  Every morphism $ f \co A \to B$ of  dagger dg algebras admits a factorisation  $A \xra{p} \breve{B} \xra{r} B $ with $p$ quasi-free and $r$ a surjective quasi-isomorphism, and a factorisation  $A \xra{q} \tilde{B} \xra{s} B $ with $q$ a quasi-free quasi-isomorphism and $s$ surjective in strictly positive degrees.
 \end{lemma}
\begin{proof}
 Since $B_0$ is a dagger algebra, it admits a surjection $K\<x_1,\ldots,x_n\>^{\dagger} \onto B_0$ for some $n$. Combined with $f_0$, this gives a morphism $r_0\co  A_0\<x_1,\ldots,x_n\>^{\dagger} \onto B_0$, which is automatically also surjective, so we set $\breve{B}_0:= A_0\<x_1,\ldots,x_n\>^{\dagger}$, with the obvious map $p_0\co A_0 \to \breve{B}_0$.
 
 We thus have a morphism $A\ten_{A_0}\breve{B}_0  \to B$ of  CDGAs in coherent $\breve{B}_0$-modules, which is surjective in degree $0$. The standard  construction of a cofibration-trivial fibration factorisation of CDGAs then allows us to factorise this as  
 \[
 A\ten_{A_0}\breve{B}_0  \xra{p'} \breve{B} \xra{r} B
 \]
 for some CDGA $\breve{B}$ extending $\breve{B}_0$, such that $p'$ is quasi-free and $r$ a surjective quasi-isomorphism. 
 
 Noetherianity of $\breve{B}_0$ ensures that we can (non-functorially) choose finitely many generators in each degree for $\breve{B}$ over $A\ten_{A_0}\breve{B}_0  $, since every module arising in the inductive construction of a resolution is finite over  $\breve{B}_0$. This choice ensures that  the $\breve{B}_0$-modules $\breve{B}_m$ are finite, so correspond to coherent sheaves, meaning that $\breve{B}$ is indeed a  dagger dg algebra. The construction has ensured that $p$ and $r$ are of the desired form.

 For the other factorisation, just take the standard  construction of a trivial cofibration-fibration factorisation of CDGAs $A \to \tilde{B} \to B$, but with free dagger algebras rather than polynomials in degree $0$.
 \end{proof}

 

\subsubsection{Localised dg dagger affinoids}

Now, in common with dg schemes, a problem with affinoid dagger dg spaces $(X^0, \sO_X)$ is that the space $X^0$ has no geometric significance and tends to get in the way. For instance, for any open affinoid subspace $U \subset X^0$ containing $\pi^0X$, the map $(U,\sO_X|_U) \to  (X^0, \sO_X)$ is a quasi-isomorphism. 
As a consequence, the category of affinoid dagger dg spaces has too few morphisms, even after localising at quasi-isomorphisms. 
Accordingly, we now introduce a localised version.

\begin{definition}\label{genWashdef}
 Given non-negative real numbers $r_1,\ldots, r_n$, define the quasi-MW algebra 
 $
K\<\frac{x_1}{r_1}, \ldots ,\frac{x_n}{r_n}\>^{\dagger}
 $
to be the nested union
\[
 \bigcup_{\rho_i>r_i} K\<\frac{x_1}{\rho_1}, \ldots ,\frac{x_1}{\rho_n}\> = \bigcup_{\rho_i>r_i} T(\rho_1, \ldots, \rho_n)
\]
of Tate algebras.

Explicitly, elements are power series $\sum_{\nu} a_{\nu}x^{\nu}$ for $a_{\nu} \in K$, such that 
\[
 |a_{\nu}|\rho_1^{\nu_1}\ldots \rho_n^{\nu_n} \xra{|\nu| \to \infty} 0
\]
for some $\rho_i>r_i$.
\end{definition}
Note that we are allowing the numbers $r_i$ to be $0$. In particular, we allow the algebra $K\<x_1, \ldots , x_m, \frac{x_{m+1}}{0}, \ldots ,\frac{x_n}{0}\>^{\dagger}$, which can be thought of as $\Gamma(\bD^m, i^{-1}\sO_{\bD^n})$ for the inclusion $i \co \bD^m \into \bD^n$ of polydiscs, so we are looking at germs of overconvergent functions on  the dagger space $\bD^n$ restricted to $\bD^m$. 

Although only stated for the MW algebra (the case $r_i=1$ for all $i$), the proofs of \cite[\S 2]{guentzer} (which adapt R\"uckert's Basis Theorem and reduce to Tate algebras) apply for any non-negative set of radii, so   $K\<\frac{x_1}{r_1}, \ldots ,\frac{x_1}{r_n}\>^{\dagger}$ is a Noetherian factorial Jacobson ring, even when some of the $r_i$ are $0$. 

\begin{definition}
 Define a  quasi-dagger algebra to be a quotient algebra $A=K\<\frac{x_1}{r_1}, \ldots ,\frac{x_n}{r_n}\>^{\dagger}/I$ for some $r_i \ge 0$. Define the associated affinoid quasi-dagger space $\Sp(A)$ to be the set of maximal ideals of $A$, equipped with the obvious structure sheaf on open localised affinoid subdomains. 
 
 The category of  quasi-dagger algebras is then defined by letting morphisms be all $K$-algebra homomorphisms between   quasi-dagger algebras, and the category of affinoid quasi-dagger spaces is it opposite.
 \end{definition}
 
 In other words, affinoid quasi-dagger spaces $(X,\sO_X)$ are ringed spaces of  the form $(\bar{X},i^{-1}\sO_Y)$ for closed immersions $i \co \bar{X} \to Y$  of affinoid dagger spaces.

 \begin{definition}\label{completiondef}
  Given a  quasi-dagger algebra $A=K\<\frac{x_1}{r_1}, \ldots ,\frac{x_n}{r_n}\>^{\dagger}/I$, we define the affinoid algebra $\bar{A}$ to be the completion of $A$ with respect to the quotient semi-norm on $A$ induced by the norm 
  \[
  \|\sum a_{j_1,\ldots,j_n} x_1^{j_1}\ldots x_n^{j_n}\|_{\uline{r}}:= \sup_{j_1, \ldots,j_n} |a_{j_1,\ldots,j_n}|r_1^{j_1}\ldots r_n^{j_n}
 \]
 on $K\<\frac{x_1}{r_1}, \ldots ,\frac{x_n}{r_n}\>$.
 \end{definition}
Note that in contrast to the situation for dagger algebras in \cite[\S 1]{GrosseKloenne}, the completion map $A \to \bar{A}$ will not be injective if any of the radii $r_i$ are $0$.

\begin{proposition}\label{indaffinoidprop}
 There is a natural fully faithful functor from the category of quasi-dagger algebras to the ind-category of affinoid algebras.
 \end{proposition}
\begin{proof}
%
Since the  algebras $K\<\frac{x_1}{r_1}, \ldots ,\frac{x_n}{r_n}\>^{\dagger}$ are all Noetherian, any   quasi-dagger algebra $A$ can be expressed  as a quotient of the form $ K\<\frac{x_1}{r_1}, \ldots ,\frac{x_n}{r_n}\>^{\dagger}/(f_1, \ldots, f_m)$. The elements $f_i$ must all lie in $ K\<\frac{x_1}{\rho_1}, \ldots ,\frac{x_n}{\rho_n}\>$ for some $\uline{\rho}>\uline{r}$, giving rise to  a direct system
\[
 \{K\<\frac{x_1}{\rho_1'}, \ldots ,\frac{x_n}{\rho_n'}\>/(f_1, \ldots, f_m) \}_{\uline{\rho}\ge \uline{\rho}' >\uline{r}};
\]
this is the ind-affinoid algebra to which our functor $G$ sends $A$.

We now need to establish natural isomorphisms $\Hom_{\Alg_K}(A,B) \cong \Hom_{\ind(\Alg_K)}(GA,GB)$; this will make $G$ a fully faithful functor, and ensure that the choices in the previous paragraph did not matter up to isomorphism. Writing $A(\uline{\rho}'):=K\<\frac{x_1}{\rho'_1}, \ldots ,\frac{x_n}{\rho'_n}\>/(f_1, \ldots, f_m)$ and similarly for $B$, this means that
 we need a natural isomorphism
\[
 \Hom_{\Alg_K}(A,B) \cong\Lim_{\uline{\rho}\ge \uline{\rho}' >\uline{r}}\LLim_{\uline{\sigma}\ge \uline{\sigma}' >\uline{s}} \Hom_{\Alg_K}(A(\uline{\rho}'), B(\uline{\sigma}')).
\]

Since $A \cong \LLim_{\uline{\rho}\ge \uline{\rho}' >\uline{r}}A(\uline{\rho}')$, 
we already have
\[
 \Hom_{\Alg_K}(A,B) \cong\Lim_{\uline{\rho}\ge \uline{\rho}' >\uline{r}}\Hom_{\Alg_K}(A(\uline{\rho}'), B),
\]
so it suffices to prove that for any affinoid algebra $C$, the natural map
\[
\LLim_{\uline{\sigma}\ge \uline{\sigma}' >\uline{s}} \Hom_{\Alg_K}(C,  B(\uline{\sigma}')) \to  \Hom_{\Alg_K}(C, B)
\]
is an isomorphism.

We now adapt the proof of \cite[Lemma 1.8]{GrosseKloenne}. 
As in the proof of \cite[6.1.5/4]{BoschGuentzerRemmertNonArchanalysis}, $C$ is topologically generated by a finite set $T$ of topological nilpotents. Given a 
$K$-algebra homomorphism $\phi \co C \to B$, we can look at the induced homomorphism $\bar{\phi} \co C \to \bar{B}$, for the completion $\bar{B}$ of Definition \ref{completiondef}; this is necessarily continuous, since it is a homomorphism of affinoid algebras.
The elements $\bar{\phi}(t)$ for $t \in T$ are topologically nilpotent and hence power-bounded (in the sense that $\{t^n~:~ n \in \N\}$ is a bounded set); since $\bar{B}$ is affinoid, \cite[6.2.3/1]{BoschGuentzerRemmertNonArchanalysis} then implies that $\|\bar{\phi}(t)\|_{\uline{s}}\le 1$ for all $t \in T$, with topological nilpotence implying that for some $n$ we have $\|\bar{\phi}(t)^n\|_{\uline{s}}< 1$ for all $t \in T$. 

Since $B=\LLim_{\uline{\sigma}'>\uline{s}}$, there exists some $\uline{\sigma}'$ for which the elements  $\phi(t) \in B$ all lift to elements in $\widetilde{\phi(t)} \in B(\uline{\sigma}')$. Now,  since $\|b\|_{\uline{s}}=\lim_{\uline{\sigma}\to \uline{s}}\|b\|_{\uline{\sigma}}$ for all $b \in B(\uline{\sigma}')$, we deduce that for some $\uline{\sigma}' \ge \uline{\sigma}''> \uline{s}$, we have $\|\widetilde{\phi(t)}^i\|_{\sigma''}<1$ for all $t \in T$ and $n \le i <2n$. The images in $B(\uline{\sigma}'')$ of the elements 
$\widetilde{\phi(t)}$ are thus topologically nilpotent for all $t \in T$. In particular, they are power-bounded so must  have $\|\widetilde{\phi(t)}\|_{\sigma''}\le 1$, by \cite[6.2.3/1]{BoschGuentzerRemmertNonArchanalysis}. We have thus constructed a commutative diagram
\[
 \xymatrix{K\<T\> \ar[r]^{\tilde{\phi}} \ar@{->>}[d] & B(\uline{\sigma}'') \ar[d]\\ C \ar[r]^{\phi} & B}
\]

Since $K\<T\>$ is Noetherian, the kernel of $K\<T \> \to C$ is finitely generated, by a set $U$, say. The image of $\tilde{\phi}(U)$ in $B$ is $0$; since $U$ is finite, this means that the image must be $0$ in  $B(\uline{\sigma}''')$ for some $ \uline{\sigma}'' \ge \uline{\sigma}'''> \uline{s}$. That gives us a homomorphism $C \to B(\uline{\sigma}''')$ lifting $\phi$, thus establishing surjectivity.

To see that the map is injective, consider another choice $\widehat{\phi(t)} \in B(\uline{\sigma}')$ of lift of $\phi(t)$ for each $t$. Since these have the same image in $B$, there must be some $\uline{\sigma}' \ge \uline{\sigma}^*> \uline{s}$ for which $\widehat{\phi(t)}$ and $\widetilde{\phi(t)} $ have the same image in  $B(\uline{\sigma}^*)$; the set $T$ being finite, we can choose the same $ \uline{\sigma}^*$ for all $T$. Our element of $\LLim_{\uline{\sigma}\ge \uline{\sigma}' >\uline{s}} \Hom_{\Alg_K}(C,  B(\uline{\sigma}')) $ is thus unchanged by this choice, establishing injectivity.
\end{proof}

\begin{definition}\label{dgqdaggerdef}
 Define an affinoid quasi-dagger dg space $X$ over $K$ to consist of an affinoid quasi-dagger space $X^0$ over $K$  together with an $\sO_{X^0}$-CDGA $\sO_{X,\ge 0}$ in coherent sheaves on $X^0$, with $\sO_{X,0}=\sO_{X^0}$. We then define an  quasi-dagger dg algebra $A$ over $K$ to be a dg $K$-algebra of the form $\Gamma(X^0,\sO_X)$ for an affinoid quasi-dagger dg space $X$; this is equivalent to saying that $A_0$ is a quasi-dagger algebra and the $A_0$-modules $A_m$ are all finite.
 
 We  say that an affinoid quasi-dagger dg space $X$ is a localised affinoid dagger dg space if  
 the vanishing locus $\pi^0X$  of $\delta$ is dagger affinoid and 
  the closed immersion $i \co \pi^0X \to X^0$ 
  gives an isomorphism on the underlying sets of points. We then define a localised  dagger dg algebra $A$ over $K$ to be a dg $K$-algebra of the form $\Gamma(X^0,\sO_X)$ for a localised affinoid dagger dg space $X$.

A morphism $f \co X \to Y$ of affinoid quasi-dagger dg spaces 
consists of a morphism $f^0 \co X^0 \to Y^0$ of affinoid quasi-dagger spaces, together with a morphism $f^{\sharp} \co (f^0)^*\sO_Y\to \sO_X$ of CDGAs in coherent sheaves on $X^0$. Equivalently a morphism $A \to B$ of  quasi-dagger dg algebras 
is just a homomorphism of dg $K$-algebras.

We denote the category of localised dagger dg algebras by $ dg_+\Affd\Alg^{\loc,\dagger}_K$.
 \end{definition}
 
 \begin{definition}\label{Xlocdef}
  Given an affinoid dagger dg space $X$, we define the associated localised affinoid dagger dg space $X_{\loc}$ as follows. The affinoid quasi-dagger space $X_{\loc}^0$ consists of the topological space underlying $\pi^0X$ equipped with the sheaf $i^{-1}\sO_{X^0}$, for the closed immersion $i \co \pi^0X \to X^0$. The sheaf $\sO_{X_{\loc}}$ of CDGAs on $X^0_{\loc}$ is then given by $i^{-1}\sO_X$. If $A=\Gamma(X^0,\sO_X)$, we then write $A^{\loc}$ for the localised  dagger dg algebra $\Gamma(\pi^0X,i^{-1}\sO_X)$.
  \end{definition}
 
Note that Noetherianity ensures that every localised  affinoid dagger dg space is isomorphic to $X_{\loc}$ for some affinoid dagger dg space $X$, since the ideal defining $\pi^0X \subset X^0$ must have finitely many generators, which converge on some larger polydisc. 
  

 \begin{definition}
  We say that a morphism $f \co X \to Y$ of affinoid quasi-dagger dg spaces is a quasi-isomorphism if it induces an isomorphism $\pi^0f\co \pi^0X \to \pi^0Y$ on underived truncations together with isomorphisms $f^{-1}\sH_n(\sO_Y) \cong \sH_n(\sO_X)$ of sheaves on $X^0$.  
 \end{definition}

The following is an immediate consequence of the sheaves $\sH_n(\sO_X)$ being supported on $\pi^0X$ 
 \begin{lemma}
  For any affinoid dagger dg space $X$, the canonical map $X_{\loc} \to X$ is a quasi-isomorphism.
 \end{lemma}

 The definitions of quasi-free morphisms and strictly closed immersions now adapt to affinoid quasi-dagger dg spaces in the obvious way, and the analogue of Lemma \ref{coffactnlemma} also holds, since  quasi-dagger algebras are Noetherian. 

%

 \subsection{Derived mapping spaces and comparison with dg EFC algebras}\label{cfEFCsn}
 
 We now recall some definitions from \cite{DStein}, based on \cite{CarchediRoytenbergHomological}.
 
 \begin{definition}\label{EFCnonArchdef}
Define a  $K$-algebra $A$ with entire functional calculus (or EFC $K$-algebra for short)  to be a product-preserving set-valued  functor $\bA_K^n \mapsto A^n$  on the full subcategory of rigid analytic varieties with objects the affine spaces $\{\bA_K^n\}_{n \ge 0}$. 
\end{definition}
Thus  an EFC $K$-algebra $A$  is a commutative $K$-algebra equipped with a systematic and consistent way of evaluating expressions of the form 
\[
 \sum_{m_1, \ldots, m_n=0}^{\infty} \lambda_{m_1, \ldots, m_n} a_1^{m_1}\cdots a_n^{m_n}
\]
 in $A$ whenever the coefficients $\lambda_{m_1, \ldots, m_n} \in K$ satisfy 
 \[
 \lim_{\sum m_i \to \infty } |\lambda_{m_1, \ldots, m_n}|^{1/\sum m_i}= 0.
 \]

Examples of EFC $K$-algebras include rings of functions on rigid $K$-analytic spaces, and any $K$-algebra colimits of such rings (in particular, rings of functions on dagger $K$-analytic spaces arise in this way). By \cite[Proposition \ref{DStein-fpalgprop2}]{DStein}, taking rings of functions gives  a contravariant equivalence of categories between globally finitely presented Stein spaces over $K$ and finitely presented EFC $K$-algebras. Since the EFC monad preserves filtered colimits, this also gives a contravariant equivalence between EFC $K$-algebras and the pro-category of globally finitely presented Stein spaces.

 \begin{definition}\label{EFCDGAdef}
Define a non-negatively graded EFC-differential graded $K$-algebra (EFC-DGA for short)  to be a chain complex $A_{\bt}=(A_{\ge 0}, \delta)$ of $K$-vector spaces equipped with:
\begin{itemize}
 \item an associative graded multiplication, graded-commutative in the sense that $ab= (-1)^{\bar{a}\bar{b}}ba$ for all $a,b \in A$, where $\bar{a}$ is the parity of $a$ (i.e. the degree modulo $2$), and 
\item an enhancement of the $K$-algebra structure on $A_0$ to an EFC $K$-algebra structure,
\end{itemize}
such that $\delta$ is a graded derivation in the sense that $\delta(ab)= \delta(a)b + (-1)^{\bar{a}}a\delta(b)$ for all $a,b \in A$.
\end{definition}
 
\begin{proposition}\label{daggertoEFCprop}
 There is a natural fully faithful functor from  quasi-dagger  algebras $A$ to  EFC algebras. 
\end{proposition}
 \begin{proof}
  By Proposition \ref{indaffinoidprop}, $A$ naturally has the structure of an ind-affinoid algebra, and hence an ind-Banach algebra. Since every Banach algebra carries an entire functional calculus, this gives $A_0$ an EFC-algebra structure.
  Morphisms automatically preserve the EFC algebra structure, since morphisms of affinoid algebras are automatically continuous.
  
  In other words, Proposition \ref{indaffinoidprop} gives us a fully faithful functor from  quasi-dagger  algebras $A$ to ind-affinoid  algebras $\{A(\uline{\rho})\}_{\uline{\rho}>\uline{r}}$, and there is a natural fully faithful functor from affinoid algebras to EFC algebras, so we have a fully faithful  composite functor from  quasi-dagger  algebras to ind-EFC algebras, and hence a functor to EFC algebras by taking colimits. We therefore need to show that for such algebras $A,B$, the natural map
  \[
   \Hom_{\ind(EFC)}(\{A(\uline{\rho})\}_{\uline{\rho}>\uline{r}},\{B(\uline{\sigma})\}_{\uline{\sigma}>\uline{s}})\to \Hom_{EFC}(\LLim_{\uline{\rho}>\uline{r}}A(\uline{\rho}),\LLim_{\uline{\sigma}>\uline{s}}B(\uline{\sigma})),
  \]
is an isomorphism, where the left-hand side expands out as
\[
 \Lim_{\uline{\rho}>\uline{r}}\LLim_{\uline{\sigma}>\uline{s}}\Hom_{EFC}(A(\uline{\rho}),B(\uline{\sigma})).
\]

Since affinoid algebras are not obviously finitely presented as EFC algebras, we proceed by reducing to Stein algebras. Instead of regarding the MW algebra    
 $ K\<\frac{x_1}{r_1}, \ldots ,\frac{x_n}{r_n}\>^{\dagger}
 $
as the nested union
$
 \bigcup_{\rho_i>r_i} K\<\frac{x_1}{\rho_1}, \ldots ,\frac{x_1}{\rho_n}\>$ of Tate algebras, we can look at it as the nested union $ \bigcup_{\rho_i>r_i} O(\Delta(\rho_1,\ldots,\rho_n))$ of Stein algebras, where $O(\Delta(\rho_1,\ldots,\rho_n)):=\Lim_{\varrho_i<\rho_i} K\<\frac{x_1}{\varrho_1}, \ldots ,\frac{x_1}{\varrho_n}\>$ is the ring of analytic functions on the open polydisc $ \Delta(\rho_1,\ldots,\rho_n)$ with radii $\rho_i$. Note that for $\uline{\varrho}<\uline{\rho}$ we have natural maps 
 \[
 K\<\frac{x_1}{\rho_1}, \ldots ,\frac{x_1}{\rho_n}\>\to O(\Delta(\rho_1,\ldots,\rho_n))\to K\<\frac{x_1}{\varrho_1}, \ldots ,\frac{x_1}{\varrho_n}\>,
 \]
 so the direct systems $\{K\<\frac{x_1}{\rho_1}, \ldots ,\frac{x_1}{\rho_n}\>  \}_{\uline{\rho}>\uline{r}}$
 and $\{O(\Delta(\rho_1,\ldots,\rho_n))\}_{\uline{\rho}>\uline{r}}$ are isomorphic as ind-algebras.
 
 Setting $A(\uline{\rho}'):=A(\uline{\rho})\ten_{K\<\frac{x_1}{\rho_1}, \ldots ,\frac{x_1}{\rho_n}\>}O(\Delta(\rho_1,\ldots,\rho_n))$, we then have an ind-EFC algebra isomorphism $\{A(\uline{\rho})'\}_{\uline{\rho}>\uline{r}}\cong \{A(\uline{\rho})\}_{\uline{\rho}>\uline{r}}$. Since $A(\uline{\rho})'$ is a globally finitely presented Stein algebra, it is of finite presentation as an EFC-algebra by \cite[Proposition \ref{DStein-fpalgprop2}]{DStein}, so
 \[
  \Hom_{EFC}(A(\uline{\rho}'), \LLim_{\uline{\sigma}>\uline{s}}B(\uline{\sigma}))\cong \LLim_{\uline{\sigma}>\uline{s}}\Hom_{EFC}(A(\uline{\rho}'), B(\uline{\sigma})),
 \]
and passing to limits completes the proof.
   \end{proof}

   Applying Proposition \ref{daggertoEFCprop} in degree $0$ gives:
\begin{corollary}\label{daggertoEFCcor}
 There is a natural fully faithful functor from  quasi-dagger dg algebras  to  EFC-DGAs. 
\end{corollary}

We now set about establishing similar statements for the corresponding $\infty$-categories given by inverting quasi-isomorphisms.
The first subtlety we encounter is that for the standard model structure of \cite[Proposition \ref{DStein-stdmodelprop}]{DStein}, none of the objects with which we work is cofibrant.  
However, there is a Quillen-equivalent model structure  which resolves this, constructed as follows.

\begin{definition}\label{locdef}
 Given a morphism $ f \co A \to B$ of EFC-algebras, as a special case of \cite[Definition \ref{DStein-locdef}]{DStein} define the localisation $(A/B)^{\loc}$ of $A$ along $B$ as follows. If $A$ and $B$ are finitely presented, then $f$ takes the form 
 $O(V) \to O(U)$
 for a morphism $g \co U \to V$ of globally finitely presented Stein spaces, and we set
\[
 (A/B)^{\loc}:= \Gamma(U,g^{-1}\sO_V).
\]

For the general case, we write the morphism as a filtered colimit of morphisms $f(i) \co A(i) \to B(i)$ of finitely presented EFC-algebras,  indexed by some poset $I$, and then set $(A/B)^{\loc}:=\LLim_{i \in I} (A(i)/B(i))^{\loc} $.
\end{definition}

\begin{remark}
For a localised  dagger dg algebra  
$A$, observe that by definition we have $A_0\cong (A_0/\H_0A)^{\loc}$, identifying $A$ with the underlying EFC-DGA given by Corollary \ref{daggertoEFCcor}.
\end{remark}

The following is  \cite[Proposition \ref{DStein-locmodelprop}]{DStein} specialised to our setting; the final statements follow because localised dagger affinoid spaces are constructed as inverse limits of open Stein subspaces of affine space, regarding the closed dagger polydisc $\bD^n$ as a limit of open discs.
\begin{proposition}\label{locmodelprop}
 There is a cofibrantly generated model structure (the local model structure) on the category of those non-negatively graded EFC-DGAs $A_{\bt}$ with $A_0\cong (A_0/\H_0A)^{\loc}$, in which  weak equivalences are quasi-isomorphisms and
 fibrations are surjective in strictly positive chain degrees.  The inclusion  functor to the category of all  non-negatively graded EFC-DGAs is then a right Quillen equivalence.

 For any open immersion $U \to V$ of Stein spaces, the corresponding morphism of  Stein algebras is a cofibration in this model structure, as are transfinite compositions of such, and compositions of these with quasi-free morphisms of EFC-DGAs. In particular,  any localisation of a  quasi-free  dagger dg algebra is cofibrant in this model structure.
\end{proposition}

With respect to the model structure of Proposition \ref{locmodelprop}, Lemma \ref{coffactnlemma} ensures that we have cofibrant replacement within the subcategory of EFC-DGAs associated to  localised  dagger dg algebras. Moreover, all objects in that model structure are fibrant.

\begin{lemma}\label{locmaplemma}
 Given a small category $\C$ and a subcategory $\cW$, take an object $A \in \C$ and assume that we have an augmented cosimplicial diagram $\tilde{A}^{\bt} \to A$ in $\C$ such that 
 \begin{enumerate}
  \item the morphisms $\tilde{A}^m \to A$ all lie in $\cW$, and
  \item for any morphism $B \to C$ in $\cW$, the map of simplicial sets given in level $m$ by
  \[
   \Hom_{\C}(\tilde{A}^m,B) \to \Hom_{\C}(\tilde{A}^m,C)
  \]
is a weak equivalence.
 \end{enumerate}
 Then in the simplicial localisation $L_{\cW}(\C)$ of $\C$ at $\cW$, the simplicial set-valued functor
 \[
  \HHom_{L_{\cW}(\C)}(A,-)
 \]
is weakly equivalent to $ m\mapsto \Hom_{\C}(\tilde{A}^m,B)$.
\end{lemma}
\begin{proof}
 We can consider the model category of restricted diagrams from \cite[\S 2.3.2]{hag1}, applied to our category, so objects are functors from $\C$ to simplicial sets, and fibrant objects are objectwise fibrant functors which send morphisms in $\cW$ to weak equivalences. Writing $h_A:=\Hom_{\C}(A,-)$, in that model structure the morphisms $h_A \to h_{\widetilde{A'}^m} $ are all weak equivalences, since the maps $\widetilde{A'}^m \to A$ are in $\cW$. The map $h_A \to \ho\LLim_{m \in \Delta}h_{\widetilde{A'}^m}$ is thus a weak equivalence, but the latter is just the functor $H_{\tilde{A}}$ sending  $B$ to the simplicial set $m  \mapsto \Hom_{\C}(\widetilde{A'}^m,B)$.
 
By hypothesis,    $H_{\tilde{A}}$  sends morphisms in $\cW$ to weak equivalences, so taking objectwise fibrant replacement gives us a weakly equivalent functor $ H_{\tilde{A}}'$ which is fibrant in the model category of restricted diagrams, and   $h_A \to H_{\tilde{A}}'$ is fibrant replacement. By \cite{DKEquivsHtpyDiagrams}, as interpreted in  \cite[Theorem  2.3.5]{hag1}, this means that $H_{\tilde{A}}(B)$ is weakly equivalent to the space of maps from $A$ to $B$ in $L_{\cW}(\C)$. 
\end{proof}

\begin{proposition}\label{affdsubEFCprop}
 The functor from localised  dagger dg algebras to EFC-DGAs given by Corollary \ref{daggertoEFCcor} induces a fully faithful functor on simplicial categories after simplicial localisation at quasi-isomorphisms, as does its restriction to quasi-free localised  dagger dg algebras.
\end{proposition}
\begin{proof}
This effectively follows with the same reasoning as \cite{DKfunction}, since Lemma \ref{coffactnlemma} means that localised  dagger dg algebras come close to being a model subcategory of the model category in Proposition \ref{locmodelprop}. However, since it is not closed under finite limits, we now give the details. 

Given a dagger dg algebra $A$, repeated application of Lemma \ref{coffactnlemma}  gives us a quasi-free cosimplicial resolution $\tilde{A}^{\bt}$ in the following sense, using Reedy category techniques as in \cite[\S 5]{hovey}. Firstly, each $\tilde{A}^n$ is a quasi-free dagger dg algebra, and moreover the latching maps $(\pd^0,\pd^1) \co \tilde{A}^0\hten^{\pi}\tilde{A}^0 \to \tilde{A}^1$ etc. are quasi-free morphisms of dagger dg algebras, where $\hten^{\pi}$ is the completed projective tensor product (corresponding to product of affinoid dagger spaces). Secondly, the degeneracy maps are all quasi-isomorphisms and we have a quasi-isomorphism  $\tilde{A}^0 \to A$.

Now, any localised dagger dg algebra $A'$ is of the form $A^{\loc}$ for some dagger dg algebra $A$, in the notation of Definition \ref{Xlocdef}, with underlying EFC algebra $A' \cong  A\ten_{A_0}(A_0/\H_0A)^{\loc}$. The construction above then gives us a quasi-free cosimplicial resolution $\widetilde{A'}^{\bt}$ of $A'$ by setting $\widetilde{A'}^m:=(\widetilde{A}^m)^{\loc}$. On 
the underlying EFC algebras, this is a Reedy cofibrant cosimplicial resolution with respect to the model structure of Proposition \ref{locmodelprop}. In particular, for any EFC-DGA $B$ which is fibrant in that model structure, the space of maps $\oR\map_{EFC,DG}(A,B)$ is weakly equivalent to the simplicial set $m  \mapsto \Hom_{EFC,DG}(\widetilde{A'}^m,B) $.

%
%

Since Corollary \ref{daggertoEFCcor} gives a fully faithful functor from  quasi-dagger dg algebras to EFC-DGAs, we have $\Hom_{\Alg_K}(\widetilde{A'}^m,C)\cong \Hom_{DG,EFC}(\widetilde{A'}^m,C)$ for all quasi-dagger dg algebras $C$. If  $C$ is a 
localised dagger dg algebra, its underlying EFC-DGA is  
fibrant in the model structure of Proposition \ref{locmodelprop}. Since $A'$ is a cofibrant cosimplicial resolution, we deduce that the functor $H_{\tilde{A}}(-)$ sending  $C$ to the simplicial set $m  \mapsto \Hom_{DG\Alg_K}(\widetilde{A'}^m,C)$ is a model for $\oR\map_{EFC,DG}(A,-)$, and in particular sends quasi-isomorphisms in $ dg_+\Affd\Alg^{\loc,\dagger}_K$  to weak equivalences.
We may therefore appeal to Lemma \ref{locmaplemma}, from which it follows  that  $H_{\tilde{A}}(-)$  is also a model for the mapping space $\oR\map_{dg_+\Affd\Alg^{\loc,\dagger}_K}(A,-)$.

 Finally, since the objects of $\tilde{A}$ are all quasi-free, exactly the same reasoning applies to the category of restricted diagrams from quasi-free localised dagger dg algebras to simplicial sets. This means that  for $A$ and $B$ quasi-free, $H_{\tilde{A}}(B)$ is also  weakly equivalent to the space of maps from $A$ to $B$ in simplicial localisation at quasi-isomorphisms of the category of quasi-free localised dagger dg algebras.
\end{proof}

\begin{proposition}\label{affdsubEFCprop2}
Under the fully faithful functor from localised dagger dg algebras to EFC-DGAs in Corollary \ref{daggertoEFCcor} and Proposition \ref{affdsubEFCprop}, the essential image consists of those EFC-DGAs $A_{\bt}$ for which  $\H_0A$ is a dagger algebra and the $\H_0A$-modules $\H_mA$ are all finite.
\end{proposition}
\begin{proof}
 For any localised dagger dg algebra $C$, we have that $\H_0C$ is a dagger algebra, and the modules $\H_mB$ are finite, since coherent.
 
 Given an EFC-DGA $A$ satisfying the conditions above, it thus suffices to construct a dagger dg algebra $C$ quasi-isomorphic to $A$, since we can then localise $C$ at $\H_0A$ to give a quasi-isomorphic localised dagger dg algebra.  We begin by replacing $A$ with the quasi-isomorphic EFC-DGA $A\ten_{A_0}(A_0/\H_0A)^{\loc}$.
 
 Now, since $\H_0A$ is a dagger algebra, there exists a quasi-free dagger algebra $C(0)$ mapping surjectively to $\H_0A$,  and a surjection $C(0) \to \H_0A$ of EFC algebras. Since we have localised $A$ and $C(0)$ is quasi-free, this lifts to a morphism $f(0) \co C(0) \to A_0$ of EFC algebras, and hence to a map $C(0) \to A$ of EFC-DGAs which is surjective on $\H_0$.
 
Now assume inductively that we have constructed a sequence $C(0) \to C(1) \to \ldots \to C(n)$ of quasi-free morphisms  dagger dg algebras, with  $C(i)$ generated over $C(i-1)$ by generators  in degree $i$. Also assume that we have a morphism $f(n) \co C(n) \to A$ which gives isomorphisms on $\H_{< n}$ and is surjective on $\H_n$; this hypotheses amounts to saying that $\H_i\cone(f(n))=0$ for all $i \le n$. 
Since $\H_{n+1}A$ is a finite $\H_0A$-module and $C(n)$ is a complex of finite modules over the Noetherian ring $C(0)=C(n)_0$, it follows that $\H_{n+1}\cone(f(n))$ is a finite $C(0)$-module. We can therefore pick a finite set $S$ of generators and lift them to $\z_{n+1}\cone(f(n))$, giving us a map 
\[
(\delta,f(n+1))\co S \to \{(a,b) \in C(n)_n\by A_{n+1} ~:~ \delta a =0,\, f(n)(a)=\delta b\},
 \]
 and hence, placing $S$ in degree $n+1$ a map $C(n+1):=(C(n)[S],\delta) \to A$ satisfying the hypotheses, which completes the inductive step.
 
  The dagger dg algebra $C:=\bigcup_nC(n)$ then has finitely many generators in each level, so is levelwise finitely generated over $C_0=C(0)$, making it a quasi-free localised dagger dg algebra, and it is equipped with a quasi-isomorphism $C \to A$ of EFC-DGAs.
\end{proof}

The significance of Propositions \ref{affdsubEFCprop} and \ref{affdsubEFCprop2} is that we can use localised affinoid dagger dg spaces as the building blocks for derived dagger stacks satisfying a coherence condition, and hence for partially proper derived $K$-analytic stacks (e.g.  derived $K$-analytic stacks without boundary) satisfying coherence conditions, as in \cite[\S 4.2]{DStein}.

\subsection{Dagger dg spaces and stacks}\label{spacesn}
 
\subsubsection{Definitions} 
 
 \begin{definition}
  Define a $K$-dagger dg space $X$ to be a pair $(\pi^0X,\sO_X)$ where $\pi^0X$ is a $K$-dagger space in the sense of \cite[2.12]{GrosseKloenne} and $\sO_X$ is a presheaf of  quasi-dagger dg $K$-algebras (Definition \ref{dgqdaggerdef})
  on the site of open affinoid subdomains of $\pi^0X$, such that the homology presheaf $\H_0\sO_X$ is just $\sO_{\pi^0X}$, and the homology presheaves $\H_i\sO_X$ are all coherent $\sO_{\pi^0X}$-modules.
 \end{definition}

 \begin{example}
  Given an affinoid dagger dg space $X=(X^0,\sO_X)$, there is an associated dagger dg space given by $(\pi^0X, \iota^{-1}\sO_X)$, for the closed immersion $\iota \co \pi^0X \to X^0$. 
 \end{example}

 \begin{definition}
  A morphism $f \co X \to Y$ of  $K$-dagger dg spaces is said to be a quasi-isomorphism if it induces an isomorphism $\pi^0f \co \pi^0X \to \pi^0Y$ of dagger spaces and  isomorphisms $\H_i(f^{-1}\sO_Y) \to \H_i\sO_X$ for all $i$.
 \end{definition}

 Similarly:
 \begin{definition}\label{dgstackdef}
  Define a $K$-dagger dg analytic Deligne--Mumford  stack $X$ to be a pair $(\pi^0X,\sO_X)$ where $\pi^0X$ is a $K$-dagger analytic Deligne--Mumford stack and $\sO_X$ is a presheaf of  quasi-dagger dg $K$-algebras 
  on the site of dagger affinoid spaces \'etale  over $\pi^0X$, such that the homology presheaf $\H_0\sO_X$ is just $\sO_{\pi^0X}$, and the homology presheaves $\H_i\sO_X$ are all coherent $\sO_{\pi^0X}$-modules.
 \end{definition}
 
 \begin{remark}\label{Nstackrmk}
  Similar definitions exist for  analytic Deligne--Mumford  $\infty$-stacks,
  in which case the \'etale  site has higher categorical structure. 
  
  For an alternative characterisation of dg dagger spaces and stacks, and a generalisation to derived Artin stacks, we can use the approach via \v Cech nerve-type constructions as in \cite{stacks2} and 
  \cite[\S 6]{2021lect}.
  Instead of defining  a presheaf $\sO_X$ on a site associated to $\pi^0X$, we can just take a hypercover $Z_{\bt}$ of $\pi^0X$ with each $Z_n$ a disjoint union of dagger affinoid spaces $U$, and then give a dg dagger  algebra $\sO_{X,\bt}(U)$ for each $U$, such that $\H_0 \sO_{X,\bt}(U) \cong O(U)$ and $\H_i\sO_{X,\bt}(V)\cong \H_i\sO_{X,\bt}(U)\ten_{O(U)}O(V)$ for each morphism $V \to U$. Equivalences are then generated by certain hypercovers. If we restrict to compact stacks with compact (higher) diagonals, we can take each $Z_n$ to be dagger  affinoid rather than a disjoint union of such.
   \end{remark}

   Combining Corollary \ref{daggertoEFCcor}, Proposition \ref{affdsubEFCprop2} and \cite[Remark \ref{DStein-DHomrmk}]{DStein}, we get:
 \begin{corollary}\label{globalEFCdaggercor}
   The dagger-analytic analogue of the $\infty$-category of derived $K$-analytic spaces  from \cite{PortaYuNonArch} is equivalent to the simplicial localisation at quasi-isomorphisms of the category of $K$-dagger dg spaces.
   
   Moreover, the $\infty$-category of partially proper derived $K$-analytic spaces from \cite{PortaYuNonArch} is equivalent to the simplicial localisation at quasi-isomorphisms of the category of partially proper $K$-dagger dg spaces.
   
   The analogous statements for derived  $K$-dagger analytic Deligne--Mumford and Artin ($\infty$-)stacks also hold. In particular, they can be regarded as functors from localised  dagger dg $K$-algebras to simplicial sets.
 \end{corollary}
Here, we are saying that $X$ is partially proper if and only its underived truncation $\pi^0X$ is so; essentially this means that the space does not have a boundary.

\subsubsection{Representability} 
 
 \begin{definition}
  We denote the category of localised  dagger dg $K$-algebras by $dg_+\Affd\Alg_K^{\dagger,\loc}$. We then denote its full subcategory of objects which are bounded as chain complexes by $dg_+\Affd\Alg_K^{\dagger,\loc, \flat}$.
 \end{definition}

 \begin{definition}
Say that a simplicial set-valued functor $F \co dg_+\Affd\Alg_K^{\dagger,\loc} \to s\Set$  is homotopy-preserving if it maps quasi-isomorphisms to weak equivalences. 
\end{definition}
 
 \begin{definition}\label{sq0def}
We say that  a map $A \to B$ in $dg_+\Affd\Alg_K^{\dagger,\loc}$ is a square-zero extension if it is surjective and the kernel $I$ is square-zero, i.e. satisfies $I^2=0$. 
\end{definition}

\begin{lemma}
 If  $A \to B$ and $C \to B$ are surjective maps in  $dg_+\Affd\Alg_K^{\dagger,\loc}$, with $A \to B$  a square-zero extension, then the fibre product exists $A\by_BC \to C$ in $dg_+\Affd\Alg_K^{\dagger,\loc}$. Similar statements hold for dagger and quasi-dagger dg algebras. 
\end{lemma}
\begin{proof}
 First observe that if $D  \to C$ is a square-zero extension of  $K$-CDGAs with kernel $J$, and each $C_0$-module $J_r$ is finitely generated, then $D \in dg_+\Affd\Alg_K^{\dagger,\loc}$, with generators in degree $0$ given by combining those for $J_0$ with lifts of those for $C_0$; the dg quasi-dagger  algebra $D$ is localised because the maps $\Sp(C_0) \to \Sp(D_0)$ and $\Sp(\H_0C) \to \Sp(\H_0D)$  give isomorphisms on the underlying sets of points. 
 
 Now observe that our hypotheses imply that $A\by_BC \to C$ is a surjection with kernel $I$. Moreover, $I$ is levelwise finitely generated as a $C_0$-module (since it is as a $B_0$-module, and $C_0$ surjects onto $B_0$), so $A \by_BC \in dg_+\Affd\Alg_K^{\dagger,\loc}$. Beware that this would not be true if $C \to B$ were not surjective.
\end{proof}


\begin{definition}\label{hhgsdef}
We say that a 
functor 
$$
F\co dg_+\Affd\Alg_K^{\dagger,\loc} \to s\Set
$$
is homogeneous if  for all square-zero extensions $A \to B$ and all surjections $C \to B$ in $dg_+\Affd\Alg_K^{\dagger,\loc}$, the natural map
$$
F(A\by_BC) \to F(A)\by^h_{F(B)}F(C)
$$
to the homotopy fibre product is a weak  equivalence.
\end{definition}

\begin{remark}
This terminology is based on that from \cite{drep}, which was inspired by earlier usage in derived deformation theory, such as \cite{Man2}, and is a natural generalisation of Schlessinger's conditions for set-valued deformation functors from \cite[Theorem 2.11]{Sch} and \cite[2.2, S1]{Artin}.
Note however that in algebraic setting of \cite{drep}, the morphism $C \to B$ was not required to be  surjective; here, we have imposed surjectivity to ensure that the fibre product exists in our category. 
Homotopy-homogeneity differs from the notion of infinitesimal cohesion in \cite{lurie} in that we only require one of the morphisms to be nilpotent; our notion does not appear in \cite{lurie}, but its influence  is such that nowadays homogeneity is frequently referred to as ``infinitesimal cohesion on one factor''.
\end{remark}

\begin{definition}\label{Cohdef}
  Given a dg algebra $A$, we say that an $A$-module $M$ is levelwise finitely generated if as a graded $A$-module it has a  generating set with finitely many elements in each degree. We then let  $dg_+\Coh_A \subset dg_+\Mod_A$ be the category of levelwise finitely generated modules.
 \end{definition}
The significance of this condition is that if $A$ is an  quasi-dagger dg algebra and $M$ a levelwise f.g. $A$-module, then $A \oplus M$ is also an  quasi-dagger dg algebra, where the multiplication is defined so that $M$ is square-zero. 
 
\begin{definition}\label{Tdef}
Given a homotopy-preserving homogeneous  functor $F\co dg_+\Affd\Alg_K^{\dagger,\loc} \to s\Set$, an object $A \in dg_+\Affd\Alg_K^{\dagger,\loc}$   and a point $x \in F(A)$, define the tangent functor $T_xF$
\[
 T_xF \co dg_+\Coh_A \to s\Set,
\]
by
$$
T_xF(M):= F(A\oplus M)\by^h_{F(A)}\{x\},
$$
where  $A \oplus M$ is given the multiplication $(a_1,m_1)(a_2,m_2):=(a_1a_2,a_1m_2 +m_1a_2)$.

As for instance in \cite[Lemma \ref{drep-adf}]{drep}, the space $T_xF(M_{[-1]})$ deloops $T_xF(M)$, so we may  define tangent cohomology groups by $\DD^{n-i}_x(F,M):= \pi_i (F(A\oplus M[n])\by^h_{F(A)}\{x\})$.
\end{definition}

\begin{definition}\label{Fcotdef}
 In the setting of Definition \ref{Tdef}, we say that $F$ has a coherent cotangent complex $\bL^{F,x}$ at $x$ if there is a levelwise finitely generated  $A$-module $ \bL^{F,x}$ in chain complexes, bounded below in  chain degrees, representing $T_x(F)$ homotopically  in the sense that the simplicial mapping space 
 \[
  \oR\map_{dg\Mod_{A}}(\bL^{F,x},-) 
 \]
is weakly equivalent to $T_x(F)$ when restricted to $dg_+\Coh_{A}$. 
\end{definition}
In particular, this means that 
\[
\pi_i T_x(F)(M)\cong \EExt^{-i}_{A}(\bL^{F,x},M)
\]
for all $M \in dg_+\Coh_{A}$.

\begin{lemma}\label{cottranslemma}
 In the setting of Definition \ref{Tdef}, if $f \co A \to B$ is a morphism in $dg_+\Affd\Alg_K^{\dagger,\loc}$ and $x \in F(A)$ a point at which  $F$ has a coherent cotangent complex, then there is a natural quasi-isomorphism
 $\bL^{F,f_*x} \simeq  \bL^{F,x}\ten_A^{\oL}B$ of $B$-modules. 
 \end{lemma}
 \begin{proof}
For any morphism $f \co A \to B$ in $dg_+\Affd\Alg_K^{\dagger,\loc}$ and any $M \in dg_+\Coh_A$, flat over $A$ as a graded module, the $B$-module $M\ten_AB$ is levelwise finitely generated over $B$, so the morphism $A\oplus M \to B \oplus (M\ten_AB)$ in  $dg_+\Affd\Alg_K^{\dagger,\loc}$ gives us a map $T_x(F)(M) \to T_{f_*x}(F)(M\ten_AB)$. Since $F$ is homotopy-preserving, this gives us a map $T_x(F)(M) \to T_{f_*x}(F)(M\ten_A^{\oL}B)$ in the derived category for all $M \in dg_+\Coh_A$. By universality, this induces a morphism $ \bL^{F,f_*x} \to  \bL^{F,x}\ten_A^{\oL}B$.

Now, when $f$ is surjective, for any $M\in dg_+\Coh_B$, we have $A \oplus M \in dg_+\Affd\Alg_K^{\dagger,\loc}$, and then homotopy-homogeneity applied to the fibre product $(A \oplus M)\cong (B\oplus M)\by_BA$ gives us a weak equivalence $T_x(F)(M) \to T_{f_*x}(F)(M)$. By adjunction, this implies that the map  
$ \bL^{F,f_*x} \to  \bL^{F,x}\ten_A^{\oL}B$
is a quasi-isomorphism. Since $F$ is homotopy-preserving, the same conclusion holds whenever $\H_0f \co \H_0A \to \H_0B$ is surjective.

In general, we can factorise $f$ as a composite $A \to A\<\frac{z_1}{r_1}, \ldots, \frac{z_m}{r_m}\>^{\dagger} \onto B$, so by transitivity it remains to consider the case when $f$ is of the form $ A \to A\<\frac{z}{r}\>^{\dagger}$. We have a morphism $ \eta\co  \bL^{F,f_*x} \to  \bL^{F,x}\ten_AA\<\frac{z}{r}\>^{\dagger}$ in $dg_+\Coh_{A\<\frac{z}{r}\>^{\dagger}}$, and know that this becomes a quasi-isomorphism on base change along $A\<\frac{z}{r}\>^{\dagger} \to C$ whenever the induced map $\H_0A \to \H_0C$ is surjective. 
We can now argue as for instance in the proof  of \cite[Theorem 7.4.1]{lurie}, using MW algebras instead of polynomial rings. 
If we look at $\cone(\eta)$, then its homology groups are finite $\H_0A\<\frac{z}{r}\>^{\dagger}$-modules, but the base change result above implies that the support of the lowest non-zero homology group contains no closed points, which is a contradiction.
 \end{proof}

 \begin{definition}
  We say that $F\co dg_+\Affd\Alg_K^{\dagger,\loc} \to s\Set$ is nilcomplete if for all $B \in dg_+\Affd\Alg_K^{\dagger,\loc}$,  applying $F$ to the Postnikov tower of $B$ gives an equivalence
\[
 F(B) \to \ho\Lim_n F(B/\tau_{>n}B),
\]
where $\tau_{>n}B \subset B$ is the dg ideal given by good truncation in degrees above $n$.
 \end{definition}

 \begin{lemma}\label{cotexistslemma}
  If $F\co dg_+\Affd\Alg_K^{\dagger,\loc} \to s\Set$ is a homotopy-preserving, homogeneous, nilcomplete functor such that for all dagger  algebras $A$ and all points  $x \in F(A)$, the groups $\DD^i_x(F, A)$ are all  finitely generated $A$-modules and vanish for $i\ll 0$, then $F$ has coherent cotangent complexes $\bL^{F,y}$ at  all points $y \in F(B)$ for all $B \in dg_+\Affd\Alg_K^{\dagger,\loc}$. 
  \end{lemma}
\begin{proof}
 By \cite[Proposition 1.5]{GrosseKloenne}, MW algebras are regular. Any $B \in dg_+\Affd\Alg_K^{\dagger,\loc}$  is almost of finite presentation over an MW algebra by Lemma \ref{coffactnlemma}, so has a dualising module by \cite[Theorem 3.6.8]{lurie}.
 Although our tangent functor is only defined on coherent complexes, the relevant sections of the proof of \cite[Theorem 3.6.9]{lurie} establish the existence of $\bL^{F,y}$ provided that $\DD^i_y(F, M)$ is a finitely generated $\H_0B$-module for all finitely generated $\H_0B$-modules $M$ and that there exists some $n$ with  $\DD^i_y(F, M)=0$ for all $i<-n$ and all such $M$. We may  take a projective resolution of $M$, and then homogeneity combines with nilcompleteness to imply that these conditions hold provided $\DD^i_x(F, \H_0B) $ is finitely presented, and vanishes for $i<-n$. 
 
 Since $B \to \H_0B$ is surjective, homogeneity gives $\DD^i_X(F,M) \cong \DD^i_{\bar{x}}(F,M)$, where $\bar{x}$ is the image of $x$ in $F(\H_0B)$. Thus the conditions are satisfied by hypothesis.
\end{proof}

\begin{lemma}\label{affinecotlemma}
 For any  $A \in dg_+\Affd\Alg_K^{\dagger,\loc}$, the functor 
 $B \mapsto    \oR\map_{dg_+\Affd\Alg_K^{\dagger,\loc}}(A,B)$ has coherent cotangent complexes at all points.
\end{lemma}
\begin{proof}
Proposition \ref{affdsubEFCprop} allows us to take the mapping space in the model category of EFC-DGAs from Proposition \ref{locmodelprop}.
By Lemma \ref{coffactnlemma}, we may assume that $A$ is quasi-free, hence cofibrant in that model structure, so by Lemma \ref{qufreeOmegalemma} below, the cotangent complex at a point $x \co A \to B$ is represented by $\Omega^1_A\ten_AB$, where $\Omega^1_A$ is the coherent cotangent module of Definition \ref{OmegaAffddef}.
\end{proof}

We are now in a position to state a weak derived representability result. 

\begin{corollary}\label{lurierep2}
A homotopy-preserving functor  $F\co dg_+\Affd\Alg_K^{\dagger,\loc} \to s\Set$ is a dagger-analytic derived Artin $n$-stack 
if and only if 
 the following conditions hold
\begin{enumerate}
 
\item The restriction $\pi^0F \co \Affd\Alg_K^{\dagger} \to s\Set$ to underived dagger  algebras is represented by a dagger-analytic Artin  $n$-stack.

\item
$F$ is homogeneous.

\item $F$ is nilcomplete.

\item \label{shf2} 
for all dagger  algebras $A \in \Affd\Alg_K^{\dagger}$, all $x \in F(A)_0$ and all \'etale morphisms $f:A \to A'$, the maps
\[
\DD_x^*(F, A)\ten_AA' \to \DD_{fx}^*(F, A')
\]
are isomorphisms.

\item for all  dagger  algebras  $A$  and all $x \in F(A)$, the groups $\DD^i_x(F, A)$ are all  finitely generated $A$-modules. 

\end{enumerate}
\end{corollary}
\begin{proof}
 By Lemma \ref{cotexistslemma}, the conditions imply that $F$ has coherent cotangent complexes $\bL^{F,x}$ at  all points $y \in F(A)$ for all $A \in dg_+\Affd\Alg_K^{\dagger,\loc}$, and by Lemma \ref{cottranslemma} we have quasi-isomorphisms  $\bL^{F,f_*x}\simeq \bL^{F,x}\ten^{\oL}_AB$ for all morphisms $f \co A \to B$ in  $dg_+\Affd\Alg_K^{\dagger,\loc}$. Arguing as in the proof of \cite[Corollary \ref{drep-lurierep2}]{drep}, it follows from  \cite[Proposition \ref{drep-sheafresult}]{drep} that $F$ is an \'etale hypersheaf. The final stage of the proof of \cite[Theorem 7.1]{PortaYuRep} (itself based on the algebraic setting of  \cite[Theorem C0.9]{hag2}, based on \cite{lurie}) then adapts directly to give the desired result. 
\end{proof}

\begin{remarks}
 This result is significantly weaker than the representability results of \cite{lurie,drep} in that it assumes that $\pi^0F$ is representable. Direct analogues of  stronger forms of the representability theorem cannot exist in our setting, because  $dg_+\Affd\Alg_K^{\dagger,\loc}$ does not have filtered colimits or contain complete local rings. However, an analogue of the combined formal effectiveness and finite presentation conditions would be to require that 
 for every complete local $K$-algebra $A$ with residue field finite over $K$, the morphism 
 \[
  \LLim_{\substack{A'\subset A\\ \text{dagger} \\ \text{affinoid}}} F(A') \to \ho\Lim_m F(A/\m_A^m)
 \]
is a weak equivalence; a stronger representability result incorporating such a condition is plausible, but may be of limited use.
 
There is a simpler version of Corollary \ref{lurierep2} in which we just take  $F\co dg_+\Affd\Alg_K^{\dagger,\loc,\flat} \to s\Set$ to be a homotopy-preserving functor on bounded objects, and consequently drop the nilcompleteness condition. Such functors correspond to nilcomplete functors on $dg_+\Affd\Alg_K^{\dagger,\loc}$, the correspondence given by setting $F(B) := \ho\Lim_n F(B/\tau_{>n}B)$.
  \end{remarks}

 The following is an  immediate consequence of Corollary \ref{lurierep2}
\begin{corollary}\label{opencutcor}
 Assume that $F\co dg_+\Affd\Alg_K^{\dagger,\loc} \to s\Set $ is homotopy-preserving, nilcomplete and homogeneous. Take a  $K$-dagger analytic Artin $n$-stack  $\fX$, together with  a natural transformation  $\fX \to \pi^0F$ of functors $\Affd\Alg_K^{\dagger,\loc} \to s\Set $
 which is formally \'etale in the sense that for all square-zero extensions $A \to B$, the map
 \[
  \fX(A) \to F(A)\by^h_{F(B)}\fX(B)
 \]
is a weak equivalence. 

Assume that for all points $x$ in image of $\fX(A) \to \pi^0F(A)$, the groups $\DD^i_x(F, A)$ are all  finitely generated $A$-modules, and that  
 for  all \'etale morphisms $f:A \to A'$, the maps
\[
\DD_x^*(F, A)\ten_AA' \to \DD_{fx}^*(F, A')
\]
are isomorphisms.

Then the functor
\[
 A \mapsto F(A)\by^h_{F(\H_0A)}\fX(\H_0A)
\]
on $dg_+\Affd\Alg_K^{\dagger,\loc}$ is representable by an $n$-geometric (resp. $\infty$-geometric)  $K$-dagger dg analytic Artin stack $\fX$
\end{corollary}

\subsection{Pro-\'etale sheaves associated to affinoid dagger dg spaces}\label{proetdaggersn}

For our purposes, the great advantage of dagger algebras over EFC algebras is that the former are Noetherian and consequently equipped with well-behaved canonical topologies\footnote{Although free algebras induce a colimit topology on EFC algebras,   Corollary \ref{preservehtpycor} needs Noetherianity to guarantee closure of submodules.}, which we now exploit to produce condensed algebras. From now on, we assume that the valuation on our base field $K$ is discretely valued, so the ring $\cO_K:=\{\lambda \in K ~:~ |\lambda|\le 1\}$ is a DVR with maximal ideal $\m_K:=\{\lambda \in K ~:~ |\lambda|< 1\}$, 
and the topology on $K$ induced by the norm is the $\varpi$-adic topology, where $\varpi$ is an element of $\m_K$ of maximum norm (and hence a generator of that ideal). 


Now, every affinoid $K$-algebra $A$ is a Banach $K$-algebra, and every finite $A$-module $M$ then inherits the structure of a Banach space (up to Banach space isomorphism), since it admits a surjection $A^n \onto M$ for some $n$, with kernel necessarily closed by \cite[Proposition 5.2.7.1]{BoschGuentzerRemmertNonArchanalysis}.


\begin{lemma}\label{exactmodsaffdlemma}
 Every surjection $f \co M \onto N$ of finite modules over an affinoid $K$-algebra  $A$ admits a continuous $K$-linear section.
\end{lemma}
\begin{proof}
Since the kernel $L$ of $f$ is (necessarily) closed, we have a topological $A$-linear isomorphism $M/L \cong N$. 
 The norm on $N$ is thus equivalent to the quotient norm induced from $M$. With respect to these norms, we can then let $M^o:=\{m \in M~:~ |m|\le 1\}$ and similarly for $N^o \subset N$. These are $\cO_K$-modules, with $M= \bigcup_n \varpi^{-n}M^o \cong M^o\ten_{\cO_K}K$ and similarly for $N$, where $\varpi$  generates $\m_K$.
 
 Now, since the norm on $N$ is ultrametric and discrete (taking the same values as the norm on $K$),  the topology on the  $\cO_K$-module $N^o$ is the $\varpi$-adic topology. Pick a basis for the $\cO_K/\m_K$-vector space  $ N^o/\varpi N^o = N^o/\m_KN^o$, and lift to a set $S$ of elements of $N^o$. Then we have a topological $\cO_K$-linear isomorphism $\Lim_n ((\cO_K/\varpi^n) .S) \to N^o$, since $N^o$ is $\varpi$-adically complete. 
 
 Picking pre-images of $S$ in $M^o$ then gives us a continuous $\cO_K$-linear section of $f \co M^o \to N^o$, and hence a continuous $K$-linear section of $f \co M\to N$ on tensoring with $K$.
 \end{proof}

\begin{proposition}\label{exactmodsaffdprop}
 Given a pro-finite set $S$ and an affinoid $K$-algebra $A$, the functor
 \[
  \Hom_{\cts}(S,-)
 \]
on the abelian category of finite $A$-modules is exact. 

In consequence, for all finite $A$-modules $M$ the natural map $ \Hom_{\cts}(S,A)\ten_AM \to  \Hom_{\cts}(S,M)$ is an isomorphism.
 \end{proposition}
\begin{proof}
The functor is obviously left exact, preserving finite limits. Lemma \ref{exactmodsaffdlemma} implies that for any short exact sequence of finite $A$-modules, the underlying  sequence of topological $K$-modules splits. Since $\Hom_{\cts}(S,-)$ depends only on the topological abelian group structure, left exactness thus guarantees exactness.

The final statement then follows by applying the functor to a finite presentation $\coker(A^m \to A^n)$ of $M$, noting that finite $A$-modules are all finitely presented because $A$ is Noetherian.
\end{proof}

 

\begin{definition}\label{ulinedef}
 Given  a topological  $K$-vector space $V$, we let $\uline{V}$ be the functor from pro-finite sets to  $K$-vector spaces given by 
 \[
  S \mapsto \Hom_{\cts}(S,V),
 \]
 the space of continuous functions from $S$ to $V$.
 
 We extend this definition to  quasi-dagger dg algebras $A$ and their modules via Proposition \ref{indaffinoidprop}, writing $A$ as a filtered colimit $A=\LLim_{i \in I}A(i) $ of dg affinoid algebras,   and setting
 \[
\uline{A}(S):= \LLim_{i \in I} \uline{A(i)}(S),
 \]
and similarly for modules.
 \end{definition}


\begin{corollary}\label{preservehtpycor}
 For any pro-finite set $S$, the functor $A \mapsto \uline{A}(S)$ from quasi-dagger dg algebras to CDGAs over $\uline{K}(S)$ preserves quasi-isomorphisms. The functor is moreover naturally isomorphic  to the functor $A \mapsto \uline{A_0}(S)\ten_{A_0}A$. 
 \end{corollary}
\begin{proof} 
Because filtered colimits are exact, Proposition \ref{exactmodsaffdprop} also applies to finite modules over Noetherian filtered colimits of affinoid algebras, hence over  quasi-dagger algebras; Noetherianity means every finite $A_0$-module is finitely presented, so induced by a module over some affinoid subalgebra. Since the $A_0$-modules $A_n$ are all finite, it follows immediately that  $\uline{A}(S)\cong \uline{A_0}(S)\ten_{A_0}A$, giving the second statement.

By applying Lemma \ref{coffactnlemma}, we see that any quasi-isomorphism $A \to B$ of quasi-dagger dg algebras admits a factorisation $A \to C \to B$ 
into quasi-isomorphisms for which the first map admits a retraction and the second is surjective. It thus suffices to prove that the functor $A \mapsto \uline{A}(S)$ sends surjective quasi-isomorphisms to quasi-isomorphisms.
 
 Now, given a surjective quasi-isomorphism $A \to B$ of quasi-dagger dg algebras, the $A_0$-modules $A_n$ and $B_n$ are all finite. Since the functor $M \mapsto \uline{M}(S)$ is exact on finite $A_0$-modules, it preserves quasi-isomorphisms, giving the first statement.
 \end{proof}

 \begin{definition}
  Given a  scheme $X$ and a quasi-dagger dg algebra $A$, we define the sheaf $\uline{A}_X$ of CDGAs on the affine pro-\'etale site $X_{\pro\et}^{\aff}$ of $X$ (see \cite[Definition 4.2.1]{BhattScholzeProEtale}) by
  \[
   U \mapsto \uline{A}(\pi_0U),
  \]
where $\pi_0U$ is the pro-finite set of components of  the quasi-compact quasi-separated scheme $U$, constructed as in \cite[\S 2]{BhattScholzeProEtale}.  It follows from \cite[Lemma 4.2.12]{BhattScholzeProEtale} that the presheaves $\uline{A}_{X}$ are indeed sheaves.
 \end{definition}
 
 We can use this to construct moduli functors of various flavours of sheaf on the pro-\'etale site:
 
 \begin{definition}\label{FXproetdef}
  Given a scheme $X$ and a  functor $F \co dg_+\CAlg_K \to s\Set$ from differential graded-commutative $K$-algebras in non-negative chain degrees to simplicial sets, define the functor
  \[
   F(X_{\pro\et},-) \co dg_+\Affd\Alg^{\loc,\dagger}_K \to s\Set
  \]
from  localised dagger dg algebras to simplicial sets by
\[
 A \mapsto \oR\Gamma(X_{\pro\et},F(\uline{A}_X))
\]  
where $\oR\Gamma$ is the right-derived functor of the global sections functor $\Gamma$ in simplicial sets.
 \end{definition}

 \begin{example}\label{BGXex}
  If $G$ is an algebraic group over $K$, then we can let $F$ be the derived stack $BG$, parametrising $G$-torsors. The functor $ BG(X_{\pro\et},-)$ then parametrises $G$-torsors on $X_{\pro\et}$. In particular, when $A$ is a dagger algebra (concentrated in degree $0$), $ BG(X_{\pro\et},A)$ is the nerve of the groupoid of $G(\uline{A}_X)$-torsors on $X_{\pro\et}$. 
  
  If $X$ is locally topologically Noetherian and connected, with a geometric point $x$, then for $A \in \Affd\Alg^{\dagger}_K$, this means that  we have the nerve
  \[
   BG(X_{\pro\et},A) \simeq B[\Hom(\pi_1^{\pro\et}(X,x), G(A))/G(A)],
  \]
  of the groupoid of continuous group homomorphisms,
where $G(A)$ is topologised using the topology on $A$ (take the coarsest topology for which the maps $G(A) \to A$ given by elements of $\sO_G$ are all continuous),  $ \pi_1^{\pro\et}$ is the pro-\'etale fundamental group of \cite[\S 7]{BhattScholzeProEtale}, and $G(A)$ acts on the set $\Hom(\pi_1^{\pro\et}(X,x), G(A))$ by conjugation.

 \end{example}

 \begin{remark}\label{antoniormk}
In \cite{antonioPadic,antonioElladic}, Ant\'onio constructed derived $\ell$-adic analytic moduli stacks of  local systems on smooth schemes $X$ over algebraically closed fields of finite characteristic. There, objects are defined using the \'etale site 
by functorially reinterpreting modules over derived affinoid algebras as ind-pro-finite groups to define the moduli functor. One advantage of the pro-\'etale site is that it removes the need for such a reduction.

If we  extend Ant\'onio's functor to derived dagger affinoids by passing to colimits, it will  thus map to our functor  $B\GL_n(X_{\pro\et},-)$ via the pullback $\nu^*\co \mathrm{Shv}(X_{\et}) \to \mathrm{Shv}(X_{\pro\et}) $. Moreover,  \cite[Corollary 5.1.6]{BhattScholzeProEtale} ensures that the comparison map between the respective derived moduli stacks is \'etale, since $\nu^*$ is faithful on constructible objects.

Moduli of $\infty$-local systems of perfect complexes were also considered in \cite{antonioPadic,antonioElladic}, and are similarly related to our construction $\Perf(X_{\pro\et},-)$ for $\Perf$ the derived moduli stack of perfect complexes.

 \end{remark}

 \begin{lemma}\label{proethtpylemma}
  \begin{enumerate}
   \item If $F$ is homotopy-preserving in the sense that it maps quasi-isomorphisms to weak equivalences, then so is $F(X_{\pro\et},-)$.
   
   \item Assume that the natural map $F(A'\by_{B'}C')\to F(A')\by^h_{F(B')}F(C')$ is a weak equivalence for all surjections $A' \to B' \la C'$ of CDGAs with $\ker(A'_0 \to B'_0)$ nilpotent; in particular this holds if $F$  is homotopy-homogeneous in the sense of \cite{drep}. Then for any surjective morphisms $A \to B \la C$ of localised dagger dg algebras, with $\ker(A_0 \to B_0)$ nilpotent, the object $A\by_BC$ is also a localised dagger dg algebra and the natural map
 \[
  F(X_{\pro\et},A\by_BC)\to F(X_{\pro\et},A)\by^h_{F(X_{\pro\et},B)}F(X_{\pro\et},C)
 \]
 is a weak equivalence.
  \end{enumerate}
 \end{lemma}
\begin{proof}
 The first statement is an immediate consequence of Corollary \ref{preservehtpycor}.

 For the second statement, begin by noting that since both $A \to B$ and $C \to B$ are surjective, the modules $(A_n\by_{B_n}C_n)$ must  all be finite over $A_0 \by_{C_0} B_0$, so $A \by_BC$ is a quasi-dagger dg algebra. Since $A_0 \to B_0$ is a nilpotent surjection, the map $A_0\by_{B_0}C_0 \to C_0$ is also a nilpotent surjection. Any element $x$ in $\ker(A_0\by_{B_0}C_0) \to \H_0(A\by_BC)$ maps to $0$ in $\H_0C$, so $1+x$ maps to a unit in $C_0$, since $C$ is a localised dagger algebra. Nilpotence of $ A_0\by_{B_0}C_0 \to C_0$ then implies that $1+x$ is a unit, so $A\by_BC$ is also localised.
 
 The maps $\uline{A}_X \to \uline{B}_X \la \uline{C}_X$ are now objectwise surjections of CDGAs by Proposition \ref{exactmodsaffdprop}, and the first is objectwise nilpotent. Since the functor $A \mapsto \uline{A}_X$ preserves finite limits, we have $\uline{A\by_BC}_X \cong \uline{A}_X\by_{\uline{B}_X}\uline{C}_X$.  The hypothesis on $F$ then gives 
 \[
  F(\uline{A\by_BC}_X) \simeq F(\uline{A}_X)\by^h_{F(\uline{B}_X)}F(\uline{C}_X),
 \]
and the desired statement then follows because derived global sections preserve homotopy limits.
 \end{proof}

 
 \section{Shifted symplectic structures}\label{sympsn}
 
 
 \subsection{Structures on EFC algebras}
 
 We now recall how standard algebraic definitions adapt to the analytic setting, as sketched in \cite[\S 4.4]{DStein}.
 
 \begin{definition}\label{Omegadef}
 Given an EFC-DGA $A$, define the complex $\Omega^1_A$ to be the $A$-module in chain complexes representing  the functor 
 \[
 M \mapsto \Hom_{EFC-DGA}(A, A \oplus M)\by_{\Hom_{EFC-DGA}(A,A)}\{\id\}
 \]
 of closed EFC derivations from $A$ to $M$ of degree $0$. Here, the EFC structure on $A \oplus M$ is determined by requiring it to be a group object over $A$; explicitly, for a holomorphic function $f(z_1, \ldots, z_n)$ in $n$ variables, we set
\[
 f(a_1+m_1, \ldots, a_n+m_n):= f(a_1, \ldots, a_n) + \sum_i \frac{\pd f}{\pd z_i}(a_1, \ldots, a_n)m_i,
\]
which in particular means that the multiplication on $M$ is $0$.
 
 Given a morphism $R \to A$ of EFC-DGAs, define $\Omega^1_{A/R}$ to be the cokernel of $\Omega^1_R\ten_RA \to \Omega^1_A$.
\end{definition}

As in \cite[\S \ref{DStein-cotsn}]{DStein}, we can then follow the approach of \cite{Q} to give EFC cotangent complexes.

\begin{definition}
 Denote by $A \mapsto (A, \oL\Omega^1_A)$ the left-derived functor of the functor  $A \mapsto (A, \Omega^1_A)$ from EFC-DGAs to  the category of pairs  $(A,M)$ of   EFC-DGAs and modules. We refer to $\oL\Omega^1_A$ as the cotangent complex of $A$. Given a morphism $R \to A$ of EFC-DGAs, write $\oL\Omega^1_{A/R}$ for the cone of the natural map $ \oL\Omega^1_R\ten_R^{\oL}A \to  \oL\Omega^1_A$. 
\end{definition}

\begin{definition}\label{DRdef}
Given an EFC-DGA $A$,  write $\Omega^p_{A}:= \L^p_{A}\Omega^1_{A}$, and define the de Rham complex $\DR(A)$ to be the product total cochain complex of the double complex
\[
 A \xra{d} \Omega^1_{A} \xra{d} \Omega^2_{A}\xra{d} \ldots,
\]
so the total differential is $d \pm \delta$.

We define the Hodge filtration $F$ on  $\DR(A)$ by setting $F^p\DR(A) \subset \DR(A)$ to consist of terms $\Omega^i_{A}$ with $i \ge p$.

Define $\oL\DR(A)$ to be $\DR(\tilde{A})$ for any cofibrant replacement of $A$.
\end{definition}

Properties of the product total complex ensure that a map $f \co A \to B$ induces a filtered quasi-isomorphism $\DR(A) \to \DR(B)$ whenever the maps $\Omega^p_{A} \to \Omega^p_{B}  $ are quasi-isomorphisms, which will happen whenever $f$ is a weak equivalence between cofibrant EFC-DGAs.

Note that if $\tilde{A}$ is a cofibrant replacement for $A$, there is a natural $\tilde{A}$-linear quasi-isomorphism $\Omega^1_{\tilde{A}}\to \oL\Omega^1_A$.

\begin{definition}\label{presymplecticdef}
Define an $n$-shifted pre-symplectic structure $\omega$ on an EFC-DGA $A$ to be an element
\[
 \omega \in \z^{n+2}F^2\oL\DR(A).
\]
\end{definition}

Explicitly, this means that $\omega$ is given by an infinite sum $\omega = \sum_{i \ge 2} \omega_i$, with $\omega_i \in (\Omega^i_{\tilde{A}})_{i-n-2}$ and $d\omega_i = \pm\delta \omega_{i+1}$.

\begin{definition}
 Define an $n$-shifted symplectic structure $\omega$ on $A$ to be an $n$-shifted pre-symplectic structure $\omega$ for which the component $\omega_2 \in \z^n\Omega^2_{\tilde{A}}$ induces a quasi-isomorphism
\[
 \omega_2^{\sharp} \co \HHom_{\tilde{A}}(\Omega^1_{\tilde{A}},\tilde{A}) \to (\Omega^1_{\tilde{A}})_{[-n]},
\]
with $\oL\Omega^1_{A} $  perfect as an $A$-module. 
\end{definition}

\begin{definition}\label{PreSpdef}
 Define the space of $n$-shifted pre-symplectic structures on an EFC-DGA $A$ to be the simplicial set $\PreSp(A,n)$ given by Dold--Kan denormalisation of the chain complex
\[
 \tau_{\ge 0}(\oL F^2\DR(A)^{[n+2]}),
\]
where $\tau_{\ge 0}$ denotes good truncation in non-negative chain degrees.

Set $\Sp(A,n) \subset \PreSp(A,n)$ to consist of the symplectic structures --- this is a union of path-components.
\end{definition}

Observe that for any morphism $f \co A \to B$ of EFC-DGAs, we have a natural map $\PreSp(A,n) \to \PreSp(B,n)$, but that this does not restrict to the subspaces of symplectic structures unless $f$ gives a quasi-isomorphism $B\ten^{\oL}_A\oL\Omega^1_A\to \oL\Omega^1_B$ on cotangent complexes, i.e. unless $f$ is homotopy  ind-\'etale.

\begin{definition}\label{Lagdef}
Take a morphism $f \co A \to B$ of  EFC-DGAs, with an  $n$-shifted pre-symplectic structure $\omega$ on $A$. We then define an isotropic structure on $B$ relative to $\omega$ to be an element $(\omega, \lambda)$ of 
\[
 \z^{n+1}\cone(\oL F^2\DR(A)\to \oL F^2\DR(B))
\]
lifting $\omega$.

 This structure is called Lagrangian if $\oL \Omega^1_{A} $ and $\oL \Omega^1_{B}$ are perfect as an $A$-module and a $B$-module respectively, if $\omega$ is symplectic, and if contraction with the image $(\omega_2,\lambda_2)$ of $(\omega,\lambda)$ in  $\z^{n-1}\cone(\Omega^2_{\tilde{A}} \to \Omega^2_{\tilde{B}} )$ 
 induces a quasi-isomorphism
\[
 (f \circ \omega_2^{\sharp}, \lambda_2^{\sharp}) \co 
 \cone(\HHom_{\tilde{B}}(\Omega^1_{\tilde{B}},\tilde{B}) \to \HHom_{\tilde{A}}(\Omega^1_{\tilde{A}},\tilde{B}))  \to (\Omega^1_{\tilde{B}})_{[-n]}.
\] 
\end{definition}

 \begin{definition}\label{Isodef}
  Given a morphism $A \to B$ of  EFC-DGAs, define 
  the space $\Iso(A,B;n)$
  of  $n$-shifted isotropic structures on the pair $(A,B)$ to be the simplicial set given  by Dold--Kan denormalisation of the chain complex 
\[
 \tau_{\ge 0}(\cone(\oL F^2\DR(A)  \to \oL F^2\DR(B))^{[n+1]}).
\]

Set $\Lag(A,B;n) \subset \Iso(A,B;n)$ to consist of the   Lagrangians on symplectic structures --- this is a union of path-components.
\end{definition}

\subsection{Structures on quasi-dagger algebras}

\begin{lemma}\label{EFCOmegalemma}
If $A$ is the EFC-algebra underlying the quasi-dagger algebra $K\<\frac{x_1}{r_1}, \ldots ,\frac{x_n}{r_n}\>^{\dagger}$, for $r_1, \ldots,r_n\ge 0$, then 
\[
 \Omega^1_A \cong \bigoplus_{i=1}^n A.dx_i,
\]
and the natural map $\oL\Omega^1_A \to \Omega^1_A$ is a quasi-isomorphism.
\end{lemma}
\begin{proof}
The free EFC-algebra on $n$-generators is given by the ring 
\[
O(\bA^n_K) \cong \Lim_{s_1, \ldots s_n \to \infty}K\<\frac{x_1}{s_1}, \ldots ,\frac{x_n}{s_n}\>^{\dagger}
\]
of analytic functions on affine $n$-space, so
\[
 \Omega^1_{O(\bA^n_K)} \cong \bigoplus_{i=1}^n O(\bA^n_K).dx_i.
\]
Since the inclusions $\Delta(\rho_1,\ldots,\rho_n) \subset \bA^n_K$ of open polydiscs are open immersions,
\cite[Lemma \ref{DStein-Cunramlemma}]{DStein} then implies that
\[
 \Omega^1_{O(\Delta(\rho_1,\ldots,\rho_n) )} \cong \bigoplus_{i=1}^n O(\Delta(\rho_1,\ldots,\rho_n)).dx_i,
\]
while \cite[Lemma \ref{DStein-cotcoflemma}]{DStein} implies that 
\[
 \oL \Omega^1_{O(\Delta(\rho_1,\ldots,\rho_n) )} \simeq \Omega^1_{O(\Delta(\rho_1,\ldots,\rho_n) )}.
\]
The result then follows by passing to the filtered colimit over $\rho_i >r_i$.
\end{proof}

\begin{definition}\label{OmegaAffddef}
 Given a   quasi-dagger DGA $A$, we define the $A$-module  $\Omega^1_A$ such that  $d \co A \to \Omega^1_A$ is the universal $K$-linear derivation from $A$ to  
 $A$-modules $M$  in complexes which are  \emph{coherent} in the sense of Definition \ref{Cohdef},
  generalising the definition for dagger algebras in \cite[4.1]{GrosseKloenne}; this exists with similar reasoning. 
\end{definition}
Beware that this is not usually the same as the algebraic module $\Omega^1_{A^{\alg}}$ of K\"ahler differentials, which is universal with respect to $K$-linear derivations from $A$ to  \emph{all} $A$-modules in complexes; there is a canonical map from $\Omega^1_{A^{\alg}} \to \Omega^1_A$, by universality.

\begin{lemma}\label{qufreeOmegalemma}
 If $A$ is a quasi-free quasi-dagger DGA, then the $A$-module $\Omega^1_A$ is a model for the cotangent complex of the EFC-DGA $B$ underlying $A$. 
\end{lemma}
\begin{proof}
Since $A_0$ is a free quasi-dagger algebra, it is smooth, so \cite[4.1]{GrosseKloenne} implies that $\Omega^1_{A_0}$ is the free $A_0$-module of Lemma \ref{EFCOmegalemma}, which by that lemma is a model for the cotangent complex $\oL \Omega^1_{B_0}$ of the EFC-algebra underlying $A_0$.

Now $A$ is freely generated as an $A_0$-algebra, with generators in strictly positive degrees, so we have a short exact sequence
\[
 0 \to \Omega^1_{A_0}\ten_{A_0}A \to \Omega^1_A \to \Omega^1_{A/A_0} \to 0.
\]
Moreover, \cite[Lemma \ref{DStein-cotalglemma}]{DStein} implies that $\oL \Omega^1_{B/B_0}$ is given by the cotangent complex for commutative algebras, which is just $\Omega^1_{A/A_0}$ by quasi-freeness. Thus the short exact sequence above must be quasi-isomorphic to the exact triangle
\[ 
 \oL\Omega^1_{B_0}\ten_{B_0}B \to \oL\Omega^1_B \to \oL\Omega^1_{B/B_0} \xra{[1]} \ldots, 
\]
and hence $\oL\Omega^1_B \simeq  \Omega^1_A$. 
\end{proof}
Again, beware that the EFC cotangent complex will not be the same as the cotangent complex for commutative algebras, as the CDGAs underlying  quasi-free quasi-dagger DGAs are not sufficiently close to being cofibrant, unlike the underlying EFC-DGAs. 

The following is now immediate, and allows us to  interpret shifted symplectic structures in terms of dagger affinoid constructions:
\begin{corollary}
Given a quasi-free quasi-dagger DGA $A$ with underlying EFC-DGA $B$, the de Rham complex $\oL\DR(B)$ of Definition \ref{DRdef} 
is quasi-isomorphic to the
 product total cochain complex of the double complex $\DR(A)$ given by
\[
 A \xra{d} \Omega^1_{A} \xra{d} \Omega^2_{A}\xra{d} \ldots,
\]
where $\Omega^p_{A}:= \L^p_{A}\Omega^1_{A}$

Moreover the complex  $F^p\oL\DR(B)$ is quasi-isomorphic to the  subcomplex of  $\DR(A)$ to consisting of terms $\Omega^i_{A}$ with $i \ge p$.
\end{corollary}

\begin{corollary}\label{Spcor}
Given a quasi-dagger DGA $A$ with underlying EFC-DGA $B$, and a quasi-isomorphism $\tilde{A} \to A$ from a quasi-free quasi-dagger DGA, the space $\PreSp(B,n)$   of $n$-shifted pre-symplectic structures on $B$ is equivalent to the 
Dold--Kan denormalisation of the chain complex
\[
 \tau_{\ge 0}( F^2\DR(A)^{[n+2]}). 
\]

The space $\Sp(B,n) \subset \PreSp(B,n)$ then corresponds to the subspace of points $\omega$ 
for which the component $\omega_2 \in \z^n\Omega^2_{\tilde{A}}$ induces a quasi-isomorphism
\[
 \omega_2^{\sharp} \co \HHom_{\tilde{A}}(\Omega^1_{\tilde{A}},\tilde{A}) \to (\Omega^1_{\tilde{A}})_{[-n]},
\]
with $\Omega^1_{\tilde{A}} $  perfect as an $\tilde{A}$-module. 
\end{corollary}

From now on, we will simply refer to these as (pre-)symplectic structures on $A$, denoted $\PreSp(A,n)$ and $\Sp(A,n)$.

\subsection{Structures on dagger dg spaces and stacks} 

\begin{definition}\label{SpDMdef} 
Given a $K$-dagger dg space or $K$-dagger dg DM stack $X=(\pi^0X,\sO_X)$ as in \S \ref{spacesn}, we define  the spaces $\PreSp(X,n)$ and $\Sp(X,n)$ of $n$-shifted pre-symplectic and symplectic structures on $X$ as follows.
First, take an objectwise cofibrant replacement $\tilde{\sO}_X$, in the model structure of Proposition \ref{locmodelprop},  of the EFC-DGA underlying  $\sO_X$ (so each $\tilde{\sO}_X(U)$ is cofibrant, e.g. a quasi-free quasi-dagger dg algebra, and $\tilde{\sO}_X \to \sO_X$ is a quasi-isomorphism).
We then set 
\[
 \PreSp(X,n):= \oR\Gamma(\pi^0X, \PreSp(\tilde{\sO}_X)) \quad \Sp(X,n):= \oR\Gamma(\pi^0X, \Sp(\tilde{\sO}_X)).
\]

\end{definition}

 
 \begin{definition}\label{daggerLagdef}
Similarly, given a morphism $f \co X \to Y$ of $K$-dagger dg spaces or $K$-dagger dg DM stacks, we generalise Definition \ref{Isodef} to define the space  $\Iso(X,Y;n)$   of $n$-shifted isotropic structures on $X$ over $Y$ to be the homotopy fibre product
\[
 \Iso(X,Y;n):=\PreSp(Y,n)\by^h_{\PreSp(X,n)}\{0\},
\]
so that  isotropic structures on $X$ over a fixed $n$-shifted pre-symplectic structure $\omega$ on $Y$ are given by the homotopy fibre $\{f^*\omega\}\by^h_{\PreSp(X,n)}\{0\}$ of $\Iso(X,Y;n)\to \PreSp(Y,n)$ over $\omega$.

We then define the space  $\Lag(X,Y;n) \subset \Iso(X,Y;n)$   of $n$-shifted  Lagrangian structures on $X$ over $Y$ to consist of the path components of objects $(\omega,\lambda)$ which are non-degenerate in the sense that $\omega$ is symplectic  and  contraction with the image $(\omega_2,\lambda_2)$ of $(\omega,\lambda)$ in  $\z^{n-1}\cone(\Omega^2_{\tilde{\sO}_Y} \to \Omega^2_{\tilde{\sO}_X} )$ 
 induces a quasi-isomorphism
\[
 (f^{\sharp} \circ \omega_2^{\sharp}, \lambda_2^{\sharp}) \co 
 \cone(\HHom_{\tilde{\sO}_X}(\Omega^1_{\tilde{\sO}_X},\tilde{\sO}_X) \to \HHom_{f^{-1}\tilde{\sO}_Y}(f^{-1}\Omega^1_{\tilde{\sO}_Y},\tilde{\sO}_X))  \to (\Omega^1_{\tilde{\sO}_X})_{[-n]}.
\] 
 \end{definition}

\begin{remark}
 As a consequence of Corollary \ref{Spcor}, these correspond to shifted symplectic and Lagrangian structures in the sense of \cite[\S 4.4]{DStein}  on the derived analytic spaces and stacks associated to $K$-dagger dg spaces and $K$-dagger dg DM stacks by Corollary \ref{globalEFCdaggercor}.
 \end{remark}
 
 For $K$-dagger dg Artin stacks, we now have to take a little more care, as non-degeneracy is not preserved by smooth morphisms. Given such a stack $X=(\pi^0X,\sO_X)$, with $\sO_X$ a sheaf on the site of affinoids which are smooth over $\pi^0X$, the cotangent complex $\bL^X$ is  a Cartesian $\sO_X$-module in chain complexes representing the functor of derived derivations. It is
 determined by the property that for all Cartesian $\sO_X$-modules $M$ with $M(U) \in dg_+\Coh_{\sO_X(U)}$, the space of maps $(\pi^0X, \sO_X \oplus M) \to (\pi^0X, \sO_X)$ extending the identity map is the Dold--Kan denormalisation of the good truncation of   
 $\oR\HHom_{\sO_X}(\bL^X,M)$.

 \begin{definition}
  Given a  $K$-dagger dg Artin stack $X$  we define  the space $\PreSp(X,n)$ of $n$-shifted pre-symplectic  structures on $X$ by
$
 \PreSp(X,n):= \oR\Gamma(\pi^0X, \PreSp(\tilde{\sO}_X)).
$
 
We then define  $\Sp(X,n) \subset \PreSp(X,n)$ to consist of the path components of objects $\omega$ which are non-degenerate in the sense that the essentially unique  morphism $\oR\hom_{\sO_X}(\bL^X, \sO_X)_{[n]} \to \bL^X$ making the diagrams
\[
 \begin{CD}
  \oR\sHom_{\sO_X}(\bL^X, \sO_X)(U)_{[n]} @>>>  \bL^X(U) \\
  @VVV @AAA \\
  \oR\sHom_{\tilde{\sO}_X(U)}(\Omega^1_{\tilde{\sO}_X(U)}, \sO_X(U))_{[n]} @>{\omega_2^{\sharp}}>> \Omega^1_{\tilde{\sO}_X(U)}
 \end{CD}
\]
commute is a quasi-isomorphism, where $\bL$ denotes the cotangent complex. 
  \end{definition}

  We can then define shifted isotropic and Lagrangian structures exactly as in Definition \ref{daggerLagdef}, with the non-degeneracy conditions being imposed on cotangent complexes.
 
 \begin{remark}\label{Nstackrmk2}
 As in Remark \ref{Nstackrmk},
 similar definitions exist for $N$-stacks, the only difference being that derived global sections are taken over   sites with higher categorical structure. 
  
  Alternatively, we could proceed as in \cite[\S \ref{poisson-DMsn}]{poisson} and define these structures  via \v Cech nerve-type constructions as in \cite{stacks2} and 
 \cite[\S 6]{2021lect} . Given a suitable simplicial hypercover $Z_{\bt}$ of $\pi^0X$, the definitions then reduce to taking cosimplicial homotopy limits, so that 
\[
 \PreSp(X,n):=\ho\Lim_{i \in \Delta} \PreSp(\Gamma(Z_i, \sO_X(Z_i),n),
\]
with the symplectic condition having to be imposed globally on path components in the  Artin  case. 

For more general homotopy-preserving functors $F \co dg_+\Affd\Alg_K^{\dagger,\loc} \to s\Set$, we can generalise this definition to look at the space $\oR\map(F,  \PreSp(-,n))$ of maps of homotopy-preserving presheaves from $F$ to $\PreSp(-,n)$, thus functorially associating an $n$-shifted pre-symplectic structure on $A$ to each point in $F(A)$. In order to define a subspace of shifted symplectic structures, we need $F$ to moreover be homogeneous with a cotangent complex in order to formulate the non-degeneracy condition. See \S \ref{hgsSpsn} for a slightly subtler, but more effective, definition of a space of symplectic structures on a homogeneous homotopy-preserving functor; it  admits a canonical map from the space  considered here, with the two agreeing for derived DM stacks.  


\end{remark}

 \section{Shifted symplectic structures associated to pro-\'etale sheaves}\label{proetsn}  

 \subsection{Pre-symplectic structures}
 
\begin{lemma}\label{anDRlemma}
 Given a quasi-free quasi-dagger dg algebra $A$, we have natural maps
 \[
  F^p\oL\DR(\uline{A}(S)^{\alg}) \to \uline{F^p\DR(A)}(S),
 \]
 functorial in  pro-finite sets $S$, where the domain is the algebraic derived de Rham cohomology of the CDGA $\uline{A}(S)$, as in \cite{poisson}.
 \end{lemma}
 \begin{proof}
  We automatically  have a natural map $  F^p\oL\DR(\uline{A}(S)^{\alg}) \to  F^p\DR(\uline{A}(S)^{\alg})$. Since the derivation $d \co A \to \Omega^1_A$ is continuous, it induces a derivation $\uline{A}(S) \to \uline{\Omega^1_A}(S)$, and hence an $\uline{A}(S)$-linear map $\Omega^1_{\uline{A}(S)^{\alg}} \to \uline{\Omega^1_A}(S)$, by universality. On passing to alternating powers, we then get compatible maps $\Omega^j_{\uline{A}(S)^{\alg}} \to \uline{\Omega^j_A}(S)$, and the desired morphisms arise by passing to product total complexes.
 \end{proof}
%
%
%

 \begin{proposition}\label{traceprop} 
Let $\ell$ be the unique integral prime in $\m_K$, and assume that $X$ is a topologically Noetherian scheme with a constructible  $\Zl$-complex $\bD$ such that the \'etale  cohomology groups $\H^i (X_{\et}, \bD/\ell^n)$ vanish for $i>d$ and are equipped with compatible  maps 
\[
\tr \co \H^d (X_{\et}, \bD/\ell^n)\to \Z/\ell^n.
\]

Then given a quasi-dagger dg $K$-algebra $A$ and a finite $A$-module $M$ in complexes, we have a natural map
\[
 \oR\Gamma(X_{\pro\et},\uline{M}_X\hten_{\Zl}\bD ) \to M[-d] 
\]
in the $\infty$-category of cochain complexes, depending only on the structure of $M$ as a  
complex of topological abelian groups.
 \end{proposition}
\begin{proof}
The hypotheses give us $\Z/\ell^n$-linear zigzags
\[
 \oR\Gamma(X_{\et}, \bD/\ell^n) \xla{\sim} \tau^{\le d}  \oR\Gamma(X_{\et}, \bD/\ell^n) \to \H^d(X_{\et}, \bD/\ell^n)[-d] \to \Z/\ell^n[-d].
\]
Since the rings $\Z/\ell^n$ are discrete, \cite[Corollary 5.1.6]{BhattScholzeProEtale} then gives $ \oR\Gamma(X_{\et}, \bD/\ell^n) \simeq \oR\Gamma(X_{\pro\et}, \bD/\ell^n)$. We thus have a map
\[
 \oR\Gamma(X_{\pro\et}, \uline{N}_X\hten_{\Zl}\bD) \to N[-d]
\]
for all finite $\ell$-torsion abelian groups $N$. Passing to derived inverse limits and completing the tensor product, this map extends to pro-$\ell$ abelian groups $N$. 

Given a finite module $N$ over an  affinoid $K$-algebra $B$, as in Lemma \ref{exactmodsaffdlemma} we can take elements of norm $\le 1$ to give a pro-$\ell$ algebra $B^o$ and a finite $B^o$-module (in particular pro-$\ell$) $N^o$ with $N\cong N^o\ten \Q$. 
By \cite[Lemma 6.8.12]{BhattScholzeProEtale}, the functor $\oR\Gamma(X_{\pro\et},-)$ commutes with filtered colimits, giving us
\[
 \oR\Gamma(X_{\pro\et}, \uline{N}_X\hten_{\Zl}\bD) \to N[-d]
\]
for all such modules $N$. Passing to filtered colimits again, this extends to the case where $B$ is a dg  affinoid $K$-algebra, and $N$ a $B$-module in complexes.

Now, our quasi-dagger dg $K$-algebra $A$ is a filtered colimit of affinoid  dg $K$-algebras $B$, and $M$ is of the form $A\ten_BN$ for some such $B$ and some finite $B$-module $N$. Passing to filtered colimits yet again gives us the desired map
\[
  \oR\Gamma(X_{\pro\et}, \uline{M}_X\hten_{\Zl}\bD) \to M[-d]
\]
of complexes.
 \end{proof}

 \begin{examples}\label{traceex}
 If $X$ is  an $\ell$-coprime proper scheme, then trace maps $\tr$ of the form required in Proposition \ref{traceprop} arise from the six functors formalism (and specifically Poincar\'e duality)  whenever we have a form of duality on the base. Examples include the cases:
 
 \begin{enumerate}[itemsep=5pt, parsep=0pt]
 \item
 $X$ is a proper separated scheme over a separably closed field $k$ prime to $\ell$. In particular, when $X$ is moreover smooth  of dimension $m$ over $k$, we just have $\bD\simeq \uline{\Zl(m)}$, with trace
 \[
 \H^{2m}(X_{\pro\et},\uline{\Zl(m)})\to \Zl
 \]
as in \cite[Theorem VI.11.1]{Mi} (Poincar\'e duality).

  \item\label{tracelocex} $X$ is a proper separated scheme over a local field $k$ containing $\ell^{-1}$ with finite residue field, 
where the trace is given  by combining relative Poincar\'e duality with local Tate duality (e.g. \cite[proof of Proposition II.1.5]{milneArithDuality} or \cite[\S II.5.5, Proposition 19]{galoisienne} and the subsequent comment). 
 When $X$ is moreover smooth of dimension $m$ over $k$, we have $\bD\simeq \uline{\Zl(m+1)}$, with trace 
 \[
 \H^{2m+2}(X_{\pro\et},\uline{\Zl(m+1)})\to \Zl;
 \]
 for $f \co X \to \Spec k$, this arises as
 \begin{align*}
  \oR\Gamma(X_{\pro\et},\uline{\Zl(m+1)}) \simeq &\oR\Gamma((\Spec k)_{\pro\et}, \oR f_*\uline{\Zl(m+1)})\\
&\to \oR\Gamma((\Spec k)_{\pro\et}, \uline{\Zl(1)})[-2m] \tag{Poincar\'e}\\
&\to \Zl[-2-2m] \tag{local Tate}.
 \end{align*}

\item $X$ is a proper separated scheme over a finite field $k$ prime to $\ell$. This follows similarly, by combining relative Poincar\'e duality \cite[Proposition VI.11.8]{Mi} with duality for $\Gal(k)\cong \hat{\Z}$-representations.  In particular, when $X$ is moreover smooth  of dimension $m$ over $k$, we have $\bD\simeq \uline{\Zl(m)}$, with trace
\[
\H^{2m+1}(X_{\pro\et},\uline{\Zl(m)})\to \Zl. 
\]

\item\label{tubular} An example with a similar flavour, but not strictly of the form in Proposition \ref{traceprop}, is given by starting from a variety $U$ over one of the fields $k$ above, and taking an open immersion $j \co U \into \bar{U}$ into a complete variety, with complement $i \co Z \to \bar{U}$.  If we have a trace map $\oR \Gamma_c(U_{\pro\et}, \bD)[r] \to \Zl$ from cohomology with compact supports, then there is a composite trace map
\[
 \oR\Gamma(Z_{\pro\et}, i^*\oR j_*\bD) \to \oR\Gamma_c(U _{\pro\et}, \bD)[1]  \to \Zl[1-r],
\]
via the equivalence 
\[
\oR\Gamma(Z_{\pro\et}, i^*\oR j_*\bD) \simeq \cone( \oR\Gamma_c(U _{\pro\et}, \bD) \to  \oR\Gamma(U _{\pro\et}, \bD) ) ,
\]
which follows from \cite[Prop III.1.2.9]{Mi}. The proof of Proposition \ref{traceprop} thus gives us a natural composite  map
\[
 \oR\Gamma(Z_{\pro\et}, i^*\oR j_*(\uline{M}_U\hten_{\Zl}\bD) ) \to M[1-r]
\]
for any finite $A$-module $M$ in complexes over a quasi-dagger dg $K$-algebra $A$.

Unlike cohomology with compact supports, the cohomology theory $ \oR\Gamma(Z_{\pro\et}, i^*\oR j_* -)$ carries a cup product; the resulting pairing induced by the trace  then corresponds to the Poincar\'e duality pairing between $\oR\Gamma(U,-)$ and $\oR\Gamma_c(U,-)$. This leads to traces acting on cohomology one degree lower than those above.
Explicitly, when $U$ is smooth of dimension $m$  over a separably closed base field, we have $\bD\simeq \uline{\Zl(m)}$ with trace given by the composite 
\[
 \H^{2m-1}(Z_{\pro\et},i^*\oR j_* \uline{\Zl(m)})\to \H^{2m}_c(U_{\pro\et},\uline{\Zl(m)}) \to  \Zl;
\]
there are similar statements for local and finite fields, but with different twists and shifts.

In fact, we can interpret this example in terms of  Deligne's oriented fibre product 
$Z \overleftarrow{\by}_{\bar{U}}U$,
known as the 
punctured
tubular neighbourhood of $Z$ in $\bar{U}$, since that satisfies
\[
\oR\Gamma(Z_{\et} \overleftarrow{\by}_{\bar{U}_{\et}}U_{\et},\pr_U^*L) \simeq \oR\Gamma(Z_{\et}, i^*\oR j_*L)
\]
for constructible sheaves $L$ on $U$, where $\pr_U \co Z \overleftarrow{\by}_{\bar{U}}U \to U$ is the second projection; see for instance \cite[2.2, 2.5 and 6.4.4]{illusieVanishing}. This almost fits into the framework of  Proposition \ref{traceprop}, except that we have to work with the \'etale and pro-\'etale sites of  $Z \overleftarrow{\by}_{\bar{U}}U$ rather than of a locally Noetherian scheme. The construction $M \mapsto \uline{M}$ there would be defined  using the cosheaf $\pi_0 \circ \pr_U$ of pro-finite sets, so
\[
\uline{M}_{Z \overleftarrow{\by}_{\bar{U}}U}:=\Hom_{\cts}(\pi_0 \circ \pr_U(-),M)= \pr_U^*\Hom_{\cts}(\pi_0(-),M)=\pr_U^*\uline{M}_U.
\]

The 
punctured
tubular neighbourhood is a pro-\'etale analogue of the  boundary $\pd U$ at infinity of $U$ featuring for instance in \cite[Definition 4.2]{PantevToenLocSys}.
\end{enumerate}
   \end{examples}
 
\begin{corollary}\label{tracecor}
If $X$ is  a topologically Noetherian scheme satisfying the conditions of Proposition \ref{traceprop} with an isomorphism $\bD\cong \Zl$, 
then for any 
 $n$-shifted pre-symplectic (in the terminology of \cite{poisson}) derived $\infty$-geometric Artin stack $F \co dg_+\CAlg_K \to s\Set$,
 the   functor
  \[
   F(X_{\pro\et},-) \co dg_+\Affd\Alg^{\loc,\dagger}_K \to s\Set
  \]
of Definition \ref{FXproetdef} carries a functorial $(n-d)$-shifted pre-symplectic structure at all points; in particular,  this implies that any formally \'etale 
map  $Y \to F(X_{\pro\et},-)$ from a dg dagger-analytic Artin $\infty$-stack $Y$ induces an $(n-d)$-shifted pre-symplectic structure on $Y$.
\end{corollary}
 \begin{proof}
  The space $\PreSp(F,n)$ of  $n$-shifted pre-symplectic structures on $F$ is the space of maps from $F$ to the functor $\PreSp^{\alg}(-,n)\co dg_+\CAlg_K \to s\Set$ defined similarly to $\PreSp(-,n)$ in Definition \ref{PreSpdef}, but based on algebraic cotangent complexes $\oL\Omega^1_{A^{\alg}}$.

 For any  quasi-free quasi-dagger dg algebra $A$, Lemma \ref{anDRlemma} combines with Proposition \ref{traceprop} to give us maps 
 \[
\oR\Gamma(X_{\pro\et},  F^2\oL\DR(\uline{A}_X^{\alg})) \to \oR\Gamma(X_{\pro\et}, \uline{F^2\DR(A)}_X) \to F^2\DR(A)[-d],
\] 
and hence
\[
 \oR\Gamma(X_{\pro\et},\PreSp^{\alg}(\uline{A}_X^{\alg},n)) \to \PreSp(A,n-d).
\]

Since $F$ maps quasi-isomorphisms to weak equivalences, Proposition \ref{affdsubEFCprop}  applies, and in particular it suffices to work with the restriction of  $F(X_{\pro\et},-)$ to objects which are quasi-free.
Because $ \PreSp^{\alg}(F,n) \simeq \oR\map(F,\PreSp^{\alg}(-,n))$, the maps above combine to give a composite natural transformation
\[
 \PreSp^{\alg}(F,n)   \by  F(X_{\pro\et},-) \to \PreSp^{\alg}(X_{\pro\et},-,n) \to  \PreSp(-,n-d),
\]
of functors on quasi-free localised dagger dg algebras. 

By adjunction, we can rephrase this as a morphism
\[
 \PreSp^{\alg}(F,n) \to \oR\map(F(X_{\pro\et},-), \PreSp(-,n-d))
\]
in the $\infty$-category of simplicial sets, where the space of maps is taken in the $\infty$-category of homotopy-preserving simplicial set-valued  functors on $dg_+\Affd\Alg^{\loc,\dagger}_K$.
\end{proof}

\subsection{Duality  over cyclotomic fields --- Selmer and Iwasawa theory}\label{cycdualitysn}

In order to apply Proposition \ref{traceprop} to construct shifted  pre-symplectic structures on conventional derived moduli stacks of $\ell$-adic sheaves, we need the dualising object $\bD$ to be isomorphic to the trivial pro-sheaf $\Zl$. In practice, this only tends to occur for such duality theories over separably closed fields. However, in order to trivialise Tate twists it suffices for the base to contain the roots $\mu_{\ell^{\infty}}$ of unity, leading to examples of a more arithmetic flavour.  In this brief interlude, we develop the necessary theory, involving subtler duality results which do not  fit directly into  the six functors formalism.


We first fix some notation, partly based on that in \cite{nekovarSelmerComplexes}. 
%
%
%
$\Gamma$ will be a pro-finite group isomorphic to $\Zl$, with $\ell$-completed group ring $\Lambda:= \Zl[\Gamma]$. If $u \in \Zl[\Gamma]$  is a  generator for $\Gamma$ and $t=u-1$, then $\Lambda \cong \Zl \llb t \rrb$.
Let $I$ be the kernel of the augmentation map $\Lambda \to \Zl$ (sending all elements of $\Gamma$ to $1$), so $I\cong t\Lambda$. We then write $I^*:=\Hom_{\Lambda}(I,\Lambda)$, so $I^* \cong t^{-1}\Lambda$.

\begin{lemma}\label{traceIwlemma}
 If $f \co X_{\infty} \to X$ is a pro-\'etale Galois cover with Galois group $\Gamma$ isomorphic to $\Zl$, then for any \'etale sheaf $N$ in abelian $\ell$-power torsion groups on $X$, we have a trace map
 \[
  \tr_{\Gamma} \co 
  f_{*}f^{-1}N
  \to \Gamma\ten_{\Zl}N[1].
 \]
\end{lemma}
\begin{proof} 
 We  have a $\Gamma$-equivariant isomorphism $I/I^2 \cong \Gamma$ (with $\Gamma$ acting trivially on itself), under which $\gamma \in \Gamma$ corresponds to the class $\gamma -1 +I^2$. Hence  $I^*$  is a $\Gamma$-equivariant  extension of  $\Hom_{\Zl}(\Gamma,\Zl) \cong t^{-1}\Zl$ by $\Lambda$.

Write $\tilde{\Lambda}$ for the pro-local system   on $X$ associated to the regular $\Gamma$-action on $\Lambda$, and similarly for $I$ and $I^*$. 
Explicitly, for subgroups $\ell^m\Gamma \subset \Gamma$ of finite index and $p_m \co X_{\infty}/\ell^m\Gamma \to X$ the induced map, we have $\tilde{\Lambda}= \Lim_n\Lim_m p_{m!}(\Z/\ell^n)$; since $p_m$ is finite \'etale, $p_{m!}=p_{m*}$.
Thus we have an isomorphism
\[
 \Hom_{\cts}(\tilde{\Lambda},N) \cong \LLim_m \LLim_n\Hom( p_{m!}(\Z/\ell^n),N) \cong \LLim_m p_{m*}p_m^{-1}N \cong f_*f^{-1}N
\]
of $\ell$-power torsion \'etale sheaves on $X$. We thus have an $\ell$-power torsion \'etale sheaf $\Xi(N):=\Hom_{\cts}(\tilde{I}^*,N)$ which is canonically an extension of $ f_*f^{-1}N$ by $\Gamma\ten_{\Zl}N$, so the corresponding extension class gives us the trace map $\tr_{\Gamma}$.
 \end{proof}

\begin{corollary}\label{tracepropinfty}
 In the setting of Proposition \ref{traceprop}, if $f \co X_{\infty} \to X$ is a pro-\'etale Galois cover with Galois group $\Gamma$ isomorphic to $\Zl$, then given a quasi-dagger dg $K$-algebra $A$ and a finite $A$-module $M$ in complexes, we have a natural map
\[
\tr_{X_{\infty},X} \co \oR\Gamma(X_{\infty,\pro\et},\uline{M}_X\hten_{\Zl}f^{-1}\bD ) \to \Gamma\ten_{\Zl}M[1-d] 
\]
in the $\infty$-category of cochain complexes, depending only on the structure of $M$ as a  
complex of topological abelian groups.
\end{corollary}
\begin{proof}
 As in the proof of Proposition \ref{traceprop}, we have compatible maps $\tr_X \co \oR\Gamma(X_{\et}, \bD/\ell^n) \to  (\Z/\ell^n)[-d]$, which we can compose with the map $\tr_{\Gamma}$ of Lemma \ref{traceIwlemma} to give 
 \[
 \tr_{X_{\infty},X} \co \oR\Gamma(X_{\infty,\et},f^{-1}\bD/\ell^n) \to (\Gamma/\ell^n\Gamma)[1-d],
 \]
and the remainder of the proof then follows exactly as in Proposition \ref{traceprop}.
\end{proof}

We now introduce some more notation, allowing for more general bases. Write $\Gamma_n:= \ell^n\Gamma$ and $\Lambda_n:= \Zl[\Gamma_n]$. Then  $u^{\ell^n}$ is a  generator for $\Gamma_n$, and for $t_n:=u^{\ell^n}-1$ we have $\Lambda_n \cong \Zl \llb t_n \rrb$. We then let
 $I_n$ be the kernel of the augmentation map $\Lambda_n \to \Zl$ and $I_n^*:=\Hom_{\Lambda_n}(I_n,\Lambda_n)$, so $I_n^* \cong t_n^{-1}\Lambda_n$.

\begin{lemma}\label{cyctraceindeptlemma}
 The trace map in Corollary \ref{tracepropinfty} is independent of the base, in the sense that for the quotient $f_n \co X_{\infty} \to X_n$ by 
$\Gamma_n \subset \Gamma$,
we have an equivalence between $\tr_{X_{\infty},X_n}$ and the composition of $\tr_{X_{\infty},X}$ with the isomorphism  
$\ell^n \co \Gamma \to \Gamma_n$. 
\end{lemma}
\begin{proof}
We need to compare the map $\tr_{\Gamma}$ from the proof of Corollary \ref{tracepropinfty} with the composition  
\[
 f_*f^{-1}N \xra{p_{n*}\tr_{\Gamma_n}p_n^{-1}}  \Gamma_n\ten_{\Zl}p_{n*}p_n^{-1}N[1] \xra{\tr_{X_n/X} }\Gamma_n\ten_{\Zl}N[1],
\]
where $\tr_{X_n/X} \co p_{n*}p_n^{-1}\bD \to \bD$ is the relative trace of $p_n \co X_n \to X$. 

Write $\Xi_n(N)$ for the extension $\Hom_{\cts}(\tilde{I}_n^*,N)$ of $f_{n*}f_n^{-1}N$ by $\Gamma_n\ten_{\Zl}N$ coming from the proof of Lemma \ref{traceIwlemma} applied to $f_n$. 
Now taking $N$ to be an \'etale sheaf in abelian $\ell$-power torsion groups on $X$, we have
\[
\Hom_{\cts}(I_n^*\ten_{\Lambda_n} \tilde{\Lambda},N) \cong \Hom_{\cts}(p_{n!}\tilde{I}_n^*,N) \cong p_{n*}\Xi_n(p_n^{-1}N),
\]
an extension of $f_{*}f^{-1}N$ by $ \Gamma_n\ten_{\Zl}p_{n*}p_n^{-1}N\cong \Hom_{\Lambda_n}(\tilde{\Lambda}, \Gamma_n\ten_{\Zl}N)$. 

Now, the relative trace map $p_{n*}p_n^{-1}N \cong \Hom(\tilde{\Lambda}\ten_{\tilde{\Lambda}_n}\Zl,N) \to N $ is given by evaluation at the sum $\phi_n$ of the elements of $\Gamma/\Gamma_n$. The extension of $f_*f^{-1}N$ by $\Gamma_n\ten_{\Zl}N$ corresponding to  the composite trace $\tr_{X_n/X} \circ p_{n*}\tr_{\Gamma_n}p_n^{-1}$ is thus given by maps out of the pullback
\[
 I_n^*\ten_{\Lambda_n} \tilde{\Lambda} \xra{\alpha} \Hom_{\Zl}(\Gamma_n,\Zl)\ten_{\Zl} \Zl[\widetilde{\Gamma/\Gamma_n}] 
 \xla{\phi_n} \Hom_{\Zl}(\Gamma_n,\Zl),
\]
where $\alpha$ is the quotient map $-\ten_{\Lambda_n}\Zl$.

 Since  $t_n =(t+1)^{\ell^n} -1$ and  $\phi_n=\sum_{i=0}^{\ell^n -1} [u]^i$, we have $\phi_n=[t_n/t]$. In co-ordinate terms, $I_n^*\ten_{\Lambda_n} \Lambda \cong t_n^{-1}\L$, and then we have 
 \[
 \alpha(t_n^{-1}f(t))(u^{\ell^n})= [f(t)] \in \Lambda/(t_n)= \Zl[\Gamma/\Gamma_n].
 \]
 Thus the pullback consists of elements $t_n^{-1}f(t)$ such that $t_n/t$ divides $f(t)$, so is given by  $t^{-1}\tilde{\Lambda}= \tilde{I}^*$. The map $I^* \to \Hom_{\Zl}(\Gamma_n,\Zl) $ to the second factor  sends $t^{-1}$ to the dual of  $u^{\ell^n}$, so corresponds to the composition of the canonical map $ I^* \to \Hom_{\Zl}(\Gamma,\Zl)$ with the isomorphism $\ell^n \co \Gamma \to \Gamma_n$, as required.
\end{proof}

We now look to understand duality for this trace map. The key observation is that for a finite $\ell$-group $N$, we have a $\Zl\llb t \rrb$-linear isomorphism $  N(\!(t)\!)/N\llb t \rrb \xra{\simeq} \Hom_{\Zl}(\Zl\llb t \rrb, N)$ given by multiplication followed by evaluation at the  coefficient of $t^{-1}$.
Writing $N^{\vee}$ for the Pontrjagin dual $\Hom_{\Zl}(N,\Ql/\Zl)$, the tensor-Hom adjunction gives a map
\[
N^{\vee}\llb t \rrb\hten \Hom_{\Zl}(\Zl\llb t \rrb, N) \to N^{\vee}\hten N \to \Ql/\Zl,
\]
inducing an isomorphism $N^{\vee}\llb t \rrb \to \Hom_{\Zl}(\Zl\llb t \rrb, N)^{\vee}$.

We can thus think of the module $N (\!(t)\!)$ as an extension of $  N(\!(t)\!)/N\llb t \rrb \cong f_*f^{-1}N$   by the Pontrjagin dual construction $N\llb t \rrb  \cong (f_*f^{-1}(N^{\vee}))^{\vee}$, and regard the latter as $f_!f^{-1}N$ since $f_!f^{-1}$ is left adjoint to $f_*f^{-1}$ whenever both are defined.  This extension only exists as a  Tate object in abelian sheaves, since it combines ind- and pro-objects, induced by $f_*f^{-1}$ and $f_!f^{-1}$ respectively, because $f$ does not have finite type.


\begin{lemma}\label{cycdualitylemma}
 Take a discrete $\Z/\ell^m$-algebra $R$ and \'etale sheaves $M,N$ of $R$-modules on $X$.  The following are naturally quasi-isomorphic:
\begin{enumerate}
\item  the complex given by the cone of the map 
 \[
  \tau \co \oR\HHom_{R,X_{\infty}}(f^{-1}M, f^{-1}N)[-1] \to \Gamma\ten_{\Zl}\oR\HHom_{R,X} ( f_*f^{-1}M, N) 
  \]
induced by the trace map $f_*f^{-1}N \to \Gamma\ten_{\Zl}N[1]$   of Lemma \ref{traceIwlemma},

\item the complex $\oR\HHom_{R[\Gamma]}(R(\Gamma -1),  \oR\HHom_{R,X_{\infty}}(f^{-1}M, f^{-1}N)  )$, for the localisation $R(\Gamma -1)$ of $R[\Gamma]$ at elements $\gamma -1$ for $\gamma \in \Gamma \setminus \{1\}$,

\item when $M$ is constructible, the complex $\Gamma\ten_{\Zl}\oR\HHom_{R,X} ( f_*f^{-1}M, N)\ten_{ R[\Gamma]}R(\Gamma -1)$,

\item the complex $\oR\Lim_n \oR\HHom_{R,X_n}(p_n^{-1}M, \Xi_n(p_n^{-1}N))$, for $\Xi_n$ as in the proof of Lemma \ref{cyctraceindeptlemma}.
\end{enumerate}
\end{lemma}
\begin{proof}
Under the isomorphism $R[\Gamma] \cong R \llb t \rrb$ given by a choice $u=t+1$ of generator for $\Gamma$, the localisation $R(\Gamma -1)$ is simply given by 
$R(\!(t)\!)$; the elements $u^n-1$ are all invertible in the local ring $(\Z/\ell^m)(\!(t)\!)$ because their images in the residue  field $\bFl(\!(t)\!)$ are non-zero.
 
Given a morphism $\theta \co K[-1] \to L$ of  $R \llb t \rrb$-modules with:
\begin{itemize}
 \item[(a)] $t$ acting quasi-isomorphically on $\cone(\theta)$, such that
\item[(b)] $K\ten_{R\llb t \rrb} R(\!(t)\!)\simeq 0$, resp. 
\item[(c)] $\oR\HHom_{R\llb t \rrb} (R(\!(t)\!),L)\simeq 0$, 
\end{itemize}
 we automatically have quasi-isomorphisms
\begin{align*}
  \cone(\theta) \to \cone(\theta)\ten_{R\llb t \rrb} R(\!(t)\!) \la  L\ten_{R\llb t \rrb} R(\!(t)\!),  \text{ resp. }\\
  \cone(\theta) \la \oR\HHom_{R\llb t \rrb} (R(\!(t)\!),\cone(\theta)) \to \oR\HHom_{R\llb t \rrb} (R(\!(t)\!),K).
 \end{align*}
Since $R[\Gamma] \cong R \llb t \rrb$, equivalence of the first three statements in the lemma will follow if we can establish these properties when $K$ and $L$ are the second and third complexes above.

Now, $\oR\HHom_{R,X_{\infty}}(f^{-1}M, f^{-1}N) \simeq \oR\HHom_{R,X}(M, f_*f^{-1}N)$ and since   $f_*f^{-1}N \cong \sHom_{\pro(\Ab)}(\tilde{\Lambda}, N)$ as in the proof of Lemma \ref{traceIwlemma}, 
we have 
%
\begin{align*}
&\cone(R\llb t \rrb \to t^{-1}R\llb t \rrb )\ten_{R\llb t \rrb}\oR\HHom_{R,X_{\infty}}(f^{-1}M, f^{-1}N)\\
&\simeq \oR\HHom_{R,X}(M, \cone(f_*f^{-1}N \xra{t} f_*f^{-1}N) )\\
&\simeq \oR\HHom_{R,X}(M, \sHom_{\pro(\Ab)}(\cone(t\tilde{\Lambda}\to \tilde{\Lambda}),N))[1]\\
&\simeq \oR\HHom_{R,X}(M,N)[1].
\end{align*}

Similarly, since the cocone of $t$ acting on $ f_*f^{-1}M$ is quasi-isomorphic to $M$, 
we have
\[
 \cone(tR\llb t \rrb \to R\llb t \rrb )\ten_{R\llb t \rrb} \oR\HHom_{R,X} ( f_*f^{-1}M, N) \simeq \oR\HHom_{R,X}(M,N)
\]

Thus $\tau$ induces an quasi-isomorphism on the respective cones with respect to $t$, so $t$ acts quasi-isomorphically on $\cone(\tau)$, giving property (a). 

 Now, 
\[
\oR\HHom_{R\llb t \rrb} (R(\!(t)\!), \oR\HHom_{R,X} ( f_*f^{-1}M, N)) \simeq   \oR\HHom_{R,X} ( R(\!(t)\!)\ten_{R\llb t \rrb}  f_*f^{-1}M, N)),
\]
but the elements of  $ f_*f^{-1}M$  are all $t$-power torsion, so this is acyclic, giving (c).

Similarly,  
\[
R(\!(t)\!)\ten_{R\llb t \rrb}\oR\HHom_{R,X}(M, f_*f^{-1}N) \simeq  \oR\HHom_{R,X}(M, R(\!(t)\!)\ten_{R\llb t \rrb}f_*f^{-1}N) 
\]
when $M$ is constructible, and this is again acyclic, giving (b).

It remains to prove the equivalence with  $\oR\Lim_n \oR\HHom_{R,X}(M, p_{n*}\Xi_n(p_n^{-1}N))$. The comparison of Lemma \ref{cyctraceindeptlemma} gives a map to the latter complex, which we can express in co-ordinates 
as
\begin{align*}
 &\oR\Lim_n \oR\HHom_{R,X}(M,   \Hom_{\Lambda_n,\cts}(\tilde{\Lambda},  \Hom_{\cts}(t_n^{-1}\Lambda_n ,N))\\
 &\simeq  \oR\Lim_n \oR\HHom_{R,X}(M,\Hom_{\cts}(t_n^{-1}\tilde{\Lambda},N)).
\end{align*}
Since $t_n \equiv t^{\ell^n} \mod \ell$ and $R$ is a $\Z/\ell^m$-algebra, the natural inclusion $\LLim_n t_n^{-1}R\llb t \rrb \to R(\!(t)\!)$ is an isomorphism, so this is quasi-isomorphic to $\oR\HHom_{R\llb t \rrb}(R (\!(t)\!),  \oR\HHom_{R,X}(M, f_*f^{-1}N)  )$, as required.
%
%
%
%
\end{proof}

\begin{definition}
 Given an $\ell$-adically complete $\Zl$-algebra $R$, define the ring $R\{\Gamma -1\}$ to be the $\ell$-completed localisation of $R[\Gamma]$ at the elements $\gamma -1$ for $\gamma \in \Gamma\setminus \{1\}$. 
\end{definition}
Explicitly, for the isomorphism $R[\Gamma] \cong R\llb t \rrb$, we have  $R\{\Gamma -1\} \cong  \Lim_n\left( (R/\ell^n)(\!(t)\!)\right)$, a ring usually denoted by   $R\{\!\{t\}\!\}$ (see e.g.\ \cite{zhukovHigherDimLocalFields}). It consists of elements $\sum_{i=-\infty}^{\infty} a_i t^i$ with $a_i \in R$ such that $a_i \to 0$ as $i \to -\infty$. In particular,  $ R\{\!\{t\}\!\}$ is larger than the ring $ R(\!(t)\!)$ of Laurent series, and  the elements $t_n= u^{\ell^n}-1$ are not invertible in the latter; they are invertible in $ \Zl\{\!\{t\}\!\}$ and hence $R\{\!\{t\}\!\} $ because their images in the residue field $\bFl(\!(t)\!)$ are non-zero.

\begin{lemma}\label{constrcycdualitylemma}
Assume that we have an $\ell$-torsion-free $\ell$-adically complete commutative Noetherian ring $R$, together with  constructible $\hat{R}_X$-complexes $M,N$ in the sense of \cite[\S 6.5]{BhattScholzeProEtale}. 

 The following are then naturally quasi-isomorphic:
\begin{enumerate}
\item  the complex given by the cone of the map 
 \[
  \tau \co \oR\HHom_{\uline{R}_{X_{\infty}}}(f^{-1}M, f^{-1}N)[-1] \to \Gamma\ten_{\Zl}\oR\HHom_{\uline{R}_X} ( f_*f^{-1}M, N) 
  \]
induced by the trace maps $f_*f^{-1}(N/\ell^n) \to \Gamma\ten_{\Zl}(N/\ell^n)[1]$   of Lemma \ref{traceIwlemma},

\item the complex $\oR\HHom_{R[\Gamma]}(R\{\Gamma -1\},  \oR\HHom_{\uline{R}_{X_{\infty}}}(f^{-1}M, f^{-1}N)  )$, 

\item when $\oR\HHom_{\uline{R}_X} ( f_*f^{-1}M, N)$ is perfect as an $R[\Gamma]$-module, the complex $\Gamma\ten_{\Zl}\oR\HHom_{R,X} ( f_*f^{-1}M, N)\ten_{ R[\Gamma]}^{\oL}R\{\Gamma -1\}$.
\end{enumerate}
\end{lemma}
\begin{proof}
Since $R$ is $\ell$-adically complete, writing $R_n:= R/\ell^n$ we have 

\begin{align*}
\oR\HHom_{\uline{R}_X}(M, N)& \simeq \oR\Lim_n \oR\HHom_{\uline{R_n}_X}(M\ten_{\uline{R}_X}^{\oL}\uline{R_n}_X  , N\ten_{\uline{R}_X}^{\oL}\uline{R_n}_X), \\
&\simeq\oR\Lim_n \oR\HHom_{\uline{R}_X}(M, N)\ten_R^{\oL}R_n,
\end{align*}
with similar expressions for the other complexes. Since the modules are constructible, we can then rephrase in terms of \'etale 
$R_n$-module sheaves
and substitute in Lemma \ref{cycdualitylemma}. 

The first equivalence then follows because
\begin{align*}
& \oR\HHom_{R[\Gamma]}(R\{\Gamma -1\},  \oR\HHom_{\uline{R}_{X_{\infty}}}(f^{-1}M, f^{-1}N)  )\\ & \simeq \oR\Lim_n  \oR\HHom_{R_n[\Gamma]}(R\{\Gamma -1\}/\ell^n,  \oR\HHom_{\uline{R}_{X_{\infty}}}(f^{-1}M, f^{-1}N)  \ten_R^{\oL}R_n)\\
&\simeq  \oR\Lim_n  \oR\HHom_{R_n[\Gamma]}(R_n(\Gamma -1),  \oR\HHom_{\uline{R_n}_{X_{\infty}}}(f^{-1}(M\ten_{\uline{R}_X}^{\oL}\uline{R_n}_X), f^{-1}(N\ten_{\uline{R}_X}^{\oL}\uline{R_n}_X))  )
 \end{align*}
with the equivalences $R\{\!\{t\}\!\}\ten^{\oL}_{R}R_n  \simeq R\{\!\{t\}\!\}/\ell^n \cong R_n(\!(t)\!)$ following from $\ell$-torsion-freeness of $R$. 


For the second equivalence,     we just need a quasi-isomorphism  
\[
\oR\HHom_{\uline{R}_X} ( f_*f^{-1}M, N)\ten_{ R\llb t \rrb}^{\oL}R \{\!\{t\}\!\} \simeq \oR\Lim_n (\oR\HHom_{\uline{R}_X} ( f_*f^{-1}M, N)\ten_{ R\llb t \rrb}^{\oL}R_n(\!(t)\!)).
\]
Because we are acting on a perfect $R\llb t \rrb$-module by hypothesis, this reduces to the equivalences
\[
R \{\!\{t\}\!\} \cong \Lim_n R_n(\!(t)\!) \simeq \oR\Lim_n R_n(\!(t)\!).
\]
\end{proof}

\begin{remark}
Lemma \ref{constrcycdualitylemma} is expressed in terms of global sections for simplicity of exposition. For sheaf-level statements, we would have to regard the objects $(f_!f^{-1}\sHom_R(M,N))\ten_{\Zl\llb t \rrb} \Zl\{\!\{t\}\!\}$ and $\sHom_{\Zl\llb t \rrb}( \Zl\{\!\{t\}\!\},f_*f^{-1}\sHom_R(M,N)) $  as  $2$-Tate objects in finite group sheaves in the sense of \cite[\S 7]{BraunlingGroechenigWolfsonTate}, a perspective which can also be used to describe their rationalisations. This is because the Laurent series from Lemma \ref{cycdualitylemma} are already $1$-Tate objects, so the passage from finite to $\ell$-adic produces $2$-Tate objects.
\end{remark}

Writing $\bD_R(M):= \sHom_{\uline{R}_X}(M,  \uline{R}_X\ten_{\Zl}\bD)$, we have:

\begin{proposition}\label{dualpropinfty}
Assume the trace maps $\tr \co \H^d (X_{\et}, \bD/\ell^n)\to \Z/\ell^n$ as in  Proposition \ref{traceprop} induce a perfect pairing on constructible $\Z/\ell^n$-modules. For any constructible $\uline{R}_X$-complex $M$,  the cone of the natural map
\[
 \oR\Gamma(X_{\infty},f^{-1}\bD_R(M)) \to \oR\HHom_R(\oR\Gamma(X_{\infty},f^{-1}M),R)[1-d]
\]
induced by the trace $\tr_{X_{\infty},X}$ of Corollary \ref{tracepropinfty} is then quasi-isomorphic to both:
\begin{enumerate}
\item the complex $\oR\HHom_{R[\Gamma]}(R\{\Gamma -1\}, \oR\Gamma(X_{\infty},f^{-1}\bD_R(M)) )$, and
\item  the complex $\Gamma\ten_{\Zl}\oR\HHom_{R,X} ( f_*f^{-1}M, \bD)\ten_{ R[\Gamma]}^{\oL}R\{\Gamma -1\}$.
\end{enumerate}
\end{proposition}
\begin{proof}
 Since $f^{-1}\bD_R(M)\cong \sHom_{\uline{R}_{X_{\infty}}}(f^{-1}M, f^{-1}\bD)$,   this follows immediately from Lemma \ref{constrcycdualitylemma} and the following quasi-isomorphisms, for $R_n=R/\ell^n$:
\begin{align*}
  \oR\HHom_R(\oR\Gamma(X_{\infty},f^{-1}M),R) &\simeq \oR\Lim_n \oR\HHom_R(\oR\Gamma(X_{\infty},f^{-1}M),R_n)\\
  &\simeq \oR\Lim_n \oR\HHom_{R_n}(\LLim_m \oR\Gamma(X_m,p_m^{-1}(M\ten^{\oL}_RR_n)),R_n)\\
  &\simeq \oR\Lim_n \oR\Lim_m \oR\Gamma(X, \bD_{R_n}(p_{m*}p_m^{-1}(M\ten^{\oL}_RR_n)))[-d]\\
  &\simeq \oR\Lim_n \oR\HHom_{\uline{R_n}_X}(  \LLim_m  (p_{m*}p_m^{-1}(M\ten^{\oL}_RR_n)), \bD\ten_{\Zl}\uline{R_n}_X )[-d]\\
  &\simeq \oR\HHom_{\uline{R}_X}( f_*f^{-1}M, \bD\ten_{\Zl}\uline{R}_X)[-d]. \qedhere
 \end{align*}

\end{proof}

\begin{remark}[Iwasawa cohomology]\label{Iwasawarmk}
The equivalence 
\[
 \oR\HHom_{\uline{R}_X} ( f_*f^{-1}M, N) \simeq \oR\Lim_n \oR\Lim_m \oR\HHom_{\uline{R_n}_X}(  p_{m*}p_m^{-1}(M\ten^{\oL}_RR_n), N\ten^{\oL}_RR_n)
\]  
from   the proof of Proposition \ref{dualpropinfty} is true more generally, not relying on the trace from Proposition \ref{traceprop}. Since the morphisms $p_m$ are finite \'etale covers, we can rewrite this as 
 $ \oR\Lim_m\oR\Lim_n \oR\HHom_{\uline{R_n}_X}(  M\ten^{\oL}_RR_n, p_{m*}p_m^!(N\ten^{\oL}_RR_n)) $. 
 
In particular, $\oR\HHom_{\uline{R}_X} ( f_*f^{-1}\uline{R}_X, N) \simeq \oR\Lim_m \oR\Gamma(X_{m,\pro\et}, p_m^!N)$, which is given by  the derived limit of the trace maps
\[
 \oR\Lim
 ( \ldots \xra{\tr_{X_{m+1}/X_m}}   \oR\Gamma(X_{m,\pro\et}, p_m^{-1}N) \xra{\tr_{X_m/X_{m-1}}} \ldots \xra{\tr_{X_1}/X} \oR\Gamma(X_{\pro\et}, N)
 ).
\]

By analogy with \cite{nekovarSelmerComplexes}, we could thus set 
\[
\oR\Gamma_{\mathrm{Iw}}(X_{\infty}/X, N):=\oR\HHom_{\uline{R}_X} ( f_*f^{-1}\uline{R}_X, N).
\]
Lemma \ref{constrcycdualitylemma} then implies that  the morphism $\oR\Gamma(X_{\infty,\pro\et},f^{-1}N) \to\oR\Gamma_{\mathrm{Iw}}(X_{\infty}/X, N)\ten_{\Zl}\Gamma[1]  $ induced by the trace of Lemma \ref{traceIwlemma} is a quasi-isomorphism  if and only if either $\Ext^*_{R[\Gamma]}(R\{\Gamma -1\}, \oR\Gamma(X_{\infty,\pro\et},f^{-1}N))=0 $ or $\H^*(\oR\Gamma_{\mathrm{Iw}}(X_{\infty}/X, N)\ten_{R[\Gamma]}^{\oL}R\{\Gamma -1\})=0 $, and similarly after base change. 

Since $\Ql\{\!\{t\}\!\}:=\Zl\{\!\{t\}\!\}\ten_{\Zl}\Ql$ is a field, this in particular gives an isomorphism between the respective cohomologies with $\Ql$ coefficients whenever  $\H^i_{\mathrm{Iw}} 
 (X_{\infty}/X, N)\ten_{\Zl}\Ql$ is torsion for some non-zero element $f$ of $\Zl \llb t \rrb$, because then $\H^i_{\mathrm{Iw}} 
 (X_{\infty}/X, N)\ten_{\Zl\llb t \rrb}\Ql\{\!\{t\}\!\} \cong 0$. In particular, this is guaranteed when   $\H^i_{\mathrm{Iw}} 
 (X_{\infty}/X, N)\ten_{\Zl}\Ql$ is finite-dimensional over $\Ql$, taking $f$ to be the characteristic polynomial of the $t$-action. A similar statement holds for the dual cohomology theory, with an isomorphism if $\Hom_{\Zl}(\Ql,\H^i(X_{\infty,\et},N\ten_{\Zl}(\Ql/\Zl)))$ is cotorsion for some such $f$; finite-dimensionality again suffices.
 
In particular, if $X$ and $X_{\infty}$ are respectively the spectra  of a local $p$-adic field $k$ (for $p \ne \ell$) containing $\mu_{\ell}$ and its cyclotomic extension $k_{\infty}=k(\mu_{\ell^{\infty}})$,  we have finite-dimensionality for all $\Ql$ local systems $V$. This follows because  a similar argument to \cite[II.5.4]{galoisienne} gives a normal pro-finite subgroup $H$ of $\Gal(\bar{k}/k_{\infty})$ of order prime to $\ell$, with quotient $\Zl$, so $\H^*_{\cts}(\Gal(\bar{k}/k_{\infty}),V) \cong \H^*_{\cts}(\Zl, V^H)$.
 
 
 
 
 Similarly to \cite{nekovarSelmerComplexes}, on suitable coefficients a non-degenerate trace induces a perfect pairing between 
 $\oR\Gamma(X_{\infty,\pro\et},-)$ and $\oR\Gamma_{\mathrm{Iw}}(X_{\infty}/X, -)[d] $. We can then interpret pairings induced by the trace of Corollary \ref{tracepropinfty} as combining this pairing with the connecting homomorphism of the exact triangle
 \[
  \oR\Gamma_{\mathrm{Iw}}(X_{\infty}/X, -)\ten_{\Zl}\Gamma \to \oR\Gamma_{\mathrm{Iw}}(X_{\infty}/X, -)\ten_{\Zl[\Gamma]}\Zl\{\Gamma -1\} \ten_{\Zl}\Gamma \to \oR\Gamma(X_{\infty,\pro\et},-)\to.
 \]
 For $\Z/\ell^n$-modules $N$, the image of this connecting homomorphism on cohomology is the kernel of the map $\H^*_{\mathrm{Iw}}(X_{\infty}/X, N) \to \H^*_{\mathrm{Iw}}(X_{\infty}/X, N)\ten_{(\Z/\ell^n)\llb t \rrb} (\Z/\ell^n)(\!(t)\!)$, so consists of $t$-torsion elements. Our pairing of degree $d-1$ on cohomology of $X_{\infty}$ is thus induced by the generalised Cassels--Tate pairing of degree $d+1$ on torsion Iwasawa cohomology from \cite[\S 10]{nekovarSelmerComplexes} where applicable.
\end{remark}

\begin{examples}\label{traceexIwasawa1}
 As in Examples \ref{traceex}, if $X$ is  an $\ell$-coprime proper scheme, then trace maps  of the form required in Corollary \ref{tracepropinfty} arise from the six functors formalism (and specifically Poincar\'e duality)  whenever we have a form of duality on the base, so those examples all have induced traces one degree lower on adjoining $\mu_{\ell^{\infty}}$.
 
 In particular, if $X$ is a smooth, proper scheme of dimension $m$ scheme over a non-Archimedean local field $k$ containing $\mu_{\ell}$, 
 with $X_{\infty}:= X\ten_kk_{\infty}$, for $k_{\infty}:= k(\mu_{\ell^{\infty}})$, then Corollary \ref{tracepropinfty} gives a trace 
 \[
  \H^{2m+1}(X_{\infty,\pro\et},\uline{\Gamma^*(m+1)})\to \Zl,
\]
where  $\Gamma:=  \Gal(k_{\infty}/k) \cong \Zl$ and we write $\Gamma^*:= \Hom_{\Zl}(\Gamma, \Zl)$. Note that the sheaf $\uline{\Gamma^*(m+1)}_{X_{\infty}}$ is non-canonically isomorphic to $\uline{\Zl}_{X_{\infty}}$. 
 
%
\end{examples}

\subsection{Symplectic structures}  

If the shifted pre-symplectic structure on $F$ is in fact symplectic and the trace from Proposition \ref{traceprop} or Corollary \ref{tracepropinfty} leads to a duality theory, then one might expect that the pre-symplectic structure from Corollary \ref{tracecor} is in fact symplectic, as happens for analogous constructions for topological spaces as in \cite[\S 2.1]{PTVV}. 
 In order to establish similar results,
 we will be cutting down to an open subfunctor on which duality does behave well.

Fix a topologically Noetherian scheme $X$ and
 assume that we have a trace map
 \[
 \tr \co \oR\Gamma(X_{\pro\et},\uline{M}_X\hten_{\Zl}\bD ) \to M[-d] 
\]
 as in Proposition \ref{traceprop} or Corollary \ref{tracepropinfty} (where $X$ here corresponds to $X_{\infty}$ there);  $\bD$ is constructible, and we denote its $\Zl$-linear dual by $\bD^*$.

\begin{definition}\label{weakdualdef}
 Given a quasi-dagger dg algebra $A$ and a presheaf $N$ of  $\uline{A}_X$-modules in chain complexes  on $X_{\pro\et}$, we say that $N$ satisfies weak duality with respect to the trace $\tr$  if for all morphisms $A \to C$ of quasi-dagger dg algebras,  
 the map
 \[
  \oR\HHom_{\uline{A}_X}( N, \uline{C}_X)\to \oR\HHom_A(\oR\Gamma(X,N\hten_{\Zl}\bD)[d],C)
 \]
induced by the pairing
\[
 \oR\Gamma(X,N\hten_{\Zl}\bD)\ten_A^{\oL} \oR\HHom_{\uline{A}_X}( N, \uline{C}_X) \to \oR\Gamma(X, \uline{C}_X\hten_{\Zl}\bD)\to C[-d]
\]
is a quasi-isomorphism, and $\oR\Gamma(X,N\hten_{\Zl}\bD)$ is a perfect $A$-module.
\end{definition}
Note that by taking $C=A \oplus M$, we can deduce a similar quasi-isomorphism for all $M\in dg_+\Coh_A$ in place of $C$.

\begin{lemma}\label{wdopenlemma}
 Given a surjection $g \co A \to B$ of quasi-dagger dg algebras with square-zero kernel $I$, a presheaf $N$ of  $\uline{A}_X$-modules in chain complexes satisfies weak duality if and only if the presheaf $\bar{N}:=N\ten^{\oL}_{\uline{A}_X} \uline{B}_X$  of  $\uline{B}_X$-modules does so.
  \end{lemma}
\begin{proof}
 We have  an exact triangle
 \[
 \bar{N} \ten^{\oL}_{\uline{B}_X}\uline{I}_X \to N \to \bar{N} \to ,
 \]
and the result then follows immediately by substitution into Definition \ref{weakdualdef}.
 \end{proof}
 
 \begin{lemma}\label{wdnilcompletelemma}
Given a quasi-dagger dg algebra $A$, a  module $N \in dg_+\Mod_{\uline{A}_X}$  
satisfies weak duality if and only if the presheaf $\bar{N}:= N\ten^{\oL}_{\uline{A}_X} \uline{\H_0A}_X$ of $\uline{\H_0A}_X$-modules does so.
 \end{lemma}
 \begin{proof}
If $N$ satisfies weak duality, then by base change $\bar{N}$ does so. Conversely, if  $\bar{N}$ satisfies weak duality then we may apply  Lemma \ref{wdopenlemma}  to the Postnikov tower $\{A/\tau_{>k}A\}$ of $A$ (a sequence of homotopy square-zero extensions),  
since Proposition \ref{exactmodsaffdprop} gives quasi-isomorphism-invariance, and thus deduce that the $ \uline{A/\tau_{>k}A}_X$-modules  $N\ten^{\oL}_{\uline{A}_X} \uline{A/\tau_{>k}A}_X$ satisfy weak duality  for all $k$. The result now follows by writing each  quasi-dagger dg $C$-algebra as $\ho\Lim_k (C/\tau_{>k}C)$ and taking homotopy limits in  Definition \ref{weakdualdef}.
 \end{proof}

 \begin{lemma}\label{constrlemma}
Let $X$ be a quasi-compact and quasi-separated scheme, and assume that we have an $\ell$-adically complete commutative Noetherian ring $R$, together with a constructible $\hat{R}_X$-complex $L$ in the sense of \cite[\S 6.5]{BhattScholzeProEtale}. Then for any    quasi-dagger algebra $C$ equipped with a $\Zl$-algebra homomorphism $R \to C$, 
  we have $ \oR\Gamma(X_{\pro\et},L\ten^{\oL}_{\hat{R}_X}\uline{C}_X)\simeq \oR\Gamma(X_{\pro\et},L)\ten^{\oL}_RC$ and $ \oR\HHom_{\hat{R}_X}( L, \uline{C}_X)
 \simeq \oR\HHom_{\hat{R}_X}( L,\hat{R}_X)\ten^{\oL}_RC$.
 \end{lemma}
\begin{proof}
As in \S \ref{proetdaggersn}, all quasi-dagger algebras $C$ over $R$ can be written as filtered colimits
of $\ell$-adically complete $R$-modules $C(\rho)$. Since $L$ is constructible, the compactness property of \cite[Lemmas 6.3.14 and 6.8.12 ]{BhattScholzeProEtale} gives the second statement.
Similarly, cohomology of $X$ preserves filtered colimits giving the first statement. 
\end{proof}



 \begin{examples}\label{wdex}
 Roughly speaking, for traces coming from Proposition \ref{traceprop}, a sufficient condition for an $\uline{A}_X$-module $N$ on an $\ell$-coprime proper scheme $X$ to satisfy weak duality is that  its sheafification  be constructible in an appropriate sense;  by Lemma \ref{wdnilcompletelemma}, we can reduce to looking at the $\uline{\H_0A}_X$-module $N\ten^{\oL}_{\uline{A}_X} \uline{\H_0A}_X$. Also note that the category of modules satisfying weak duality is triangulated and idempotent-complete.
 Although \cite[\S 6]{BhattScholzeProEtale} is  developed for fields over $\Ql$ rather than  arbitrary dagger  algebras, most of the arguments do generalise.  
 
 If there exists an $\ell$-adically complete Noetherian ring $R$, a constructible $\hat{R}_X$-complex $L$ in the sense of \cite[\S 6.5]{BhattScholzeProEtale} and a $K$-algebra homomorphism $R \to \H_0A$, then an  $\uline{\H_0A}_X$-module $N$ with sheafification $L\ten_{\hat{R}_X}\uline{\H_0A}_X$ satisfies weak duality whenever  $X$ is one of the schemes from Examples \ref{traceex}. This follows by combining Lemma \ref{constrlemma} with the results of \cite[\S 6.5]{BhattScholzeProEtale} and the six functors formalism, which leads to equivalences 
 \[
 \oR\Gamma(X_{\pro\et}, L^* ) \simeq \oR\HHom_{\Zl}(\oR\Gamma(X_{\pro\et},  L\ten_{\uline{\Zl}}\bD)[d],\Zl)
 \]
 in each case.
 
 For traces coming from Corollary \ref{tracepropinfty}, the situation is much more subtle, with duality  requiring acyclicity of the complex in Proposition \ref{dualpropinfty}. As in Remark \ref{Iwasawarmk}, cohomological finite-dimensionality ensures this when $A=\Ql$.
 \end{examples}
 
Every derived $\infty$-geometric Artin stack $F$ has a cotangent complex $\bL^F$, consisting of suitably functorial $A$-modules $\bL^{F,A,x}$ in chain complexes for each $x \in F(A)$.

\begin{lemma}\label{FXcotlemma}
Let $X$ be  a topologically Noetherian scheme satisfying the conditions of Proposition \ref{traceprop}  or those for $X_{\infty}$ in Corollary  \ref{tracepropinfty}, and  
  $F \co dg_+\CAlg_K \to s\Set$ a derived $\infty$-geometric Artin stack.
At any point $\phi \in F(X_{\pro\et},A)$ for which the presheaf $\bL^{F,\uline{A}_X, \phi} $  of $\uline{A}_X$-modules satisfies weak duality in the sense of Definition \ref{weakdualdef}, the functor
\[
T_{\phi}(F(X_{\pro\et},-),-) \co M \mapsto F(X_{\pro\et},A\oplus M)\by^h_{F(X_{\pro\et},A)}\{\phi\}
\]
on levelwise f.g. $A$-modules $M$ is represented by the perfect complex
\[
 \oR\Gamma(X,\bL^{F,\uline{A}_X, \phi}\hten_{\Zl}\bD)[d].
\]
\end{lemma}
 \begin{proof}
 Since homotopy limits commute, we have
 \begin{align*}
  F(X_{\pro\et},A\oplus M)\by^h_{F(X_{\pro\et},A)}\{\phi\} &\simeq \oR\Gamma(X_{\pro\et},F(\uline{A}_X\oplus \uline{M}_X)\by^h_{F(\uline{A}_X)} \{\phi\})\\
  &\simeq \oR\Gamma(X_{\pro\et},N^{-1}\tau_{\ge 0}\oR\sHom_{\uline{A}_X}(\bL^{F,\uline{A}_X,\phi}, \uline{M}_X))\\
  &\simeq N^{-1}\tau_{\ge 0}\oR\HHom_{\uline{A}_X}(\bL^{F,\uline{A}_X,\phi}, \uline{M}_X),
\end{align*}
where $N^{-1}$ is Dold--Kan denormalisation and $\tau_{\ge 0}$ denotes good truncation of a chain complex. 

By weak duality, this in turn is equivalent to 
\[
 N^{-1}\tau_{\ge 0}\oR\HHom_A(\oR\Gamma(X,\bL^{F,\uline{A}_X, \phi}\hten_{\Zl}\bD)[d],M)
\]
so is represented by $\oR\Gamma(X,\bL^{F,\uline{A}_X, \phi}\hten_{\Zl}\bD)[d] $.
\end{proof}

\begin{examples}\label{Lwdex} 
If $F=B\GL_n$, then a point  $\phi \in B\GL_n(X_{\pro\et},A)$ corresponds to a sheaf $M$ of rank $n$ projective $\uline{A}_X$-modules on $X_{\pro\et}$, with $\bL^{B\GL_n,\uline{A}_X, \phi}$ being quasi-isomorphic to $\End_{\uline{A}_X}(M)[-1]$.

More generally if $F=BG$ for an affine algebraic group $G$, then a point  $\phi \in BG(X_{\pro\et},A)$ corresponds to a $G(\uline{A}_X)$-torsor $P$  on $X_{\pro\et}$, with $\bL^{BG,\uline{A}_X, \phi}[1]$ being quasi-isomorphic to the adjoint vector bundle $\ad P = (P \by (\g\ten \uline{A}_X))/G(\uline{A}_X)$, where $\g$ is the Lie algebra of $G$ equipped with its adjoint action.

Generalising in another direction, if $F$ is the derived stack $\Perf$ of perfect complexes, then a point  $\phi \in \Perf(X_{\pro\et},A)$ corresponds to a hypersheaf $M$ of perfect $\uline{A}_X$-modules on $X_{\pro\et}$, with $\bL^{F,\uline{A}_X, \phi}$ being quasi-isomorphic to $\oR\EEnd_{\uline{A}_X}(M)[-1]$.

Whenever  $X$ is one of the schemes from Examples \ref{traceex}, for each of these functors $F$ it then follows from Examples \ref{wdex} that $\bL^{F,\uline{A}_X, \phi}$ satisfies weak duality whenever  $M$ is lisse (hence constructible), resp. $\ad P$ is lisse, resp.  $M$ is constructible.
\end{examples}

 \begin{corollary}\label{tracecor2a}
Let $X$ be  a topologically Noetherian scheme satisfying the conditions of Proposition \ref{traceprop}  or those for $X_{\infty}$ in Corollary  \ref{tracepropinfty} 
with $\bD|_X=\Zl$, 
and let $F \co dg_+\CAlg_K \to s\Set$ be an
 $n$-shifted symplectic  derived Artin $\infty$-stack.
 Then there is a natural $(n-d)$-shifted symplectic structure (in the sense of Remark \ref{Nstackrmk2}) on  the full subfunctor $ F(X_{\pro\et},-)^{wd} \subset F(X_{\pro\et},-) $ consisting of points $\phi$ for which the presheaf $\bL^{F,\uline{A}_X, \phi} $   satisfies weak duality.
\end{corollary}
\begin{proof}
The  $(n-d)$-shifted pre-symplectic structure $\omega$ of Corollary \ref{tracecor} 
 pulls back along the inclusion map $ F(X_{\pro\et},-)^{wd} \subset F(X_{\pro\et},-)$
to give an $(n-d)$-shifted pre-symplectic structure, and we need to check that it is non-degenerate. By  Lemma \ref{wdopenlemma}, the inclusion map is formally \'etale, so $\bL^{F(X_{\pro\et},-)^{wd},\phi}\simeq \bL^{F(X_{\pro\et},-),\phi}$ at all points $\phi$.

  By Lemma \ref{FXcotlemma} and its proof, we have
\begin{align*}
 \bL^{F(X_{\pro\et},-),\phi} &\simeq \oR\Gamma(X,\bL^{F,\uline{A}_X, \phi})[d],\\
 \oR\HHom_{A}(\bL^{F(X_{\pro\et},-),\phi} ,A) &\simeq \oR\Gamma(X_{\pro\et},\oR\sHom_{\uline{A}_X}(\bL^{F,\uline{A}_X,\phi}, \uline{A}_X))
\end{align*}
since $\omega$ is induced by an $n$-shifted symplectic structure on $F$, the map
\[
 \omega_2^{\sharp}\co \oR\HHom_{A}(\bL^{F(X_{\pro\et},-),\phi} ,A) \to (\bL^{F(X_{\pro\et},-),\phi} )_{[d-n]}
\]
then comes from derived global sections of  the quasi-isomorphism 
\[
 \oR\sHom_{\uline{A}_X}(\bL^{F,\uline{A}_X,\phi}, \uline{A}_X)\to \bL^{F,\uline{A}_X, \phi}[n],
\]
  so is itself a quasi-isomorphism.
\end{proof}

\begin{corollary}\label{tracecor2b}
Under the conditions of Corollary \ref{tracecor2a}, take an
 (underived)  dagger Artin analytic $\infty$-stack $Y$ equipped with a 
formally \'etale morphism 
\[
 \eta \co Y \to  \pi^0F(X_{\pro\et},-)
\]
of functors $\Affd\Alg^{\loc,\dagger}_K \to s\Set $, such that at all points $\phi$ in the image of $\eta$,  the presheaf $\bL^{F,\uline{A}_X, \phi} $   satisfies weak duality.
 
  Then  the functor $\tilde{Y} \co A \mapsto Y(\H_0A)\by^h_{F(X_{\pro\et},\H_0A)}F(X_{\pro\et},A) $ on $dg_+\Affd\Alg^{\loc,\dagger}_K$
 is a dg dagger  Artin analytic $\infty$-stack carrying a natural $(n-d)$-shifted symplectic structure.
 \end{corollary}
 \begin{proof}
 Since $F$ is homotopy-preserving, nilcomplete and homogeneous, Corollary \ref{preservehtpycor} implies that the same is true of $A \mapsto F(\uline{A}_X(U))$ for all $U$; it is thus also true for $F(X_{\pro\et},-)$ by passing to homotopy limits. Moreover, Lemmas \ref{wdopenlemma} and \ref{wdnilcompletelemma} imply that the same is true of  the full subfunctor $F(X_{\pro\et},-)^{wd}$; they moreover imply that $F(X_{\pro\et},A)^{wd}\simeq F(X_{\pro\et},\H_0A)^{wd}\by^h_{F(X_{\pro\et},\H_0A)}F(X_{\pro\et},A) $, which in particular gives us a map $\tilde{Y} \to F(X_{\pro\et},-)^{wd}$.
 
  Existence of a perfect cotangent complex from Lemma \ref{FXcotlemma} then ensures that for any dagger  algebra $A$ and any $\phi \in F(X_{\pro\et},A)$ for which   $\bL^{F,\uline{A}_X, \phi} $   satisfies weak duality, the $A$-module  $\Ext^i_A(\bL^{F(X_{\pro\et},-),\phi},A)$ is finitely generated. Moreover, the universal nature of Definition \ref{weakdualdef} ensures that  $\bL^{F(X_{\pro\et},-),\phi'} \simeq \bL^{F(X_{\pro\et},-),\phi}\ten_A^{\oL}A'$ for $\phi'$ the image of $\phi$ under $F(A)\to F(A')$, so flat base change gives
\[
\Ext^i_{A'}(\bL^{F(X_{\pro\et},-),\phi},A') \cong \Ext^i_A(\bL^{F(X_{\pro\et},-),\phi},A)\ten_AA'
\]
whenever $A \to A'$ is \'etale.

Thus all the conditions of Corollary \ref{opencutcor} are satisfied, making $\tilde{Y}$   a dg dagger  Artin analytic $\infty$-stack, with cotangent complex $\bL^{\tilde{Y},\phi} \simeq \bL^{F(X_{\pro\et},-),\phi}$. It has a shifted symplectic structure induced by pulling back along the formally \'etale map  $\tilde{Y} \to F(X_{\pro\et},-)$.
%
%
%
 \end{proof}

 Satisfying weak duality is an \'etale (hence open) condition, so given any formally \'etale morphism  $Y' \to  \pi^0F(X_{\pro\et},-)$ from an
 (underived)  dagger Artin analytic $\infty$-stack $Y'$, the conditions of Corollary \ref{tracecor2b} are satisfied by the  subfunctor $Y \subset Y'$ consisting of points $\phi$ at which the presheaf $\bL^{F,\uline{A}_X, \phi} $ satisfies weak duality.

 \begin{examples}\label{Spex}
 Examples \ref{traceex}, \ref{traceexIwasawa1} and \ref{wdex} (as applied in Examples \ref{Lwdex}) now lead to some instances of $(n-d)$-shifted symplectic moduli stacks when substituted into Corollary \ref{tracecor2b}. If we take the derived stack $F$ to be $B\GL_n$, or $BG$ for any other affine algebraic group equipped with a $G$-equivariant inner product on its Lie algebra, or to be the moduli stack of perfect complexes, then $n=2$, so the corollary produces $(2-d)$-shifted symplectic structures on moduli of $G$-torsors or of complexes of pro-\'etale sheaves on $X$, provided we impose some constructibility constraints.

 In particular:
 \begin{enumerate}[itemsep=0pt, parsep=5pt]
  \item  If $X$ is a smooth proper scheme of dimension $m$ over a separably closed field prime to $\ell$, we have $(n-2m)$-shifted symplectic structures on suitable open substacks of the derived moduli stack of $F$-valued sheaves on $X_{\pro\et}$, depending on a choice of isomorphism $ \Zl(m) \cong \Zl$. Poincar\'e duality ensures that the symplectic locus of $F(X_{\pro\et},-)$ contains all $A$-valued points $\phi$ for which $\bL^{F,\uline{A}_X, \phi} $  is constructible.

When $F=\Perf$ and the base field has finite characteristic,  the comparison of Remark \ref{antoniormk}  gives an \'etale map to this from the shifted symplectic derived analytic stack of \cite[Theorem 5.13]{antonioElladic}, respecting the symplectic structures.

 \item \label{Iwasawa2} 
 If $X$ is a smooth proper scheme of dimension $m$ over a cyclotomic extension $k_{\infty}:=k(\mu_{\ell^{\infty}})$ of a 
 local field $k$ containing $\ell^{-1}$ with finite residue field, we have $(n+1-2m)$-shifted symplectic structures on suitable open substacks of the derived moduli stack of $F$-valued sheaves on $X_{\pro\et}$, depending on a choice of isomorphism $ \Zl(m+1) \cong \Gal(k_{\infty}/k)$.

 Although the pairing of Examples \ref{traceexIwasawa1} will seldom be perfect on finite coefficients, Proposition \ref{dualpropinfty} and Remark \ref{Iwasawarmk} ensure that pullbacks of $\Ql$-constructible complexes over schemes of finite type satisfy duality when the residue field has characteristic prime to $\ell$. The symplectic locus of $F(X_{\pro\et},-)$ thus contains all $\Ql$-valued points $\phi$ for which $\bL^{F,\uline{\Ql}_X, \phi} $  is constructible in that case.
 
In contrast,  for $\ell$-adic local fields,  points of the pre-symplectic derived moduli stack at which the global cotangent complex $\oL^{F(X_{\pro\et},-),\phi}$ is perfect are extremely rare, so the symplectic locus will often be empty. This is because Tate's calculations 
\cite[\S II.5.7, Theorem 5]{galoisienne}
give the Euler characteristic $\chi(k,V)$ of $\H^*_{\pro\et}(k,V)$ as $-[k:\Ql]\dim V$, while we would have to have $\chi(k,V)= \chi( k_{\infty},V)-\chi( k_{\infty},V)=0$ whenever $\chi( k_{\infty},V) $ is finite.
 
 \item  If $U$ is a smooth scheme of dimension $m$ over a separably closed field prime to $\ell$, a choice of trivialisation of $ \Zl(m)$ similarly 
 leads to an $(n+1-2m)$-shifted symplectic structure on suitable open substacks of the derived moduli stack 
 \[
F((Z \overleftarrow{\by}_{\bar{U}}U)_{\et},-)\co  A \mapsto \oR\Gamma(Z_{\pro\et}, i^*\oR j_*F(\uline{A}_X))
\]  
 of $F$-valued sheaves on the 
 punctured
 tubular neighbourhood $ Z \overleftarrow{\by}_{\bar{U}}U$ (thought of as the  boundary of $U$), where $i \co Z \to \bar{U}$ is the complement of $U$ in a compactification $j \co U \to \bar{U}$.

 Again, Poincar\'e duality ensures that the symplectic locus of $F((Z \overleftarrow{\by}_{\bar{U}}U_{\pro\et})_{\et},-)$  contains all $A$-valued points $\phi$ for which $\bL^{F,\uline{A}_{Z \overleftarrow{\by}_{\bar{U}}U}, \phi} $  is constructible in the sense of \cite[2.5]{illusieVanishing}.
 \end{enumerate}
 
 However, the hypothesis $\bD \cong \Zl$ of Corollary \ref{tracecor2a} is not satisfied by the other cases of Examples \ref{traceex}, which we will tackle by using weighted symplectic structures in \S \ref{weightedsn}, leading to Examples \ref{wSpex}.
\end{examples}

 \begin{remark}\label{pappasrmk}
When $\tilde{Y}$ is a smooth (underived) analytic space or DM stack, a $0$-shifted symplectic structure is just a symplectic structure in the classical sense. In particular, this applies when studying $\ell$-adic local systems or $G$-torsors on a smooth proper curve over an algebraically closed field, giving the symplectic structures studied in \cite{pappasVolSymplAdicLocSys}.
 %
%
 \end{remark}
 
 \subsection{Lagrangian structures}\label{Lagsn}

 We now assume that we have a morphism $f \co \pd U \to U$ of topologically Noetherian schemes (where the notation $\pd U$ is intended to be suggestive, but does not indicate a specific construction at this point)  and a constructible complex $\bD$ on $U$,  equipped with a system of trace maps 
 \[
 \cone\left(\oR\Gamma (U_{\et}, \bD/\ell^n) \to  \oR\Gamma (\pd U_{\et}, f^{-1}\bD/\ell^n)\right)[d-1]\to \Z/\ell^n
 \]
 similarly to Proposition \ref{traceprop}. 
 We will require the resulting composite pairing
 \[
 \oR\Gamma (U_{\et}, L/\ell^n) \ten \cone\left(\oR\Gamma (U_{\et},  L^*\ten \bD/\ell^n) \to  \oR\Gamma (\pd U_{\et}, f^{-1}(L^*\ten\bD/\ell^n))\right)[d-1]\to \Z/\ell^n
 \]
to be perfect for constructible complexes $L$ with dual $L^*$.

 \begin{examples}\label{traceLagex}
Trace maps $\tr$ of the form required above arise from the six functors formalism (and specifically Poincar\'e duality)  in the following situations, with $\oR\Gamma_c(U,L) \simeq \cone(\oR\Gamma(U,L) \to \oR\Gamma(\pd U,f^*L))[-1]$ in each case:

 \begin{enumerate}[itemsep=5pt, parsep=0pt]
 
 \item\label{traceLaglocex} If $F$ is a non-Archimedean local field with residue characteristic prime to $\ell$, and $\cO_F$ its ring of integers, then we can look at the morphism $\Spec F \to \Spec \cO_F$. Local Tate duality in the form of
\cite[Theorem II.1.8(b)]{milneArithDuality} or \cite[\S II.5.6, Proposition 19]{galoisienne}  then gives the desired trace pairing, with $\bD=\uline{\Zl(1)}$ and $d=3$.
 
\item More generally, if $Z$ is a smooth proper scheme of dimension $m$ over $\cO_F$, consider the morphism $Z\by_{\Spec \cO_F}\Spec F \to Z$. Local Tate duality combines with relative Poincar\'e duality \cite[Proposition VI.11.8]{Mi} to give the  desired trace pairing here, with $\bD=\uline{\Zl(m+1)}$ and $d=2m+3$. Explicitly, the trace is given by
\begin{align*}
 &\cone\left(\oR\Gamma (Z_{\et},  \uline{\Zl(m+1)}) \to  \oR\Gamma ( (Z\by_{\Spec \cO_F}\Spec F)_{\et}, \uline{\Zl(m+1)})\right)[2m+2]\\
 &\to \cone\left(\oR\Gamma ( (\Spec \cO_F)_{\et},  \uline{\Zl(1)}) \to  \oR\Gamma ( (\Spec F)_{\et}, \uline{\Zl(1)})\right)[2] \tag{Poincar\'e}\\
 &\to \Zl \tag{local Tate}.
\end{align*}

 \item If $F$ is a number field and $S$ a finite set of primes including all those dividing $\ell$, with $\cO_{F,S}$ the localisation of $\cO_F$ at $S$, then consider the morphism $\Spec \prod_{v \in S} F_v  \to \Spec \cO_{F,S}$.  Poitou--Tate duality in the form of \cite[Proposition II.2.3 and Theorem II.3.1]{milneArithDuality}  
or \cite[\S II.6.3]{galoisienne}
then gives the desired trace pairing, with $\bD=\uline{\Zl(1)}$ and $d=3$.
  
\item  More generally, if $Z$ is a smooth proper scheme of dimension $m$ over $\cO_{F,S}$, consider the morphism $Z\by_{\Spec \cO_{F.S}} \Spec \prod_{v \in S} F_v \to Z$. Poitou--Tate duality combines with relative Poincar\'e duality to give the  desired trace pairing here, with $\bD=\uline{\Zl(m+1)}$ and $d=2m+3$.
  
 \item\label{Iwasawa3} 
 There are variations of the previous examples in which we adjoin $\mu_{\ell^{\infty}}$ to the base, resulting in duality one degree lower via pairings of generalised Cassels--Tate form as in \S \ref{cycdualitysn}. Explicitly, for $F$ a non-Archimedean local field with residue characteristic prime to $\ell$, 
 we can look at  $\Spec F(\mu_{\ell^{\infty}}) \to \Spec \cO_F(\mu_{\ell^{\infty}})$. 
  The trace map of Examples \ref{traceexIwasawa1} then gives a trace pairing with $\bD=\uline{\Gamma^*(1)}\cong \uline{\Zl}$ and $d=2$. 
 
Similarly,
for $F$ a number field and $S$ a finite set of primes, we can consider the morphism $\Spec \prod_{v \in S} F_v(\mu_{\ell^{\infty}})  \to \Spec \cO_{F,S}(\mu_{\ell^{\infty}})$. An identical argument, using Poitou--Tate duality rather than local Tate duality, gives a trace pairing with $\bD= \uline{\Gamma^*(1)}\cong \uline{\Zl}$ and $d=2$. 
 
For  smooth proper schemes of dimension $m$ over those bases, we can combine these with relative Poincar\'e duality to obtain a   trace pairing with $\bD=\uline{\Gamma^*(m+1)}\cong \uline{\Zl}$ and $d=2m+2$.

 \item A situation of a similar flavour is given if $U$ is a smooth variety of dimension $m$  over a separably closed field $k$ which is prime to $\ell$, and we have an open immersion $j \co U \into X$ into a complete variety, with complement $i \co Z \to X$, so let $\pd U:=Z \overleftarrow{\by}_{X}U$, the 
 punctured
 tubular neighbourhood. The projection map $\pr_U \co \pd U \to U$ is of the desired form, for $\bD=\uline{\Zl(m)}$ and $d=2m$. This follows because $ \oR\Gamma(\pd U,\pr_U^*L) \simeq \oR\Gamma(Z,i^*\oR j_*L)$ and hence  $\oR\Gamma(X,\oR j_!L) \simeq \cone(\oR\Gamma(U,L) \to \oR\Gamma(\pd U,\pr_U^*L))[-1]$.
 
 The same example  also works over local fields $k$ containing $\ell^{-1}$  with finite residue field (resp. over finite fields $k$ prime to $\ell$), with $d= 2m+2$ (resp. $2m+1$)  and $\bD=\uline{\Zl(m+1)}$ (resp. $\uline{\Zl(m)}$), and over $k(\mu_{\ell^{\infty}})$ with $d= 2m+1$ (resp. $2m$, being separably closed) and $\bD=\uline{\Gamma^*(m+1)}\cong \uline{\Zl}$.
\end{enumerate}
   \end{examples}

\begin{example}\label{traceLagex2}
  We can also combine
  the previous types of  Examples \ref{traceLagex}, at least at a formal level. Take $U$ to be a smooth separated of dimension $m$  over a either the ring of integers $\cO_F$ of a non-Archimedean local field, or over a localisation $\cO_{F,S}$ of the ring of integers of a number field (there are also versions attaching $\mu_{\ell^{\infty}}$ to the base, replacing $d$ with $d-1$). 
    The idea then is to take a punctured tubular neighbourhood $\pd_0U \to U$ over the same base, 
  with $\pd_{01} U \to \pd_1U$ being the base change of that morphism to $F$ (resp. $\prod_{v \in S} F_v$). Then the morphisms $ \pd_{01} U \to \pd_0U$ 
   and $\pd_{01} U \to \pd_1U$   both give trace pairings of the desired form, with $d=2m+2$
  and $\bD=\uline{\Zl(m+1)}$. 
  
 More significantly, if we formally construct $\pd U$ to be the pushout  $\pd_0 U \cup_{\pd_{01}U} \pd_1U$, then the morphism 
 $\pd U \to U$ gives a trace pairing of that form, with $d=2m+3$ 
  and $\bD=\uline{\Zl(m+1)}$.
  
 Explicitly, if we have  an open immersion $j \co U \into X$ into a flat proper scheme over $\cF$, with complement $i \co Z \to X$, then we let $\pd_0 U:=Z \overleftarrow{\by}_{X}U$, the punctured tubular neighbourhood, while $\pd_1U=\coprod_v U_{F_v}$ (base change along $\cO_{F,S} \to \prod_{v \in S} F_v$, or along $\cO_F \to F$ in the local case). Thus $\pd_{01}U=  \coprod_v Z_{F_v} \overleftarrow{\by}_{X_{F_v}}U_{F_v}$.
 
For  the Galois group $G_{F,S}$ associated to $\cO_{F,S}$  and $G_v$ the Galois group of $F_v$, Poitou--Tate duality gives
 $\prod_{v \in S} \oR\Gamma(G_v,-) \simeq \cone(\oR\Gamma_c(G_{F,S},-) \to \oR\Gamma(G_{F,S},-))$, while the six functors formalism gives $\oR\Gamma(Z_{\bar{F}}, i^*\oR j_*-) \simeq \cone(\oR\Gamma_c(U_{\bar{F}},-) \to \oR\Gamma(U_{\bar{F}},-))$.
    For a constructible complex $L$ on $U$, this means that for the Galois group $G_{F,S}$ associated to $\cO_{F,S}$, we have:
    \begin{align*}
 \oR\Gamma(U,L) &\simeq \oR\Gamma(G_{F,S}, \oR\Gamma(U_{\bar{F}}, L))\\
 \oR\Gamma(\pd_0 U,f_0^*L) &\simeq \cone\left(\oR\Gamma(G_{F,S}, \oR\Gamma_c(U_{\bar{F}}, L)) \to \oR\Gamma(G_{F,S}, \oR\Gamma(U_{\bar{F}}, L)) \right)\\   
 \oR\Gamma(\pd_1 U,f_1^*L) &\simeq \cone\left(\oR\Gamma_c(G_{F,S}, \oR\Gamma(U_{\bar{F}}, L)) \to \oR\Gamma(G_{F,S}, \oR\Gamma(U_{\bar{F}}, L))\right).
     \end{align*}
Moreover
\begin{align*}
 \oR\Gamma(\pd_{01} U,f_{01}^*L)\simeq \prod_{v\in S} \oR\Gamma(Z_{F_v}, i^*\oR j_*L)),
 \end{align*}
which in turn is quasi-isomorphic to the total cone complex of
\begin{align*}
 &\oR\Gamma_c(G_{F,S}, \oR\Gamma_c(U_{\bar{F}}, L))\\
& \to \oR\Gamma_c(G_{F,S}, \oR\Gamma(U_{\bar{F}}, L))\oplus \oR\Gamma(G_{F,S}, \oR\Gamma_c(U_{\bar{F}}, L))\\
&\quad\,\ \to \oR\Gamma(G_{F,S}, \oR\Gamma(U_{\bar{F}}, L)).
\end{align*}
Thus
\begin{align*}
 \oR\Gamma(\pd U, f^*L) &\simeq  \cone\left(\oR\Gamma(\pd_0 U,f_0^*L) \oplus \oR\Gamma(\pd_1 U,f_1^*L) \to \oR\Gamma(\pd_{01} U,f_{01}^*L)\right)[-1]\\
 &\simeq
  \cone\left(\oR\Gamma_c(G_{F,S}, \oR\Gamma_c(U_{\bar{F}}, L)) \to \oR\Gamma(G_{F,S}, \oR\Gamma(U_{\bar{F}}, L))\right),
\end{align*}
 from which the desired duality for the trace pairing associated to $\pd U \to U$ follows.

\end{example}
 
 \begin{definition}\label{weakdualLagdef}
 Given a quasi-dagger dg algebra $A$ and a presheaf $N$ of  $\uline{A}_U$-modules in chain complexes  on $U_{\pro\et}$, we say that $N$ satisfies weak duality with respect to the trace $\tr$ above if $\oL f^*N:= f^{-1}N\ten^{\oL}_{f^{-1}\uline{A}_U}\uline{A}_{\pd U}$ satisfies weak duality in the sense of Definition \ref{weakdualdef}, and  for all morphisms $A \to C$ of quasi-dagger dg algebras,  
 the map
 \begin{align*}
&\cone\left( \oR\HHom_{\uline{A}_U}( N, \uline{C}_U) \to  \oR\HHom_{\uline{A}_{\pd U}}( \oL f^* N, \uline{C}_{\pd U})\right)\\
&\to \oR\HHom_A(\oR\Gamma(U,N\hten_{\Zl}\bD)[d-1],C)
 \end{align*}
induced by the pairing
\begin{align*}
 \oR\Gamma(U,N\hten_{\Zl}\bD)\ten_A^{\oL} \cone\left( \oR\HHom_{\uline{A}_U}( N, \uline{C}_U) \to \oR\HHom_{\uline{A}_{\pd U}}( \oL f^* N, \uline{C}_{\pd U})\right)\\
 \to \cone\left(\oR\Gamma(U, \uline{C}_U\hten_{\Zl}\bD) \to \oR\Gamma({\pd U}, \uline{C}_{\pd U}\hten_{\Zl}f^{-1}\bD)\right)\xra{\tr} C[1-d]
\end{align*}
is a quasi-isomorphism, and $\oR\Gamma(U,N\hten_{\Zl}\bD)$ is a perfect $A$-module.
\end{definition}

\begin{examples}\label{wdexLag}
 Similarly to Examples \ref{wdex}, if there exists an $\ell$-adically complete Noetherian ring $R$, a constructible $\hat{R}_U$-complex $L$ in the sense of \cite[\S 6.5]{BhattScholzeProEtale} and a $K$-algebra homomorphism $R \to \H_0A$, then an  $\uline{\H_0A}_U$-module $N$ with sheafification $L\ten_{\hat{R}_U}\uline{\H_0A}_U$ satisfies weak duality whenever $(U,\pd U)$ is one of the pairs of schemes or sites from Examples \ref{traceLagex} for which the pairing comes from the six functors formalism. 
 
 For pairings over cyclotomic extensions induced by the methods of \S \ref{cycdualitylemma}, we have additional torsion of finiteness conditions to impose, as in Proposition \ref{dualpropinfty} and Remark \ref{Iwasawarmk}.
\end{examples}

The proofs of Lemmas \ref{wdopenlemma} and \ref{wdnilcompletelemma} then adapt to give the same statements verbatim in this setting. Corollary \ref{tracecor2a} adapts to give:
 
 \begin{corollary}\label{tracecor2Laga}
Take     a morphism $\pd U\to U$ of topologically Noetherian schemes which has a trace pairing as above with $\bD=\uline{\Zl}$, 
and  $F \co dg_+\CAlg_K \to s\Set$ an
 $n$-shifted symplectic  derived Artin $\infty$-stack. Then for the full subfunctors   $F({\pd U}_{\pro\et},-)^{wd}  \subset F({\pd U}_{\pro\et},-)$ and $F( U_{\pro\et},-)^{wd}\subset F( U_{\pro\et},-)$ consisting of points $\phi$ for which   the presheaf $\bL^{F,\uline{A}_{\pd U}, \phi} $ (resp. $\bL^{F,\uline{A}_{U}, \phi} $)  satisfies weak duality in the sense of Definition \ref{weakdualdef} (resp. Definition \ref{weakdualLagdef}), the natural map
\[ 
 F( U_{\pro\et},-)^{wd} \to F({\pd U}_{\pro\et},-)^{wd}
 \]
 carries a natural Lagrangian structure with respect to the $(n-d+1)$-shifted symplectic structure on $F({\pd U}_{\pro\et},-)^{wd}$ given by Corollary \ref{tracecor2a}.
 \end{corollary}
 
 Corollary \ref{tracecor2b} then adapts to give:
  \begin{corollary}\label{tracecor2Lagb}
In the setting of Corollary \ref{tracecor2Laga},  take a morphism $W \to Y$ of (underived)  dagger Artin analytic $\infty$-stacks, equipped with a 
commutative diagram
\[
 \begin{CD}
 W @>{\eta'}>>  \pi^0F( U_{\pro\et},-)\\
 @VVV @VV{f^*}V\\
   Y @>{\eta}>>  \pi^0F({\pd U}_{\pro\et},-)
 \end{CD}
\]
of functors $\Affd\Alg^{\loc,\dagger}_K \to s\Set $.
Assume moreover that  the horizontal maps $\eta,\eta'$ are formally \'etale and  that at all points $\phi$ in the image of $\eta'$ (resp. $\eta$),  the presheaf $\bL^{F,\uline{A}_{\pd U}, \phi} $ (resp. $\bL^{F,\uline{A}_{U}, \phi} $)  satisfies weak duality in the sense of Definition \ref{weakdualdef} (resp. Definition \ref{weakdualLagdef}).
 
  Then  the functor $\tilde{W} \co A \mapsto W(\H_0A)\by^h_{F({\pd U}_{\pro\et},\H_0A)}F({\pd U}_{\pro\et},A) $ on $dg_+\Affd\Alg^{\loc,\dagger}_K$
 is a dg dagger  Artin analytic $\infty$-stack. It  carries a natural Lagrangian structure with respect to the $(n-d+1)$-shifted symplectic structure on $\tilde{Y}$ given by Corollary \ref{tracecor2b}.
 \end{corollary}

 \begin{examples}\label{SpexLag}
 Examples \ref{traceLagex} and \ref{wdexLag} (applied as in Examples \ref{Lwdex}) now lead to many instances of Lagrangians in weighted shifted symplectic moduli stacks when substituted into Corollary \ref{tracecor2Lagb}.
 
 If we take the $n$-shifted symplectic derived stack $F$ to be $B\GL_n$, or $BG$ for any other affine algebraic group equipped with a $G$-equivariant inner product on its Lie algebra, or to be the moduli stack of perfect complexes, then $n=2$, so the corollary produces $(3-d)$-shifted 
 Lagrangian structures on moduli of $G$-torsors or of complexes of pro-\'etale sheaves on $U$, provided we impose some constructibility constraints.
 
 In particular, if $U$ is a smooth proper scheme of dimension $m$ over a separably closed field prime to $\ell$, we have $(n+1-2m)$-shifted Lagrangian structures on suitable open substacks of the derived moduli stack of $F$-valued sheaves on $U_{\pro\et}$, depending on a choice of isomorphism $ \Zl(m) \cong \Zl$. These stacks are Lagrangian over the derived moduli stack 
 of $F$-valued sheaves on the punctured tubular neighbourhood, as in Example \ref{Spex}.(\ref{tubular}).
 Poincar\'e duality ensures that the Lagrangian locus of $F(U_{\pro\et},-)$  contains all $A$-valued points $\phi$ for which $\bL^{F,\uline{A}_U, \phi} $  is constructible.
 
The other cases of Examples \ref{traceLagex}   for which the hypothesis $\bD \cong \uline{\Zl}$ of Corollary \ref{tracecor2Laga} is satisfied are those for which the base field contains $\mu_{\ell^{\infty}}$. These produce many examples arising from pairings of generalised Cassels--Tate type as in \S \ref{cycdualitysn}, where the constraints ensuring non-degeneracy of the pairing are stricter than constructibility, but tend to be satisfied by sheaves which are Iwasawa torsion. In particular:

\begin{enumerate}[itemsep=5pt, parsep=0pt] 
 \item\label{Iwlocalex} If $K$ is a non-Archimedean local field with finite residue field prime to $\ell$,  and $K_{\infty}:=K(\mu_{\ell^{\infty}})$ with ring of integers $\cO_{K_{\infty}}$, 
then we can take $\pd U \to U$ to be the morphism $\Spec K_{\infty} \to \Spec \cO_{K_{\infty}}$, giving us   $(n-1)$-shifted Lagrangian structures  on suitable open substacks of the derived moduli stack  $F ((\Spec \cO_{K_{\infty}})_{\pro\et},-)$ over $ F ((\Spec K_{\infty})_{\pro\et},-)$, coming from local Tate duality and the trace of Lemma \ref{traceIwlemma}.  This morphism of derived stacks essentially maps from moduli of unramified  $\Gal(\bar{K}/K(\mu_{\ell^{\infty}}))$-representations to moduli of all  $\Gal(\bar{K}/K(\mu_{\ell^{\infty}}))$-representations.

As in Example \ref{Spex}.(\ref{Iwasawa2}), the Lagrangian locus of $F((\Spec \cO_{K_{\infty}})_{\pro\et},-)$ contains all $\Ql$-valued points $\phi$ for which $\bL^{F,\uline{\Ql}_{ \cO_{K_{\infty}}}, \phi} $  is the pullback of a constructible complex on a finite type scheme.

 \item\label{IwlocalZex} More generally, if $Z$ is a smooth proper scheme of dimension $m$ over $\cO_{K_{\infty}}$, we have $(n-1-2m)$-shifted Lagrangian structures on suitable open substacks of the derived moduli stack  $F(Z_{\pro\et},-)$ over $ F ((Z\ten_{\cO_K}K)_{\pro\et},-)$, since we can choose a model $Z_n$ for $Z$ over some $\cO_{K(\mu_{\ell^n})}$ and apply the traces of \S \ref{cycdualitysn} to the cover $Z \to Z_n$. Again, if  $\bL^{F,\uline{\Ql}_X, \phi} $ is the pullback of a constructible complex on a finite type scheme,  $\phi \in   F (Z_{\pro\et},\Ql)$ lies in the Lagrangian locus.

\item\label{Iwglobalex} If $K$ is a number field and $S$ a finite set of primes including all those dividing $\ell$, with $\cO_{K,S}$ the localisation of $\cO_K$ at $S$, then we have an $(n-1)$-shifted isotropic structure on  the derived moduli stack  $F( (\Spec \cO_{K(\mu_{\ell^{\infty}},S})_{\pro\et},-)$ over $ \prod_{v \in S}F ((\Spec  K(\mu_{\ell^{\infty}})_v  )_{\pro\et},-)$, coming from Poitou--Tate duality and the trace of Lemma \ref{traceIwlemma}. 

However, the necessary presence of primes $v$ dividing $\ell$ means few points lie in the non-degenerate locus, as discussed in Example \ref{Spex}.(\ref{Iwasawa2}), so there will not be many examples of suitable open substacks on which the isotropic structure is in fact Lagrangian.

 \item\label{IwglobalZex} More generally, if $Z$ is a smooth proper scheme of dimension $m$ over $\cO_{K(\mu_{\ell^{\infty}}),S}$, then we similarly  have an $(n-1-2m)$-shifted isotropic structure  on the derived moduli stack  $F( Z_{\pro\et},-)$ over $ \prod_{v \in S}F ((Z\ten_{\cO_{K,S}}K_v  )_{\pro\et},-)$. 

 \item If $U$ is a smooth proper scheme of dimension $m$ over $K(\mu_{\ell^{\infty}})$, for $K$ a local field with finite residue field, 
 we have $(n-1-2m)$-shifted  Lagrangian structures  on suitable open substacks of the derived moduli stack  $F(U_{\pro\et},-)$. These stacks are Lagrangian over the derived moduli stack 
\[
  A \mapsto \oR\Gamma(Z_{\pro\et}, i^*\oR j_*F(\uline{A}_U)) 
  \]  
 of $F$-valued sheaves on the punctured tubular neighbourhood of $U$ in its compactification, as in Example \ref{Spex}.(\ref{tubular}).
 
 When the residue field of $K$ is prime to $\ell$, constructibility of $\bL^{F,\uline{\Ql}_U, \phi} $ will again ensure that $\phi \in   F (U_{\pro\et},\Ql)$ lies in the Lagrangian locus. However, that condition is not sufficient when $K$ is an $\ell$-adic field, in which case the Lagrangian locus will be small.

\item\label{IwglobalopenZex} We can also combine these as in Example \ref{traceLagex2}. Take $U$ to be a smooth separated scheme of dimension $m$  over  either the ring of integers $\cO_{K(\mu_{\ell^{\infty}})}$, for $K$  a non-Archimedean local field with residue field prime to $\ell$, or (more rarely) over a localisation $\cO_{K(\mu_{\ell^{\infty}}),S}$ of the ring of integers of the extension $K( \mu_{\ell^{\infty}})$ of a number field $K$. For a compactification $j \co U \into \bar{U}$ of $U$ over the same base, with complement $i \co Z \to \bar{U}$, we then have an
 $(n-1-2m)$-shifted isotropic structure on the derived moduli stack  
 \[
 F(U_{\pro\et},-) \co  A \mapsto \oR\Gamma(U_{\pro\et}, F(\uline{A}_U) ). 
\]
This is isotropic over 
the  derived moduli stack sending $A$ to the homotopy fibre product of the diagram
\[
  %
  \xymatrix@R=0ex{
  \oR\Gamma(Z_{\pro\et}, i^*\oR j_*F(\uline{A}_U)) \ar[dr]\\
 & \prod_{v \in S}\oR\Gamma( Z_{K_v,\pro\et}, i^*\oR j_*F(\uline{A}_U))\\
  \ar[ur]
  \prod_{v \in S}\oR\Gamma(U_{ K_v,\pro\et}, F(\uline{A}_U))
  }
  \]  
 of $F$-valued sheaves on the boundary $\pd U$ constructed as in Example \ref{traceLagex2}, where we take  $\{K_v\}_{v \in S}=\{K\}$ in the local case.

 When $K$ is a local field,  an element $\phi \in   F (U_{\pro\et},\Ql)$ lies in the Lagrangian locus whenever  $\bL^{F,\uline{\Ql}_U, \phi} $ is the pullback of a constructible complex on a finite type scheme, similarly to example (\ref{IwlocalZex}).

However when $K$ is a global field,   the Lagrangian locus for the isotropic structures $F(U_{\pro\et},-)$ will, like those in example (\ref{Iwglobalex}), tend to be very small.
 \end{enumerate}
 
\end{examples}

\begin{example}\label{IwintersectGLCex}
  We would also like to consider  derived intersections of  the Lagrangians of Examples \ref{SpexLag}.(\ref{IwglobalZex}) and  (\ref{IwlocalZex}), but the former need primes dividing $\ell$ in the base, while the latter avoid them. 
One way to resolve this is by considering an analogue of Greenberg's local conditions \cite[0.9]{nekovarSelmerComplexes} at the primes dividing $\ell$, by choosing a Lagrangian $G \to F$.

  Given $Z$ smooth and proper of dimension $m$ over the ring of integers $\cO_{K(\mu_{\ell^{\infty}})}[\ell^{-1}]$, for $K$ a number field,  we then have an $(n-2-2m)$-shifted pre-symplectic  structure on the homotopy limit   of the diagram
  \[
  \xymatrix{
    &  \prod_{\substack{v \in S \\ v \nmid \ell}}F ((Z\ten_{\cO_{K}}\cO_v  )_{\pro\et},-) \by  \prod_{\substack{v \in S \\ v \mid \ell}}G((Z\ten_{\cO_{K}}K_v  )_{\pro\et},- ) \ar[d]\\ 
    F( (Z\ten_{\cO_K}\cO_{K,S})_{\pro\et},-)\ar[r]  &\prod_{v \in S}F ((Z\ten_{\cO_{K}}K_v  )_{\pro\et},-).
  }
  \]

When $Z= \Spec \cO_{K(\mu_{\ell^{\infty}})}[\ell^{-1}]$ itself, we might think of this as a  generalisation of Selmer complexes, 
with the resulting  $0$-shifted symplectic structure corresponding to a form of generalised Cassels--Tate pairing as in \cite[\S 10]{nekovarSelmerComplexes} and Remark \ref{Iwasawarmk}. In particular, 
the pre-symplectic form is symplectic at $\Q_p$-points $\phi$  where 
\[
\cone\big( \oR\Gamma(\cO_{K,S} (\mu_{\ell^{\infty}}),\uline{\bL}^{F, \phi})    \to \prod_{\substack{v \in S \\ v \nmid \ell}}  \oR\Gamma_c(\cO_v (\mu_{\ell^{\infty}}),\uline{\bL}^{F, \phi} )[1]    \by \prod_{\substack{v \in S \\ v \mid \ell}} \oR\Gamma( K_v(\mu_{\ell^{\infty}}) ,\uline{\bL}^{G, \phi})\big)
\]
is perfect and the complexes  $\bL^{F, \phi}$ and $\bL^{G, \phi}$ are induced by 
$G_{\cO_{K,S} (\mu_{\ell^n}) }$- and $G_{K_v(\mu_{\ell^n})}$-representations for some $n<\infty$; here $\oR\Gamma_c(\cO_v,-)[1]= \cone(\oR\Gamma(\cO_v,-) \to \oR\Gamma(K_v,-))$.   Such finiteness would be implied by the associated Iwasawa modules being torsion, as predicted by Mazur's Main Conjecture for representations coming from abelian varieties.

Instead of forcing a Lagrangian at places dividing $\ell$, it would be more satisfactory to take a crystalline construction as in the  $f$-cohomology of \cite{BlochKatoTamagawa}. However, we then  encounter the obstacle that Tate twists in crystalline cohomology are non-trivial even over cyclotomic fields, in particular because they change the Hodge filtration, so we do not have a relative trace of the required form. The weighted theory of \S \ref{weightedsn} will handle twisted duality, with a non-cyclotomic  crystalline  analogue of this construction in  \S \ref{wSpexLagcrissn}.

  %
  
 %

 
 \end{example}
 
\begin{remark}[Comparison with Selmer complexes]\label{selmer1rmk}
 To understand how the construction of Example \ref{IwintersectGLCex} generalises Selmer complexes, note that our symplectic constructions all extend to relative symplectic structures, i.e. morphisms $F \to M$  of derived Artin stacks with a non-degenerate form $\omega \in \z^{n+2}F^2\DR(F/M)$, and that for the Lagrangian case we can similarly take a morphism $N \to M$ and a Lagrangian $G \to N\by_M^hF$ relative to $N$. We can then take $M= B\Sp_{2g}$, $F= B(\Sp_{2g} \ltimes \bA^{2g})$, with $N=BU$ for $U \subset \Sp_{2g}$ the group of block upper triangular symplectic matrices and $G= B(U \ltimes \bA^g)$. Then $F$ carries a canonical $2$-shifted symplectic structure over $M$, coming from the canonical symplectic product on $\bA^{2g}$, and $G \to F\by^h_MN$ a canonical Lagrangian structure over $N$, corresponding to the Lagrangian structure on $\bA^g \subset \bA^{2g}$.

We can then apply the construction above to $G \to F$ and obtain a pre-symplectic structure by taking the homotopy fibre over a point of the same construction applied to $M \to N$. A point of the latter consists of a symplectic $G_{S,\infty}$-representation $U$ unramified at places in $S$ not dividing $\ell$, together with $G_{v,\infty}$-equivariant  Lagrangian subspaces  $U^+_v \subset U$ at all places $v \in S$. The homotopy fibre over such a point then just corresponds to the Selmer complex
\[
 \cone( \oR\Gamma(G_{S,\infty} ,U)\oplus \prod_{\substack{v \in S \\ v \nmid \ell}}  \oR\Gamma(G_{v,\infty}/I_{v,\infty} ,U)\oplus \prod_{\substack{v \in S \\ v \mid \ell}} \oR\Gamma( G_v ,U_+) \to \prod_{v \in S} \oR\Gamma(G_v,U)).
\]
\end{remark}

\subsection{Infinitesimal formality for moduli of \tps{$G$}{G}-torsors} 

\subsubsection{Deformation functors}

We can look at infinitesimal neighbourhoods of any $K$-valued point $\phi$ in one of our derived moduli functors, giving rise to the functor of derived deformations of $\phi$. This is purely algebraic in nature, essentially because Artinian $K$-algebras have unique EFC-structures and are already dagger affinoid. Explicitly, each such derived deformation functor takes inputs from the category of Artinian local $K$-CDGAs with residue field $K$, and sends $A$ to the homotopy fibre
\[
F(X_{\pro\et},A)_{\phi}:=  F(X_{\pro\et},A)\by^h_{F(X_{\pro\et},K)}\{\phi\}.
\]

 \subsubsection{DGLAs}\label{DGLAsn}
 
For simplicity, we focus on the functor $(BG)(X_{\pro\et},-)$ of moduli of $G$-torsors\footnote{Taking a general homogeneous functor instead of $BG$, \cite[Theorems 2.30 and 4.57]{ddt1} give us a hypersheaf of DGLAs instead of $ \uline{\g}_{\phi,X}$, but shifted symplectic structures are generally encoded by more complicated data than the inner products we will work with on $\g$.}, for an affine algebraic group $G$ equipped with a $G$-invariant inner product on its Lie algebra $\g$. In particular,  $ (BG)(X_{\pro\et},K)$ is equivalent to the groupoid of pro-\'etale $\uline{G(K)}_X$-torsors. With reasoning as in \cite[\S 4.5]{dmc}, the functor of derived deformations of a  torsor $\phi$ is governed by a DG Lie algebra of the form
\[
 \oR\Gamma(X_{\pro\et}, \uline{\g}_{\phi,X}),
\]
where $\g_{\phi}$ is the sheaf given by the adjoint bundle of $\phi$ (i.e. $(\phi\by \uline{\g}_X)/\uline{G(K)}_X$), and derived global sections are taken in the model category of DGLAs (if using a \v Cech complex, the simplest models for such homotopy limits involve Thom--Sullivan cochains). Specifically, following \cite{hinstack}, we say that a derived deformation functor is governed by a DGLA $L$ means if it is equivalent to the functor $A \mapsto \mmc(L\ten \m_A)$, where $\m_A$ is the maximal dg-ideal of $A$.%
\footnote{There is also an equivalent form given by the homotopy quotient of  $\mmc(L\ten \m_A)$ by the gauge action $*$ of the simplicial group $\uline{\Gg}(L\ten \m_A)$ given by $n \mapsto  \exp(L\ten \m_A\ten \Omega^{\bt}(\Delta^n))^0$,  the equivalence following by looking at tangent cohomology. There is a further equivalent form given by the nerve of the simplicial groupoid  with objects  $\mmc(L\ten \m_A)_0$ and with morphisms $\{g \in \uline{\Gg}(L\ten \m_A)~:~ g*x=y\}$ between each pair $x,y$ of objects; for further discussion and references, see \cite[Remark 4.7]{dmc}.} 

In our setting, the $G$-invariant inner product $\<-,-\>$ on the Lie algebra $\g$ induces a morphism
\[
  \oR\Gamma(X_{\pro\et}, \uline{\g}_{\phi,X}) \by  \oR\Gamma(X_{\pro\et}, \uline{\g}_{\phi,X}) \xra{\<-,-\> }  \oR\Gamma(X_{\pro\et}, \Ql) 
\]
which is compatible with the Lie bracket in the sense that $\<[x,y],z\>= \<x,[y,z]\>$. If the conditions of Proposition \ref{traceprop} are satisfied with $\bD =\uline{\Zl}$, then we can moreover compose this with the trace map $ \tr \co \oR\Gamma(X_{\pro\et}, \uline{K})\to K[-d] $ to give a shifted inner product
\[
 \<-,-\>_{\tr} \co \oR\Gamma(X_{\pro\et}, \uline{\g}_{\phi,X}) \by  \oR\Gamma(X_{\pro\et}, \uline{\g}_{\phi,X}) \to K[-d]
\]
satisfying $\<[x,y],z\>_{\tr}= \<x,[y,z]\>_{\tr}$.

\subsubsection{Shifted symplectic structures}\label{formalsympsn}

Since a Maurer--Cartan element $\alpha \in  \mmc(\oR\Gamma(X_{\pro\et}, \uline{\g}_{\phi,X}) \ten \m_A)$ defines an element of $(BG)(X_{\pro\et},A)$, Corollary \ref{tracecor} induces a $(2-d)$-shifted pre-symplectic structure $\omega_A \in F^2\DR(A)$ on $A$ whenever $X$ satisfies the conditions of Proposition \ref{traceprop} with $\bD=\Zl$. 

In order to understand this $(2-d)$-shifted pre-symplectic structure on $A$ explicitly, note that the tangent complex $T_{\alpha}(\mmc(L\ten-))$ is quasi-isomorphic to $(L\ten A, \delta +[\alpha,-])[1]$. The morphism $\bL^{\mmc(L\ten-),\alpha } \to \Omega^1_A$ then defines a closed element of degree $0$ in the complex 
\[
 \oR\HHom_A(\bL^{\mmc(L\ten-),\alpha },\Omega^1_A) \simeq (L\ten \Omega^1_A, \delta +[\alpha,-])[1]
\]
represented by $d\alpha \in (L \ten \Omega^1_A)^1$; this satisfies $\delta(d\alpha) +[\alpha,d\alpha]=0$ because $\delta \alpha + \half [\alpha,\alpha]=0$. 
 
 Analysing the proof of Corollary \ref{tracecor}, it thus follows that $\omega_A$ is given by 
 \[
 \<d\alpha, d\alpha\>_{\tr} \in \z^{2-d} (\Omega^1_A), 
 \]
 which is closed under $d$. 
 
 Although we have described this construction for Artinian CDGAs $A$, the same arguments apply to pro-Artinian CDGAs after taking suitable limits. In this infinitesimal setting, we can even work with $\Z$-graded CDGAs (playing the same role here as  the stacky CDGAs of \S \ref{stackysn}) as in \cite{Man2}, and then this description gives a $(2-d)$-shifted pre-symplectic structure on the pro-Artinian CDGA $(\widehat{\Symm}(L[1]^*), \delta +[-,-]^*)$ representing $\mmc(L\ten-)$.

\subsubsection{Formality}

Now take $K=\Ql$ and let  $X_k$ be a smooth proper variety of dimension $m$ over  an $\ell$-coprime finite field $k$, with $X:=X_k\ten_k\bar{k}$ the base change to  the algebraic closure $\bar{k}$ of $\bar{k}$. Some power of the
 Frobenius automorphism of $\bar{k}$  gives rise to a $k$-linear automorphism of $X$, so acts on the groupoid $(BG)(X_{\pro\et},\Ql)$ of pro-\'etale $G(\Ql)$-torsors. Assume moreover that  $\phi \in (BG)(X_{\pro\et},\Ql)$ is fixed by some power of Frobenius and is reductive in the sense that for every algebraic $G$-representation $V$, the local system on $X$ associated to $(\phi,V)$ is irreducible. Then,  generalising \cite[\S 2.3]{paper1}, the argument of \cite[Theorem 6.10]{weiln}\footnote{or the result itself, applied to groups of the form $R \ltimes \exp(\g \ten \m_A)$, for $R \subset G$ the Zariski closure of $\pi_1(X)$} implies that our governing DGLA is formal, in the sense that we have a zigzag of quasi-isomorphisms
 \[
  \oR\Gamma(X_{\pro\et}, \uline{\g}_{\phi,X}) \simeq \H^*(X_{\pro\et}, \uline{\g}_{\phi,X}).
 \]

\begin{proposition}\label{formalsympprop}
 The $(2-2m)$-shifted symplectic structure on $BG(X_{\pro\et},-)_{\phi}$ given by Examples \ref{SpexLag} (and depending on a choice of isomorphism $ \Zl(m) \cong \Zl$)
 is also formal, in the sense that it is induced by the perfect pairing
 \[
\<-,-\>_{\tr}\co  \H^*(X_{\pro\et}, \uline{\g}_{\phi,X}) \ten \H^*(X_{\pro\et}, \uline{\g}_{\phi,X}) \to \Ql[-2m],
 \]
 so the symplectic structure on an $A$-valued Maurer--Cartan element 
 $\alpha \in \mmc(\H^*(X_{\pro\et}, \uline{\g}_{\phi,X})\ten A)$ is $\<d\alpha,d\alpha\>_{\tr}$.
  \end{proposition}
 \begin{proof}
 The trace map for $X$ is given by $\H^{2m}(X_{\pro\et},\uline{\Zl(m)})\to \Zl$, which ($X$ being algebraically closed) we can rewrite as $ \H^{2m}(X_{\pro\et},\uline{\Zl})(m)\to \Zl$. Consequently, it has weight $-2m$ with respect to the Frobenius action, as does the shifted symplectic structure it induces.
  
  Now write $L:= \H^*(X_{\pro\et}, \uline{\g}_{\phi,X})$, so the deformation functor is pro-represented by $A:=(\widehat{\Symm}(L[1]^*), [-,-]^*)$. We can thus represent the $(2-2m)$-shifted symplectic structure as an element of 
  \[
   \z^{4-2m}F^2\DR(A) 
  \]
of weight $-2m$ with respect to the Frobenius action. 

Now the Frobenius action on $L^r$ has weight $r$ for all $r$, so  weight is at most  cochain degree in $\Symm(L[1]^*)$ and hence in $A$.
In $\Omega^p_{A}[-p] \cong S^p(L[2]^*)\ten A$, it then follows that  elements of weight $-2m$ live in cochain degree at least $2p-2m$, with equality only on $S^p(L[2]^*)\ten \Ql$.

Since the space of $(2-2m)$-shifted symplectic structures on $A$ only involves terms in $F^2\DR(A)$ of  cochain degree at most $4-2m$, the space of such structures $\omega$ of weight $-2m$ with respect to the Frobenius action is  equivalent to a   subset of $ S^2(L^*)^{-2m}$ (regarded as a discrete space). This implies that  $\omega$ lies in $\Omega^2_A$, and also that $\omega$ is determined by its image in $\Omega^2_A\ten_A\Ql$. 

To calculate this, take a universal deformation $u \in \mmc(  \oR\Gamma(X_{\pro\et}, \uline{\g}_{\phi,X}) \ten \m_A)$; this is unique up to a contractible space of choices, and encodes the data of an $L_{\infty}$-quasi-isomorphism $\Upsilon \co L \to \oR\Gamma(X_{\pro\et}, \uline{\g}_{\phi,X}) $. We know from \S \ref{formalsympsn} that the equivalence class of $\omega$ is given by $\< du,du\>_{\tr} \in \H^{2-2m}\Omega^2_A$, and we need to know its image 
\[
 \< du+\m_A,du +\m_A\>_{\tr} \in \H^{2-2m}(\Omega^2_A/\m_A\cdot \Omega^2_A).
\]
 Now observe that the element 
\begin{align*}
[du] \in  \H^1(\oR\Gamma(X_{\pro\et}, \uline{\g}_{\phi,X}) \ten(\Omega^1_A/\m_A)) &\cong \H^1(\oR\Gamma(X_{\pro\et}, \uline{\g}_{\phi,X}) \ten L[1]^*)\\
&\cong \Ext^0(L, \oR\Gamma(X_{\pro\et}, \uline{\g}_{\phi,X}) )
\end{align*}
is just the linear part of the $L_{\infty}$-quasi-isomorphism $\Upsilon$, so gives the identity on cohomology. Thus $\< du+\m_A,du +\m_A\>_{\tr} \in \H^{-2m}(S^2(L^*))\cong S^2(L^*)^{-2m}$ is   the pairing $\<-,-\>_{\tr} $  on cohomology, as required. 
\end{proof}

\begin{remark}
In accordance with the philosophy of weights \cite{poids},  the analogue of Proposition \ref{formalsympprop} for derived moduli of $G$-torsors on the analytic site of a compact K\"ahler manifold is also true. It has a much simpler proof, since trace maps are compatible with the zigzag of quasi-isomorphisms coming from the  $d\dc$- (or $\pd\bar{\pd}$)-lemma in \cite[7.6]{GM} and \cite[Lemma 2.2]{Simpson}.
 \end{remark}
 
\begin{remark}[Shifted Poisson structures]
  Since the shifted symplectic structure of Proposition \ref{formalsympprop} is strict, in the sense that it is a genuinely closed $2$-form which induces an isomorphism rather than a quasi-isomorphism, we can simply invert it to give the corresponding shifted Poisson structure.
 
 Explicitly, dualising the perfect pairing $\<-,-\>_{\tr}$ on   $L=\H^*(X_{\pro\et}, \uline{\g}_{\phi,X})$  gives us a Casimir element $\pi \in (S^2L)^{2m}$. The Poisson structure is then the biderivation of cochain degree $2m-2$ on $(\widehat{\Symm}(L[1]^*), [-,-]^*)$ given on generators $L[1]^*\by L[1]^*$ by contraction with  $\pi$.
\end{remark}

\section{Weighted shifted symplectic structures}\label{weightedsn}
 
 We now set up the theory to allow us to look  into cases where the trace maps $\H^d (X_{\et}, \bD/\ell^n)\to \Z/\ell^n$ exist on cohomology of  a  non-trivial rank $1$ local system $\bD$. Non-triviality means that results such as Corollaries \ref{tracecor} and \ref{tracecor2a} do not apply, but with some work, we can still construct a moduli functor with a  symplectic structure which is shifted in degree and twisted by a line bundle (i.e. a $\sP$-shifted symplectic structure in the terminology of \cite{BouazizGrojnowski}).
 
 If we were in the algebraic setting, looking at derived mapping stacks from a scheme or algebraic stack $X$ with dualising line bundle $\sL$, then we could take the $\bG_m$-torsor $P:= \oSpec_X (\bigoplus_{r \in \Z} \sL^{\ten r})$  on $X$ associated to $\sL$, and look at the derived mapping stack $\oR\map(P,F)$. This derived stack carries a natural $\bG_m$-action, with homotopy invariants
 \[
  \oR\map(P,F)^{h \bG_m} \simeq \oR\map([P/\bG_m],F) \simeq  \oR\map(X,F).
 \]
Moreover, a trace map on $\H^d(X,\sL^{\ten e})$ 
gives rise to a trace map on $\H^d(P,\sO_P) \cong \bigoplus_{r \in \Z}\H^d(X,\sL^{\ten r}) $ via projection to the $r=e$ factor, so an $n$-shifted pre-symplectic structure on $F$ leads to an $(n-d)$-shifted pre-symplectic structure on $\oR\map(P,F)$, though the latter is rarely representable because $P$ is not proper. 

\begin{example}
 If $X$ is the nerve $BG$ of an algebraic group, the line bundle $\sL$ corresponds to an algebraic group homomorphism $G \to \bG_m$, and then $P=[\bG_m/G]$, which is isomorphic to $B\ker(G \to \bG_m)$ if the homomorphism is surjective. The procedure of replacing $X$ with $P$ is thus analogous to replacing a Galois group $\Gal(F)$ with that of a cyclotomic extension $\Gal(F(\mu_{\infty}))$, as featured in \S \ref{cycdualitysn}. By looking at homomorphisms to $\bG_m$ rather than $\hat{\Z}^*\cong \Gal(F)/\Gal(F(\mu_{\infty})) $, we produce objects which have better finiteness properties than infinite cyclotomic extensions would; even when infinite-dimensional, they are so in a controlled way.
\end{example}

When $F$ is symplectic, this $(n-d)$-shifted pre-symplectic structure becomes symplectic under finiteness conditions on $\bigoplus_{r \in \Z}\H^*(X,\sL^{\ten r})$, which in particular require almost all of the complexes $\oR\Gamma(X,\sL^{\ten r})$ to be acyclic. We can resolve this by looking at a formal neighbourhood of $\oR\map(X,F)$ in $\oR\map(P,F)$, which carries a form of shifted symplectic structure and representability relying only on finiteness conditions on the groups $\H^*(X,\sL^{\ten r})$ separately for each $r$.

More explicitly, a $\bG_m$-equivariant derived affine scheme is given by  a CDGA $A$ equipped with a weight decomposition $A = \bigoplus_{r \in \Z} \cW_rA$, and then the space of $\bG_m$-equivariant maps from $\Spec A$ to $\oR\map(P,F) $ consists of maps 
\[
\oSpec_X (\bigoplus_{r \in \Z} \cW_rA \ten \sL^{\ten r}) \to F.
\]
The $\bG_m$-invariant substack $\oR\map(X,F)$  is then recovered by restricting to  $\bG_m$-equivariant CDGAs $A$ with $A=\cW_0A$. Its formal neighbourhood in $\oR\map(P,F) $ is  recovered by restricting to $\bG_m$-equivariant CDGAs with only finitely many non-zero weights, since in that case the ideal generated by non-zero weights is nilpotent.
This last functor has a natural affinoid analogue, since the category of coherent modules contains finite direct sums; we can also study it without needing to worry about an analytic analogue of the group scheme $\bG_m$. 

\subsection{Weighted structures}


\subsubsection{Weighted dagger dg algebras}

We now introduce the test objects on which we define moduli functors of the form above; there are entirely similar definitions in the algebraic setting.

\begin{definition}\label{wdgqdaggerdef}
 Define a weighted  dagger (resp. quasi-dagger, resp. localised dagger) dg algebra to be a dagger (resp. quasi-dagger, resp. localised dagger) dg algebra $A_{\bt}$ equipped with a decomposition
 \[
  A_{\bt}= \bigoplus_{n \in \Z} \cW_n A_{\bt},
 \]
into chain subcomplexes, with only finitely many non-zero terms, respecting the multiplication in the sense  that for $a \in \cW_m A_{\bt}$ and $b \in \cW_n A_{\bt}$, the product $ab$ lies in $\cW_{m+n}A_{\bt}$.

Accordingly, say that an affinoid  dagger (resp. quasi-dagger, resp. localised dagger)  dg space  $X=(X^0, \sO_{X,\bt})$ is weighted if it is equipped with such a decomposition on the sheaf $\sO_{X,\bt}$; this gives a contravariant equivalence between weighted algebras and weighted affinoid spaces. 

A morphism $A \to B$ of weighted quasi-dagger dg algebras 
is just a homomorphism of dg $K$-algebras which respects the weight decompositions.
 \end{definition}
 
 \begin{definition}
  We denote the category of weighted localised  dagger dg $K$-algebras by $dg_+\Affd\Alg_K^{f\bG_m,\dagger,\loc}$. We then denote its full subcategory of objects which are bounded as chain complexes by $dg_+\Affd\Alg_K^{f\bG_m,\dagger,\loc, \flat}$.
 \end{definition}
 
 \begin{definition}
  Given  a weighted quasi-dagger dg algebra $A_{\bt}$, define the quasi-dagger dg algebra $A_{\bt}/\bG_m$ to be the quotient of $A_{\bt}$ by the dg ideal generated by $\bigoplus_{n \ne 0} \cW_n A_{\bt}$; all elements of this ideal are nilpotent because the weights are bounded.
  
 For a weighted affinoid quasi-dagger dg space $X$, then define the space $X^{\bG_m}$ of $\bG_m$-invariants by the property that $ \Gamma(X^{\bG_m},\sO_{X^{\bG_m},\bt})\cong \Gamma(X,\sO_{X,\bt})/\bG_m$.
 \end{definition}

\begin{definition}\label{wcohdef}
  Given a weighted quasi-dagger dg algebra $A=\bigoplus_{n\in \Z} \cW_nA$, we let $dg\Coh_A^{f\bG_m}$ (resp. $dg_+\Coh_A^{f\bG_m}$)  be the category of those $\bG_m$-equivariant $A$-modules $M=\bigoplus_n \cW_nM$ in chain complexes (resp. non-negatively graded chain complexes) which are levelwise finitely generated and have only finitely many $\cW_nM$ non-zero.

  We then define $dg\Coh_A^{\hat{f}\bG_m}$ (resp. $dg_+\Coh_A^{\hat{f}\bG_m}$) to be the larger category consisting of those $\bG_m$-equivariant $A$-modules $M=\bigoplus_n \cW_nM$ in chain complexes (resp. non-negatively graded chain complexes) which are generated as bigraded $A$-modules by sets with finitely many elements for each pair  (weight, degree). 
\end{definition}
In particular, note that given $M \in dg\Coh_A^{\hat{f}\bG_m}$, the quotient $M/(\cW_{<r}M,\cW_{>s}M)$ of $M$ by the $A$-submodule $(\cW_{<r}M,\cW_{>s}M)$ generated by $\bigoplus_{n \notin [r,s]} \cW_nM$ lies in $dg\Coh_A^{f\bG_m}$. Moreover, if the weights of $A$ lie in the interval $[-x,y]$ for $x,y \ge 0$, then the weights of $(\cW_{<r}M,\cW_{>s}M)$ lie in the range $(-\infty, r+y)\cup(s-x,\infty)$; thus $M$ is isomorphic to the limit
$\Lim_{r,s} M/(\cW_{<-r}M,\cW_{>s}M)$  in the category of $\bG_m$-equivariant $A$-modules, since the inverse system stabilises in each weight. 

\medskip
We then have notions of weighted (quasi-)dagger spaces and analytic stacks, by introducing weights into the definitions of \S \ref{spacesn}; we will not reiterate the definitions explicitly here. One difference is that we do not have enough quasi-free objects
in our category to allow quasi-free replacement, because a free generator in non-zero weight would generate an algebra neither dagger affinoid nor with bounded weights. As a consequence, our moduli functors will only be representable by formal spaces or stacks, built from objects of $\pro(dg_+\Affd\Alg_K^{f\bG_m,\dagger,\loc})$ as in \S \ref{wfspacesn}  below. 

\subsubsection{Alternative formulations and quasi-free objects}\label{wqufreesn}

In order to get round the lack of quasi-free objects in $dg_+\Affd\Alg_K^{f\bG_m,\dagger,\loc}$, we now introduce slightly larger categories with closely related, but  more manageable, homotopy theory. 

\begin{definition}\label{bigwdgqdaggerdef}
Define $dg_+\Affd\Alg_K^{\bG_m,\dagger,\loc}$ to consist of $K$-CDGAs  $A_{\bt}= A_{\ge 0}$ equipped with a decomposition
 \[
  A_{\bt}= \bigoplus_{n \in \Z} \cW_n A_{\bt},
 \]
into chain subcomplexes (all of which are allowed to be non-zero), respecting the multiplication, such that $\cW_0A_{\bt}$ is a quasi-dagger dg algebra, $A_{\bt}$ is generated as a  $\cW_0A_{\bt}$-algebra by finitely many generators in each chain degree, and the map $\cW_0A_0 \to \H_0\cW_0A_{\bt}/\bG_m$  gives an isomorphism on the associated sets of points (i.e. of maximal ideals).

A morphism $A_{\bt} \to B_{\bt}$ in $dg_+\Affd\Alg_K^{\bG_m,\dagger,\loc}$  
 is just a homomorphism of dg $K$-algebras which respects the weight decompositions.
 \end{definition}

\begin{remarks}
 Note that the finiteness condition in Definition \ref{bigwdgqdaggerdef} only applies in each chain degree, rather than in pairs (degree, weight), so in particular for each chain degree, objects of $dg_+\Affd\Alg_K^{\bG_m,\dagger,\loc}$ only have generators in finitely many weights.

 The condition that  $\cW_0A_0 \to \H_0\cW_0A_{\bt}/\bG_m$  give an isomorphism on the associated sets of points implies that $\cW_0A_0\to \H_0\cW_0A_{\bt}$ also gives an isomorphism on the associated sets of points, so $\cW_0A_{\bt}$  is a localised dagger dg algebra. Dropping that condition would give us a weighted generalisation of the notion of quasi-dagger dg algebra, and localising at  $\H_0\cW_0A_{\bt}/\bG_m$ (rather than $\H_0A_{\bt}$ as in  Definition \ref{Xlocdef}) gives an object $A^{\loc}_{\bt} \in dg_+\Affd\Alg_K^{\bG_m,\dagger,\loc}$, this functor being left adjoint to the obvious inclusion functor.
 \end{remarks}

 \begin{definition}\label{bigqufreedef}
 We say that an object $ A_{\bt} \in dg_+\Affd\Alg_K^{\bG_m,\dagger,\loc}$ is quasi-free if, 
 in the category of $\bG_m$-equivariant graded-commutative algebras which are quasi-dagger  in weight and degree $0$, the underlying  object $A_{\#}$ given by   forgetting the differential is freely generated over a quasi-MW algebra  $K\<\frac{x_1}{r_1}, \ldots ,\frac{x_n}{r_n}\>^{\dagger}$ (see Definition \ref{genWashdef}).
 
 Similarly, say that a morphism $f \co R_{\bt} \to A_{\bt}$ in $dg_+\Affd\Alg_K^{\bG_m,\dagger,\loc}$ is quasi-free if $A_{\#}$ is freely generated in that category over  $W\ten_{R_0}R_{\#}$ for some relative quasi-MW algebra $W=R_0\<\frac{x_1}{r_1}, \ldots ,\frac{x_n}{r_n}\>^{\dagger}$.
 \end{definition}

 As the following example shows, the condition that $\cW_0A_0$ be a  quasi-dagger algebra leads to some slightly unusual behaviour for these quasi-free objects
 \begin{example}
  The quasi-free object  $ A_{\bt} \in dg_+\Affd\Alg_K^{\bG_m,\dagger,\loc}$ generated by $x \in \cW_1A_0$ and $y \in \cW_{-1}A_0$ is given by 
  \[
   \cW_iA_{\bt} = \begin{cases} K\<\frac{xy}{0}\>^{\dagger} x^i & i\ge 0 \\
                   K\<\frac{xy}{0}\>^{\dagger} y^{-i} & i\le 0,
                  \end{cases}
  \]
with the obvious multiplication. 

This has the property that $\Hom_{dg_+\Affd\Alg_K^{\bG_m,\dagger,\loc}}(A_{\bt},B_{\bt}) \cong \cW_1B_0 \by \cW_{-1}B_0$ for any $B$, since for $(u,v) \in \cW_1B_0 \by \cW_{-1}B_0$, we have  $uv \in \ker (\cW_0B \to B/\bG_m)$, so  the localisation condition ensures that it gives rise to a morphism $K\<\frac{xy}{0}\>^{\dagger} \to \cW_0B_0$.
 \end{example}
 
 The proof of Lemma \ref{coffactnlemma} adapts directly to give:
 \begin{lemma}\label{wcoffactnlemma}
  Every morphism $ f \co A \to B$ in $dg_+\Affd\Alg_K^{\bG_m,\dagger,\loc}$ admits a factorisation  $A \xra{p} \breve{B} \xra{r} B $ with $p$ quasi-free and $r$ a surjective quasi-isomorphism, and a factorisation  $A \xra{q} \tilde{B} \xra{s} B $ with $q$ a quasi-free quasi-isomorphism and $s$ surjective in strictly positive degrees.
 \end{lemma}

%
%
 
%
 
 \begin{lemma}\label{wcompletelemma}
  There is a functor $A \mapsto \hat{A}$ from $dg_+\Affd\Alg_K^{\bG_m,\dagger,\loc}$ to the pro-category $\pro(dg_+\Affd\Alg_K^{f\bG_m,\dagger,\loc,\flat})$ determined by the property that for $A \in dg_+\Affd\Alg_K^{\bG_m,\dagger,\loc}$ and $B \in dg_+\Affd\Alg_K^{f\bG_m,\dagger,\loc,\flat}$, we have
  \[
   \Hom_{\pro(dg_+\Affd\Alg_K^{f\bG_m,\dagger,\loc,\flat})}(\hat{A},B) \cong \Hom_{dg_+\Affd\Alg_K^{\bG_m,\dagger,\loc}}(A,B).
  \]

This functor sends quasi-isomorphisms to pro-quasi-isomorphisms (i.e. essentially levelwise quasi-isomorphism in the sense of \cite[\S 2.1]{isaksenStrict}), so induces a functor on the simplicial localisation at those classes of morphisms, respectively. 
 \end{lemma}
\begin{proof}
 Since  $B$ is bounded and exists in finitely many weights, any morphism  $A\to B$ must factor through $A/\tau_{>k}A/(\cW_{<-n}, \cW_{>n})$ for some $n,k$. Since $A$ is degreewise finitely generated over a localised dg dagger  algebra, those quotients all live in $dg_+\Affd\Alg_K^{f\bG_m,\dagger,\loc,\flat}$, so we just set $\hat{A}:=\{A/\tau_{>k}A/(\cW_{<-n}, \cW_{>n})\}_{n,k}$.
 
It remains to prove the final statement. Since good truncation preserves quasi-isomorphisms, we may assume that $A$ is bounded, in which case $\hat{A}$ just becomes the inverse system  $\{A/\tau_{>k}A/(\cW_{<-n}, \cW_{>n})\}$. Now, since $A$ exists in only finitely many chain degrees, it is finitely generated over $A_0$. There thus exists some generating set lying in weights $[-r,r]$, which implies that the ideal $(\cW_{\le-nr}, \cW_{\ge nr})$ is contained in $(\cW_{\ne 0})^n$, where $(\cW_{\ne 0})$ is the ideal generated by elements of non-zero weight. Moreover, any set of $n$ elements in $\cW_{\ne 0}$ contains either at least $\lceil n/2\rceil$ positively weighted elements, or at least $\lceil n/2\rceil $ negatively weighted elements, from which we conclude that $ (\cW_{\ne 0})^n \subset (\cW_{<-n/2}, \cW_{>n/2})$. Thus the systems $\{(\cW_{<-n}, \cW_{>n})\}_n$ and  $(\cW_{\ne 0})^n$ of ideals are equivalent, and $\hat{A}$ is isomorphic to the $(\cW_{\ne 0})$-adic completion $\{A/(\cW_{\ne 0})^n\}_n$.

Now, the proof of \cite[Lemma 2.13]{drep} generalises to give a pro-quasi-isomorphism $\{A/I^n\}_n \to \{B/f(I)^n\}_n$ 
for any  surjective quasi-isomorphism $f \co A \to B$ of finitely generated bounded non-negatively graded CDGAs over a Noetherian $\Q$-algebra, for any ideal $I$. Thus the functor $A \mapsto \hat{A}$ sends surjective quasi-isomorphisms to pro-quasi-isomorphisms.

Given an arbitrary quasi-isomorphism $A \to A'$, Lemma \ref{wcoffactnlemma} gives us a factorisation $A \xra{q} C \xra{s} A \by A'$ of its graph, with $s$ surjective in strictly positive degrees and $q$ a quasi-isomorphism. Thus the morphisms $s_1 \co C \to A$ and $s_2 \co C \to A'$ are surjective quasi-isomorphisms (since they are quasi-isomorphisms and are surjective in strictly positive degrees).  This  gives us pro-quasi-isomorphisms $\hat{A} \la \hat{C} \to \hat{A'}$, and  the section $\hat{q}\co \hat{A} \to \hat{C}$  of $\hat{s}_1$ is thus also a   pro-quasi-isomorphism, as is the composite $\hat{f}= \hat{s}_2 \circ \hat{q}  \co \hat{A} \to \hat{A'}$.
\end{proof}

\begin{proposition}\label{waffdsubprop1}
The inclusion functor $ dg_+\Affd\Alg_K^{f\bG_m,\dagger,\loc,\flat} \into dg_+\Affd\Alg_K^{\bG_m,\dagger,\loc}$  induces a fully faithful functor on simplicial categories after simplicial localisation at quasi-isomorphisms. More generally, for $A, B \in dg_+\Affd\Alg_K^{f\bG_m,\dagger,\loc}$, we have
\[
 \oR\map_{dg_+\Affd\Alg_K^{\bG_m,\dagger,\loc}}(A,B) \simeq \oR\Lim_k \oR\map_{dg_+\Affd\Alg_K^{f\bG_m,\dagger,\loc,\flat}}(A/\tau_{>k}A, B/\tau_{>k}B).
\]
\end{proposition}
\begin{proof}
Given $A \in  dg_+\Affd\Alg_K^{\bG_m,\dagger,\loc}$, the proof of  Proposition \ref{affdsubEFCprop}, using  Lemma \ref{wcoffactnlemma} in place of of Lemma \ref{coffactnlemma},  gives us a cosimplicial resolution $\tilde{A}^{\bt}\to A$ which is quasi-free in the sense that the latching maps are all quasi-free, and a resolution in the sense that  the degeneracy maps are all quasi-isomorphisms and we have a quasi-isomorphism  $\tilde{A}^0 \to A$. Our definition of  $dg_+\Affd\Alg_K^{\bG_m,\dagger,\loc}$ ensures that quasi-free morphisms have the left lifting property with respect to  acyclic surjections, so the functor $\Hom_{dg_+\Affd\Alg_K^{\bG_m,\dagger,\loc}}(\tilde{A}^{\bt},-) \co  dg_+\Affd\Alg_K^{\bG_m,\dagger,\loc} \to s\Set$ given in level $m$ by $\Hom_{dg_+\Affd\Alg_K^{\bG_m,\dagger,\loc}}(\tilde{A}^m,-)$ sends acyclic surjections to trivial Kan fibrations, and in particular to weak equivalences. As in the proof of Lemma \ref{wcompletelemma}, this implies that it sends all quasi-isomorphisms to weak equivalences. Thus Lemma \ref{locmaplemma} implies that 
\[
\oR\map_{dg_+\Affd\Alg_K^{\bG_m,\dagger,\loc}}(A,-) \simeq \Hom_{dg_+\Affd\Alg_K^{\bG_m,\dagger,\loc}}(\tilde{A}^{\bt},-). 
\]

Now for $B \in dg_+\Affd\Alg_K^{f\bG_m,\dagger,\loc,\flat}$,  Lemma \ref{wcompletelemma} gives 
\[
 \Hom_{dg_+\Affd\Alg_K^{\bG_m,\dagger,\loc}}(\tilde{A}^m,B)\cong \Hom_{\pro(dg_+\Affd\Alg_K^{f\bG_m,\dagger,\loc,\flat})}(\widehat{\tilde{A}^{m}},B),
\]
while for $B \in dg_+\Affd\Alg_K^{f\bG_m,\dagger,\loc}$ we have
\begin{align*}
 \Hom_{dg_+\Affd\Alg_K^{\bG_m,\dagger,\loc}}(\tilde{A}^m,B) &\cong \Lim_n \Hom_{dg_+\Affd\Alg_K^{\bG_m,\dagger,\loc}}(\tilde{A}^m,B/\tau_{>n}B)\\
 &\cong \Lim_n \Hom_{\pro(dg_+\Affd\Alg_K^{f\bG_m,\dagger,\loc,\flat})}(\widehat{\tilde{A}^{m}},B/\tau_{>n}B);
\end{align*}
quasi-freeness of $\tilde{A}^{\bt}$ implies that  $\Hom_{dg_+\Affd\Alg_K^{\bG_m,\dagger,\loc}}(\tilde{A}^{\bt},-)$ maps nilpotent surjections to  Kan fibrations so 
\[
\Lim_n \Hom_{\pro(dg_+\Affd\Alg_K^{f\bG_m,\dagger,\loc,\flat})}(\widehat{\tilde{A}^{\bt}},B/\tau_{>n}B)
\simeq \ho\Lim_n \Hom_{\pro(dg_+\Affd\Alg_K^{f\bG_m,\dagger,\loc,\flat})}(\widehat{\tilde{A}^{\bt}},B/\tau_{>n}B)
\]

Now, observe that if $A \in dg_+\Affd\Alg_K^{f\bG_m,\dagger,\loc}$, the inverse system $\hat{A}$ from  Lemma \ref{wcompletelemma} is isomorphic to $\{A/\tau_{>n}A\}_n$, so that lemma gives quasi-isomorphisms 
\[
\widehat{\tilde{A}^{m}} \to \{A/\tau_{>n}A\}_n \quad\text{ and }\quad \widehat{\tilde{A}^{m}}/\tau_{>n}\widehat{\tilde{A}^{m}}  \to A/\tau_{>n}A.
\]
In particular, 
$\widehat{\tilde{A}^{\bt}}$ is a cosimplicial resolution of $\{A/\tau_{>n}A\}_n$ in the pro-category; since $\Hom_{\pro(dg_+\Affd\Alg_K^{f\bG_m,\dagger,\loc,\flat})}(\widehat{\tilde{A}^{\bt}},-)$ sends quasi-isomorphisms in $dg_+\Affd\Alg_K^{f\bG_m,\dagger,\loc,\flat}$ to weak equivalences, the proof of Lemma \ref{locmaplemma} (embedding the pro-category in the model category of restricted diagrams)  then  implies that
\[
 \oR\map_{dg_+\Affd\Alg_K^{f\bG_m,\dagger,\loc,\flat}}(A/\tau_{>n}A,-) \simeq \Hom_{\pro(dg_+\Affd\Alg_K^{f\bG_m,\dagger,\loc,\flat})}(\widehat{\tilde{A}^{\bt}/\tau_{>n}},-)
\]
on $dg_+\Affd\Alg_K^{f\bG_m,\dagger,\loc,\flat}$, so for $B \in dg_+\Affd\Alg_K^{f\bG_m,\dagger,\loc,\flat}$ we have
\begin{align*}
 \oR\map_{dg_+\Affd\Alg_K^{\bG_m,\dagger,\loc}}(A,B) &\simeq   \ho\Lim_n \Hom_{\pro(dg_+\Affd\Alg_K^{f\bG_m,\dagger,\loc,\flat})}(\widehat{\tilde{A}^{\bt}},B/\tau_{>n}B)\\
 &\cong   \ho\Lim_n \Hom_{\pro(dg_+\Affd\Alg_K^{f\bG_m,\dagger,\loc,\flat})}(\widehat{\tilde{A}^{\bt}/\tau_{>n}},B/\tau_{>n}B)\\
  &\simeq \ho\Lim_n \oR\map_{dg_+\Affd\Alg_K^{f\bG_m,\dagger,\loc,\flat}}(A/\tau_{>n}A,B/\tau_{>n}B ).
\end{align*}

\end{proof}
 
\begin{lemma}\label{wextendlemma}
The essential image of $\Ho(dg_+\Affd\Alg_K^{f\bG_m,\dagger,\loc,\flat})\to \Ho(dg_+\Affd\Alg_K^{\bG_m,\dagger,\loc})$ is closed under square-zero extensions by complexes $I$ with homology concentrated in a single chain degree and with finitely many weights. 
 \end{lemma}
\begin{proof}
 Assume that we have a surjection $A \onto B$ in $dg_+\Affd\Alg_K^{\bG_m,\dagger,\loc}$ such that the kernel $I$, with conditions as above, squares to $0$ and $B$ is quasi-isomorphic to an object $\bar{B}$ of $dg_+\Affd\Alg_K^{f\bG_m,\dagger,\loc,\flat}$. Then we wish to show that $A$ is also quasi-isomorphic to such an object.
 
There is an obvious  $K$-CDGA structure on $\cone(I \to A)$, and this is easily seen to lie in $dg_+\Affd\Alg_K^{\bG_m,\dagger,\loc}$, with the obvious map $\cone(I \to A) \to B$ being a quasi-isomorphism. Since $I$ squares to $0$, we also have a $K$-CDGA morphism $\cone(I \to A) \to \cone(I \xra{0} B)=B \oplus I[1]$, and then $A= \cone(I \to A)\by_{ B \oplus I[1]}B$.
 
 If $\tilde{B}\to \bar{B}$ is a quasi-free replacement for $\bar{B}$, then as in the proof of Lemma \ref{wcompletelemma}, the lifting property of quasi-free morphisms with respect to acyclic surjections  gives us a map (necessarily a quasi-isomorphism) $\tilde{B} \to B$, which then further lifts to give a map $\tilde{B} \to \cone(I \to A)$ and hence $\tilde{B} \to B \oplus I[1]$. Assume that $\H_*I$ is concentrated in degree $n$. Without loss of generality, we may then replace $I$ with $\tau_{\ge n}I$ in the statements above, since $A/\tau_{\ge n}I \to A/I$ is a quasi-isomorphism. Since we are now assuming that $I$ is concentrated in degrees $\ge n$, there is a quasi-isomorphism $I \to \H_n(I)[n]$. In particular, we have $\tilde{B} \to \H_0B \oplus  \H_n(I)[n+1]$; enlarging $\tilde{B}$ if necessary, we may also assume that this map is surjective, so $A \simeq \tilde{B}\by_{ \H_0B \oplus  \H_n(I)[n+1]}\H_0B$.
 
Now, since $\H_0B \oplus  \H_n(I)[n+1] \in dg_+\Affd\Alg_K^{f\bG_m,\dagger,\loc,\flat}$, Lemma \ref{wcompletelemma} gives us a morphism $ f \co \widehat{\tilde{B}} \to   \H_0B \oplus  \H_n(I)[n+1]$ in $\pro(dg_+\Affd\Alg_K^{f\bG_m,\dagger,\loc,\flat})$. We also have a pro-quasi-isomorphism $\widehat{\tilde{B}} \to \widehat{B}$, so we can write $ \widehat{\tilde{B}}$ as an inverse system of objects mapping quasi-isomorphically to $B$. Our morphism $f$ thus factors through some such object $B'$; since $f$ is surjective, our map $f' \co B' \to \H_0B \oplus  \H_n(I)[n+1]$ is also so, and then we have
\[
 A \simeq B'\by_{ \H_0B \oplus  \H_n(I)[n+1]}\H_0B,
\]
which is an object of $dg_+\Affd\Alg_K^{f\bG_m,\dagger,\loc,\flat}$.
%
\end{proof}

 \begin{proposition}\label{waffdsubprop2}
 In the homotopy categories given by localising  at quasi-isomorphisms, the essential image of $\Ho(dg_+\Affd\Alg_K^{f\bG_m,\dagger,\loc,\flat})\to \Ho(dg_+\Affd\Alg_K^{\bG_m,\dagger,\loc})$ consists of objects $A$ for which the homology groups $\H_iA$  exist in only finitely many weights, 
 and vanish for $i\gg 0$.
\end{proposition}
\begin{proof}
 The conditions are clearly satisfied by objects in the image and are invariant under quasi-isomorphism, so it suffices to show that every object $B$ satisfying the conditions lies in the image of the relevant functor. To prove this, we will proceed by induction using the Postnikov tower $\{B/\tau_{>n}B\}_n$.
 
 
 For $n=0$, we have a quasi-isomorphism $B/\tau_{>0}B\to \H_0B$, which is a weighted localised dagger algebra, so lies in $dg_+\Affd\Alg_K^{f\bG_m,\dagger,\loc,\flat} \subset 
 dg_+\Affd\Alg_K^{\bG_m,\dagger,\loc}$. In general, the map $B/\tau_{>n+1}B\to B/\tau_{>n}B$ factorises as the composite of an acyclic surjection and a square-zero extension, so Lemma \ref{wextendlemma} implies that $B/\tau_{>n+1}B$ lies in the essential image of $dg_+\Affd\Alg_K^{f\bG_m,\dagger,\loc,\flat}$ whenever $B/\tau_{>n}B$ does so. This completes the inductive proof, since $\tau_{>nB}$ is acyclic for $n \gg 0$.
\end{proof}

The following is now an immediate consequence of Propositions \ref{waffdsubprop1} and \ref{waffdsubprop2}:

\begin{corollary}\label{waffdsubcor}
 The following simplicial categories are quasi-equivalent:
 \begin{enumerate}
  \item homotopy-preserving functors from $dg_+\Affd\Alg_K^{f\bG_m,\dagger,\loc,\flat}$ to the category of simplicial sets;
  \item  homotopy-preserving nilcomplete functors from $dg_+\Affd\Alg_K^{f\bG_m,\dagger,\loc}$ to the category of simplicial sets;
  \item homotopy-preserving nilcomplete functors from the full subcategory of $dg_+\Affd\Alg_K^{\bG_m,\dagger,\loc}$ on objects with homologically bounded weights to the category of simplicial sets.
 \end{enumerate}
 \end{corollary}

 \begin{definition}\label{wOmegaAffddef}
 Given $A \in dg_+\Affd\Alg_K^{\bG_m,\dagger,\loc}$, we define the $A$-module  $\Omega^1_A$ by the property  that  $d \co A \to \Omega^1_A$ is the universal $K$-linear derivation from $A$ to  
 $A$-modules $M$  in weighted complexes which are levelwise finitely generated. 
\end{definition}
Explicitly, we can calculate $ \Omega^1_A$ in terms of the algebraic $\bG_m$-equivariant cotangent module as  $\cone(  \Omega^1_{\cW_0A^{\alg}}\ten_{\cW_0A}A \to (\Omega^1_{\cW_0A}\ten_{\cW_0A}A) \oplus \Omega^1_{A^{\alg}})$, for the dagger  cotangent module $\Omega^1_{\cW_0A} $ of Definition \ref{OmegaAffddef}.

\begin{remark}[Weighted EFC-DGAs]\label{wEFCDGA}
 We could define a weighted EFC-DGA $A_{\bt}$ to be a $K$-CDGA $A_{\bt}$ equipped with a weight decomposition $\cW$ and a compatible EFC structure on $\cW_0A_0$. Every object of  $dg_+\Affd\Alg_K^{\bG_m,\dagger,\loc}$ has an underlying weighted EFC-DGA $A_{\bt}$, which is localised in the sense that  the natural map 
 $\cW_0A_0 \to (\cW_0A_0/(\H_0A_{\bt}/\bG_m))^{\loc}$ is an isomorphism, where $A_{\bt}/\bG_m$ is the quotient of $A_{\bt}$ by the dg EFC-ideal generated by $\cW_{\ne 0}A_{\bt}$.
 
 We could then develop weighted analogues of the results of \S \ref{cfEFCsn}, allowing us to recast weighted structures in terms of EFC-algebras, but we will instead just formulate our weighted results in the dagger affinoid world.
\end{remark}

\subsubsection{Tangent and cotangent complexes}

\begin{definition}\label{twistweightdef}
 Given $M \in dg\Coh_A^{f\bG_m}$, define $M\{r\}\in dg\Coh_A^{f\bG_m}$ to be the tensor product of $M$ with a rank $1$ $\bG_m$-representation of weight $-r$, so that  $\cW_n(M\{r\}) \cong \cW_{n+r}M$.
\end{definition}

  Generalising Definition \ref{Tdef}, we have:
\begin{definition}\label{wTdef}
Given a homotopy-preserving homogeneous  functor $F\co dg_+\Affd\Alg_K^{f\bG_m,\dagger,\loc} \to s\Set$, an object $A \in dg_+\Affd\Alg_K^{f\bG_m,\dagger,\loc}$   and a point $x \in F(A)$, define the tangent functor $T_xF$
\[
 T_xF \co dg_+\Coh_A^{f\bG_m} \to s\Set,
\]
by
$
T_xF(M):= F(A\oplus M)\by^h_{F(A)}\{x\}.
$

By the argument of \cite[Lemma \ref{drep-adf}]{drep}, the space $T_xF(M[1])$ deloops $T_xF(M)$, so we  then define tangent cohomology groups by  $\DD^{n-i}_x(F,M):= \pi_i (F(A\oplus M[n])\by^h_{F(A)}\{x\})$.
\end{definition}

\begin{definition}\label{wFcotdef}
 In the setting of Definition \ref{wTdef}, we say that $F$ has a coherent cotangent complex $\bL^{F,x}$ at $x$ if there is a weighted  $A$-module $ \bL^{F,x} \in dg\Coh_A^{\hat{f}\bG_m}$ in chain complexes,  
 bounded below in  chain degrees, representing $T_x(F)$ homotopically  in the sense that the simplicial mapping space 
 \[
  \oR\map_{dg\Coh_{A}^{\hat{f}\bG_m}}(\bL^{F,x},-) 
 \]
is weakly equivalent to $T_x(F)$ when restricted to $dg_+\Coh_{A}^{f\bG_m}$. 
\end{definition}
In particular, this means that 
\[
\pi_i T_x(F)(M)\cong \EExt^{-i}_{A}(\bL^{F,x},M)^{\bG_m}= \cW_0\EExt^{-i}_{A}(\bL^{F,x},M),
\]
for all $M \in dg_+\Coh_{A}^{f\bG_m}$. 
 
 \begin{remark}
  It is important to note that we are allowing the cotangent complex to have infinitely many generators in each degree, provided that for each degree, there are finitely many in each weight. Although the more restrictive definition would be easier to work with, for our applications of interest, it would only allow us to handle pro-\'etale sheaves on schemes over local fields, not over global fields. 
 \end{remark}

\begin{example}\label{wqufreecotex}
 Given an object $ C \in dg_+\Affd\Alg_K^{\bG_m,\dagger,\loc}$ which is quasi-free in the sense of Definition \ref{bigqufreedef}, we can look at the functor $F:=\oR\map_{dg_+\Affd\Alg^{\bG_m,\loc,\dagger}_K}(C,-)$  on $dg_+\Affd\Alg_K^{f\bG_m,\dagger,\loc}$  given by taking the mapping space in the simplicial localisation at quasi-isomorphisms.
 
 At a point $x \in F(A)$ given by a morphism $C \to A$, the functor $T_x(F)$ is given by
 \[
  M \mapsto \oR\map_{dg_+\Mod_C^{\bG_m}}(\Omega^1_C,M) \simeq \tau_{\ge 0}\HHom_{dg_+\Mod_C^{\bG_m}}(\Omega^1_C,M),
 \]
since quasi-freeness implies that $\Omega^1_C$ is cofibrant as a weighted $C$-module. Explicitly, this implies that 
\[
 \pi_i T_x(F)(M)\cong \EExt^{-i}_{C}(\Omega^1_C,M)^{\bG_m}= \cW_0\EExt^{-i}_{C}(\Omega^1_C,M),
\]
and that $F$ has a cotangent complex at $x$ given by
\[
 \Omega^1_C\ten_CA.
\]
\end{example}
 
 \begin{lemma}\label{wcotexistslemma}
  If $F\co dg_+\Affd\Alg_K^{f\bG_m,\dagger,\loc} \to s\Set$ is a homotopy-preserving, homogeneous, nilcomplete functor such that for all  dagger  algebras $A$ (regarded as living in weight $0$ and degree $0$) and all points  $x \in F(A)$, the groups $\DD^i_x(F, A\{m\})$ are all  finitely generated $A$-modules and all vanish for $i\ll 0$, then $F$ has coherent cotangent complexes $\bL^{F,y}$ at  all points $y \in F(B)$ for all $B \in dg_+\Affd\Alg_K^{f\bG_m,\dagger,\loc}$. 
  \end{lemma}
\begin{proof}
 This proceeds in much the same way as Lemma \ref{cotexistslemma}. Observe that because $C$ is a nilpotent extension of $C/\bG_m$, our finiteness hypothesis implies that $\DD^i_x(F, M)$ is finitely generated for all weighted dagger algebras $C$, all $z \in F(C)$ and all weighted coherent $C$-modules $M$, since the category of weighted coherent $C$-modules is generated by the modules $(C/\bG_m)\{m\}$.
 
 We now outline the other modifications needed to the argument of \cite[Theorem 3.6.9]{lurie} to work with weighted modules. The basic idea is still that delooping allows us to extend $T_y$ to a functor $\bT_y$ on bounded below complexes in $dg\Coh_{B}^{f\bG_m}$. Moreover, elements of $\cW_nB$ define $\bG_m$-equivariant morphisms $M\{i\} \to M\{i+n\}$ for $M \in dg\Coh_{B}^{f\bG_m}$, sufficiently canonically that $\bigoplus_n \bT_y(M\{n\})$ is naturally a $\bG_m$-equivariant $B$-module, with $ \bT_y(M\{n\})$ given weight $n$.
 
 Given a dualising complex $K_{\cW_0B}$ for $\cW_0B$, we have a dualising complex $K_B:=\oR\HHom_{\cW_0B}(B,K_{\cW_0B})$ which carries a natural $\bG_m$-action, so lies in  $dg\Coh_{B}^{f\bG_m}$ since $B$ has finitely many weights. Then we set
\[
\cW_n\bL^{F,y}:=\oR\HHom_{\cW_0B}(\bT_y(K_B\{-n\}),K_{\cW_0B}); 
\]
this has  multiplication maps $\cW_mB \ten \cW_n\bL^{F,y} \to \cW_{m+n}\bL^{F,y}$ coming from the maps $\cW_mB \ten \bT_y(K_B\{-m-n\}) \to \bT_y(K_B\{-n\})$, and these combine to give $\bL^{F,y}$ the structure of a $\bG_m$-equivariant 
$B$-module.

More generally,  the duality functor $D_{\cW_0B}M:= \oR\HHom_{\cW_0B}(M,K_{\cW_0B})$ on  $\cW_0B$-modules gives rise to a duality functor $D_B^{\bG_m}$ on $\bG_m$-equivariant $B$-modules, given by 
\[
 \cW_n(D_B^{\bG_m}M):= D_{\cW_0B}(\cW_{-n}B),
\]
with the obvious $B$-module structure on   $\bigoplus_n \cW_n(D_B^{\bG_m}M)$. 

For bounded complexes $M \in dg\Coh_{B}^{f\bG_m}$, the argument of \cite[Theorem 3.6.9]{lurie} then adapts to show that $\oR\HHom_B(M, \bigoplus_n \bT_y(K_B\{n\}))^{\bG_m} \simeq \bT_y(\oR\HHom_B(M,K_B))$, for the natural $\bG_m$-action on $\oR\HHom_B(M,K_B)$, since the statement holds when $M$ is of the form $B\{n\}$. For bounded coherent complexes $N \in dg\Coh_{B}^{f\bG_m}$, we thus have equivalences
\begin{align*}
 \oR\HHom_B(\bL^{F,y},N)^{\bG_m} &\simeq \oR\HHom_B(D_B^{\bG_m}N, D_B^{\bG_m} \bL^{F,y})^{\bG_m}\\
 &\simeq \oR\HHom_B(D_B^{\bG_m}N, \bigoplus_n \bT_y(K_B\{n\}))^{\bG_m}\\
 &\simeq \bT_y(\oR\HHom_B(D_B^{\bG_m}N,K_B))\\
 &\simeq \bT_y(\oR\HHom_B(D_B^{\bG_m}N,D_B^{\bG_m}B))\\
 &\simeq \bT_y(\oR\HHom_B(B,N))\simeq    \bT_y(N),
\end{align*}
as required,
noting that $N$ only contains finitely many weights by hypothesis. The statement then extends to arbitrary $N \in dg_+\Coh_{B}^{f\bG_m}$ by nilcompleteness.
%
\end{proof}

\subsection{Weighted shifted symplectic structures on weighted dagger dg algebras}
 
 \subsubsection{Weighted pre-symplectic structures}
 
 The category $dg_+\Affd\Alg_K^{f\bG_m,\dagger,\loc}$ of weighted localised  dagger dg $K$-algebras does not  contain any quasi-free objects with non-zero weights, but the results of \S \ref{wqufreesn} (with a similar argument to Example \ref{wqufreecotex}) imply that the functor $A \mapsto (A,\Omega^1_A)$ and its alternating powers $A \mapsto (A,\Omega^n_A)$ admit left-derived functors taking values in the category of pairs $(A,M)$ for $M \in dg_+\Coh_A^{\hat{f}\bG_m}$. 
 
 On objects, we can construct the derived functors explicitly by taking a quasi-free resolution $C$ of $A$ in the larger category $dg_+\Affd\Alg_K^{\bG_m,\dagger,\loc}$, then setting 
 \[
 \oL\Omega^n_A:= \Omega^n_C\ten_CA \in dg_+\Coh_A^{\hat{f}\bG_m}.
 \]
 In particular, note that although $A$ has finitely many weights, the resolution $C$ can have generators in infinitely many weights as the degree increases, which is why $\oL\Omega^1_A$ lies in $dg_+\Coh_A^{\hat{f}\bG_m}$ rather than $dg_+\Coh_A^{f\bG_m}$.
 
 \begin{definition}\label{wDRdef}
Given  a weighted algebra 
$A \in dg_+\Affd\Alg_K^{\bG_m,\dagger,\loc}$,  define the de Rham complex $\DR(A)$ to be the $\bG_m$-equivariant complex $\DR(A):=\bigoplus_n \cW_n\DR(A)$ given by setting $\cW_n\DR(A)$ to be
the product total cochain complex of the double complex
\[
 \cW_nA \xra{d} \cW_n\Omega^1_{A} \xra{d} \cW_n\Omega^2_{A}\xra{d} \ldots,
\]
so the total differential is $d \pm \delta$.

We define the Hodge filtration $F$ on  $\DR(A)$ by setting $F^p\cW_n\DR(A) \subset \cW_n\DR(A)$ to consist of terms $\cW_n\Omega^i_{A}$ with $i \ge p$.

Define $\oL\DR(A)$ to be $\DR(\tilde{A})$ for any quasi-free replacement of $A$. 
\end{definition}
 
 \begin{definition}\label{wPreSpdef}
 Define the space $\cW_m\PreSp(A,n)$ of $n$-shifted pre-symplectic structures of weight $m$ on an object $A \in dg_+\Affd\Alg_K^{\bG_m,\dagger,\loc}$ to be the simplicial set 
 given by Dold--Kan denormalisation of the chain complex
\[
 \tau_{\ge 0}(\oL F^2\cW_m\DR(A)^{[n+2]}). 
\]

\end{definition}

 \begin{definition}\label{wIsodef}
  Given a morphism $A \to B$ in $dg_+\Affd\Alg_K^{\bG_m,\dagger,\loc}$, define 
  the space $\cW_m\Iso(A,B;n)$
  of  $n$-shifted isotropic structures of weight $m$ on the pair $(A,B)$ to be the simplicial set given  by Dold--Kan denormalisation of the chain complex 
\[
 \tau_{\ge 0}(\cone(\oL F^2\cW_m\DR(A)  \to \oL F^2\cW_m\DR(B))^{[n+1]}).
\]
\end{definition}

 \subsubsection{Weighted symplectic structures}

 \begin{definition}\label{cHomdef}
  Given a $\bG_m$-equivariant CDGA $A_{\bt}=\bigoplus_n\cW_nA_{\bt}$, define the internal $\Hom$ functor $\cHom$ on the category of $\bG_m$-equivariant $A$-modules by setting
  \[
   \cW_n\cHom(M,N):=\HHom_A(M, N\{n\})^{\bG_m},
  \]
the complex of $A$-linear $\bG_m$-equivariant maps from $M$ to $N\{n\}$, 
  for $N\{n\}$ as in Definition \ref{twistweightdef}. 

  Multiplication by elements of $\cW_mA$ gives $\bG_m$-equivariant maps $N\{n\} \to N\{n+m\}$, which combine to give the $\bG_m$-equivariant $A$-module structure on $\cHom(M,N):=\bigoplus_n \cW_n\cHom(M,N)$.
 \end{definition}

 \begin{definition}\label{wSpdef}
 For $A \in dg_+\Affd\Alg_K^{\bG_m,\dagger,\loc}$,  say that an $n$-shifted pre-symplectic structure $\omega \in  \z^{n+2}F^2\cW_n\DR(\tilde{A})$ of weight $m$ is \emph{symplectic} if the induced map
 \[
  \omega_2^{\sharp} \co \cHom_{\tilde{A}}(\Omega^1_{\tilde{A}},\tilde{A}) \to (\Omega^1_{\tilde{A}})_{[-n]}\{m\}
 \]
is a quasi-isomorphism. 

We then define the space $\cW_m\Sp(A,n)$ of $n$-shifted symplectic structures of weight $m$ to be the subspace of $\cW_m\PreSp(A,n)$ (a union of path components) consisting of symplectic objects.
 \end{definition}

 \begin{definition}\label{wLagdef}
Given a morphism  $f \co A \to B$ in $dg_+\Affd\Alg_K^{\bG_m,\dagger,\loc}$, we say that  an element $(\omega, \lambda)$ of 
\[
 \z^{n+1}\cone(F^2\cW_m\DR(\tilde{A})\to \oL F^2\cW_m\DR(B))
\]
is Lagrangian of weight $m$  if  $\omega \in \cW_m\Sp(A,n)$ (i.e. is symplectic)
 and if contraction with the image $(\omega_2,\lambda_2)$ of $(\omega,\lambda)$ in  $\z^{n-1}\cone(\cW_m\Omega^2_{\tilde{A}} \to \cW_m\Omega^2_{\tilde{B}} )$ 
 induces a quasi-isomorphism
\[
 (f \circ \omega_2^{\sharp}, \lambda_2^{\sharp}) \co 
 \cone(\cHom_{\tilde{B}}(\Omega^1_{\tilde{B}},\tilde{B}) \to \cHom_{\tilde{A}}(\Omega^1_{\tilde{A}},\tilde{B}))  \to (\Omega^1_{\tilde{B}})_{[-n]}\{m\}.
\] 
 
Set $\cW_m\Lag(A,B;n) \subset \cW_m\Iso(A,B;n)$ to consist of the   Lagrangian structures --- this is a union of path-components.
  
 \end{definition}
 
\begin{remark}\label{wNstackrmk2}
 For general homotopy-preserving functors $F \co dg_+\Affd\Alg_K^{f\bG_m,\dagger,\loc} \to s\Set$, we can adapt Remark \ref{Nstackrmk2} and let  $\cW_m\PreSp(F,n)$ be  the space $\oR\map(F,  \cW_m\PreSp(-,n))$ of maps of homotopy-preserving presheaves from $F$ to $\cW_m\PreSp(-,n)$, thus functorially associating an $n$-shifted pre-symplectic structure of weight $m$ on $A$ to each point in $F(A)$. 
 
 In order to define a subspace of shifted symplectic structures of weight $m$, we need $F$ to moreover be homogeneous with a cotangent complex, and then we can let  $\cW_m\Sp(F,n)\subset  \cW_m\PreSp(F,n)$ consist of the objects $\omega$ which are non-degenerate in the sense that the maps 
 \[
 (\omega_x)_2^{\sharp} \co \cHom_{\tilde{A}}(\Omega^1_{\tilde{A}},\tilde{A}) \to (\Omega^1_{\tilde{A}})_{[-n]}\{m\},
\]
for all $x \in F(A)$ are induced by compatible quasi-isomorphisms
\[
 \cHom_{\tilde{A}}(\bL^{F,x},\tilde{A}) \to (\bL^{F,x})_{[-n]}\{m\};
\]
when $F$ has an \'etale cover by objects of $dg_+\Affd\Alg_K^{f\bG_m,\dagger,\loc}$, this amounts to saying that those objects carry compatible shifted symplectic structures of weight $m$.

We can make entirely similar constructions for isotropic and Lagrangian structures, with the former given as a mapping space over the category of arrows in $dg_+\Affd\Alg_K^{f\bG_m,\dagger,\loc}$.
\end{remark}

\subsection{Weighted formal dg dagger spaces}\label{wfspacesn}

Since our functors $F$ will typically have cotangent complexes with infinitely many weights (i.e.  lying in $ dg\Coh_{B}^{\hat{f}\bG_m}$ rather than  $dg\Coh_{B}^{f\bG_m}$), we cannot just take weighted dagger dg algebras (or dually weighted dg dagger affinoid spaces) as building blocks. Instead, we work with formal weighted algebras and spaces, so take inverse systems $\{A(\alpha)\}_{\alpha}$ of localised weighted dagger dg algebras for which the system $\{\H_0A(\alpha)/\bG_m\}_{\alpha}$ is constant. Equivalently, this means looking at  direct systems $X= \{X(\alpha)\}_{\alpha}$ of weighted localised dg dagger affinoid spaces for which the system $\{\pi^0X(\alpha)^{\bG_m}\}_{\alpha}$ of $\bG_m$-invariants in the underived truncation is constant; we can then glue these  using just the topology on $\pi^0X^{\bG_m}$. 

\subsubsection{Weighted formal dagger dg algebras}

\begin{definition}
 Define a weighted formal localised dagger dg algebra to be an object $A=\{A(i)\}_{i \in I}$ of the pro-category $\pro(dg_+\Affd\Alg_K^{f\bG_m,\dagger,\loc, \flat})$ for which the pro-object $\H_0A/\bG_m=\{\H_0A(i)/\bG_m\}_{i \in I} $ is isomorphic to a constant filtered system (i.e.  lies in the essential image of $ \Affd\Alg_K \to \pro(\Affd\Alg_K)$).
\end{definition}

\begin{example}
 If we start with a dagger  algebra $C$ in weight $0$ and introduce free variables $x,y$ of weights $1,-1$, then the free weighted formal  dagger $K$-algebra $D$ over $C$ generated by $x,y$ is given by the CDGA
 \[
  \cW_nD= \begin{cases} 
           x^nC\llb xy \rrb & n \ge 0\\
           y^{-n} C\llb xy \rrb & n \le 0,
          \end{cases}
          \]
regarded as a limit of the weighted  dagger $K$-algebras $D/(\cW_{<-m}D,\cW_{>m}D)$ given by
\[
 \cW_n(D/(\cW_{<-m}D,\cW_{>m}D)) = \begin{cases} 
           x^n(C[ xy]/(xy)^{m+1-n}) & 0 \le n \le m\\
           y^{-n} (C[xy]/(xy)^{m+1+n})  & -m \le  n \le 0\\
           0 & n \notin [-m,m].
          \end{cases}
\]
This has the property that 
\[
 \Hom_{\pro(\Affd\Alg_K^{f\bG_m,\dagger})}(D,E) \cong \Hom_{\Affd\Alg_K^{\dagger}}(C,\cW_0E)\by\cW_1E\by \cW_{-1}E
 \]
 for all weighted  dagger $K$-algebras $E$.
\end{example}

\begin{definition}
 Define simplicial mapping spaces $\oR\map_{ \pro(dg_+\Affd\Alg_K^{f\bG_m,\dagger,\loc, \flat})} $ by localisation at pro-quasi-isomorphisms (i.e. essentially levelwise quasi-isomorphism in the sense of \cite[\S 2.1]{isaksenStrict}). Explicitly, for $A= \{A(i)\}_{i\in I}$ and $B=\{B(j)\}_{j \in J}$, we have
 \[
  \oR\map_{ \pro(dg_+\Affd\Alg_K^{f\bG_m,\dagger,\loc, \flat})}(A,B) \simeq \ho\Lim_j \LLim_i \oR\map_{ dg_+\Affd\Alg_K^{f\bG_m,\dagger,\loc, \flat}}(A(i),B(j)).
 \]
\end{definition}

Note that Corollary \ref{waffdsubcor}  implies that the obvious left adjoint functors from  $dg_+\Affd\Alg_K^{f\bG_m,\dagger,\loc}$  and $dg_+\Affd\Alg_K^{\bG_m,\dagger,\loc}$ to $\pro(dg_+\Affd\Alg_K^{f\bG_m,\dagger,\loc, \flat})$ induce 
fully faithful left-derived simplicial functors on the simplicial localisations. This allows us to regard weighted formal localised dagger dg algebras as giving  a natural enlargement of the categories we have been studying so far.

\begin{definition}\label{whfet}
 Say that a morphism $A \to B$ in $\pro(dg_+\Affd\Alg_K^{f\bG_m,\dagger,\loc, \flat})$ is \emph{homotopy formally \'etale} (resp. \emph{homotopy formally smooth}) if it induces weak equivalences (resp. $\pi_0$-surjections)
 \[
  \oR\map(B,C)\to  \oR\map(A,C)\by^h_{ \oR\map(A,D)} \oR\map(B,D)
 \]
for all nilpotent surjections $C \to D$ in $dg_+\Affd\Alg_K^{f\bG_m,\dagger,\loc, \flat}$, where $\oR\map$ denotes $ \oR\map_{ \pro(dg_+\Affd\Alg_K^{f\bG_m,\dagger,\loc, \flat})}$.
\end{definition}

\begin{remark}\label{weightedcot}
There are many equivalent characterisations of Definition \ref{whfet}. 
 Filtering by powers of the kernel, it suffices to know that the condition holds for surjections $C \to D$ with square-zero kernel. 
 
 The constructions of \cite{Q} are sufficiently general to give a theory of cotangent complexes $\bL^A$ for objects $A$ of $\pro(dg_+\Affd\Alg_K^{f\bG_m,\dagger,\loc, \flat})$, existing as Beck modules, meaning that $A \oplus \bL^A$ is a group object in the slice category $ \pro(dg_+\Affd\Alg_K^{f\bG_m,\dagger,\loc, \flat})$; in particular, we can regard $\bL^A$ is an $A$-module in pro-(bounded chain complexes over $K$). Standard arguments then show that for $A \to B$ to be homotopy formally \'etale amounts to saying that $\bL^{B/A}$ is pro-quasi-isomorphic to $0$, while being homotopy formally smooth amounts to saying that $\bL^{B/A}$ satisfies the left lifting property with respect to surjections.
 
 There is a form of base change for Beck modules, and we can rewrite the last characterisation of homotopy formally \'etale maps as saying that $\bL^B \simeq \bL^A\ten_AB$. We can reduce this further to say that  $\bL^B\ten_BC \simeq  \bL^A\ten_AC$ for all $ C \in B \da dg_+\Affd\Alg_K^{f\bG_m,\dagger,\loc, \flat}$, or even just for $C \in \H_0B \da \Affd\Alg_K^{f\bG_m,\dagger,\loc, \flat}$, via a Postnikov induction argument.  
 These characterisations tie in with Definition \ref{wFcotdef} because for $F:=  \oR\map_{ \pro(dg_+\Affd\Alg_K^{f\bG_m,\dagger,\loc, \flat})}(A,-)$ and $x \in F(C)$, we have $\bL^{F,x} \simeq\bL^A\ten_AC $.
\end{remark}




\subsubsection{Weighted formal dagger spaces}

 \begin{definition}\label{wdaggerdgspacedef}
  Define a weighted formal $K$-dagger dg space $X$ to be a pair $(\pi^0X^{\bG_m},\sO_X)$ where $\pi^0X^{\bG_m}$ is a $K$-dagger space  and $\sO_X$ is a presheaf of weighted formal localised dagger dg algebras 
  on the site of open affinoid subdomains of $\pi^0X^{\bG_m}$, 
    such that the homology presheaf $\H_0\sO_X/\bG_m$ is just $\sO_{\pi^0X^{\bG_m}}$,
    and such that for an inclusion $U \subset V$ of open affinoid subdomains of $\pi^0X^{\bG_m}$, the map
\[
 \sO_X(V) \to \sO_X(U) 
\]
is homotopy formally \'etale.
     \end{definition}

 \begin{remark}\label{wNstackrmk}
 There are entirely similar definitions for weighted formal $K$-dagger dg Deligne--Mumford and Artin stacks, using the \'etale and smooth sites, and for  $N$-stacks, proceeding as in Remark \ref{Nstackrmk}.
   \end{remark}

 \begin{definition}
  A morphism $f \co X \to Y$ of   weighted formal $K$-dagger dg spaces is said to be a quasi-isomorphism if it induces an isomorphism $\pi^0f^{\bG_m} \co \pi^0X^{\bG_m} \to \pi^0Y^{\bG_m}$ of formal weighted dagger spaces and  a pro-quasi-isomorphism $f^{-1}\sO_Y \to \sO_X$.
 \end{definition}

\subsection{Representability}\label{wrepsn}

Although not strictly necessary for the study of shifted symplectic structures on analytic moduli functors, we now include a representability result which applies to most examples of interest.

\begin{proposition}\label{wlurierep2}
A homotopy-preserving functor  $F\co dg_+\Affd\Alg_K^{f\bG_m,\dagger,\loc, \flat} \to s\Set$ is a weighted formal  dagger-analytic Artin derived $n$-stack 
with a coherent cotangent complex
if and only if  
 the following conditions hold
\begin{enumerate}
 
\item The restriction $\pi^0F^{\bG_m} \co \Affd\Alg_K^{\dagger} \to s\Set$ to underived dagger  algebras is represented by a dagger-analytic Artin $n$-stack.

\item
$F$ is homogeneous.


%

\item \label{wshf2} 
for all weighted dagger  algebras  $A\in \Affd\Alg_K^{f\bG_m,\dagger}$, 
all $x \in F(A)_0$ and all \'etale morphisms $f:A \to A'$, the maps
\[
\DD_x^*(F, A)\ten_AA' \to \DD_{fx}^*(F, A')
\]
are isomorphisms.

\item for all dagger  algebras $A$ (regarded as living in weight $0$ and degree $0$) and all $x \in F(A)$, 
the groups $\DD^i_x(F, A\{m\})$ are all  finitely generated $A$-modules. 

\end{enumerate}
\end{proposition}
\begin{proof}
By Lemma \ref{wcotexistslemma}, the conditions ensure that 
 $F$ has coherent cotangent complexes $\bL^{F,y}$ at  all points $y \in F(B)$ for all $B \in dg_+\Affd\Alg_K^{f\bG_m,\dagger,\loc, \flat}$. 
 Reasoning as in the proof of Corollary \ref{lurierep2}, it follows that $F$ is an \'etale hypersheaf. It then only remains to establish the existence of smooth atlases. 
 
 Corollary \ref{lurierep2} implies that the functor $F^{\bG_m}$ given by the restriction of  $F$ to $dg_+\Affd\Alg_K^{\dagger,\loc}$ (regarded as objects concentrated in weight $0$) is representable. Given a smooth atlas $u_0 \co U_{(0)} \to F^{\bG_m}$, we now inductively construct  a sequence
 \[
U_{(0)} \to U_{(1)} \to U_{(2)} \to  \ldots \to \to  F
 \]
 of morphisms of presheaves $U_{(n)} =\oR\Spec A_{(n)}$ on $dg_+\Affd\Alg_K^{f\bG_m,\dagger,\loc, \flat}$, for quasi-free filtered pro-algebras $(A_{(n)},\Fil)$ in $\pro(dg_+\Affd\Alg_K^{f\bG_m,\dagger,\loc})$   such that:
 \begin{itemize}
 \item each algebra $A_{(n)}$ is equipped with a multiplicative decreasing filtration $\Fil$ with $\Fil^0A_{(n)}=A_{(n)}$ and $\cW_{\ne 0}A_{(n)} \subset \Fil^1A_{(n)}$,
  \item each morphism $A_{(n+1)} \to A_{(n)}$ respects the filtrations $\Fil$ and gives a filtered quasi-isomorphism $A_{(n+1)}/\Fil^n \to A_{(n)}/\Fil^n$, and

\item for each morphism $u_n \co U_{(n)} \to F$, the complex
\[
(\bL^{U_{(n)}/F}\ten^{\oL}_{A_{(n)}}\H_0A_{(0)})/\Fil^{n+1}
\]
is quasi-isomorphic to a complex of projective $\H_0A_{(0)}$-modules concentrated in non-positive chain degrees, where the filtration is defined by combining the filtration on $\bL^{U_{(n)}}$ (induced from $\Fil$ on $A_{(n)}$) with the filtration on $u_0^*\bL^F$ given by $\Fil^0(u_0^*\bL^F)= u_0^*\bL^F$, $\Fil^1(u_0^*\bL^F)= \cW_{\ne 0}(u_0^*\bL^F)$ and $\Fil^2(u_0^*\bL^F)= 0$.
\end{itemize}
The required atlas will then be given by  $\LLim_n U_{(n)} \to F$.

On $A_{(0)}$, we define the filtration by $\Fil^1A_{(0)}=0$, so $\bL^{U_{(0)}/F}/\Fil^1 \simeq \cW_0\bL^{U_{(0)}/F} \simeq \bL^{U_{(0)}/F^{\bG_m}}$ satisfies the  condition above.

Our inductive construction relies on the following lemma, which is more commonly applied to the Postnikov tower. Writing $A\oplus_dM:=(A\oplus M, \delta \pm d)$, the square-zero extension of $A$ by $M$ induced by a derivation $d \co A \to M[1]$, we have:
\begin{lemma*}
 Take a $\Q$-CDGA $A$ equipped with a decreasing multiplicative filtration $\Fil$ such that $\Fil^0A=A$, together with  a compatible filtered $A$-module $M$ in complexes such that $M=\Fil^jM$ for $j>0$. Then the cotruncated relative cotangent complex $\bL(A/A\oplus_dM)/\Fil^{2j}$ is quasi-isomorphic to $M[1]/\Fil^{2j}$, so $(\bL(A/A\oplus_dM)\ten_A^{\oL}\gr_{\Fil}^0A)/\Fil^{j+1} \simeq \gr_{\Fil}^jM[1]$. 
\end{lemma*}
\begin{proof}[Proof of lemma]
In the projective model category $\C_{2j}$ of filtered $\Q$-CDGAs $(B, \Fil)$ with $\Fil^0B=B$ and $\Fil^{2j}B=0$, the complex $\bL(A/A\oplus_dM)/\Fil^{2j}$ is a model for the cotangent complex of $A/\Fil^{2j}$ over $(A\oplus_dM)/\Fil^{2j}$. We may assume that $A$ is cofibrant as a filtered $\Q$-CDGA and that $M$ is a filtered cofibrant $A$-module, and then observe that  $\tilde{A}:=\cone(M \to A\oplus_dM)/\Fil^{2j}$ is a cofibrant replacement for $A/\Fil^{2j}$ over $(A\oplus_dM)/\Fil^{2j}$ in $\C_{2j}$. Then $(\Omega^1_{\tilde{A}/(A\oplus_dM)}\ten_{\tilde{A}}A)/\Fil^{2j} \cong  M[1]/\Fil^{2j}$, as required.
\end{proof}
 
Now, assume that we have constructed $(A_{(n-1)},\Fil)$ and  $u_{n-1} \co U_{(n-1)} \to F$, with $U_{(n-1)}= \oR\Spec A_{(n-1)}$.   The obstruction to taking $U_{(n)}$ to be $U_{(n-1)}$ is determined by the  complex  $\gr_{\Fil}^n(\bL^{U_{(n-1)}/F}\ten^{\oL}_{A_{(n-1)}}\H_0A_{(0)})$.
 Choose a complex $L$ of $A_{(0)}$-modules, freely  generated by generators in non-positive chain degrees, 
 and an $A_{(0)}$-linear morphism $L\to  \gr_{\Fil}^n(\bL^{U_{(n-1)}/F}\ten^{\oL}_{A_{(n-1)}}A_{(0)})$ with cone quasi-isomorphic to $M[1]$ for some connective complex $M$. We thus have an $A_{(0)}$-linear exact triangle
 \[
  M[1] \to \cone\left(L \to  (\bL^{U_{(n-1)}/F}\ten^{\oL}_{A_{(n-1)}}A_{(0)})/\Fil^{n+1}\right) \to  (\bL^{U_{(n-1)}/F}\ten^{\oL}_{A_{(n-1)}}A_{(0)})/\Fil^n \to,
 \]
which splits since the inductive hypothesis on $U_{(n-1)}$ combined with connectivity of $M$ ensures the relevant $\Ext$ group vanishes.

A choice of splitting induces a map $\bL^{U_{(n-1)}/F} \to M[1]$; denote the associated derivation on $A_{(n-1)}$ by $d$ and   set $A_{(n)}$ to be a quasi-free replacement of $A_{(n-1)}\oplus_dM$, with $M$ given the filtration $\Fil^nM=M$, $\Fil^{n+1}M=0$. Thus $F(A_{(n)})\simeq F(A_{(n-1)}\by^h_{((\id,d), F(A_{(0)} \oplus M[1], (\id,0))}F(A_{(0)})$, so we can lift our given element  $u_{n-1} \in F(A_{(n-1)})$  to an element $ u_n \in  F(A_{(n)})$ because the derivation $d$ factors through $\bL^{U_{(n-1)}/F}$ by construction.

Applying the lemma gives an  exact triangle
\[
(\bL^{U_{(n)}/F}\ten^{\oL}_{A_{(n)}}A_{(0)})/\Fil^{n+1} \to  (\bL^{U_{(n-1)}/F}\ten^{\oL}_{A_{(n-1)}}A_{(0)})/\Fil^{n+1} \to M[1]\to,
\]
or equivalently  an exact triangle
\[
(\bL^{U_{(n)}/F}\ten^{\oL}_{A_{(n)}}A_{(0)})/\Fil^{n+1} \to (\bL^{U_{(n-1)}/F}\ten^{\oL}_{A_{(n-1)}}A_{(0)})/\Fil^n \to L[1]\to,
\]
where again the final map is induced by the choice of splitting. Since $(\bL^{U_{(n-1)}/F}\ten^{\oL}_{A_{(n-1)}}A_{(0)})/\Fil^n$ and $L$
 are both quasi-isomorphic to complexes freely generated as $A_{(0)}$-modules by generators in non-positive degrees, it follows that $(\bL^{U_{(n)}/F}\ten^{\oL}_{A_{(n)}}A_{(0)})/\Fil^{n+1}$ must also be so, completing the inductive step.
\end{proof}

\section{Weighted shifted  symplectic structures associated to pro-\'etale sheaves}\label{wproetsn}

We now return to the questions outlined at the start of \S \ref{weightedsn}, 
applying the techniques of that section 
to our examples of interest.

As in \S \ref{proetdaggersn},  assume that the valuation on our base field $K$ is discretely valued, so the ring $\cO_K$ is a DVR with maximal ideal $\m_K$. 
Let $\ell$ be the unique integral prime in $\m_K$.

%
%

\subsection{Weighted pre-symplectic structures}

\begin{definition}

Given a scheme $X$, a locally free rank $1$ $\ell$-adic lisse sheaf $\bE$ on $X$,  and a graded  topological  $K$-vector space $V= \bigoplus_{r \in \Z} \cW_rV$, we let $\uline{V}_X(\bE)$ be the  pro-\'etale sheaf $\bigoplus_{n \in \Z} \uline{\cW_nV}_X\hten_{\Zl}\bE^{\ten n}$ on $X$.
\end{definition}

This allows  the following modification of Definition \ref{FXproetdef}:
 \begin{definition}\label{wFXproetdef}
  For $(X,\bE)$ as above,
and a  functor $F \co dg_+\CAlg_K \to s\Set$ from differential graded-commutative $K$-algebras in non-negative chain degrees to simplicial sets, define the functor
  \[
   F(X_{\pro\et},\bE,-) \co dg_+\Affd\Alg^{f\bG_m\loc,\dagger}_K \to s\Set
  \]
from weighted  localised dagger dg algebras to simplicial sets by
\[
A \mapsto \oR\Gamma(X_{\pro\et},F(\uline{A}_X(\bE)),
\]  
where $\oR\Gamma$ is the right-derived functor of the global sections functor $\Gamma$ in simplicial sets. 
Note that the hypotheses on $A$ imply that the direct sum  in the definition of $\uline{A}_X(\bE)$   is finite.
 \end{definition}

 \begin{example}[$G$-torsors]\label{wBGXex}
  If $G$ is an algebraic group over $K$, then we can let $F$ be the derived stack $BG$, parametrising $G$-torsors. The functor $ BG(X_{\pro\et},\bE,-)$ can then be thought of as parametrising $G$-torsors on  $\bE\setminus \{0\}$ over $X_{\pro\et}$.
  In particular, when $A$ is a weighted dagger algebra (concentrated in chain degree $0$), the simplicial set  $ BG(X_{\pro\et},\bE, A)$ is the nerve of the groupoid of $G(\uline{A}_X(\bE))$-torsors on $X_{\pro\et}$. 
  
If $X$ is locally topologically Noetherian and connected, with a geometric point $x$, then $\bE$ corresponds to a continuous $\Zl$-representation $E$ of  the pro-\'etale fundamental group  $ \pi_1^{\pro\et}(X,x)$  of \cite[\S 7]{BhattScholzeProEtale}.  For $A \in \Affd\Alg_K^{\dagger,f\bG_m}$ (i.e. $A=A_0$), the simplicial set  $BG(X_{\pro\et},\bE, A)$ is then equivalent to the nerve of the groupoid of continuous sections of the group homomorphism
\[
    \pi_1^{\pro\et}(X,x) \ltimes G(\bigoplus_{r\in \Z} \cW_rA\hten_{\Zl}E^{\ten r})  \to     \pi_1^{\pro\et}(X,x)                                                                                                                                                                                                                                                                                                                                                                                                                                                                                                                                                                                                                                                                                                                                                                                                                                              \]
 given by projection to the first factor.
 \end{example}

 \begin{example}[Commutative groups]\label{wBGmXex}

 As a special case of Example \ref{wBGXex}, if $G$ is commutative with associated (abelian) Lie algebra $\g$, then for any $A \in \Affd\Alg_K^{\dagger,f\bG_m}$ we have a group isomorphism
 \begin{align*}
 G( \bigoplus_r \cW_rA\hten_{\Zl}E^{\ten r}) &\cong G(\cW_0A)  \by ( \g\ten_K(\bigoplus_{r\ne 0} \cW_rA\hten_{\Zl}E^{\ten r}))
 \end{align*}
 given  by $g\exp(v) \mapsfrom (g,v)$, which is well-defined because the ideal $I$ generated by $\cW_{\ne 0}A$ is nilpotent. To see that
 this is an isomorphism, observe that $\exp \co I\ten_K \g \to G(A)$ is injective and that $I= \cW_0I \oplus  \cW_{\ne 0}A$,
 with
 isomorphisms $G(A)/\exp(\g\ten_KI) \cong G(A/I)= G(\cW_0A/\cW_0I) \cong G(\cW_0A)/\exp(\g\ten_K\cW_0I)$.
 Since $A$ has finitely many non-zero weights, also note that the sum is finite, hence a product.

 Passing to right-derived functors, for any $A \in dg_+\Affd\Alg_K^{\dagger,f\bG_m}$ we then have a group equivalence
 \begin{align*}
 \oR G( \bigoplus_r \cW_rA\hten_{\Zl}E^{\ten r}) &\simeq G(\cW_0A)  \by  \prod_{r \ge 0} N^{-1} (\g\ten_K(\cW_rA\hten_{\Zl}E^{\ten r})),
 \end{align*}
 where $N^{-1}$ is Dold--Kan denormalisation.

 Since $\oR\Gamma$ commutes with products, the weighted moduli functor thus decouples as
 \[
  BG(X_{\pro\et},\bE, A) \simeq BG(X_{\pro\et}, \cW_0A) \by \prod_{r \ge 0} N^{-1}\tau_{\ge 0}( \oR\Gamma(X_{\pro\et},  \g\ten_K \uline{\cW_rA}\hten_{\Zl}\bE^{\ten r})_{[-1]}),
 \]
where 
$\tau_{\ge 0}$ denotes good truncation of a chain complex.

 For instance, if $k$ is a finite field prime to $\ell$ then the weighted derived moduli stack of $G$-torsors (for $G$ commutative) is concentrated in weight $0$, in the sense that the natural map $ BG((\Spec k)_{\pro\et}, \cW_0A) \to  BG((\Spec k)_{\pro\et},\Zl(1), A) $ is an equivalence, since tangent cohomology $\H^*(k, \cW_nA(n))$  vanishes in all  non-zero weights $n$ (being cohomology of the Galois group $\hat{\Z}$ acting as $|k|^n$). 

 The next simplest example is for a non-Archimedean local field $F$ of residue characteristic 
 prime to $\ell$, where for $G$ commutative  we have  
 \[
 BG((\Spec F)_{\pro\et},\Zl(1), A) \simeq  BG((\Spec F )_{\pro\et}, \cW_0A) \by  
 N^{-1}\tau_{\ge 0}(\g \ten_K (\cW_1A \oplus \cW_1A_{[1]} )).   
 \]
 This follows because cohomology vanishes in all other weights by local Tate duality \cite[\S II.5.5, Proposition 19]{galoisienne}, which gives an exact triangle
 \[
  \oR\Gamma(k, \Ql(r)) \to \oR\Gamma(K_v, \Ql(r)) \to  \oR\Gamma(k, \Ql(1-r))^*[-2]\to
 \]
for the residue field $k$, with the term on the left (resp. right) being acyclic for $r \ne 0$ (resp. $r\ne 1$).

  \end{example}

Corollary \ref{tracecor} now generalises to give us the following corollary of Proposition \ref{traceprop}, applicable to all of the cases in Examples \ref{traceex} (taking $\bE=\uline{\Zl(1)}$):
 
\begin{corollary}\label{wtracecor}
If $X$ is  a topologically Noetherian scheme and $\bE$ a locally free rank $1$ $\ell$-adic lisse sheaf on $X$, equipped with compatible trace maps
\[
\tr \co \H^d (X_{\et}, \bE^{\ten m}/\ell^n)\to \Z/\ell^n.
\]
satisfying the conditions of Proposition \ref{traceprop},
then for any 
 $n$-shifted pre-symplectic (in the terminology of \cite{poisson}) derived $\infty$-geometric Artin stack $F \co dg_+\CAlg_K \to s\Set$,
the functor
  \[
   F(X_{\pro\et},\bE,-) \co dg_+\Affd\Alg^{\loc,\dagger}_K \to s\Set
  \]
of Definition \ref{FXproetdef} carries a functorial $(n-d)$-shifted pre-symplectic structure of weight $m$ at all points; in particular,  this implies that any formally \'etale 
map  $Y \to F(X_{\pro\et},-)$ from a dg dagger-analytic Artin $\infty$-stack $Y$ induces an $(n-d)$-shifted pre-symplectic structure of weight $m$ on $Y$.
\end{corollary}
 \begin{proof}
 For any   $A \in dg_+\Affd\Alg^{f\bG_m\loc,\dagger}_K $, Lemma \ref{anDRlemma} combines with Proposition \ref{traceprop} to give us maps 
 \[
\oR\Gamma(X_{\pro\et},  F^2\oL\DR(\uline{A}_X(\bE))) 
\to \oR\Gamma(X_{\pro\et}, \uline{F^2\oL\DR(A) }_X(\bE))
\xra{\tr} \cW_mF^2\oL\DR(A)[-d],
\] 
via the projection $ \uline{F^2\oL\DR(A) }_X(\bE) \to \uline{  \cW_m F^2\oL\DR(A) }_X\hten_{\Zl}\bE^{\ten m}$.
Hence we have a map
\[
 \oR\Gamma(X_{\pro\et},\PreSp^{\alg}(\uline{A}_X(\bE)^{\alg} ,n) \to \cW_m\PreSp(A,n-d).
\]

Similarly to Corollary \ref{tracecor}, because $ \PreSp^{\alg}(F,n) \simeq \oR\map(F,\PreSp^{\alg}(-,n))$, the maps above combine to give a composite  transformation
\[
 \PreSp^{\alg}(F,n)   \by  F(X_{\pro\et},\bE,A) \to 
 \oR\Gamma(X_{\pro\et},\PreSp^{\alg}( \uline{A}_X(\bE)^{\alg} ,n)
 \to  \cW_m\PreSp(-,n-d),
\]
natural in $A$.
By adjunction, we can rephrase this as a morphism
\[
 \PreSp^{\alg}(F,n) \to \oR\map(F(X_{\pro\et},\bE,-), \cW_m\PreSp(-,n-d)).
\]
as required.
\end{proof}

\subsection{Weighted symplectic structures}

In order to establish non-degeneracy, we need a weighted version of the notion of weak duality from Definition \ref{weakdualdef}. This involves looking at a weighted  quasi-dagger dg algebra $A$ and a presheaf $N$ of  $\uline{A}_X(\bE)$-modules in chain complexes  on $X_{\pro\et}$. The complex $\oR\Gamma(X,N)$ is not then an $A$-module, but the complex $\bigoplus_{r \in \Z}  \oR\Gamma(X,N\hten_{\Zl}\bE^{\ten - r} )$ is, via the map  $\uline{A}_X \to \bigoplus_r \uline{A}_X(\bE)\hten_{\Zl}\bE^{\ten - r}$ given by 
\[
\uline{\cW_rA}_X \to (\uline{\cW_rA}_X(\bE)\hten_{\Zl}\bE^{\ten r}) \hten_{\Zl}\bE^{\ten -r} \subset \uline{A}_X(\bE)\hten_{\Zl}\bE^{\ten -r}.
\]
Accordingly, we regard $\bigoplus_{r \in \Z} N\hten_{\Zl}\bE^{\ten - r}$ as a $\bG_m$-equivariant module, with the term $ N\hten_{\Zl}\bE^{\ten - r}$ having weight $r$.

\begin{definition}\label{wweakdualdef}
 Given a weighted  quasi-dagger dg algebra $A$ and a presheaf $N$ of  $\uline{A}_X(\bE)$-modules in chain complexes  on $X_{\pro\et}$, we say that $N$ satisfies weak duality with respect to the trace $\tr$  if for all morphisms $A \to C$ of weighted quasi-dagger dg algebras, 
 the map 
 \[
  \oR\HHom_{\uline{A}_X(\bE)}( N,\uline{C}_X(\bE) )\to \oR\HHom_A(\bigoplus_r \oR\Gamma(X,N\hten_{\Zl}\bD\hten_{\Zl}\bE^{\ten - r} )[d],\bigoplus_r \cW_rC)^{\bG_m}
 \]
induced by the pairing
\begin{align*}
(\bigoplus_r \oR\Gamma(X, N\hten_{\Zl}\bD\hten_{\Zl}\bE^{\ten -r}))\ten_A^{\oL} \oR\HHom_{\uline{A}_X(\bE)}( N, \uline{C}_X(\bE)) \\
\to \oR\Gamma(X,\bigoplus_{r} \uline{C}_X(\bE)\hten_{\Zl}\bD\hten_{\Zl}\bE^{\ten -r})\to C[-d]
\end{align*}
is a quasi-isomorphism.  
Here, our pairing factors through the obvious  projection maps  $\bigoplus_{r} \uline{C}_X(\bE)\hten_{\Zl}\bE^{\ten -r} \to \bigoplus_{r} \uline{\cW_rC}_X$.
\end{definition}
Note that by taking $C=A \oplus M$, we can deduce a similar quasi-isomorphism for all $M\in dg_+\Coh_A^{f\bG_m}$ in place of $C$.

\begin{remark}\label{wweakrmk}
 When $A$ lives in weight  $0$,  we have $\uline{A}_X(\bE)= \uline{A}_X$, but beware that Definition \ref{wweakdualdef} is then still a more general statement than Definition \ref{weakdualdef}, because the algebras $C$ can still have non-zero weights. Explicitly, the condition reduces to
 \[
   \bigoplus_r \oR\HHom_{\uline{A}_X}( N,\uline{\cW_rC}_X\hten_{\Zl}\bE^{\ten r} )\simeq \bigoplus_r\oR\HHom_A(\oR\Gamma(X,N\hten_{\Zl}\bD\hten_{\Zl}\bE^{\ten - r} )[d], \cW_rC)
 \]
for all $C$, so it amounts to saying that the $\uline{A}_X$-modules $N\hten_{\Zl}\bE^{\ten r}$ all satisfy weak duality in the sense of  Definition \ref{weakdualdef}.
 
\end{remark}

There are then  immediate weighted analogues of Lemmas  \ref{wdopenlemma} and \ref{wdnilcompletelemma}; the latter becomes:
 \begin{lemma}\label{wwdnilcompletelemma}
Given a weighted quasi-dagger dg algebra $A$, a  module $N \in dg_+\Mod_{\uline{A}_X(\bE)}$ 
satisfies weak duality  in the sense of Definition \ref{wweakdualdef} if and only if the presheaves  $ \uline{(\H_0A/\bG_m)}_X\ten^{\oL}_{\uline{A}_X(\bE)}N\hten_{\Zl}\bE^{\ten r}$ of $\uline{\H_0A/\bG_m}_X$-modules all satisfy weak duality in the sense of Definition \ref{weakdualdef}.
 \end{lemma}

\begin{examples}\label{wwdex}
  Roughly speaking, a sufficient condition for an $\uline{A}_X(\bE)$-module $N$ on an $\ell$-coprime proper scheme $X$ to satisfy weak duality is that  its sheafification  is constructible in an appropriate sense;  by the analogue of Lemma \ref{wdnilcompletelemma}, we can reduce to looking at the $\uline{(\H_0A/\bG_m)}_X$-module $N\ten^{\oL}_{\uline{A}_X(\bE)} \uline{(\H_0A/\bG_m)}_X$ (noting that $\uline{(\H_0A/\bG_m)}_X=  \uline{(\H_0A/\bG_m)}_X(\bE)$ since $\H_0A/\bG_m$ lives in weight $0$). Also note that the category of modules satisfying weak duality is triangulated and idempotent-complete.
 
In the setting of Lemma \ref{constrlemma}, with a constructible $\hat{R}_X$-complex $L$ and a homomorphism $R \to \cW_0A$, consider the $\uline{A}_X(\bE)$-complex $N:=L\ten^{\oL}_{\hat{R}_X}\uline{A\{m\}}_X(\bE)$ (for $A\{m\}$ as in Definition \ref{twistweightdef}). This (and hence any object of the triangulated category generated by such objects) satisfies weak duality, with the following reasoning. 
Since
\[
 \uline{A\{m\}}_X(\bE)= \bigoplus_r \uline{\cW_{m+r}A}\hten_{\Zl}\bE^{\ten r}= \uline{A}_X(\bE)\hten_{\Zl}\bE^{\ten -m}, 
\]
Lemma \ref{constrlemma} implies that 
\begin{align*}
 \oR\Gamma(X_{\pro\et}, N 
 \hten_{\Zl}\bE^{\ten -r} ) 
 &\simeq \bigoplus_s \oR\Gamma(X_{\pro\et},L\hten_{\Zl}\bE^{\ten s-r-m}  )\ten^{\oL}_R \cW_{s}A,
\end{align*}
so 
\begin{align*}
 &\oR\HHom_A(\bigoplus_r  \oR\Gamma(X,N\hten_{\Zl}\bD\hten_{\Zl}\bE^{\ten -r} ), \bigoplus_r \cW_rC)^{\bG_m}\\
 &\simeq   \oR\HHom_A( \bigoplus_r\oR\Gamma(X_{\pro\et},L\hten_{\Zl}\bD\hten_{\Zl}\bE^{\ten -r-m}  )\ten^{\oL}_R A, \bigoplus_r \cW_rC)^{\bG_m}\\
 &\simeq\bigoplus_r  \oR\HHom_R( \oR\Gamma(X_{\pro\et},L\hten_{\Zl}\bD\hten_{\Zl}\bE^{\ten -r-m}  ),  \cW_rC), 
\end{align*}
noting that the sum is finite by the hypothesis on $C$.

Similarly, Lemma \ref{constrlemma} gives
\begin{align*}
 \oR\HHom_{\uline{A}_X(\bE)}(N ,
 \uline{C}_X(\bE))
& \simeq  \oR\HHom_{\uline{A}_X(\bE)}( L\ten^{\oL}_{\hat{R}_X}\uline{A}_X(\bE), \uline{C}_X(\bE) \hten_{\Zl}\bE^{\ten m})\\
& \simeq \bigoplus_r \oR\HHom_{\hat{R}_X}( L, \hat{R}_X\hten_{\Zl}\bE^{\ten m+r})\ten^{\oL}_R \cW_rC,
 \end{align*}
The natural map 
\[
  \oR\HHom_{\uline{A}_X(\bE)}( N,\uline{C}_X(\bE) )\to \oR\HHom_A(\bigoplus_r \oR\Gamma(X,N\hten_{\Zl}\bD\hten_{\Zl}\bE^{\ten - r} )[d],\bigoplus_r \cW_rC)^{\bG_m}
 \]
is thus a quasi-isomorphism, since it reduces to the Verdier duality quasi-isomorphisms
 \[
  \oR\HHom_{\hat{R}_X}( L, \hat{R}_X\hten_{\Zl}\bE^{\ten m+r})\to \oR\HHom_R( \oR\Gamma(X_{\pro\et},L\hten_{\Zl}\bD\hten_{\Zl}\bE^{\ten -r-m}  ),R).
 \]
\end{examples}

\begin{lemma}\label{wFXcotlemma}
Let $X$ be  a topologically Noetherian scheme satisfying the conditions of Proposition \ref{traceprop}, and 
  $F \co dg_+\CAlg_K \to s\Set$ a derived $\infty$-geometric Artin stack.
At any point $\phi \in F(X_{\pro\et},\bE,A)$ for which the presheaf $\bL^{F,\uline{A}_X(\bE), \phi} $  of $\uline{A}_X(\bE)$-modules satisfies weak duality in the sense of Definition \ref{weakdualdef}, the functor
\[
T_{\phi}(F(X_{\pro\et},\bE,-),-) \co M \mapsto F(X_{\pro\et},\bE, A\oplus M)\by^h_{F(X_{\pro\et}, \bE,A)}\{\phi\}
\]
on weighted coherent $A$-modules $M$ is represented by the weighted complex
\[
\bigoplus_r \oR\Gamma(X, \bL^{F,\uline{A}_X(\bE), \phi}\hten_{\Zl}\bD\hten_{\Zl}\bE^{\ten - r} )[d].
\]
\end{lemma}
\begin{proof}
 This follows in much the same way as Lemma \ref{FXcotlemma}. Explicitly,   we have
  \begin{align*}
  &F(X_{\pro\et},\bE,A\oplus M)\by^h_{F(X_{\pro\et},\bE,A)}\{\phi\} \\
  &\simeq N^{-1}\tau_{\ge 0}\oR\HHom_{\uline{A}_X(\bE)}(\bL^{F,\uline{A}_X(\bE),\phi},  \uline{M}_X(\bE)  ).
\end{align*}

By weak duality, this in turn is equivalent to 
\[
 N^{-1}\tau_{\ge 0}\oR\HHom_A(\bigoplus_r\oR\Gamma(X,\bL^{F,\uline{A}_X(\bE), \phi}\hten_{\Zl}\bD\hten_{\Zl}\bE^{\ten - r})[d],M)^{\bG_m}
\]
so is represented by $\bigoplus_r \oR\Gamma(X,\bL^{F,\uline{A}_X, \phi}\hten_{\Zl}\bD\hten_{\Zl}\bE^{\ten - r})[d] $.
 \end{proof}

Corollaries \ref{tracecor2a}, \ref{tracecor2b} now have the following immediate generalisations:
\begin{corollary}\label{wtracecor2a}
Let $X$ be  a topologically Noetherian scheme satisfying the conditions of Proposition \ref{traceprop} with $\bD=\bE^{\ten m}$, 
and  $F \co dg_+\CAlg_K \to s\Set$ an
 $n$-shifted symplectic  derived Artin $\infty$-stack. 
 Then there is a natural $(n-d)$-shifted symplectic structure of weight $m$ (in the sense of Remark \ref{wNstackrmk2}) on  the full subfunctor $ F(X_{\pro\et},\bE,-)^{wd} \subset F(X_{\pro\et},\bE,-) $ consisting of points $\phi$ for which the presheaf $\bL^{F,\uline{A}_X(\bE), \phi} $   satisfies weak duality.
\end{corollary}

\begin{corollary}\label{wtracecor2b}
Under the conditions of Corollary \ref{wtracecor2a}, take an
 (underived)  dagger Artin analytic $\infty$-stack $Y$ equipped with a 
formally \'etale morphism 
\[
 \eta \co Y \to  \pi^0F(X_{\pro\et},-)
\]
of functors $\Affd\Alg^{\loc,\dagger}_K \to s\Set $, such that at all points $\phi$ in the image of $\eta$,  the $\uline{A}_X$-modules $\bL^{F,\uline{A}_X, \phi}\hten_{\Zl}\bE^{\ten r} $   satisfy weak duality for all $r$.
 
  Then  the functor $\tilde{Y} \co A \mapsto Y(\H_0A/\bG_m)\by^h_{F(X_{\pro\et},\H_0A/\bG_m)}F(X_{\pro\et},\bE,A) $ on $dg_+\Affd\Alg^{f\bG_m,\loc,\dagger}_K$
 is a formal weighted dg dagger  Artin analytic $\infty$-stack carrying a natural $(n-d)$-shifted symplectic structure of weight $m$.
 \end{corollary}
 
  \begin{examples}\label{wSpex}
 Examples \ref{traceex} and \ref{wwdex} (applied as in Examples \ref{Lwdex}) now lead to many instances of weighted $(n-d)$-shifted symplectic moduli stacks when substituted into Corollary \ref{wtracecor2b}, taking in most of the cases we had to exclude from  Examples \ref{Spex}. If we take the derived stack $F$ to be $B\GL_n$, or $BG$ for any other affine algebraic group equipped with a $G$-equivariant inner product on its Lie algebra, or to be the moduli stack of perfect complexes, then $n=2$, so the corollary produces $(2-d)$-shifted symplectic structures on moduli of $G$-torsors or of complexes of pro-\'etale sheaves on $X$, provided we impose some constructibility constraints.
 
 These all exist on suitable open substacks (given by the formula of Corollary \ref{wtracecor2b}) of the formal weighted derived moduli stack $F(X_{\pro\et},\Zl(1),-)$, which we can think of as parametrising $F$-valued sheaves on the $\bG_m$-torsor on $X_{\pro\et}$ given by the Tate motive. In each of the following cases, the symplectic locus contains all points $\phi \in F(X_{\pro\et},\Zl(1),A)$ for which $\bL^{F,\uline{A}_X, \phi} $  is constructible.

 In particular: 
 \begin{enumerate}[itemsep=0pt, parsep=5pt]
  \item\label{wlocalex} If $X$ is a smooth proper scheme  of dimension $m$  over a non-Archimedean local field $k$ containing $\ell^{-1}$, 
  we have $(n-2m-2)$-shifted symplectic structures of weight $m+1$ on suitable open substacks $\tilde{Y}$ of $F(X_{\pro\et},\Zl(1),-)$.

 \item If $X$ is a smooth proper scheme  of dimension $m$   over a finite field $k$ prime to $\ell$, we  have $(n-2m-1)$-shifted symplectic structures of weight $m$ on suitable open substacks of $F(X_{\pro\et},\Zl(1),-)$.
 
 \item If $U$ is a smooth scheme  of dimension $m$ over one of the local (resp. finite) fields $k$ above, then we have  $(n-2m-1)$-shifted  (resp. $(n-2m)$-shifted) symplectic structures of weight $m+1$ (resp. weight $m$) 
 on suitable open substacks of the derived moduli stack 
 \[
 A \mapsto \oR\Gamma(Z_{\pro\et}, i^*\oR j_*F(\uline{A}_X(\bE)))
\]  
 of $F$-valued sheaves on the punctured tubular neighbourhood $ Z \overleftarrow{\by}_{\bar{U}}U$ (thought of as the  boundary of $U$), where $i \co Z \to \bar{U}$ is the complement of $U$ in a compactification $j \co U \to \bar{U}$.
 \end{enumerate}
\end{examples}

 \begin{remark}\label{locformrmk}
 In the terminology of \cite{BouazizGrojnowski}, we can regard our notion of $n$-shifted symplectic structure of weight $m$ on a space $Y$  as being an analytic analogue of a $\sP$-shifted  symplectic structure on $[Y/\bG_m]$ relative to the morphism $f \co [Y/\bG_m] \to B\bG_m$, taking $\sP= f^*\sO(m)[n]$. The proof of \cite{BouazizGrojnowski} adapts to analytic settings, so for negative shifts gives  local models  for these $\sP$-shifted symplectic derived spaces as products of  $\sP$-shifted twisted cotangent bundles with quadratic bundles, the latter being trivial unless $n \equiv 2 \mod 4$; their results refine and generalise the $(-1)$-shifted case addressed in \cite{BBBJdarboux}).  
 
 In our setting, an affinoid $\sP$-shifted twisted cotangent bundle is given by taking a formal weighted localised dg dagger algebra $B$, then forming  the formal  weighted dg algebra
 \[
(  \Symm_B( \oR\cHom_B(\oL\Omega^1_B,B)\{-m\}[-n] ), \delta + df \lrcorner-),
 \]
where $f \in \z_{1+n}\cW_mB$ is twisting the differential. A quadratic bundle is a formally  weighted coherent  $B$-module $M$ with a symmetric pairing $M\ten^{\oL}_BM \simeq B\{-m\}[-n]$ (so \'etale locally, $M$ can be taken to be  trivial unless $n \equiv 2 \mod 4$ and $m$ is even); 
the general local model for an  $n$-shifted symplectic structure of weight $m$  is then given by tensoring the twisted cotangent algebra above with $\Symm_BM$.
 \end{remark}

\subsection{Weighted Lagrangian structures}

We now adapt the constructions of \S \ref{Lagsn} to incorporate weights and twists. As in that section, we take a morphism
$f \co \pd U \to U$ of topologically Noetherian schemes (where the notation $\pd U$ is intended to be suggestive, but does not indicate a specific construction at this point)  and a constructible complex $\bD$ on $U$,  equipped with a system of trace maps 
 \[
 \cone\big(\oR\Gamma (U_{\et}, \bD/\ell^n) \to  \oR\Gamma (\pd U_{\et}, f^{-1}\bD/\ell^n)\big)[d-1]\to \Z/\ell^n
 \]
inducing a perfect pairing between  $\cone(\oR\Gamma (U_{\et},  -) \to  \oR\Gamma (\pd U_{\et},f^{-1}-))$ and  $\oR\Gamma (U_{\et}, -)$ for constructible sheaves and their duals.

\begin{definition}\label{wweakdualLagdef}
 Given a quasi-dagger dg algebra $A$ and a presheaf $N$ of  $\uline{A}_U(\bE)$-modules in chain complexes  on $U_{\pro\et}$, we say that $N$ satisfies weak duality with respect to the trace $\tr$ above if the $\uline{A}_{\pd U}(\bE)$-module $\oL f^*N:= f^{-1}N\ten^{\oL}_{f^{-1}\uline{A}_U(\bE)}\uline{A}_{\pd U}(\bE)$ satisfies weak duality in the sense of Definition \ref{wweakdualdef}, and if  for all morphisms $A \to C$ of quasi-dagger dg algebras,  
 the map
 \begin{align*}
&\cone\big( \oR\HHom_{\uline{A}_U(\bE)}( N, \uline{C}_U(\bE)) \to  \oR\HHom_{\uline{A}_{\pd U}(\bE)}( \oL f^* N, \uline{C}_{\pd U}(\bE))\big)\\
&\to
 \oR\HHom_A(\bigoplus_r \oR\Gamma(U,N\hten_{\Zl}\bD\hten_{\Zl}\bE^{\ten - r} )[d-1],\bigoplus_r \cW_rC)^{\bG_m}
 \end{align*}
induced by the composite pairing
\begin{align*}
 &\oR\Gamma(U,N\hten_{\Zl}\bD\hten_{\Zl}\bE^{\ten -r})\ten_A^{\oL} 
 \cone\big( \oR\HHom_{\uline{A}_U(\bE)}( N, \uline{C}_U(\bE)) \to \oR\HHom_{\uline{A}_{\pd U}(\bE)}( \oL f^* N, \uline{C}_{\pd U}(\bE))\big)\\ 
 &\to \bigoplus_{r}\cone\big(\oR\Gamma(U, \uline{C}_U\hten_{\Zl}\bD\hten_{\Zl}\bE^{\ten -r})\to \oR\Gamma({\pd U}, \uline{C}_{\pd U}\hten_{\Zl}f^{-1}\bD\hten_{\Zl}\bE^{\ten -r})\big)\\
 &\xra{\tr} C[1-d]
\end{align*}
is a quasi-isomorphism.
\end{definition}

Corollaries \ref{tracecor2Laga} and \ref{tracecor2Lagb} then adapt along the lines of Corollaries \ref{wtracecor2a} and \ref{wtracecor2b} to give:
 
 \begin{corollary}\label{wtracecor2Laga}
Take     a morphism $\pd U\to U$ of topologically Noetherian schemes which has a trace pairing as above with $\bD=\bE^{\ten m}$, 
and  take $F \co dg_+\CAlg_K \to s\Set$ an
 $n$-shifted symplectic  derived Artin $\infty$-stack. Then for the full subfunctors   $F({\pd U}_{\pro\et},\bE,-)^{wd}  \subset F({\pd U}_{\pro\et},\bE,-)$ and $F( U_{\pro\et},\bE,-)^{wd}\subset F( U_{\pro\et},\bE,-)$ consisting of points $\phi$ for which   the presheaf $\bL^{F,\uline{A}_{\pd U}(\bE), \phi} $ (resp. $\bL^{F,\uline{A}_{U}(\bE), \phi} $)  satisfies weak duality in the sense of Definition \ref{wweakdualdef} (resp. Definition \ref{wweakdualLagdef}), the natural map
\[ 
 F( U_{\pro\et},\bE,-)^{wd} \to F({\pd U}_{\pro\et},\bE,-)^{wd}
 \]
 carries a natural Lagrangian structure with respect to the $(n-d+1)$-shifted symplectic structure of weight $m$ on $F({\pd U}_{\pro\et},-)^{wd}$ given by Corollary \ref{wtracecor2a}.
 \end{corollary}

 \begin{corollary}\label{wtracecor2Lagb}
In the setting of Corollary \ref{wtracecor2Laga},  take a morphism $W \to Y$ of (underived)  dagger Artin analytic $\infty$-stacks, equipped with a 
commutative diagram
\[
 \begin{CD}
 W @>{\eta'}>>  \pi^0F( U_{\pro\et},-)\\
 @VVV @VV{f^*}V\\
   Y @>{\eta}>>  \pi^0F({\pd U}_{\pro\et},-)
 \end{CD}
\]
of functors $\Affd\Alg^{\loc,\dagger}_K \to s\Set $.
Assume moreover that  the horizontal maps $\eta$ are formally \'etale and  that at all points $\phi$ in the image of $\eta'$ (resp. $\eta$),  the presheaves $\bL^{F,\uline{A}_{\pd U}, \phi} \hten_{\Zl}\bE^{\ten r} $  (resp. $\bL^{F,\uline{A}_{U}, \phi}  \hten_{\Zl}\bE^{\ten r}$)  satisfy weak duality in the sense of Definition \ref{weakdualdef} (resp. Definition \ref{weakdualLagdef}) for all $r$.
 
  Then  the functor $\tilde{W} \co A \mapsto W(\H_0A/\bG_m)\by^h_{F({\pd U}_{\pro\et},\H_0A/\bG_m)}F({\pd U}_{\pro\et},A) $ on $dg_+\Affd\Alg^{\loc,\dagger}_K$
 is a formal weighted dg dagger  Artin analytic $\infty$-stack. It  carries a natural Lagrangian structure with respect to the $(n-d+1)$-shifted symplectic structure of weight $m$ on $\tilde{Y}$ given by Corollary \ref{wtracecor2b}.
 \end{corollary}
 
 \begin{examples}\label{wSpexLag}
 Examples \ref{traceLagex} and \ref{wdexLag} now lead to many instances of Lagrangians in weighted shifted symplectic moduli stacks when substituted into Corollary \ref{tracecor2Lagb}. If we take the $n$-shifted symplectic derived stack $F$ to be $B\GL_n$, or $BG$ for any other affine algebraic group equipped with a $G$-equivariant inner product on its Lie algebra, or to be the moduli stack of perfect complexes, then $n=2$, so the corollary produces $(3-d)$-shifted 
 Lagrangian structures on moduli of $G$-torsors or of complexes of pro-\'etale sheaves on $U$, provided we impose some constructibility constraints.  In each of the following cases, the Lagrangian locus contains all points $\phi \in F(U_{\pro\et},\Zl(1),A)$ for which $\bL^{F,\uline{A}_U, \phi} $  is constructible. 

\begin{enumerate}[itemsep=5pt, parsep=0pt]
 \item\label{localex} If $K$ is a non-Archimedean local field with finite residue field prime to $\ell$, and $\cO_K$ its ring of integers, then we can take $\pd U \to U$ to be the morphism $\Spec K \to \Spec \cO_K$, giving us   $(n-2)$-shifted Lagrangian structures of weight $1$ on suitable open substacks of the derived moduli stack  $F ((\Spec \cO_K)_{\pro\et},\Zl(1),-)$ over $ F ((\Spec K)_{\pro\et},\Zl(1),-)$, coming from local Tate duality.  If we overlook weighted aspects, this morphism of derived stacks essentially maps from moduli of unramified local Galois representations to moduli of all local Galois representations.

 \item\label{localZex} More generally, if $Z$ is a smooth proper scheme of dimension $m$ over $\cO_K$, we have $(n-2-2m)$-shifted Lagrangian structures of weight $m+1$ on suitable open substacks of the derived moduli stack  $F(Z_{\pro\et},\Zl(1),-)$ over $ F ((Z\ten_{\cO_K}K)_{\pro\et},\Zl(1),-)$.

\item\label{globalex} If $K$ is a number field and $S$ a finite set of primes including all those dividing $\ell$, with $\cO_{K,S}$ the localisation of $\cO_K$ at $S$, then we have $(n-2)$-shifted Lagrangian structures of weight $1$ on suitable open substacks of the derived moduli stack  $F( (\Spec \cO_{K,S})_{\pro\et},\Zl(1),-)$ over $ \prod_{v \in S}F ((\Spec  K_v  )_{\pro\et},\Zl(1),-)$, coming from Poitou--Tate duality.

 \item\label{globalZex} More generally, if $Z$ is a smooth proper scheme of dimension $m$ over $\cO_{K,S}$, then we have $(n-2-2m)$-shifted Lagrangian structures of weight $m+1$ on suitable open substacks of the derived moduli stack  $F( Z_{\pro\et},\Zl(1),-)$ over $ \prod_{v \in S}F ((Z\ten_{\cO_{K,S}}K_v  )_{\pro\et},\Zl(1),-)$.

 \item If $U$ is a smooth proper scheme of dimension $m$ over a local (resp. finite) field  prime to $\ell$, we have $(n-1-2m)$-shifted (resp. $n-2m$-shifted)  Lagrangian structures of weight $m+1$ (resp. $m$)  on suitable open substacks of the derived moduli stack  $F(U_{\pro\et},\Zl(1),-)$. These stacks are Lagrangian over the derived moduli stack 
\[
  A \mapsto \oR\Gamma(Z_{\pro\et}, i^*\oR j_*F(\bigoplus_r  \uline{\cW_rA}_U(r))) 
  \]  
 of $F$-valued sheaves on the punctured tubular neighbourhood of $U$ in its compactification, as in Examples \ref{Spex}.
 
 \item\label{globalopenZex} We can also combine these as in Example \ref{traceLagex2}. Take $U$ to be a smooth separated scheme of dimension $m$  over a either the ring of integers $\cO_K$ of a non-Archimedean local field, or over a localisation $\cO_{K,S}$ of the ring of integers of a number field. For a compactification $j \co U \into \bar{U}$ of $U$ over the same base, with complement $i \co Z \to \bar{U}$, we then have 
 $(n-2-2m)$-shifted Lagrangian structures of weight $m+1$ on suitable open substacks of the derived moduli stack  
 \[
 F(U_{\pro\et},\Zl(1),-) \co  A \mapsto \oR\Gamma(U_{\pro\et}, F(\bigoplus_r  \uline{\cW_rA}_U(r))).
\]
These are Lagrangian over 
the  derived moduli stack sending $A$ to the homotopy fibre product of the diagram
\[
  %
  \xymatrix@R=0ex{
  \oR\Gamma(Z_{\pro\et}, i^*\oR j_*F(\bigoplus_r  \uline{\cW_rA}_U(r))) \ar[dr]\\
 & \prod_{v \in S}\oR\Gamma( Z_{K_v,\pro\et}, i^*\oR j_*F(\bigoplus_r  \uline{\cW_rA}_U(r)))\\
  \ar[ur]
  \prod_{v \in S}\oR\Gamma(U_{ K_v,\pro\et}, F(\bigoplus_r  \uline{\cW_rA}_U(r)))
  }
  \]  
 of $F$-valued sheaves on the boundary $\pd U$ constructed as in Example \ref{traceLagex2}, where we take  $\{K_v\}_{v \in S}=\{K\}$ in the local case.
 
  \item\label{intersectGLCZex}  
  We also use these examples to give more shifted symplectic structures, as in Example \ref{IwintersectGLCex}. Given $Z$ smooth and proper of dimension $m$ over a ring of integers $\cO_K$ and a Lagrangian $G \to F$,   we  can take a derived intersection  of the Lagrangians of  (\ref{globalZex}) and  (\ref{localZex}), together with the Lagrangian induced by $G$ at places dividing $\ell$,
    to give an $(n-3-2m)$-shifted symplectic  structure of weight $m+1$ on the derived fibre product   of the diagram
  \[
%
\xymatrix@R=0ex{    \prod_{\substack{v \in S\\ v \nmid \ell}} F((Z\ten_{\cO_{K}}\cO_v  )_{\pro\et},\Zl(1),-) \ar[r] & \prod_{\substack{v \in S\\ v \nmid \ell}}  F((Z\ten_{\cO_{K}}K_v  )_{\pro\et},\Zl(1),-) \\
F( (Z\ten_{\cO_K}\cO_{K,S})_{\pro\et},\Zl(1),-)\ar[ur] \ar[dr] \\
\prod_{\substack{v \in S\\ v \mid \ell}} G((Z\ten_{\cO_{K}}K_v  )_{\pro\et},\Zl(1),-) \ar[r] & \prod_{\substack{v \in S\\ v \mid \ell}}  F((Z\ten_{\cO_{K}}K_v  )_{\pro\et},\Zl(1),-).
}
\]
Instead of using a Lagrangian $G$ to provide local conditions at places dividing $\ell$, in \S \ref{wSpexLagcrissn} we will incorporate crystalline data  to give a more generally applicable construction.
 \end{enumerate}
 
\end{examples}

\begin{example}[Moduli of $\bG_m$-torsors]\label{wBGmXex2}
 In the toy example $F=B\bG_m$, the description of Example \ref{wBGmXex} says that
 \[
 B\bG_m(X_{\pro\et},\uline{\Zl(1)}, A) \simeq B\bG_m(X_{\pro\et}, \cW_0A) \by \prod_{r \ge 0} N^{-1}\tau_{\ge 0}( \oR\Gamma(X_{\pro\et},  \uline{\cW_rA}(r))_{[-1]}).
 \]

 When $X= \Spec K_v$, for a non-Archimedean local field $K_v$ with finite residue field prime to $\ell$, all terms vanish except $r=0,1$,
 and local Tate duality \cite[\S II.5.2, Theorem 2]{galoisienne}
 gives an equivalence $\oR\Gamma((\Spec K_v)_{\pro\et}, \uline{\Ql}(1)) \simeq \oR\Gamma((\Spec K_v)_{\pro\et}, \uline{\Ql})^*[-2]$, so
 the functor $B\bG_m((\Spec K_v)_{\pro\et},\Zl(1),-)$ is equivalent to  the cotangent stack of the weight $0$ part $B\bG_m((\Spec K_v)_{\pro\et},-)$ with cotangent space having weight $1$. Under this equivalence, the  $0$-shifted symplectic structure coming from local Tate duality corresponds to the canonical symplectic structure on the cotangent stack.
 On  the Lagrangian $B\bG_m((\Spec \cO_v)_{\pro\et},\Zl(1),-)$, all the terms with $r \ne 0$ vanish, so it is identified under that equivalence with the open derived substack $B\bG_m((\Spec \cO_v)_{\pro\et},-)$ (the locus of unramified representations) of the canonical Lagrangian $B\bG_m((\Spec K_v)_{\pro\et},-)$ in the cotangent stack.

 \smallskip
 
 For $\ell$-adic local fields $K_v$ the functor is significantly larger, since $\H^1_{\pro\et}(K_v,\Ql(r))$ is non-zero (in fact isomorphic to $K_v$) for all $r$ (see e.g. \cite[Proposition 7.3.10]{NeukirchSchmidtWingberg}). Explicitly, since $\oR\Gamma((\Spec K_v)_{\pro\et}, \uline{\Ql}(r)) \simeq \oR\Gamma((\Spec K_v)_{\pro\et}, \uline{\Ql}(1-r))^*[-2]$ by local Tate duality, we can reinterpret every term of strictly positive weight as the cotangent bundle of a term of non-positive weight. This gives an equivalence between $B\bG_m((\Spec K_v)_{\pro\et},\Zl(1),-)$ and the derived (weighted) cotangent stack of either
 \begin{align*}
  &B\bG_m((\Spec K_v)_{\pro\et},-) \by \prod_{r<0} \H^1((\Spec K_v)_{\pro\et}, \Ql(r)) \quad\text{ or}\\
   &B\bG_m((\Spec K_v)_{\pro\et},-) \by \prod_{r>1} \H^1((\Spec K_v)_{\pro\et}, \Ql(r)),
 \end{align*}
where the  cotangent bundle of the weight $r$ term is placed in weight $1-r$. Here, we are writing $\H^1((\Spec K_v)_{\pro\et}, \Ql(r))\cong K_v$ for the affine space  $\bA^{[K_v:\Ql]}$ placed in weight $r$.

 \smallskip
 
Similarly, for  a localisation $\cO_{K,S}$ of the ring of integers of a number field at a finite set  $S$ including all primes dividing $\ell$,
Example \ref{wBGmXex}  combines with Poitou--Tate duality
\[
 \oR\Gamma_{\pro\et}(\cO_{K,S}, \vv)^* \simeq \cone(\oR\Gamma_{\pro\et}( \cO_{K,S}, \vv^*(1)) \to  \prod_{v \in S}B\bG_m((\Spec  K_v  )_{\pro\et},\vv^*(1))[2]
\]
to express $B\bG_m((\Spec \cO_{K,S})_{\pro\et},\Zl(1),-)$ as the derived (weighted) conormal stack of the morphism
\begin{align*}
  &B\bG_m((\Spec \cO_{K,S})_{\pro\et},-) \by \prod_{r<0} \oR\Gamma_{\pro\et}(\cO_{K,S}, \Ql(r))[1] \\
&\to (\prod_{v \in S}  B\bG_m((\Spec K_v)_{\pro\et},-)) \by (\prod_{\substack{r<0,\\ v \mid \ell}} \H^1_{\pro\et}(K_v, \Ql(r))),
  \end{align*}
where the conormal bundle of a weight $r$ term is placed in weight $1-r$. It can likewise be rewritten as the derived (weighted) conormal stack of the corresponding morphism with $r>1$ instead of $r<0$; since
 $\H^2_{\pro\et}(\cO_{K,S}, \Ql(r))$ vanishes for $r>1$ by \cite{souleKthCorpsNomCohoet}, that morphism simplifies to
 \begin{align*}
  &B\bG_m((\Spec \cO_{K,S})_{\pro\et},-) \by \prod_{r>1} \H^1_{\pro\et}(\cO_{K,S}, \Ql(r)) \\
&\to (\prod_{v \in S}  B\bG_m((\Spec K_v)_{\pro\et},-)) \by (\prod_{\substack{r>1,\\ v \mid \ell}} \H^1_{\pro\et}(K_v, \Ql(r))).
  \end{align*}


\end{example}

\section{Crystalline constructions and Selmer functors}\label{crissn} 

The Lagrangians coming from local Tate duality in Examples \ref{SpexLag}.(\ref{Iwlocalex}) only exist when the residue characteristic $p$ is prime to $\ell$. In order to incorporate  $\ell$-adic local fields, and thus combine this duality with Poitou--Tate duality to form Lagrangian intersections, we must instead incorporate crystalline structure. 

Specifically, the cohomology theory providing Lagrangians in $p$-adic Galois cohomology of a finite extension $K$ of $\Q_p$ is the $f$-cohomology of \cite{BlochKatoTamagawa}.  
\begin{definition}\label{Cfbtdef}
 Define the complex $C^{\bt}_f$ of $\Q_p$-vector spaces to be the cocone
 \begin{align*}
F^0B_{\dR} \oplus  B_{\cris} &\to B_{\dR} \oplus B_{\cris}\\
 (x,y)  &\mapsto (x-y, (\phi-1)y),
\end{align*}
for period rings $B_{\cris}$ and $B_{\dR}$ as in \cite{Bcrisillusie}.   
\end{definition}

As explained in \cite[1.19]{nekovarPadicHeight} and \cite[\S 6.1]{bettsMotAnab}, 
the key property is that $C^{\bt}_f$ is  a resolution of the Galois module $\Q_p$.
For any Galois representation $V$, we thus have  $\bH^*(K, C^{\bt}_f\ten V) \cong \H^*(K,V)$, but when $V$ is a de Rham representation, the cohomology of the subcomplex $(C^{\bt}_f\ten V)^{G_K}$ of $G_K$ invariants corresponds  \cite[1.19]{nekovarPadicHeight} to $\H^*_f(K,V)$, which exactly annihilates its dual  in $\H^*(K,V)$ by \cite[Proposition 3.8]{BlochKatoTamagawa}. 

Instead of looking at moduli constructions over $\cO_K$ as in  Examples \ref{SpexLag}.(\ref{Iwlocalex}), which are not Lagrangian for $\ell=p$, we will thus form constructions modelled on $(C^{\bt}_f\ten_{\Q_p} V)^{G_K}$ via the crystalline and de Rham correspondences. 

\subsection{Analogues of filtered \tps{$\Phi$}{Phi}-modules}\label{MFsn}
  
 Writing $K_0=W(k)[p^{-1}]$ where $k$ is the residue field of $K$, 
  the category of crystalline Galois representations of $K$  is equivalent to a full subcategory of Fontaine's category $\mathrm{MF}_{K}(\Phi)$ of filtered $\Phi$-modules.  The latter consists of  triples $(M,\Phi,\Fil)$, where $M$ is a $K_0$-vector space, $\Phi$ is a  Frobenius-semilinear injective endomorphism of $M$, 
and $\Fil$ is a decreasing filtration  on $M\ten_{K_0}K$ which is exhaustive and Hausdorff.  The natural cohomology theory for this category associates to such data the cochain complex
\[
\Fil^0(M\ten_{K_0}K)\oplus M \xra[ (x-y,\Phi(y)-y)]{ (x,y) \mapsto  } (M\ten_{K_0}K)\oplus M;
\]
in particular,  note that $\H^0$ of the complex is $\Hom_{\mathrm{MF}_{K}(\Phi)}(K_0,M)$, where $K_0$ is given trivial  Frobenius action and filtration $\gr_{\Fil}^0K=K$.

We can generalise this to give the following notion of filtered $\Phi$-objects taking values in any derived stack $F$, along similar lines to  the  constructions $U(B_f^{\bt})^{G_K}$  of \cite[\S 6.1]{bettsMotAnab}. First, via Rees constructions as in \cite[\S 5]{Simfil}, a  filtered $K$-vector space is equivalent to a vector bundle on the stack $[\bA^1_K/\bG_m]$, which generalises to $F_{\Hod}$ below; incorporating Frobenius then gives rise to $F_{\mathrm{MF}(\Phi)}$.

\begin{definition}\label{MFPhidef}
 Given  a non-Archimedean local field $K$ of residue characteristic $p$, with ring of integers $\cO_K$ and residue field $k$, and a derived $\infty$-stack  $F \co dg_+\CAlg_{\Q_p} \to s\Set$ over $\Q_p$, define $F_{\mathrm{MF}(\Phi)}(K,-)$ as follows.
 
 First, for any $\Q_p$-CDGA  $A$ we simply define crystalline and de Rham functors associated to $\Spec K$ by  $F_{\cris}(K,A):= F(A\ten_{\Q_p}K_0)$ and $F_{\dR}(A):= F(A\ten_{\Q_p}K)$; these generalise the moduli of $K_0$- and $K$-vector spaces, respectively.
 
Next, we  define a Hodge functor $F_{\Hod}(K,-)\co dg_+\CAlg_{\Q_p} \to s\Set $ parametrising 
filtered $K$-linear objects of $F$ by
\[
 F_{\Hod}(K,A):= \oR\map( [\bA^1_{K\ten_{\Q_p}A}/\bG_m],F) = \ho\Lim_{n \in \Delta}F(A\ten_{\Q_p}K[s, u_1^{\pm}, \ldots, u_n^{\pm}]).
\]

We then define $F_{\mathrm{MF}(\Phi)}(K,-) \co dg_+\CAlg_{\Q_p} \to s\Set $ by 
setting $F_{\mathrm{MF}(\Phi)}(K,A)$ to be the homotopy equaliser
\[
\xymatrix@1{
 F_{\Hod}(K,A) \by F_{\cris}(K,A) \ar@<0.5ex>[rr]^{ (x,y) \mapsto (x,\sigma^*y) } \ar@<-0.5ex>[rr]_{ (x,y) \mapsto (y,y) } &&  F_{\dR}(K,A) \by F_{\cris}(K,A) }
\]
where $\sigma$ denotes Frobenius and the map $F_{\Hod} \to F_{\dR}$ comes from the inclusion $\Spec (K\ten_{\Q_p}A) \cong [\bG_{m,K\ten_{\Q_p}A}/\bG_m] \into [\bA^1_{K\ten_{\Q_p}A}/\bG_m]$.
 \end{definition}
  
  Taking $F$ to be   the derived  stack $B\GL_n$, it follows that the space  $(B\GL_n)_{\mathrm{MF}(\Phi)}(K,\Q_p) $ of $\Q_p$-valued points  is equivalent to the  nerve of the maximal subgroupoid of $\mathrm{MF}_{K}(\Phi) $ on objects which have rank $n$ and for which the map $\Phi \co \sigma^*M \to M $ is an isomorphism.
  
We will also need a weighted version:
\begin{definition}
 Given a $\bG_m$-equivariant CDGA $A$ and any of the constructions $F_?$ of Definition \ref{MFPhidef}, where $? \in \{\cris,\dR,\Hod,\mathrm{MF}(\Phi) \}$,  define   $F_?(K, \Z_p(1),A)$ to be given by the same formula as $F_?(K,A)$, except that $\bG_m$ and $\Phi$ are now taken to act non-trivially on $A$ in the obvious way.
\end{definition}
For stacks of modules, this corresponds to  taking a $(\bigoplus_r \cW_rA(r)) \ten_{\Q_p}K_0$-module $M$ with a Frobenius semilinear endomorphism $\Phi\co M \to M$ inducing an isomorphism $\sigma^*M \to M$, together with   a filtration on  $M\ten_{K_0}K$ compatible with the grading on $A$; note that Frobenius $\sigma$ is  here acting on $(\bigoplus_r \cW_rA(r))$ non-trivially  via Tate twists.

The functors $F_{\mathrm{MF}(\Phi)}(K,-)$ and $F_{\mathrm{MF}(\Phi)}(K, \Z_p(1),-)$ are manifestly homogeneous whenever $F$ is so, and 
we can then describe tangent complexes of these functors.
Given $x \in F_{\mathrm{MF}(\Phi)}(K, \Z_p(1),A)$, we have an underlying point $x_{\cris} \in F_{\cris}(K,A)$, and hence a cotangent complex $\bL^{F,A\ten K_0,x_{\cris}}$, which is a complex of $A\ten K_0$-modules equipped with a quasi-automorphism $\Phi$. We also have a point $x_{\Hod} \in F_{\Hod}(K,A)$ and hence a quasi-coherent complex $\bL^{F,x_{\Hod}}:=\bL^{F,[\bA^1_{K\ten_{\Q_p}A}/\bG_m]  ,x_{\Hod}}$  over $[\bA^1_{K\ten_{\Q_p}A}/\bG_m]$, whose restriction to the generic point $[(\bG_m)_{K\ten_{\Q_p}A}/\bG_m] \simeq \Spec (K\ten_{\Q_p}A)$ is quasi-isomorphic to  $\bL^{F,x_{\cris}}:=\bL^{F,A\ten K_0,x_{\cris}}$. We then immediately have:
\begin{lemma}\label{MFtgtlemma}
The tangent complex  of $F_{\mathrm{MF}(\Phi)}(K, \Z_p(1),-)$ at $x$ sends a graded $A$-module $M$ to the cocone of
 \begin{align*}
 &\oR\cW_0\HHom_{A\ten K_0}(\bL^{F,x_{\cris}},M\ten K_0) \oplus\oR\cW_0\HHom_{[\bA^1_{K\ten_{\Q_p}A}/\bG_m] }(\bL^{F,x_{\Hod}},M\ten \sO_{[\bA^1_K/\bG_m]})\\
&\to \oR\cW_0\HHom_{A\ten K_0}(\bL^{F,x_{\cris}},M\ten K_0) \oplus  \oR\cW_0\HHom_{A\ten K}(\bL^{F,x_{\cris}}\ten_{K_0}K,M\ten K),
 \end{align*}
 where as usual the map is $(y,z) \mapsto (\sigma^*y-y, y-z)$.
\end{lemma}

Note that when the tangent complexes are associated as filtered $\phi$-modules to crystalline $G_K$-representations, the expression of Lemma \ref{MFtgtlemma} is the same as the complex $(V_{\bt}\ten_{\Q_p}C^{\bt}_f)^{G_K}$ calculating the  cohomology $\H^*_f(K,-)$ from \cite{BlochKatoTamagawa}. More precisely, if we have a complex $V_{\bt}$ of  crystalline $G_K$-representations, with $K_0[\phi]$-linear complex $D_{\cris}(V_{\bt})= (V_{\bt}\ten_{\Q_p}B_{\cris})^{G_K}$, and   filtered $K$-linear complex  $D_{\dR}(V_{\bt})= (V_{\bt}\ten_{\Q_p}B_{\dR})^{G_K}\cong D_{\cris}(V_{\bt})\ten_{K_0}K$, then $(V_{\bt}\ten_{\Q_p}C^{\bt}_f)^{G_K}$ is the cocone of
\[
D_{\cris}(V_{\bt}) \oplus \oR\Gamma([\bA^1/\bG_m], \xi( D_{\dR}(V_{\bt}) ,\Fil)) \to D_{\cris}(V_{\bt}) \oplus D_{\dR}(V_{\bt}),
\]
for the map $(y,z) \mapsto (\sigma^*y-y, y-z)$, where $\xi$ denotes the Rees construction, i.e. the $\bG_m$-equivariant $K[s]$-module  $\xi(D,\Fil)= \bigoplus_i \Fil^iD s^{-i}$.

  \subsection{Analogues of \tps{$B$}{B}-pairs} \label{Bpairsn}
  
The functor $F_{\mathrm{MF}(\Phi)}(K,-)$ does not map directly to $F ((\Spec \cO_K)_{\pro\et},-)$, so we now need to introduce an intermediate functor which receives maps from both functors, the second \'etale   and the  first Lagrangian, at least on suitable subfunctors.

When applied to moduli functors such as $B\GL_n$, this intermediate functor parametrises objects of a slight generalisation of Berger's category of $B$-pairs from \cite{bergerBpairsPhiGamma}, which is equivalent to the category of $(\Phi,\Gamma)$-modules.

Since the period ring $B_{\dR}$ carries a filtration, we can form its Rees construction $\xi(B_{\dR},\Fil)$, which is a graded $K[s]$-algebra, leading to the following definitions.

\begin{definition}
 Given a derived $\infty$-stack  $F \co dg_+\CAlg_{\Q_p} \to s\Set$ over $\Q_p$, define $F^{\mathsf{B}}(K,\Z_p(1),-)\co dg_+\Affd\Alg_{\Q_p}^{f\bG_m,\dagger,\loc} \to s\Set$ as follows.

For $\bE:=\Z_p(1)$, we first define
\begin{align*}
F^{\mathsf{B}}_{\Hod}(K,\bE,A)&:=   \oR\Gamma( (\Spec K)_{\pro\et},\oR\map( [  \Spec ( \xi(\uline{B_{\dR} \hten_{\Q_p}A}(\bE)  ,\Fil)
)/\bG_m],F) ),\\
F^{\mathsf{B}}_{\cris}(K,\bE,A)&:=   \oR\Gamma( (\Spec K)_{\pro\et}, F(\uline{ B_{\cris} \hten_{\Q_p}A}
(\bE)  ) ),\\
F^{\mathsf{B}}_{\dR}(K,\bE,A)&:=   \oR\Gamma( (\Spec K)_{\pro\et}, F(\uline{ B_{\dR}\hten_{\Q_p}A}
(\bE)  ) ),
\end{align*}
and then define 
$F^{\mathsf{B}}(K,\bE, A)$ to be the homotopy equaliser
\[
\xymatrix@1{
 F^{\sfB}_{\Hod}(K,\bE,A)  \by F_{\cris}^{\sfB}(K,\bE,A) \ar@<0.5ex>[rr]^{ (x,y) \mapsto (x,\sigma^*y) } \ar@<-0.5ex>[rr]_{ (x,y) \mapsto (y,y) } &&  F_{\dR}^{\sfB}(K,\bE,A) \by F_{\cris}^{\sfB}(K,\bE,A). }
\]

\end{definition}

%
%
In other words, $F^{\mathsf{B}}$ parametrises data consisting of a $\phi$-equivariant $\uline{B_{\cris}}$-object, a filtered $\uline{B_{\dR}}$-object and an equivalence between the underlying $\uline{B_{\dR}}$-objects.
%

\begin{example}\label{BGLnBex}
 In the case $F=B\GL_n$, the space $(B\GL_n)^{\mathsf{B}}_{\cris}(K,\bE,A)$ parametrises locally free  $B_{\cris} \hten_{\Q_p}A(\bE)$-modules $M_{\cris}$ of rank $n$ on $\Spec K_{\pro\et}$. The space $(B\GL_n)^{\mathsf{B}}_{\dR}(K,\bE,A)$ is similar, replacing $B_{\cris}$ with $B_{\dR}$. The space $F^{\mathsf{B}}_{\Hod}(K,\bE,A)$ parametrises filtered $B_{\dR} \hten_{\Q_p}A(\bE)$-modules $M_{\Hod}$, where the filtration is required to be compatible with that on $B_{\dR}$, such that $M$ is locally free of rank $n$ and $\gr_{\Fil}M$ is locally free of rank $n$ over  $\gr_{Fil}(B_{\dR}\hten A)$.

 Thus $B\GL_n^{\mathsf{B}}(K,\bE, A)$ is  the space of quadruples $( M_{\Hod},M_{\cris},\theta,\phi)$ for $M_{\Hod}$, $M_{\cris}$ as above, with  $\theta \co M_{\Hod} \cong \mathsf{B}_{\dR}\hten_{B_{\cris}}M_{\cris}$ a $B_{\dR} \hten_{\Q_p}A(\bE)$-linear (unfiltered) isomorphism and $\phi \co \sigma^*M_{\cris} \to M_{\cris}$ an isomorphism.
\end{example}

\begin{definition} \label{MFtoBdef}
 For $A \in dg_+\Affd\Alg_{\Q_p}^{f\bG_m,\dagger,\loc}$, define natural maps
\begin{align*}
F_{\cris}(K,\Z_p(1),A) &\to F_{\cris}^{\mathsf{B}}(K, \Z_p(1), A),\\
F_{\dR}(K, \Z_p(1), A) &\to     F_{\dR}^{\mathsf{B}}(K, \Z_p(1), A),\\ 
 F_{\Hod}(K, \Z_p(1), A) &\to F^{\mathsf{B}}_{\Hod}(K, \Z_p(1), A),\text{ and hence}\\
F_{\mathrm{MF}(\Phi)}(K, \Z_p(1), A) &\to F^{\mathsf{B}}(K, \Z_p(1),  A)
\end{align*}
as follows. For   
$ A(\ast):= \bigoplus_{n \in \Z}\cW_nA(n)$,
the first three are respectively the maps
\begin{align*}
 F(K_0\ten_{\Q_p}A(\ast)) &\to \oR\Gamma_{\pro\et}( \Spec K, F(\uline{ B_{\cris} \hten_{\Q_p}A(\ast)}  ) ),\\
 F(K\ten_{\Q_p}A(\ast)) &\to \oR\Gamma_{\pro\et}( \Spec K, F(\uline{ B_{\dR} \hten_{\Q_p}A(\ast)}  ) ),\\
 \oR\map( [\bA^1_{ K\ten_{\Q_p}A(\ast)}/\bG_m],F) &\to\oR\Gamma_{\pro\et}( \Spec K,\oR\map( [  \Spec ( \uline{\xi(B_{\dR}  ,\Fil) \hten_{\Q_p}A(\ast)}
)/\bG_m],F) ),
 \end{align*}
induced by the isomorphisms
\begin{align*}
K_0(n) &= D_{\cris}(\Q_p(n)) = \Gamma_{\pro\et}(\Spec  K, \uline{ B_{\cris}(n)}),\\
K(n) &= D_{\dR}(\Q_p(n))= \Gamma_{\pro\et}(  \Spec K, \uline{ B_{\dR}(n)}), \quad \text{  and}\\\
s^nK[s] &=   ( \bigoplus_{i\in \Z} s^{-i} \Fil^{i+n}B_{\dR})^{G_K} 
=\Gamma_{\pro\et}(\Spec  K,\uline{\xi(B_{\dR}(n) ,\Fil)} ),
\end{align*}
for $D_?$ as in \cite[\S 1.3]{Bcrisillusie}. The fourth follows by taking homotopy fibre products.
 \end{definition}

\begin{definition} 
Define the natural map
\[
 F ((\Spec K)_{\pro\et},\Z_p(1),-) \to F^{\mathsf{B}}(K, \Z_p(1),  -)
\]
to be that induced by the maps from $\Q_p$ to the diagram 
\[
\xymatrix@1{ 
\Fil^0B_{\dR}\oplus B_{\cris} \ar@<0.5ex>[rr]^{ (x,y) \mapsto (x,\sigma^*y) } \ar@<-0.5ex>[rr]_{ (x,y) \mapsto (y,y) } &&  B_{\dR} \oplus B_{\cris}. 
}
\]
\end{definition}

\begin{lemma}\label{crisetalelemma}
 If $F$ has perfect cotangent complex, then the map  $F ((\Spec K)_{\pro\et},\Z_p(1),-) \to F^{\mathsf{B}}(K,\Z_p(1), -)$ is formally \'etale.
\end{lemma}
\begin{proof}
 Given a point $x \in F ((\Spec K)_{\pro\et},\bE,A)$ and an $A$-module $M$, the tangent complex $T_xF_{\Hod}^{\sfB}(K, \bE, M)$ of Definition \ref{Tdef} is given by applying $\oR\Gamma((\Spec K)_{\pro\et},-)$ to the pro-\'etale sheaf
 \[
  \oR\Gamma\big([\bA^1/\bG_m], T_x\big(F, \uline{M}(\bE)\ten_{\Q_p}\uline{\xi(B_{\dR}  ,\Fil)}\big)\big).
 \]
Since the cotangent complex of $F$ is perfect, the latter is just
\[
 \oR\Gamma\big([\bA^1/\bG_m], T_x\big(F, \uline{M}(\bE)\big)\ten_{\Q_p}\uline{\xi(B_{\dR}  ,\Fil)}\big)= T_x(F, \uline{M}(\bE))\ten_{\Q_p}\uline{\Fil^0B_{\dR}}.
\]

Replacing $\Fil^0B_{\dR}$ with $B_{\dR}$ and $B_{\cris}$ gives the tangent complexes $T_xF^{\sfB}_{\dR}$ and $T_xF^{\sfB}_{\cris}$, so by definition $T_xF^{\sfB}(K, \bE, M)$ is given by applying
$\oR\Gamma((\Spec K)_{\pro\et},T_x(F, \uline{M}(\bE))\ten_{\Q_p}-)$ to
the homotopy equaliser of
the resolution for $\Q_p$ above. Since  $T_xF ((\Spec K)_{\pro\et},\Z_p(1),M)\simeq \oR\Gamma((\Spec K)_{\pro\et},T_x(F, \uline{M}(\bE)))$, it follows that our map is formally \'etale.
\end{proof}

\begin{definition}\label{moduliassocdef}
 Define $F^{\cris}((\Spec K)_{\pro\et},\Z_p(1),-)$ to be the homotopy pullback 
 $F_{\mathrm{MF}(\Phi)}(K, \Z_p(1), -)\by^h_{F^{\mathsf{B}}(K, \Z_p(1),  -)}F ((\Spec K)_{\pro\et},\Z_p(1),-)$. 
\end{definition}
We can think of this as derived moduli of associations, i.e. associated Galois and crystalline data, or simply as moduli of $F$-valued crystalline representations.

Lemma \ref{crisetalelemma} implies that $F^{\cris}((\Spec K)_{\pro\et},\Z_p(1),-)$ is formally \'etale over $F_{\mathrm{MF}(\Phi)}(K, \Z_p(1),-)  $.

\begin{example}\label{BGLncrisex}
Applied to $F=B\GL_n$, and substituting in Example \ref{BGLnBex}, Definition \ref{moduliassocdef} unwinds to the following.
Writing $A(\ast):=\bigoplus_n A(n)$ for $A$ equipped with the usual filtration and Frobenius for Tate twists,
 an object of $F^{\cris}((\Spec K)_{\pro\et},\Z_p(1),A)$ consists of:
\begin{itemize}[itemsep=5pt, parsep=0pt]
 \item a pro-\'etale locally free rank $n$ $\uline{A}(\Z_p(1))$-module $\vv$ on $\Spec K$ (the object of $F ((\Spec K)_{\pro\et},\Z_p(1),-)$),
\item a triple $(M,\Phi,\Fil)$ (the object of $F_{\mathrm{MF}(\Phi)}(K, \Z_p(1),A)$), where $M$ is a locally free rank $n$ $K_0\ten_{\Ql}A(\ast)$-module, $\Phi$ is a  Frobenius-semilinear injective endomorphism of $M$,
and $\Fil$ is a decreasing filtration  on $M\ten_{K_0}K$ which is exhaustive, Hausdorff and compatible with the filtration on $A$,
\item a Frobenius-equivariant isomorphism $ \alpha \co \vv\hten_{\Q_p}\uline{B_{\cris}} \to M\hten_K\uline{B_{\cris}}$ such that $\alpha\hten_{B_{\cris}}B_{\dR}$ is a filtered isomorphism.
 \end{itemize}
\end{example}

\begin{example}\label{wBGmXexcris}
 In the special case $F=B\bG_m$, the isomorphism
 $\bG_m( \bigoplus_r \cW_rA(r)) \cong \bG_m(\cW_0A)  \by (\bigoplus_{r\ne 0} \cW_rA(r))$ of Example \ref{wBGmXex} gives us an equivalence
\begin{align*}
 (B\bG_m)_{\mathrm{MF}(\Phi)} (K, \Z_p(1),A)&\simeq (B\bG_m)_{\mathrm{MF}(\Phi)} (K, \cW_0A) \by \prod_{r \ne 0} N^{-1}\cW_rA\ten_{\Q_p} \H^1_f(K,\Q_p(r))\\
 &\simeq (B\bG_m)_{\mathrm{MF}(\Phi)} (K, \cW_0A) \by \prod_{r > 0} N^{-1}\cW_rA\ten_{\Q_p}K,
\end{align*}
where the second equivalence follows because $\Fil^0K(r)$ is $0$ for $r>0$ and $K$ for $r<0$, so the cone of $(\id, \Phi -\id) \co K_0(r) \to (K(r)/\Fil^0) \oplus K_0(r)$ is quasi-isomorphic to $K[1]$ for $r>0$ and $0$ for $r<0$.

Applying Lemma \ref{crisetalelemma}, it thus follows that
\[
B \bG_m^{\cris}((\Spec K)_{\pro\et},\Z_p(1),A) \simeq B\bG_m^{\cris}((\Spec K)_{\pro\et},\cW_0A) \by\prod_{r > 0} N^{-1}\cW_rA\ten_{\Q_p}K.
\]
\end{example}

\subsection{Trace, duality and Lagrangians} 

The relation between the $\mathrm{MF}(\Phi)$ and $B$-pair constructions is very similar to that between $U$ and $\pd U$ in \S \ref{Lagsn}.

For the complex $C_f^{\bt}$ from Definition \ref{Cfbtdef} and the quasi-isomorphism  $\Q_p \to C_f^{\bt}$, we have
the following analogue of Example \ref{traceLagex}.(\ref{traceLaglocex}):
\begin{lemma}\label{cristracelemma}
The trace map  of Example \ref{traceex}.(\ref{tracelocex}) factors through a trace
\[
 \tr \co \cone\big( C_f^{\bt}(1)^{G_K} \to \oR\Gamma((\Spec K)_{\pro\et}, \Q_p(1))\big) \to \Q_p[-2].
\]

Given a quasi-dagger dg $K$-algebra $A$ and a finite $A$-module $M$ in complexes, this induces a natural map
\[
  \cone\big( (C_f^{\bt}\hten_{\Q_p}M(1))^{G_K} \to \oR\Gamma((\Spec K)_{\pro\et},\uline{M}\hten_{\Zl}\Zl(1))\big) \to M[-2] 
\]
in the $\infty$-category of cochain complexes, depending only on the structure of $M$ as a  
complex of topological abelian groups.
\end{lemma}
\begin{proof}
 The first statement follows immediately because $C_f^{\bt}(1)$ is concentrated in cochain degrees $[0,1]$, with  $\oR\Gamma((\Spec K)_{\pro\et}, \Q_p(1)) \simeq \oR\Gamma((\Spec K)_{\pro\et}, C_f^{\bt}(1))$. The second statement then follows exactly as in Proposition \ref{traceprop}.
\end{proof}
Explicitly, observe that $(C_f^{\bt}\hten_{\Q_p}M(1))^{G_K}$ is just the complex $K_0\ten_{\Q_p}M \xra{(\id,(p^{-1}\phi -1)\ten \id} (K\oplus K_0)\ten_{\Q_p}M$.

\begin{corollary}\label{crisLagcor}
If   $F \co dg_+\CAlg_{\Q_p} \to s\Set$ is an
  $n$-shifted symplectic  derived Artin $\infty$-stack, then the morphism $F^{\cris}((\Spec K)_{\pro\et},\Z_p(1),-) \to F((\Spec K)_{\pro\et},\Z_p(1),-)$ from Definition \ref{moduliassocdef} is isotropic with respect to the $(n-2)$-shifted symplectic structure of 
 Example \ref{wSpex}.(\ref{wlocalex}).
  
  For the full subfunctor   $F^{\cris}((\Spec K)_{\pro\et},\Z_p(1),-)^{wd}  \subset F^{\cris}((\Spec K)_{\pro\et},\Z_p(1),-)$  consisting of points $x$ at which the trace of Lemma \ref{cristracelemma} induces a perfect pairing on cotangent complexes similar to Definition \ref{wweakdualLagdef}, the natural map
\[ 
 F^{\cris}((\Spec K)_{\pro\et},\Z_p(1),-)^{wd} \to F((\Spec K)_{\pro\et},\Z_p(1),-)^{wd} 
 \]
 carries a natural Lagrangian structure with respect to the $(n-2)$-shifted symplectic structure of weight $1$ given by Corollary \ref{wtracecor2a} and Example \ref{wSpex}.(\ref{wlocalex}).

 In particular, this includes all points valued in weighted dgas $A$ which are finite-dimensional $\Q_p$-vector spaces.
\end{corollary}
\begin{proof}
 The first part follows by exactly the same argument as Corollary \ref{wtracecor2Laga}. 
 
 At any $A$-valued point $x \in F^{\cris}((\Spec K)_{\pro\et},\Z_p(1),A)^{wd}$, we have a constructible pro-\'etale $\uline{A}$-module $L_{\bt}^{\et}:=\bL^{F,\uline{A}_{\Spec K}(\uline{\Z}_p(1)), x_{\et}}$ and a perfect $A\ten_{\Q_p}K_0$-module $L_{\bt}^{\cris}:= \bL^{F_{\cris},A(\Z_p(1)), x_{\cris}}$ equipped with a Frobenius semilinear endomorphism $\phi$ and a homotopy filtration $\Fil$ on $L_{\bt}^{\cris}\ten_{K_0}K $. 
 Since $F$ has perfect cotangent complexes, the sheaves $\H_i(L_{\bt}^{\et})$ are finite $A$-modules. The hypothesis on $A$ then implies that the  $A$-linear objects $(\H_i(L_{\bt}^{\cris}),\phi, \Fil)$ are associated via Fontaine's functors to finite-dimensional  Galois representations $V_i$, which are thus crystalline. 
 
 The duality required for the final part now follows from the exact triangle
 \[
  (C_f^{\bt}\ten_{\Q_p}V)^{G_K} \to \oR\Gamma((\Spec K)_{\pro\et}, V) \to ((C_f^{\bt}\ten_{\Q_p}V^*(1))^{G_K})^* [-2]\to 
 \]
from \cite[Proposition 3.8]{BlochKatoTamagawa}. 
\end{proof}

\subsection{Selmer constructions}\label{wSpexLagcrissn}
 
  We also use these examples to give more shifted symplectic structures, taking derived intersection  of the Lagrangians of  Examples \ref{wSpexLag}.(\ref{globalex}) and  \ref{wSpexLag}.(\ref{localex}) and Corollary \ref{crisLagcor} to give an $(n-3)$-shifted symplectic  structure of weight $1$ on the derived fibre product $F^{\mathrm{Sel}}(\cO_{K,S},\Z_p(1),-)^{wd}$   of the diagram
  \[
  \xymatrix@C=-20ex{
    &  \prod_{\substack{v \in S \\ v \nmid p}} F (( \cO_v  )_{\pro\et},\Z_p(1),-)^{wd}  \by  \prod_{\substack{v \in S \\ v \mid p}} F^{\cris}(( K_v)_{\pro\et},\Z_p(1),-)^{wd} \ar[dr]\\
   F( ( \cO_{K,S})_{\pro\et},\Z_p(1),-)^{wd} \ar[rr]  & &\prod_{v \in S}F ( (K_v  )_{\pro\et},\Z_p(1),-)^{wd}
  }
  \]
    whenever $F$ is $n$-shifted symplectic, where we write $\cO_{\pro\et}:= (\Spec \cO)_{\pro\et}$.
This leads to many more applications than the variant with analogues of Greenberg's local conditions in  Example \ref{wSpexLag}.(\ref{intersectGLCZex}), since it does not rely on a Lagrangian in $F$ to provide restrictions at the primes dividing $p$.

When  $F$ is $B\GL_n$ or  the stack of perfect complexes, the resulting stack is equipped with  a $(-1)$-shifted symplectic structure, and parametrises derived moduli of $G_{K,S}$-representations which are unramified at primes in $S$ not dividing $\ell$, and crystalline at primes dividing $\ell$. These functors might be thought of as  Selmer complexes with non-abelian  coefficients.

\begin{example}\label{wBGmXex3}
  In the toy example $F=B\bG_m$, the descriptions of Examples \ref{wBGmXex}, \ref{wBGmXex2} and \ref{wBGmXexcris} give the following characterisations:
  \begin{itemize}[itemsep=5pt, parsep=0pt]
   \item for all $v$, $B\bG_m(( K_v)_{\pro\et},\Z_p(1),-)$ as the derived cotangent stack of $B\bG_m(( K_v)_{\pro\et},\Z_p(1),\cW_{\le 0}(-))$, with the latter equivalent to $B\bG_m^{\cris}(( K_v)_{\pro\et},\cW_{0}(-))$ when $v \nmid p$;
  \item when $v \nmid p$, $B\bG_m(( \cO_v  )_{\pro\et},\Z_p(1),-)$  as the locus of unramified representations $B\bG_m(( \cO_v)_{\pro\et},\cW_0(-)) \subset B\bG_m((K_v)_{\pro\et},\cW_0(-)) $, or equivalently as the conormal bundle of that (open) inclusion;
  \item when $v \mid p$, $B\bG_m^{\cris}((K_v)_{\pro\et},\Z_p(1),-)\simeq B\bG_m^{\cris}(( K_v)_{\pro\et},\Z_p(1),\cW_{\ge 0}(-))$ as the conormal bundle of the morphism
  \begin{align*}
 &B\bG_m^{\cris}(( K_v)_{\pro\et},\cW_0(-))\simeq  \\
 &B\bG_m^{\cris}((K_v)_{\pro\et},\Z_p(1),\cW_{\le 0}(-))   \to B\bG_m(( K_v)_{\pro\et},\Z_p(1),\cW_{\le 0}(-)),
  \end{align*}
using the cohomological duality of $\H^*_f$ as in the proof of Corollary \ref{crisLagcor}. 
  \end{itemize}
In all of these cases, we are placing cotangent and conormal bundles of weight $r$ terms in weight $1-r$.

Passing to the homotopy fibre product and substituting, it follows that   $(B\bG_m)^{\mathrm{Sel}}(\cO_{K,S},\Z_p(1),-)$ is equivalent to the  $(-1)$-shifted cotangent stack of $(B\bG_m)^{\mathrm{Sel}}(\cO_{K,S},\Z_p(1),\cW_{\le 0}(-))$, and that we can describe the latter as
\[
 (B\bG_m)^{\mathrm{Sel}}(\cO_{K,S},\cW_0(-)) \by \prod_{r<0} N^{-1}\tau_{\ge 0}(\H^2_c(G_{K,S}, \Q_p(r))\ten_{\Q_p}\cW_r(-) [-1]),
\]
where $\oR\Gamma_c(G_{K,S},V)$ is the cocone of $\oR\Gamma_{\pro\et}(\cO_{K,S},\uline{V})\to \prod_{v \in S} \oR\Gamma_{\pro\et}(K_v,\uline{V})$, using that $\H^1_c (G_{K,S}, \Q_p(r))=0$ for $r<0$ by \cite{souleKthCorpsNomCohoet}.

Taking duals of the cohomology groups, Poitou--Tate duality allows us to rewrite this as saying that $(B\bG_m)^{\mathrm{Sel}}(\cO_{K,S},\Z_p(1),-)$ is equivalent to the  $(-1)$-shifted cotangent stack of
\[
 (B\bG_m)^{\mathrm{Sel}}(\cO_{K,S},\cW_0(-)) \by \prod_{r>1} N^{-1} \H^1_{\pro\et}(\cO_{K,S}, \uline{\Q_p}(r))\ten_{\Q_p}\cW_r(-).
\]
\end{example}

\subsection{Higher-dimensional analogues} \label{highercrissn}

Finally, there exist versions of these functors for a  smooth proper variety $Z$ over a $p$-adic local field $K$ admitting a model $\cZ$ over the Witt vectors $W(k)$ of the residue field, given by combining the duality of Lemma \ref{cristracelemma} with the Poincar\'e duality of \cite[\S VII.2]{Be} and the \'etale-crystalline comparison theorem of \cite[5.6]{Hop}.
For these functors, we can set (for  $\cZ_k:=\cZ\ten_{W(k)}k$)
\[
 F_{\cris}(Z,A):= \oR\Gamma(\cZ_{k,\ind\cris}/W(k),  F(\sO_{\cris}\ten_{\Z_p}A)), 
 \quad  F_{\dR}(Z,A):=\oR\Gamma( Z_{\dR}^{\an}, F(\sO_{Z,\dR}\ten_{\Q_p})),
\]
where $\cZ_{k,\ind\cris}$ is the variant of the crystalline site allowing formal schemes as thickenings (so the tensor product $\sO_{\cris}\ten_{\Z_p}A$ is non-trivial), and  where   $Z_{\dR}^{\an}$ is  the analytification of the de Rham stack of \cite{Simfil}. We also use  Simpson's Hodge stack $Z_{\Hod}$ in place of $\bA^1_K$, so take   
\[
 F_{\Hod}(Z,A):=\oR\Gamma([Z_{\Hod}/\bG_m], F(\sO\ten_{\Q_p} A)).
\]
The other formulae of \S\S \ref{MFsn}, \ref{Bpairsn} then adapt to this generality. 

In particular, this (or rather the obvious generalisation to weighted algebras) allows us to replace the local conditions at primes dividing $\ell$ in Example \ref{wSpexLag}.\ref{intersectGLCZex} with crystalline terms instead of using a Lagrangian in $F$.

\section{Weighted shifted Poisson structures}\label{poisssn}

In this section, we introduce shifted Poisson structures on weighted dg algebras, twisted by a weight, and establish a comparison with weighted symplectic structures. We also outline analogues for derived stacks and formulations of deformation quantisation.  For these purposes, there is nothing special about the dagger affinoid setting in which we are working, and obvious analogues of these results hold in particular in derived algebraic geometry.

\subsection{Polyvectors}

In this section, we fix an object $A$ of the  pro-category $\pro(dg_+\Affd\Alg_K^{\bG_m,\dagger,\loc})$, which in applications will be quasi-free. The following is adapted from \cite[Definition \ref{poisson-poldef}]{poisson}:
 

%
\begin{definition}\label{poldef} 
Define the graded cochain complex of $m$-twisted $n$-shifted multiderivations 
on $A$ by
\[
 \widehat{\Pol}(A,n,m):=  \prod_p \cHom_A(\CoS_A^p((\Omega^1_{A} \{m\})_{[-n-1]}),A), 
\]
for the internal $\Hom$-complex $\cHom$ of Definition \ref{cHomdef} and twist $\{m\}$ as in Definition \ref{twistweightdef}, and for $\CoS_A^p(M) =\Co\Symm^p_A(M)= (M^{\ten_A p})^{\Sigma_p}$. 
Note that the only effect of the twisting is to affect the weights, but that the product has to be taken in each weight separately.
This has a $\bG_m$-equivariant graded-commutative  multiplication following the usual conventions for symmetric powers.

The Lie bracket on $\cHom_A(\Omega^1_{A/R},A) = \cHom_A(\Omega^1_{A/R}\{m\},A)\{m\}$ then extends to give a  $\bG_m$-equivariant bracket (the Schouten--Nijenhuis bracket)
\[
[-,-] \co \widehat{\Pol}(A,n,m)\{m\} \by \widehat{\Pol}(A,n,m)\{m\}\to \widehat{\Pol}(A,n)^{[-1-n]}\{m\},
\]
determined by the property that it is a bi-derivation with respect to the multiplication operation. 


Note that the cochain differential $\delta$ on $\widehat{\Pol}(A,n,m)$ can be written as $[\delta,-]$, where $\delta \in \cW_m\widehat{\Pol}(A,n,m)^{n+2}$ 
is the element defined by the derivation $\delta$ on $A$.
\end{definition}

\begin{definition}\label{Fdef}
Define a decreasing filtration $F$ on  $\widehat{\Pol}(A,n,m)$ by 
\[
 F^i\widehat{\Pol}(A,n,m):= \prod_{j \ge i}  \cHom_A( \CoS_A^j((\Omega^1_{A/R}\{m\})_{[-n-1]}),A);
\]
this has the properties that $\widehat{\Pol}(A,n,m)= \Lim_i \widehat{\Pol}(A,n,m)/F^i$, with $[F^i,F^j] \subset F^{i+j-1}$, $\delta F^i \subset F^i$, and that $F^i F^j \subset F^{i+j}$.
\end{definition}

Observe that this filtration makes $F^2\widehat{\Pol}(A,n,m)^{[n+1]}\{m\}$ into a $\bG_m$-equivariant pro-nilpotent DGLA, and hence makes $ \cW_mF^2\widehat{\Pol}(A,n,m)^{[n+1]}$ into a pro-nilpotent DGLA.

\subsubsection{Poisson structures}

\begin{definition}\label{mcPLdef}
 Given a   DGLA $L$, define the the Maurer--Cartan set by 
\[
\mc(L):= \{\omega \in  L^{1}\ \,|\, \delta\omega + \half[\omega,\omega]=0 \in  \bigoplus_n L^{2}\}.
\]

Following \cite{hinstack}, define the Maurer--Cartan space $\mmc(L)$ (a simplicial set) of a nilpotent  DGLA $L$ by
\[
 \mmc(L)_n:= \mc(L\ten_{\Q} \Omega^{\bt}(\Delta^n)),
\]
where 
\[
\Omega^{\bt}(\Delta^n)=\Q[t_0, t_1, \ldots, t_n,\delta t_0, \delta t_1, \ldots, \delta t_n ]/(\sum t_i -1, \sum \delta t_i)
\]
is the commutative dg algebra of de Rham polynomial forms on the $n$-simplex, with the $t_i$ of degree $0$.
\end{definition}
%

\begin{definition}\label{poissdef}
If $A$ is quasi-free, define an  $n$-shifted Poisson structure of weight $m$ on $A$ to be an element of
\[
 \mc( \cW_mF^2 \widehat{\Pol}(A,n,m)^{[n+1]}), 
\]
 and the space $\cP(A,n,m)$ of   $n$-shifted Poisson structures of weight $m$ on $A$ to be given by the simplicial 
set
\[
 \cP(A,n,m):= \Lim_p \mmc( \cW_mF^2 \widehat{\Pol}(A,n,m)^{[n+1]}/F^{p+2}).
\]

\end{definition}


\begin{remark}
Observe that elements of $\cP_0(A,n,m)= \mc(\cW_mF^2 \widehat{\Pol}(A,n,m)^{[n+1]})$ consist of infinite sums $\pi = \sum_{i \ge 2}\pi_i$ with $\bG_m$-equivariant maps
\[
 \pi_i \co \CoS_A^i((\Omega^1_{A/R})_{[-n-1]}\{m\}) \to A_{[-n-2]}\{m\}
\]
 satisfying $\delta(\pi_i) + \half \sum_{j+k=i+1} [\pi_j,\pi_k]=0$. This is precisely the condition which ensures that $\pi$ defines an $L_{\infty}$-algebra structure on $A_{[-n]}\{m\}$. 

\end{remark}

\begin{definition}
We say that an  $n$-shifted Poisson structure $\pi = \sum_{i \ge 2}\pi_i $ of weight $m$ on $A$ is non-degenerate if $\pi_2 \co \CoS_A^2((\Omega^1_{A}\{m\})_{[-n-1]}) \to A_{[-n-2]}\{m\}$ induces a quasi-isomorphism 
\[
\pi_2^{\sharp}\co  (\Omega^1_{A})_{[-n]}\{m\} \to \cHom_A(\Omega^1_{A},A).
\]

Define $\cP(A,n,m)^{\nondeg}\subset \cP(A,n,m)$ to consist of non-degenerate elements --- this is a union of path-components.
\end{definition}

\begin{remark}
 Beware that this non-degeneracy condition can only be satisfied if $n \le 0$. The same phenomenon arises for symplectic structures, with derived dagger affinoids (and more generally derived dagger DM stacks) only capable of carrying non-positively shifted structures. For derived Artin stacks, the formulation of shifted Poisson structures is more subtle, allowing for positive shifts --- see \S \ref{hgsSpsn}.
\end{remark}

\subsection{Comparison of weighted Poisson and symplectic structures}

%
%
%
%
%


\begin{theorem}\label{compatthmLie}
For a quasi-free object $A$ of the pro-category $\pro(dg_+\Affd\Alg_K^{\bG_m,\dagger,\loc})$, there are canonical weak equivalences
\[
 \cW_m\Sp(A,n) \simeq \cP(A,n,m)^{\nondeg}
\]
of simplicial sets. 

In particular,  the sets of equivalence classes of $n$-shifted symplectic structures of weight $m$ and of  non-degenerate $n$-shifted Poisson structures of weight $m$ $A$ are isomorphic. 
 \end{theorem}
\begin{proof}
The proof of \cite[Corollary \ref{poisson-compatcor2}]{poisson} adapts, \emph{mutatis mutandis}. We now summarise the main steps in this setting.

Each Poisson structure $\pi \in \cP(A,n,m)$  gives rise to a Poisson cohomology complex
\[
T_{\pi} \widehat{\Pol}(A,n,m),
\]
defined as the $\bG_m$-equivariant cochain complex  given by the derivation  $\delta+[\pi,-]$ acting on $\widehat{\Pol}(A,n,m) $. There is also a canonical element $\sigma(\pi) \in\z^{n+2} \cW_mT_{\pi}\widehat{\Pol}(A,n)$  given by 
\[
 \sigma(\pi)= \sum_{i \ge 2}  (i-1)\pi_i.
\]

The key construction is then given by the compatibility map  
\begin{align*}
 \mu(-,\pi) \co \DR(A) &\to T_{\pi} \widehat{\Pol}_{\pi}(X,n,m) \\
a df_1 \wedge \ldots df_p &\mapsto a[\pi,f_1]\ldots [\pi,f_p]
\end{align*}
of filtered $\bG_m$-equivariant cochain complexes. When $\pi$ is non-degenerate, this map is necessarily a $\bG_m$-equivariant quasi-isomorphism, and the symplectic structure associated to $\pi$ is given by 
\[
 \mu(\pi,-)^{-1}\sigma(\pi) \in \H^{n+2}F^2\cW_m\DR(A).
\]
By analogy with \cite{KhudaverdianVoronov}, we can regard  the inverse map $\mu(\pi,-)^{-1}$ as  a homotopy Legendre transform. 

Establishing that this gives an equivalence between symplectic and Poisson structures relies on obstruction theory associated to filtered DGLAs, building the Poisson form $\pi=\pi_2+\pi_3 + \ldots$ inductively from the  symplectic form $\omega=\omega_2+\omega_3 + \ldots$ by solving the equation $\mu(\omega, \pi)\simeq \sigma(\pi)$ up to coherent homotopy; for a readable summary of the argument from \cite{poisson}, see \cite[\S 2.5]{safronovPoissonLectures}.
\end{proof}

\begin{remark}[$\sP$-shifted Poisson structures]\label{PshiftedPoissonrmk}
The $\sP$-shifted symplectic structures of \cite{BouazizGrojnowski} are much more general than we consider here, since they take $\sP$ to be a shifted line bundle with a flat connection, then define shifted symplectic structures in terms of the de Rham complex of $\sP$. Under additional finiteness conditions, our weighted shifted symplectic structures on $X$ fit into this picture by looking at $\sP$-shifted  symplectic structures relative to the morphism $f \co [X/\bG_m] \to B\bG_m$, taking $\sP= f^*\sO(m)[n]$.

Our results so far easily adapt to $\sP$-shifted symplectic structures whenever $\sP$ is the pullback of a shifted line bundle on the base, but here seems a suitable place to indicate how to formulate $\sP$-shifted symplectic structures for more general $\sP$.
We work with the filtered DGLA
\[
\widehat{\Pol}(A,\sP):=\prod_{p} \HHom_A(\CoS_A^p(\Omega^1_{A}\ten_A\sP),\sP),
\]
with Lie bracket given in terms of the connection $\nabla \co \sP \to \Omega^1_{A}\ten_A\sP$  by $[f,g]= f\circ (\nabla \circ g) \mp g \circ (\nabla \circ f)$, where the first $\circ$ is interpreted as a sum over all substitutions. This is a module over the graded algebra $ \prod_{p} \HHom_A(\CoS_A^p(\Omega^1_{A}\ten_A\sP),A)$, and the bracket is a biderivation with respect to this multiplication.

The generalisation of Theorem \ref{compatthmLie} then holds to give
\[
 \mmc(\DR(\sP))^{\nondeg} \simeq \mmc(\widehat{\Pol}(A,\sP))^{\nondeg}
\]
so we regard Maurer--Cartan elements $\pi \in F^2 \widehat{\Pol}(A,\sP)$ as $\sP$-shifted Poisson structures.
In particular a $\sP$-shifted Poisson structure $\pi$  gives rise to an element $\sigma(\pi)\in T_{\pi}F^2 \widehat{\Pol}(A,\sP)$ and a filtered compatibility map
\[
 \mu(-,\pi) \co \DR(\sP) \to T_{\pi}\widehat{\Pol}(A,\sP), 
\]
analogous to the map $\cW_m\DR(A) \to T_{\pi} \cW_m\widehat{\Pol}_{\pi}(X,n,m)$; these give rise to the comparison. 

Finally, composition with the connection $\nabla \co \sP \to \Omega^1_{A}\ten_A\sP$ allows us to interpret elements of $\HHom_A(\CoS_A^j(\Omega^1_{A}\ten_A\sP),\sP)$ as $j$-ary operations on $\sP$. Thus a $\sP$-shifted Poisson structure is the same as an $L_{\infty}$-algebra structure $\{[-]_i\}$ on $\sP$ for which each operation $[-]_i$ is a $\nabla$-multiderivation in the sense that each operation  $[p_1,p_2, \ldots, p_{i-1},-]_i \co \sP \to \sP$ acts as $v \lrcorner \nabla$ for some tangent vector $v \in T_A$.  

\end{remark}

\subsection{Co-isotropic structures}\label{coisosn}
In the algebraic setting of \cite{MelaniSafronovI}, 
an $n$-shifted co-isotropic structure on a morphism $f \co A \to B$ is defined to consist of an $n$-shifted Poisson structure on $A$, an $(n-1)$-shifted Poisson  structure $\pi$ on $B$ and a lift of $f$ to a strong homotopy $P_{n+1}$-algebra morphism $A \to (\widehat{\Pol}(B,n-1),\delta +[\pi,-])$ to the complex of polyvectors with differential twisted by $[\pi,-]$. There is an interpretation in terms of additivity showing that the latter data are equivalent to a strongly homotopy associative action of $A$ on the $P_n$-algebra $B$.
 
Unfortunately, the usual formulation of s.h. $P_{n+1}$-algebra morphisms involves cofibrant replacement of $A$ as a $P_{n+1}$-algebra, which does not have a natural analogue in analytic settings. However,  suitable spaces of morphisms should be constructible by applying  operadic Koszul duality and working with the coloured operad of polydifferential operators as described in \cite[Remark \ref{DQpoisson-EFCcoisormk}]{DQpoisson},
incorporating weights in the weighted case.
 
Given such a setup, the equivalence between (weighted) shifted Lagrangians and non-degenerate (weighted) shifted co-isotropic structures established in 
\cite{MelaniSafronovII} should have a fairly straightforward analytic analogue, phrased in terms of structures on (weighted) EFC-algebras. 

%

\subsection{Global Poisson structures}

Functoriality for shifted Poisson structures is fairly subtle, but for homotopy formally  \'etale morphisms $A \to B$, the constructions of \cite[\S \ref{poisson-DMsn}]{poisson} adapt to our setting. As in \cite[\S \ref{poisson-DMdiagramsn}]{poisson}, 
we have a natural notion of a space $ \cP(A \to B,n,m)$ of $n$-shifted Poisson structures of weight $m$ on the diagram $A \to B$, with a natural equivalence $ \cP(A \to B,n,m) \to\cP(A,n,m) $ and a natural map $\cP(A \to B,n,m) \to\cP(B,n,m) $; this leads to an $\infty$-functor on the category of homotopy formally \'etale morphisms of formal weighted dg dagger affinoids. 

Passing to homotopy limits as in \cite[Definition \ref{poisson-inftyFXdef}]{poisson}, we can then define the space of $n$-shifted symplectic structures on a weighted formal dagger dg space (or DM $\infty$-stack) by setting
\[
 \cP(X,n,m):= \oR\Gamma((\pi^0X^{\bG_m})_{\et}, \cP(\tilde{\sO}_X,n,m)).
\]

The equivalence generalising Theorem \ref{compatthmLie} is then sufficiently canonical to give us an equivalence
\[
 \Sp(X,n,m) \simeq \cP(X,n,m)^{\nondeg}.
\]

For Artin stacks, we encounter the obstacle that Poisson structures are not functorial with respect to smooth morphisms, which we will address in \S \ref{hgsSpsn}.

\begin{examples}\label{wPoissex}
Whenever $Y$ is a $K$-dagger space (or even a $K$-dagger DM $\infty$-stack),  
 Theorem \ref{compatthmLie}  gives us $r$-shifted Poisson structures of weight $m$ associated to each  $r$-shifted symplectic structure of weight $m$ on the functors
 \[
 \tilde{Y} \co A \mapsto Y(\H_0A/\bG_m)\by^h_{F(X_{\pro\et},\H_0A/\bG_m)}F(X_{\pro\et},\Zl(1),A) 
 \]
in Examples \ref{wSpex}.  

Moreover, the comparison anticipated in \S \ref{coisosn} would give  $(r-1)$-shifted Poisson structures of weight $m$ on $\tilde{W}$ associated to each  $r$-shifted Lagrangian structure of weight $m$
on $F( U_{\pro\et},\bE,-)^{wd}$
in Examples \ref{wSpexLag}, whenever $W$ is a  $K$-dagger DM $\infty$-stack over $\pi^0F(X_{\pro\et},-)$.

These results will also extend to cases where $Y$ is a $K$-dagger Artin $\infty$-stack, relying on the formulation of shifted Poisson structures on derived Artin stacks in \S \ref{hgsSpsn} below, as in Examples \ref{wPoissexArtin}.
\end{examples}

\subsection{Quantisations} \label{quantsn}

As detailed in \cite{DQpoisson} for the $0$-shifted case, and outlined in general in  \cite[\S 4.4]{DStein} and \cite[Remarks \ref{DQDG-htpyfdrmkd1} and \ref{DQDG-htpyfdrmkd2}]{DQDG}, we can formulate shifted quantisations in  derived analytic settings using polydifferential operators, and algebraic quantisation results will then all adapt.

To quantise  Poisson structures of weight $m$, the results also adapt, with the same proofs. The only modification we have to make is to work in a  setting  with an additional grading and to set the formal variable $\hbar$ to have weight $m$.    For positively shifted structures, quantisations take the form of $E_{n+1}$-algebra deformations, whose existence can be inferred directly from Kontsevich formality. Our examples  of interest tend to carry non-positively shifted structures, which are more subtle:

\begin{itemize}[itemsep=5pt, parsep=0pt]
 \item Given a $0$-shifted Poisson structure of weight $m$ on $Z$, the proofs of \cite{DQnonneg,DQpoisson} (the latter includes   EFC  settings, so we need only incorporate weights) adapt to give us a quantisation in the form of a curved almost commutative $\bG_m$-equivariant $A_{\infty}$-deformation $\tilde{\sO}_Z$ of $\{\sO_Z[\hbar]/\hbar^k\}_k$ over $\{K[\hbar]/\hbar^k\}_k$, with $\hbar$ of weight $m$. The curvature can be interpreted as giving an algebroid quantisation, and leads to a deformation (over $\{K[\hbar]/\hbar^k\}_k$) of the category of $\bG_m$-equivariant line bundles on $X$.
 
 \item Given a $(-1)$-shifted symplectic structure of weight $m$ on $W$,   the proof of \cite{DQvanish} will adapt to give us a quantisation of the square root $\sL$ of the canonical bundle on $W$ (or of any line bundle with a right $\sD$-module structure on its square)
  in the form of a 
  $\bG_m$-equivariant differential operator $\Delta$ on $\{\sL[\hbar]/\hbar^k\}_k$ over $\{K[\hbar]/\hbar^k\}_k$, with $\hbar$ of weight $m$. This operator has constraints on the orders of its coefficients, and corresponds to a homotopy $BV$-algebra structure when $\sL=\sO_W$ and $\Delta(1)=0$. 
    On inverting $\hbar$, $\Delta$ gives a complex related to vanishing cycles as in Example \ref{vanishex} below. 
  Euler characteristics of such complexes are used in complex-analytic settings to recover the Behrend function used in enumerative questions such as Donaldson--Thomas theory. 
  
  \item Given a $0$-shifted Lagrangian structure of weight $m$ on $W$ over $Z$, the previous two examples will combine compatibly as in \cite{DQLag}, with the $\sO_Z$-module structure on $\sL$ extending to give an  $\tilde{\sO}_Z$-module structure on $(\{\sL[\hbar]/\hbar^k\}_k, \delta +\Delta)$. On taking the limit over $k$ and inverting $\hbar$, the resulting structures resemble objects of an algebraic Fukaya category.

  \item If the structure sheaf carries a right $\sD$-module structure, then there is also a notion of quantisation for $(-2)$-shifted symplectic structures, given by solutions of a quantum master equation as in \cite{DQ-2}. The only modification necessary in the weighted case is to stipulate that the solution has weight $m$ with respect to the $\bG_m$-action.  

  \item Exact shifted symplectic structures correspond to shifted Poisson structures with formal derivation as in  \cite{DQexact}. The formal derivations quantise to $\hbar\pd_{\hbar}$ derivations, as described there.
  \end{itemize}

\begin{example}\label{vanishex} 
 The prototypical example of a $(-1)$-shifted symplectic structure is given by a derived critical locus, and there is a similar construction for weighted structures  
 as a special case of Remark \ref{locformrmk}. Let  $Y=(Y^{\bG_m},\sO_Y)$ be a formally smooth weighted formal $K$-dagger space, and $f \in \Gamma(Y,\cW_m\sO_Y)$ a function of weight $m$. For the tangent sheaf $\sT_Y$ on $Y$, the derived critical locus $W$ of $f$ is then a weighted formal  $K$-dagger dg space with $\pi^0W^{\bG_m}$ the vanishing locus of $df$ on $Y^{\bG_m}$, and $\sO_W=( \Symm_{\sO_Y}(\sT_Y\{-m\}_{[-1]}),\delta)$, where the differential $\delta$ given on generators as contraction with $df$ (a map  $\sT_Y\{-m\}\to \sO_Y$).
 
 The $(-1)$-shifted symplectic structure of weight $m$ on $W$ is then given by the closed $2$-form arising from the identification of $\sT_Y$ with the dual of $\Omega^1_{Y}$. If $\sO_Y$ has local co-ordinates $y_i$, then $\sO_W$ has co-ordinates $y_i,\eta_i=\pd_{y_i}$, and the $2$-form on $W$ is $\omega =\sum_i dy_i \wedge d\eta_i \in \ker d \cap \ker \delta$. If $y_i \in \cW_{n_i}\sO_Y$, then $\eta_i \in \cW_{m-n_i}\sO_W$,  
 so 
 \[
  \omega \in \cW_m(\Omega^2_W\cap \ker \delta \cap \ker d)_1 \subset \cW_m\z^1F^2\DR(W). 
 \]
Since this is non-degenerate, it is thus a $(-1)$-shifted symplectic structure of weight $m$.
 
Now, a square root $\sL$ of the dualising bundle on $W$ is simply given by the pullback of $\Omega^{\dim Y}_Y$ to $W$, which is the complex
\[
 (\Omega^{\#}_Y\{m\#\}, df\wedge{})\{-m\dim Y\}[\dim Y].
\]
Reasoning as in \cite[\S 4.2
]{DQvanish}, a differential operator quantising  $\{\sL[\hbar]/\hbar^k\}_k$ is just the rescaled de Rham differential $\hbar d$, giving us the twisted de Rham complex
\[
 (\Omega^{\#}_Y\{m\#\}[\hbar]/\hbar^k, \hbar d +df\wedge{})\{-m\dim Y\}[\dim Y]. 
\]

When $m\ne 0$, the $\bG_m$-equivariant limit over all $k$ is
\[
   (\Omega^{\#}_Y\{m\#\}[\hbar], \hbar d +df\wedge{})\{-m\dim Y\}[\dim Y] \cong  \hbar^{\dim Y}(\hbar^{-\#}\Omega^{\#}_Y[\hbar],d+ \hbar^{-1}df \wedge{})[\dim Y].
\]
In particular, if $m=\pm 1$ then
on inverting $\hbar$ we just have a copy of the twisted de Rham complex $(\Omega^{\#}_Y, d+df\wedge{})$ in every weight. 
Because $f$ has non-zero weight, its exponential does not lie in $\sO_Y$, so this is  not isomorphic to the de Rham complex of $Y$.
It behaves more like an algebraic twisted de Rham complex, computing rapid decay cohomology.  

When $m=0$, the complex given by inverting $\hbar$ is $(\Omega^{\#}_Y(\!(\hbar)\!), d+\hbar^{-1}df\wedge{}) $, which computes vanishing cycles.

For the exact structure on $W$ given by $\sum_i dy_i \wedge \eta_i - f$, \cite[Proposition \ref{DQexact-DCritprop}]{DQexact} shows that the essentially unique self-dual quantisation corresponds to the complex above equipped with the $\hbar\pd_{\hbar}$-connection which acts  on $\Omega^p_Y$ as multiplication by  $(\dim Y/2 -p) -\hbar^{-1}f$. On the twisted de Rham complexes above, that corresponds to the operator  $\hbar\pd_{\hbar}- \hbar^{-1}f \pm w- \dim Y/2$ in weight $w$ if $m =\pm 1$, and the operator $\hbar\pd_{\hbar}- \hbar^{-1}f - \dim Y/2$ if $m=0$.
\end{example}

\begin{remarks}
 There are similar approaches to quantisation for shifted weighted symplectic structures on the derived Artin analytic stacks of the next section. 
  We can use the toy example of derived moduli of $\bG_m$-torsors as described in Examples \ref{wBGmXex}, \ref{wBGmXex2}, \ref{wBGmXex3} as a test case for the programme outlined in 
 \cite[\S 10]{KimArithGauge}. 
 
 For derived $0$-shifted cotangent stacks, there is a canonical  quantisation given by rings of differential operators. The (unweighted)  moduli stack of $\bG_m$-representations of a local Galois group is just the product of a smooth one-dimensional analytic space with $B\bG_m$, so we can describe the dg category of  perfect complexes  over the  quantisation of the  weighted derived moduli stack by substituting in  
 \cite[Definition 3.12 and Example 3.15]{DerCotMain}, but  taking $\hbar$ to have weight $1$ so power series do not appear.
 
 For the derived intersection of  \S \ref{wSpexLagcrissn}, corresponding to the major focus of  \cite[\S 10]{KimArithGauge},  Example \ref{wBGmXex3} 
 gives a description as a weighted $(-1)$-shifted derived cotangent stack, so we can describe a quantisation as in Example \ref{vanishex} (with $f=0$).

 However,   this has the complication that there are infinitely many terms in the product of stacks, so na\"ively inverting $\hbar$ will not simply give us the  de Rham complex of the unweighted moduli stack shifted by its virtual dimension; although the other terms in the product are derived affine spaces, so have trivial de Rham cohomology, their combined contribution will annihilate  this localised  complex by shifting it to infinite degree.  Moreover, because $\sum_{r>0} \dim  \H^1_{\pro\et}(\cO_{K,S}, \Ql(r))$ is infinite,  the usual construction  does not yield a line bundle $\sL$ with a right $\sD$-module structure on $\sL^{\ten 2}$, as needed for self-duality theory, and
 the simplified description of the  self-dual $\hbar\pd_{\hbar}$-connection in Example \ref{vanishex} also fails. These suggest that extracting useful invariants  will be a delicate process, since in general it will not be possible simply to quotient out the higher weights. 


 \end{remarks}

\subsection{Poisson structures for derived Artin stacks}\label{hgsSpsn}

\subsubsection{Stacky derived affines}\label{stackysn}

The lack of functoriality of tangent spaces and Poisson structures with respect to smooth morphisms makes their definition for derived Artin stacks fairly subtle. Instead of working with functors on CDGAs, the solution is to work with algebras in double complexes, where the chain direction encodes the derived structure and the cochain direction encodes the stacky structure. For details, see  \cite{smallet2} (and in particular  \S \ref{smallet2-Fermatsn}). We can then incorporate weights in much the same way as the incorporation of $\Z/2\Z$-gradings in \cite{DQDG}.

The following is adapted from \cite[Definition 3.2]{poisson}:
\begin{definition}
We define a stacky  quasi-dagger  dg algebra  to be a chain cochain complex $A= A^{\ge 0}_{\ge 0}$ of $K$-vector spaces equipped with a commutative product $A\ten A \to A$ and unit $K \to A$, such that  $A^0_0$ is a quasi-dagger algebra, and the $A^0_0$-modules $A^i_j$ are all finite. 

A morphism $A \to B$ of  stacky  quasi-dagger  dg algebra is then said to be a weak equivalence if the morphisms $A^i \to B^i$ are all quasi-isomorphisms.

There are obvious formal and weighted generalisations, which we will not write down.
\end{definition}

On double complexes $V^{\bt}_{\bt}$ combining both chain and cochain gradings, we denote the chain and cochain differentials  by $\delta$ and $\pd$ respectively,  regarding the cochain differential $\pd$ as stacky structure and the chain differential $\delta$ as derived structure.

\begin{definition}
 Given a chain cochain complex $V$, define the cochain complex $\hat{\Tot} V \subset \Tot^{\Pi}V$ as a subset of the product total complex by
\[
(\hat{\Tot} V)^m := (\bigoplus_{i < 0} V^i_{i-m}) \oplus (\prod_{i\ge 0}   V^i_{i-m})
\]
with differential $\pd \pm \delta$. This is sometimes referred to as the Tate realisation.
\end{definition}

In order to pass from double complexes to complexes in a fashion which behaves well with respect to weak equivalences and tensor operations, we use the following, which appear as \cite[Definitions 3.7 and 3.8]{poisson}. 
\begin{definition}
 Given a chain cochain complex $V$, define the cochain complex $\hat{\Tot} V \subset \Tot^{\Pi}V$ by
\[
(\hat{\Tot} V)^m := (\bigoplus_{i < 0} V^i_{i-m}) \oplus (\prod_{i\ge 0}   V^i_{i-m})
\]
with differential $\pd \pm \delta$.
\end{definition}

\begin{definition}
 Given $A$-modules $M,N$ in chain cochain complexes, we write $\cHom_A(M,N)$ for the  internal $\Hom$ object in the category of $A$-modules in double complexes. Explicitly,
\[
 \cHom_A(M,N)^i_j=  \Hom_{A^{\#}_{\#}}(M^{\#}_{\#},N^{\#[i]}_{\#[j]}),
\]
with differentials  $\pd f:= \pd_N \circ f \pm f \circ \pd_M$ and  $\delta f:= \delta_N \circ f \pm f \circ \delta_M$,
where $V^{\#}_{\#}$ denotes the bigraded vector space underlying a chain cochain complex $V$. 

We then define the  $\Hom$ complex $\hat{\HHom}_A(M,N)$ by
\[
 \hat{\HHom}_A(M,N):= \hat{\Tot} \cHom_A(M,N).
\]
\end{definition}

\subsubsection{Poisson structures}\label{Artinpoisssn}

We can then extend the definition of shifted Poisson structures to stacky   quasi-dagger  dg algebras, by using $\hat{\HHom}$ in place of $\HHom$.  
Unwinding the definitions, this means  that a Poisson structure on a stacky dg algebra $A$ gives rise to a sequence $\pi_2, \pi_3, \ldots$ of multiderivations on the product total complex $\hat{\Tot} A =\Tot^{\Pi}A$, defining a shifted $L_{\infty}$-algebra structure. However, since the multiderivations lie in the spaces defined using $\hat{\HHom}$, they come with boundedness restrictions from the original cochain direction: if we filter $\hat{\Tot} A$ by setting  $\Fil^p \hat{\Tot}A:= \Tot^{\Pi} A^{\ge p}$, then each component $\pi_k$  must be $\Fil$-bounded in the sense that for some integer $r$,  each $\Fil^p$ is mapped to $\Fil^{p+r}$. 

There is a notion of homotopy \'etale for  morphisms $A \to B$ of stacky dg algebras, which essentially amounts to saying that $\hat{\Tot}(\oL\Omega^1_A\ten^{\oL}_AB^0) \to \hat{\Tot}(\oL\Omega^1_B\ten^{\oL}_AB^0)$ is a quasi-isomorphism. This gives sufficient flexibility to ensure functoriality of shifted Poisson structures with respect to such morphisms, as in \cite[\S \ref{poisson-Artindiagramsn}]{poisson}.

In particular, the analogue of Theorem \ref{compatthmLie} still holds, giving a correspondence between $n$-shifted symplectic structures of weight $m$ and non-degenerate $n$-shifted Poisson structures of weight $m$ on formal weighted stacky   localised dagger  dg algebras. The comparison arguments outlined in \S \ref{coisosn} between $n$-shifted Lagrangian structures of weight $m$ and non-degenerate $n$-shifted co-isotropic structures of weight $m$ will also extend to stacky algebras, by incorporating weights in the arguments of \cite{MelaniSafronovII} and \cite[\S \ref{DQpoisson-coisosn1}]{DQpoisson}.

\subsubsection{Denormalisation}

The denormalisation functor $D$ from cochain complexes to cosimplicial vector spaces combines with the Eilenberg--Zilber shuffle product to give a functor from stacky dg algebras to cosimplicial dg algebras; for explicit formulae, see \cite[\S \ref{smallet2-denormsn}]{smallet2}. The iterated  codegeneracy maps $D^nA \to A^0$ are  $n$-nilpotent, which in particular implies that this functor  sends stacky  (quasi-)dagger  dg algebras $A$ to cosimplicial   (quasi-)dagger  dg algebras $DA$. 

\begin{definition}\label{Dlowerdef}
 Given a functor $F$ from   (formal weighted)  (quasi-)dagger  dg algebras   simplicial sets, we define a functor $D_*F$ on   (formal weighted)   stacky  (quasi-)dagger  dg algebras as the homotopy limit
\[
 D_*F(B):= \ho\Lim_{n \in \Delta} F(D^nB).
\]
\end{definition}

Thus $D_*$ naturally extends all of our moduli functors to give functors on suitable stacky dg algebras. The real power of this construction is that derived Artin stacks, and more generally homogeneous functors $F$, admit  homotopy \'etale atlases $\oR \Spec A \to D_*F$ by stacky affine objects; see \cite[\S \ref{poisson-bidescentsn}]{poisson} for derived Artin stacks and \cite[\S \ref{NCpoisson-ArtinPoisssn}]{NCpoisson} for the generalisation to homogeneous functors (the arguments there are phrased in the non-commutative setting, but other settings work in the same way, and in fact more easily).

The following is adapted from \cite[Definition \ref{NCstacks-Frigdef}]{NCstacks}:
\begin{definition}\label{Frigdefhere}
 Given a (weighted)   stacky  (quasi-)dagger  dg algebra $B \in DG^+dg_+\Alg(R)$  for which the chain complexes
$
( \oL\Omega^1_B\ten^{\oL}_{B}B^0)^i
$
are acyclic for all $i > q$, and  a simplicial functor $F$ on  (weighted)   (quasi-)dagger  dg algebra  which is homogeneous with a cotangent complex $\bL^{F,x}$  at a point $x \in F(B^0)$, we say that a point $y \in D_*F(B)$ lifting $x \in F(B^0)$ is \emph{rigid} if the induced morphism
 \[
  \bL^{F,x}\to \Tot \sigma^{\le q} (\oL\Omega^1_B\ten^{\oL}_{B}B^0)
 \]
is a quasi-isomorphism of $B^0$-modules. We denote by $(D_*F)_{\rig}(B) \subset D_*F(B)$ the space of rigid points (a union of path components).
\end{definition}

In other words, a point $y \in (D_*F)_{\rig}(B)$ corresponds to a homotopy \'etale morphism $\oR\Spec B \to D_*F$. The reason we  think of the pair $(B,y)$ as being rigid is that
  it does not deform: for any nilpotent surjection $e \co C \to B$ with a point $z \in D_*F(C)$ lifting $y$, the map $e$  has an essentially unique section $s$ with $s(y) \simeq z$. 

\subsubsection{Global Poisson structures}

On any homogeneous functor $F$, we can now define the space of $n$-shifted Poisson structures of weight $m$ by
\begin{align*}
\cP(F,n,m):= \oR\map((D_*F)_{\rig}, \cP(-,n,m)),
 \end{align*} 
where the mapping space is taken in the category of simplicial homotopy preserving functors on a  category of (weighted) stacky  (quasi-)dagger  dg algebras and homotopy \'etale morphisms.

The arguments of \cite{poisson,smallet2,NCpoisson} ensure this is consistent with our earlier definitions when $\pi^0F^{\bG_m}$ is a dagger analytic space or even a DM $\infty$-stack.
 
 The generalisation of Theorem \ref{compatthmLie} to stacky algebras gives us an equivalence
 \[
  \cP(F,n,m)^{\nondeg} \simeq \oR\map((D_*F)_{\rig}, \cW_m\Sp(-,n))
 \]
between non-degenerate Poisson structures and symplectic structures. Moreover, there is a natural map
\[
 \oR\map(F, \cW_m\PreSp(-,n))\to \oR\map((D_*F)_{\rig}, \cW_m\PreSp(-,n)),
\]
where the first mapping space is taken in the $\infty$-category of functors on localised weighted dagger dg algebras. Thus pre-symplectic structures on $F$ give rise to pre-symplectic structures on $D_*F$ for all functors $F$. In particular, we can apply to all of our examples from \S \ref{proetsn} without having to impose any representability conditions.

 There will be a similar global definition of co-isotropic structures, and global comparisons between Lagrangians and non-degenerate co-isotropic structures.
 
 \subsubsection{Global Poisson structures associated to pro-\'etale sheaves}
 
 Combining the Poisson/symplectic correspondence for stacky algebras outlined above with Corollary \ref{wtracecor2a} gives a
 non-degenerate $(n-d)$-shifted Poisson structure of weight $m$  on   $ F(X_{\pro\et},\bE,-)^{wd}$ whenever $F$ is an   $n$-shifted symplectic  derived Artin $\infty$-stack, and $X$  a topologically Noetherian scheme with dualising complex $\bE^{\ten m}[d]$.

\begin{examples}\label{wPoissexArtin}  
Generalising Examples \ref{wPoissex},
 Theorem \ref{compatthmLie} now gives us $r$-shifted Poisson structures of weight $m$ associated to each  $r$-shifted symplectic structure of weight $m$ on
\[
 F(X_{\pro\et},\Zl(1),-)^{wd} 
\]
in Examples \ref{wSpex} and \S\S \ref{wSpexLagcrissn}--\ref{highercrissn}.

Moreover, the comparison outlined in \S \ref{coisosn} would give  $(r-1)$-shifted Poisson structures of weight $m$  associated to each  $r$-shifted Lagrangian structure of weight $m$
on $F( U_{\pro\et},\bE,-)^{wd}$
in Examples \ref{wSpexLag} and \S\S \ref{wSpexLagcrissn}--\ref{highercrissn}.

Note that in this form, these results do not assume that the functors $F^{wd}$ are representable, because the use of stacky dg algebras allows us to formulate Poisson structures for any homogeneous functor.
\end{examples}


\bibliographystyle{alphanum}
\bibliography{references.bib}

\newcommand{\etalchar}[1]{$^{#1}$}
 \newcommand{\noop}[1]{} \def\cprime{$'$}
\begin{thebibliography}{BBBBJ}

\bibitem[AGK{\etalchar{+}}]{ArinkinGaitsgoryKazhdanRaskinRozenblyumYarshavsky}
D.~Arinkin, D.~Gaitsgory, D.~Kazhdan, S.~Raskin, N.~Rozenblyum, and
  Y.~Varshavsky.
\newblock The stack of local systems with restricted variation and geometric
  langlands theory with nilpotent singular support.
\newblock arXiv:2010.01906 [math.AG], 2022.

\bibitem[Ant1]{antonioPadic}
Jorge Ant\'onio.
\newblock Moduli of $p$-adic representations of a profinite group.
\newblock arXiv:1709.04275[math.AG], 2017.

\bibitem[Ant2]{antonioElladic}
Jorge Ant\'onio.
\newblock Moduli of $\ell$-adic pro-\'etale local systems for smooth non-proper
  schemes.
\newblock arXiv:1904.08001[math.AG], 2019.

\bibitem[Art]{Artin}
M.~Artin.
\newblock Versal deformations and algebraic stacks.
\newblock {\em Invent. Math.}, 27:165--189, 1974.

\bibitem[BBBBJ]{BBBJdarboux}
O.~Ben-Bassat, C.~Brav, V.~Bussi, and D.~Joyce.
\newblock A ``{D}arboux theorem'' for shifted symplectic structures on derived
  {A}rtin stacks, with applications.
\newblock {\em Geom. Topol.}, 19:1287--1359, 2015.
\newblock arXiv:1312.0090 [math.AG].

\bibitem[Ber1]{bergerBpairsPhiGamma}
Laurent Berger.
\newblock Construction de {$(\phi,\Gamma)$}-modules: repr\'esentations
  {$p$}-adiques et {$B$}-paires.
\newblock {\em Algebra Number Theory}, 2(1):91--120, 2008.

\bibitem[Ber2]{Be}
Pierre Berthelot.
\newblock {\em Cohomologie cristalline des sch{\'e}mas de caract{\'e}ristique
  $p>0$}.
\newblock Springer-Verlag, Berlin, 1974.
\newblock Lecture Notes in Mathematics, Vol. 407.

\bibitem[Bet]{bettsMotAnab}
L.~Alexander Betts.
\newblock The motivic anabelian geometry of local heights on abelian varieties.
\newblock arXiv:1706.04850, 2017.

\bibitem[BG]{BouazizGrojnowski}
E.~Bouaziz and I.~Grojnowski.
\newblock A {$d$}-shifted {D}arboux theorem.
\newblock arXiv:1309.2197v1 [math.AG], 2013.

\bibitem[BGR]{BoschGuentzerRemmertNonArchanalysis}
S~Bosch, U~G{\"u}ntzer, and R~Remmert.
\newblock {\em {Non-Archimedean} analysis}.
\newblock Grundlehren der mathematischen Wissenschaften. Springer, Berlin,
  Germany, 1984 edition, May 1984.

\bibitem[BGW]{BraunlingGroechenigWolfsonTate}
Oliver Braunling, Michael Groechenig, and Jesse Wolfson.
\newblock Tate objects in exact categories.
\newblock {\em Mosc. Math. J.}, 16(3):433--504, 2016.
\newblock With an appendix by Jan {\v{S}}{\v{t}}ov{\'{\i}}{\v{c}}ek and Jan
  Trlifaj.

\bibitem[BK]{BlochKatoTamagawa}
Spencer Bloch and Kazuya Kato.
\newblock {$L$}-functions and {T}amagawa numbers of motives.
\newblock In {\em The Grothendieck Festschrift}, pages 333--400. Springer,
  2007.

\bibitem[BPS]{DerCotMain}
Marco Benini, Jonathan~P. Pridham, and Alexander Schenkel.
\newblock Quantization of derived cotangent stacks and gauge theory on directed
  graphs.
\newblock {\em Advances in Theoretical and Mathematical Physics},
  27(5):1275–1332, 2023.
\newblock arXiv:2201.10225 [math-ph].

\bibitem[BS]{BhattScholzeProEtale}
Bhargav Bhatt and Peter Scholze.
\newblock The pro-{\'e}tale topology for schemes.
\newblock {\em Ast{\'e}risque}, 369:99--201, 2015.

\bibitem[BZSV]{BenZviSakellaridisVenkatesh}
David Ben-Zvi, Yiannis Sakellaridis, and Akshay Venkatesh.
\newblock Relative {L}anglands duality.
\newblock arXiv:2409.04677v1 [math.RT], 2024.

\bibitem[CFK]{Quot}
Ionu{\c{t}} Ciocan-Fontanine and Mikhail Kapranov.
\newblock Derived {Q}uot schemes.
\newblock {\em Ann. Sci. {\'E}cole Norm. Sup. (4)}, 34(3):403--440,
  2001\noop{1999}.
\newblock arXiv:math/9905174v2 [math.AG].

\bibitem[CR]{CarchediRoytenbergHomological}
D.~{Carchedi} and D.~{Roytenberg}.
\newblock {Homological Algebra for Superalgebras of Differentiable Functions}.
\newblock arXiv:1212.3745 [math.AG], 2012.

\bibitem[Del]{poids}
Pierre Deligne.
\newblock Poids dans la cohomologie des vari{\'e}t{\'e}s alg{\'e}briques.
\newblock In {\em Proceedings of the International Congress of Mathematicians
  (Vancouver, B. C., 1974), Vol. 1}, pages 79--85. Canad. Math. Congress,
  Montreal, Que., 1975.

\bibitem[DK1]{DKEquivsHtpyDiagrams}
W.~G. Dwyer and D.~M. Kan.
\newblock {VIII}. {E}quivalences between homotopy theories of diagrams.
\newblock In {\em Algebraic Topology and Algebraic K-Theory ({AM}-113)}, pages
  180--205. Princeton University Press, 1987.

\bibitem[DK2]{DKfunction}
W.G. Dwyer and D.M. Kan.
\newblock Function complexes in homotopical algebra.
\newblock {\em Topology}, 19(4):427--440, 1980.

\bibitem[EP]{2021lect}
J.~Eugster and J.~P. Pridham.
\newblock An introduction to derived (algebraic) geometry.
\newblock {\em Rend. Mat. Appl. (7)}, 45(1--2):1--109, 2024\noop{2021}.
\newblock arXiv:2109.14594v5.

\bibitem[Fal]{Hop}
Gerd Faltings.
\newblock Crystalline cohomology and {$p$}-adic {G}alois-representations.
\newblock In {\em Algebraic analysis, geometry, and number theory (Baltimore,
  MD, 1988)}, pages 25--80. Johns Hopkins Univ. Press, Baltimore, MD, 1989.

\bibitem[GK]{GrosseKloenne}
Elmar Grosse-Kl\"onne.
\newblock Rigid analytic spaces with overconvergent structure sheaf.
\newblock {\em J. Reine Angew. Math.}, 519:73--95, 2000.

\bibitem[GM]{GM}
William~M. Goldman and John~J. Millson.
\newblock The deformation theory of representations of fundamental groups of
  compact {K}{\"a}hler manifolds.
\newblock {\em Inst. Hautes {\'E}tudes Sci. Publ. Math.}, (67):43--96, 1988.

\bibitem[G{\"u}n]{guentzer}
Ulrich G{\"u}ntzer.
\newblock Modellringe in der nichtarchimedischen funktionentheorie.
\newblock {\em Indagationes Mathematicae (Proceedings)}, 70:334--342, 1967.

\bibitem[Hin]{hinstack}
Vladimir Hinich.
\newblock D{G} coalgebras as formal stacks.
\newblock {\em J. Pure Appl. Algebra}, 162(2-3):209--250, 2001\noop{1998}.
\newblock https://arxiv.org/abs/math/9812034.

\bibitem[Hov]{hovey}
Mark Hovey.
\newblock {\em Model categories}, volume~63 of {\em Mathematical Surveys and
  Monographs}.
\newblock American Mathematical Society, Providence, RI, 1999.

\bibitem[Ill1]{Bcrisillusie}
Luc Illusie.
\newblock Cohomologie de de {R}ham et cohomologie {\'e}tale {$p$}-adique
  (d'apr{\`e}s {G}. {F}altings, {J}.-{M}. {F}ontaine et al.).
\newblock {\em Ast{\'e}risque}, (189-190):Exp.\ No.\ 726, 325--374, 1990.
\newblock S{\'e}minaire Bourbaki, Vol. 1989/90.

\bibitem[Ill2]{illusieVanishing}
Luc Illusie.
\newblock Vanishing cycles over general bases after {P}. {D}eligne, {O}.
  {G}abber, {G}. {L}aumon and {F}. {O}rgogozo.
\newblock 2006.

\bibitem[Isa]{isaksenStrict}
Daniel~C. Isaksen.
\newblock Strict model structures for pro-categories.
\newblock In {\em Categorical decomposition techniques in algebraic topology
  ({I}sle of {S}kye, 2001)}, volume 215 of {\em Progr. Math.}, pages 179--198.
  Birkh\"auser, Basel, 2004.
\newblock arXiv:0108189[math.AT].

\bibitem[Kim]{KimArithGauge}
Minhyong Kim.
\newblock Arithmetic gauge theory: A brief introduction.
\newblock {\em Mod. Phys. Lett. A}, 33(29):1830012, September 2018.

\bibitem[Kon]{Kon}
Maxim Kontsevich.
\newblock Topics in algebra --- deformation theory.
\newblock Lecture Notes, available at
  \href{http://www.math.brown.edu/~abrmovic/kontsdef.ps}{http://www.math.brown.edu/$\sim$abrmovic/kontsdef.ps},
  1994.

\bibitem[KV]{KhudaverdianVoronov}
H.~M. Khudaverdian and Th.~Th. Voronov.
\newblock Higher {P}oisson brackets and differential forms.
\newblock In {\em Geometric methods in physics}, volume 1079 of {\em AIP Conf.
  Proc.}, pages 203--215. Amer. Inst. Phys., Melville, NY, 2008.
\newblock arXiv:0808.3406v2 [math-ph].

\bibitem[Lur1]{lurie}
J.~Lurie.
\newblock {\em Derived Algebraic Geometry}.
\newblock PhD thesis, M.I.T., 2004.
\newblock
  \href{http://hdl.handle.net/1721.1/30144}{http://hdl.handle.net/1721.1/30144}
  or
  \href{https://www.math.ias.edu/~lurie/papers/DAG.pdf}{https://www.math.ias.edu/$\sim$lurie/papers/DAG.pdf}.

\bibitem[Lur2]{lurieDAG5}
Jacob Lurie.
\newblock Derived algebraic geometry {V}: Structured spaces.
\newblock arXiv:0905.0459v1 [math.CT], 2009.

\bibitem[Man]{Man2}
Marco Manetti.
\newblock Extended deformation functors.
\newblock {\em Int. Math. Res. Not.}, (14):719--756, 2002\noop{1999}.
\newblock arXiv:math/9910071v2 [math.AG].

\bibitem[Mil1]{milneArithDuality}
J.~S. Milne.
\newblock {\em Arithmetic duality theorems}.
\newblock BookSurge, LLC, Charleston, SC, second edition, 2006.

\bibitem[Mil2]{Mi}
James~S. Milne.
\newblock {\em {\'E}tale cohomology}.
\newblock Princeton University Press, Princeton, N.J., 1980.

\bibitem[MS1]{MelaniSafronovI}
Valerio Melani and Pavel Safronov.
\newblock Derived coisotropic structures {I}: affine case.
\newblock {\em Selecta Math. (N.S.)}, 24(4):3061--3118, 2018.
\newblock arXiv:1608.01482 [math.AG].

\bibitem[MS2]{MelaniSafronovII}
Valerio Melani and Pavel Safronov.
\newblock Derived coisotropic structures {II}: stacks and quantization.
\newblock {\em Selecta Math. (N.S.)}, 24(4):3119--3173, 2018.
\newblock arXiv:1704.03201 [math.AG].

\bibitem[Nek1]{nekovarPadicHeight}
Jan Nekov{\'a}{\v{r}}.
\newblock On {$p$}-adic height pairings.
\newblock In {\em S\'eminaire de {T}h\'eorie des {N}ombres, {P}aris, 1990--91},
  volume 108 of {\em Progr. Math.}, pages 127--202. Birkh\"auser Boston,
  Boston, MA, 1993.

\bibitem[Nek2]{nekovarSelmerComplexes}
Jan Nekov{\'a}{\v{r}}.
\newblock Selmer complexes.
\newblock {\em Ast\'erisque}, (310):viii+559, 2006.

\bibitem[NSW]{NeukirchSchmidtWingberg}
J\"urgen Neukirch, Alexander Schmidt, and Kay Wingberg.
\newblock {\em Cohomology of number fields}, volume 323 of {\em Grundlehren der
  mathematischen Wissenschaften [Fundamental Principles of Mathematical
  Sciences]}.
\newblock Springer-Verlag, Berlin, second edition, 2008.

\bibitem[Pap]{pappasVolSymplAdicLocSys}
Georgios Pappas.
\newblock Volume and symplectic structure for $\ell$-adic local systems.
\newblock {\em Advances in Mathematics}, 387:107836, August 2021.

\bibitem[Pri1]{paper1}
J.~P. Pridham.
\newblock Deforming {$l$}-adic representations of the fundamental group of a
  smooth variety.
\newblock {\em J. Algebraic Geom.}, 15(3):415--442, 2006.

\bibitem[Pri2]{ddt1}
J.~P. Pridham.
\newblock Unifying derived deformation theories.
\newblock {\em Adv. Math.}, 224(3):772--826, 2010\noop{2007}.
\newblock corrigendum 228 (2011), no. 4, 2554--2556, arXiv:0705.0344v6
  [math.AG].

\bibitem[Pri3]{weiln}
J.~P. Pridham.
\newblock Galois actions on homotopy groups of algebraic varieties.
\newblock {\em Geom. Topol.}, 15(1):501--607, 2011\noop{2007}.
\newblock arXiv:0712.0928v4 [math.AG].

\bibitem[Pri4]{drep}
J.~P. Pridham.
\newblock Representability of derived stacks.
\newblock {\em J. K-Theory}, 10(2):413--453, 2012\noop{2010}.
\newblock arXiv:1011.2742v4 [math.AG].

\bibitem[Pri5]{stacks2}
J.~P. Pridham.
\newblock Presenting higher stacks as simplicial schemes.
\newblock {\em Adv. Math.}, 238:184--245, 2013\noop{2009}.
\newblock arXiv:0905.4044v4 [math.AG].

\bibitem[Pri6]{dmc}
J.~P. Pridham.
\newblock Constructing derived moduli stacks.
\newblock {\em Geom. Topol.}, 17(3):1417--1495, 2013\noop{2011}.
\newblock arXiv:1101.3300v2 [math.AG].

\bibitem[Pri7]{poisson}
J.~P. Pridham.
\newblock Shifted {P}oisson and symplectic structures on derived {$N$}-stacks.
\newblock {\em J. Topol.}, 10(1):178--210, 2017\noop{2015}.
\newblock arXiv:1504.01940v5 [math.AG].

\bibitem[Pri8]{DQ-2}
J.~P. Pridham.
\newblock Deformation quantisation for {$(-2)$}-shifted symplectic structures.
\newblock arXiv: 1809.11028v2 [math.AG], 2018.

\bibitem[Pri9]{DQnonneg}
J.~P. Pridham.
\newblock Deformation quantisation for unshifted symplectic structures on
  derived {A}rtin stacks.
\newblock {\em Selecta Math. (N.S.)}, 24(4):3027--3059, 2018\noop{2016}.
\newblock arXiv: 1604.04458v4 [math.AG].

\bibitem[Pri10]{DQvanish}
J.~P. Pridham.
\newblock Deformation quantisation for {$(-1)$}-shifted symplectic structures
  and vanishing cycles.
\newblock {\em Algebr. Geom.}, 6(6):747--779, 2019\noop{2015}.
\newblock arXiv:1508.07936v5 [math.AG].

\bibitem[Pri11]{NCpoisson}
J.~P. Pridham.
\newblock Shifted bisymplectic and double {P}oisson structures on
  non-commutative derived prestacks.
\newblock arXiv:2008.11698v3 [math.AG], 2020.

\bibitem[Pri12]{DStein}
J.~P. Pridham.
\newblock A differential graded model for derived analytic geometry.
\newblock {\em Adv. Math.}, 360:106922, 2020\noop{2018}.
\newblock arXiv: 1805.08538v1 [math.AG].

\bibitem[Pri13]{DQexact}
J.~P. Pridham.
\newblock Deformation quantisation of exact shifted symplectic structures, with
  an application to vanishing cycles.
\newblock arXiv:2511.07602, 2025.

\bibitem[Pri14]{DQpoisson}
J.~P. Pridham.
\newblock Quantisation of derived {P}oisson structures.
\newblock {\em Geom. Topol.}, 29(7):3717--3771, 2025\noop{2017}.
\newblock arXiv: 1708.00496v6 [math.AG].

\bibitem[Pri15]{DQDG}
J.~P. Pridham.
\newblock An outline of shifted {P}oisson structures and deformation
  quantisation in derived differential geometry.
\newblock {\em Diff. Geom. Appl.}, to appear\noop{2018}.
\newblock arXiv: 1804.07622v5 [math.DG].

\bibitem[Pri16]{DQLag}
Jonathan~P Pridham.
\newblock Quantisation of derived {L}agrangians.
\newblock {\em Geom. Topol.}, 26(6):2405–2489, December 2022\noop{2016}.
\newblock arXiv: 1607.01000v4 [math.AG].

\bibitem[Pri17]{NCstacks}
J.P. Pridham.
\newblock Non-commutative derived moduli prestacks.
\newblock {\em Adv. Math.}, 433:109300, 2023\noop{2020}.
\newblock arXiv:2008.11684 [math.AG].

\bibitem[Pri18]{smallet2}
J.P. Pridham.
\newblock A note on \'etale atlases for {A}rtin stacks and {L}ie groupoids,
  {P}oisson structures and quantisation.
\newblock {\em Journal of Geometry and Physics}, 203:105266, 2024\noop{2019}.
\newblock arXiv:1905.09255.

\bibitem[PT]{PantevToenLocSys}
Tony Pantev and Bertrand To{\"e}n.
\newblock Poisson geometry of the moduli of local systems on smooth varieties.
\newblock arXiv:1809.03536, 2018.

\bibitem[PTVV]{PTVV}
T.~Pantev, B.~To{\"e}n, M.~Vaqui{\'e}, and G.~Vezzosi.
\newblock Shifted symplectic structures.
\newblock {\em Publ. Math. Inst. Hautes \'Etudes Sci.}, 117:271--328,
  2013\noop{2011}.
\newblock arXiv: 1111.3209v4 [math.AG].

\bibitem[PY1]{PortaYuRep}
Mauro Porta and Tony~Yu Yue.
\newblock Representability theorem in derived analytic geometry.
\newblock arXiv:1704.01683v2 [math.AG], 2017.

\bibitem[PY2]{PortaYuNonArch}
Mauro Porta and Tony~Yu Yue.
\newblock Derived non-{A}rchimedean analytic spaces.
\newblock {\em Selecta Math. (N.S.)}, 24:609--665, 2018.
\newblock arXiv:1601.008592v2 [math.AG].

\bibitem[Qui]{Q}
Daniel Quillen.
\newblock On the (co-) homology of commutative rings.
\newblock In {\em Applications of Categorical Algebra (Proc. Sympos. Pure
  Math., Vol. XVII, New York, 1968)}, pages 65--87. Amer. Math. Soc.,
  Providence, R.I., 1970.

\bibitem[Saf]{safronovPoissonLectures}
Pavel Safronov.
\newblock {Lectures on shifted Poisson geometry}.
\newblock arXiv:1709.07698v1[math.AG], 2017.

\bibitem[Sch]{Sch}
Michael Schlessinger.
\newblock Functors of {A}rtin rings.
\newblock {\em Trans. Amer. Math. Soc.}, 130:208--222, 1968.

\bibitem[Ser]{galoisienne}
Jean-Pierre Serre.
\newblock {\em Cohomologie galoisienne}, volume~5 of {\em Lecture Notes in
  Mathematics}.
\newblock Springer-Verlag, Berlin, fifth edition, 1994.

\bibitem[Sim1]{Simfil}
Carlos Simpson.
\newblock The {H}odge filtration on nonabelian cohomology.
\newblock In {\em Algebraic geometry---Santa Cruz 1995}, volume~62 of {\em
  Proc. Sympos. Pure Math.}, pages 217--281. Amer. Math. Soc., Providence, RI,
  1997.
\newblock arXiv:alg-geom/9604005v1.

\bibitem[Sim2]{Simpson}
Carlos~T. Simpson.
\newblock Higgs bundles and local systems.
\newblock {\em Inst. Hautes {\'E}tudes Sci. Publ. Math.}, (75):5--95, 1992.

\bibitem[Sou]{souleKthCorpsNomCohoet}
C.~Soul\'e.
\newblock K-th\'eorie des anneaux d’entiers de corps de nombres et
  cohomologie \'etale.
\newblock {\em Inventiones Mathematicae}, 55(3):251–295, October 1979.

\bibitem[TV1]{hag1}
Bertrand To{\"e}n and Gabriele Vezzosi.
\newblock Homotopical algebraic geometry. {I}. {T}opos theory.
\newblock {\em Adv. Math.}, 193(2):257--372, 2005\noop{2002}.
\newblock arXiv:math/0207028 v4.

\bibitem[TV2]{hag2}
Bertrand To{\"e}n and Gabriele Vezzosi.
\newblock Homotopical algebraic geometry. {II}. {G}eometric stacks and
  applications.
\newblock {\em Mem. Amer. Math. Soc.}, 193(902):x+224, 2008\noop{2004}.
\newblock arXiv math.AG/0404373 v7.

\bibitem[Zhu]{zhukovHigherDimLocalFields}
Igor Zhukov.
\newblock Higher dimensional local fields.
\newblock In {\em Invitation to higher local fields ({M}\"{u}nster, 1999)},
  volume~3 of {\em Geom. Topol. Monogr.}, pages 5--18. Geom. Topol. Publ.,
  Coventry, 2000.
\newblock arXiv:math/0012132v1 [math.NT].

\end{thebibliography}

\end{document}